\documentclass[12pt]{article}
\usepackage[english]{babel}
\usepackage{amsmath,latexsym,amsthm,amsfonts}
\usepackage{amssymb,amscd}
\usepackage{amscd}
\usepackage{caption}
\usepackage{latexsym}

\usepackage[usenames]{color}
\usepackage{colortbl}

\usepackage{graphicx}

\usepackage{longtable}
\usepackage{amssymb}
\usepackage{multicol}
\usepackage{ifthen}
\usepackage{wrapfig}
\usepackage{hyperref}
\usepackage[all]{xy}

\numberwithin{equation}{section}

\newcommand{\hm}[1]{#1\nobreak\discretionary{}{\hbox{\ensuremath{#1}}}{}}

\renewcommand{\leq}{\leqslant}
\renewcommand{\geq}{\geqslant}

\newcommand{\systema}[2]{
$$\left\{ \begin{array}{ll}#1
\\
#2 \end{array}\right.$$ }

\newtheorem{lemma}{Lemma}

\newtheorem{theorem}{Theorem}[section]
\newtheorem{definition}{Definition}[section]
\newtheorem{proposition}[theorem]{Proposition}
\newtheorem{corollary}{Corollary}[section]
\newtheorem{notation}{Notation}[section]
\newtheorem{note}{Remark}[section]

\newtheorem{example}{Example}[section]

\newtheorem*{definition*}{Definition}
\newtheorem*{lemma*}{Lemma}
\newtheorem*{notation*}{Notation}
\newtheorem*{theorem*}{Theorem}
\newtheorem*{proposition*}{Proposition}
\newtheorem*{note*}{Remark}
\newtheorem*{corollary*}{Corollary}

\newcommand\XLast{}
\newcommand\YLast{}

\newcommand\Vidr[4]{\qbezier(#1,#2)(#1,#2)(#3,#4) \renewcommand\XLast{#3}
\renewcommand\YLast{#4}}
\newcommand\VidrTo[2]{\Vidr{\XLast}{\YLast}{#1}{#2}}

\newcommand\VidrToBold[2]{\linethickness{0.3mm}
\VidrTo{#1}{#2} \linethickness{0.1mm}}

\begin{document}

\renewcommand{\figurename}{Fig.}

\begin{center}
\large{Makar Plakhotnyk}

Postdoctoral researcher at S\~ao Paulo university, Brazil

Mail:\, makar.plakhotnyk@gmail.com

\LARGE{\textbf{Topological conjugation of one dimensional maps}}

\large{Draft of the book}

\end{center}

\vskip 2cm

\begin{abstract}
Topological conjugateness of one dimensional unimodal dynamical
systems, which are generated by interval $[0,\, 1]$ into itself
maps are studied. We study the smoothness and differentiability of
the conjugacy of symmetrical and non-symmetrical tent maps. Also
we prove the extremal property of the length of the graph of the
conjugacy of symmetrical and non-symmetrical tent maps.

\end{abstract}

\newpage

\section{Preface}

During the 10 years, from 2004 till 2013, Volodymyr Fedorenko was
a lector of the course of Dynamical systems at Mechanical and
mathematical faculty of National Taras Shevchenko University of
Kyiv.

All the time, when he considered the topic of the topological
conjugation of one dimensional maps, he considered the continuous
unimodal $[0,\, 1]\rightarrow [0,\, 1]$ maps

$$ f(x) = \left\{\begin{array}{ll}
2x,& x< 1/2;\\
2-2x,& x\geqslant 1/2,
\end{array}\right.$$and $g(x) = 4x(1-x),$
which are conjugated via $ h(x) = \sin^2\left(\frac{\pi
x}{2}\right). $

\begin{figure}[htbp]
\begin{minipage}[h]{0.3\linewidth}
\begin{center}
\begin{picture}(100,125)
\put(100,0){\line(0,1){100}} \put(0,100){\line(1,0){100}}
\put(0,0){\vector(1,0){120}} \put(0,0){\vector(0,1){120}}
\put(0,0){\line(1,2){50}} \put(50,100){\line(1,-2){50}}
\put(40,50){f(x)}
\end{picture}\end{center}
\centerline{a)}
\end{minipage}
\hfill
\begin{minipage}[h]{0.3\linewidth}
\begin{center}
\begin{picture}(100,125)
\put(100,0){\line(0,1){100}} \put(0,100){\line(1,0){100}}
\put(0,0){\vector(1,0){120}} \put(0,0){\vector(0,1){120}}
\qbezier(0,0)(50,200)(100,0) \put(40,50){$g$(x)}
\end{picture}\end{center}
\centerline{b)}
\end{minipage}
\hfill
\begin{minipage}[h]{0.3\linewidth}
\begin{center}
\begin{picture}(100,125)
\put(100,0){\line(0,1){100}} \put(0,100){\line(1,0){100}}
\put(0,0){\vector(1,0){120}} \put(0,0){\vector(0,1){120}}
\put(20,50){$h$(x)} \put(0,0){\circle*{2}} \put(1,0){\circle*{2}}
\put(2,0){\circle*{2}} \put(3,0){\circle*{2}}
\put(4,0){\circle*{2}} \put(5,1){\circle*{2}}
\put(6,1){\circle*{2}} \put(7,1){\circle*{2}}
\put(8,2){\circle*{2}} \put(9,2){\circle*{2}}
\put(10,2){\circle*{2}} \put(11,3){\circle*{2}}
\put(12,4){\circle*{2}} \put(13,4){\circle*{2}}
\put(14,5){\circle*{2}} \put(15,5){\circle*{2}}
\put(16,6){\circle*{2}} \put(17,7){\circle*{2}}
\put(18,8){\circle*{2}} \put(19,9){\circle*{2}}
\put(20,10){\circle*{2}} \put(21,10){\circle*{2}}
\put(22,11){\circle*{2}} \put(23,12){\circle*{2}}
\put(24,14){\circle*{2}} \put(25,15){\circle*{2}}
\put(26,16){\circle*{2}} \put(27,17){\circle*{2}}
\put(28,18){\circle*{2}} \put(29,19){\circle*{2}}
\put(30,21){\circle*{2}} \put(31,22){\circle*{2}}
\put(32,23){\circle*{2}} \put(33,25){\circle*{2}}
\put(34,26){\circle*{2}} \put(35,27){\circle*{2}}
\put(36,29){\circle*{2}} \put(37,30){\circle*{2}}
\put(38,32){\circle*{2}} \put(39,33){\circle*{2}}
\put(40,35){\circle*{2}} \put(41,36){\circle*{2}}
\put(42,38){\circle*{2}} \put(43,39){\circle*{2}}
\put(44,41){\circle*{2}} \put(45,42){\circle*{2}}
\put(46,44){\circle*{2}} \put(47,45){\circle*{2}}
\put(48,47){\circle*{2}} \put(49,48){\circle*{2}}
\put(50,50){\circle*{2}} \put(51,52){\circle*{2}}
\put(52,53){\circle*{2}} \put(53,55){\circle*{2}}
\put(54,56){\circle*{2}} \put(55,58){\circle*{2}}
\put(56,59){\circle*{2}} \put(57,61){\circle*{2}}
\put(58,62){\circle*{2}} \put(59,64){\circle*{2}}
\put(60,65){\circle*{2}} \put(61,67){\circle*{2}}
\put(62,68){\circle*{2}} \put(63,70){\circle*{2}}
\put(64,71){\circle*{2}} \put(65,73){\circle*{2}}
\put(66,74){\circle*{2}} \put(67,75){\circle*{2}}
\put(68,77){\circle*{2}} \put(69,78){\circle*{2}}
\put(70,79){\circle*{2}} \put(71,81){\circle*{2}}
\put(72,82){\circle*{2}} \put(73,83){\circle*{2}}
\put(74,84){\circle*{2}} \put(75,85){\circle*{2}}
\put(76,86){\circle*{2}} \put(77,88){\circle*{2}}
\put(78,89){\circle*{2}} \put(79,90){\circle*{2}}
\put(80,90){\circle*{2}} \put(81,91){\circle*{2}}
\put(82,92){\circle*{2}} \put(83,93){\circle*{2}}
\put(84,94){\circle*{2}} \put(85,95){\circle*{2}}
\put(86,95){\circle*{2}} \put(87,96){\circle*{2}}
\put(88,96){\circle*{2}} \put(89,97){\circle*{2}}
\put(90,98){\circle*{2}} \put(91,98){\circle*{2}}
\put(92,98){\circle*{2}} \put(93,99){\circle*{2}}
\put(94,99){\circle*{2}} \put(95,99){\circle*{2}}
\put(96,100){\circle*{2}} \put(97,100){\circle*{2}}
\put(98,100){\circle*{2}} \put(99,100){\circle*{2}}
\put(100,100){\circle*{2}}
\end{picture} \end{center}
\centerline{c)}
\end{minipage}
\hfill \caption{Graphics} \label{fig:1}
\end{figure}
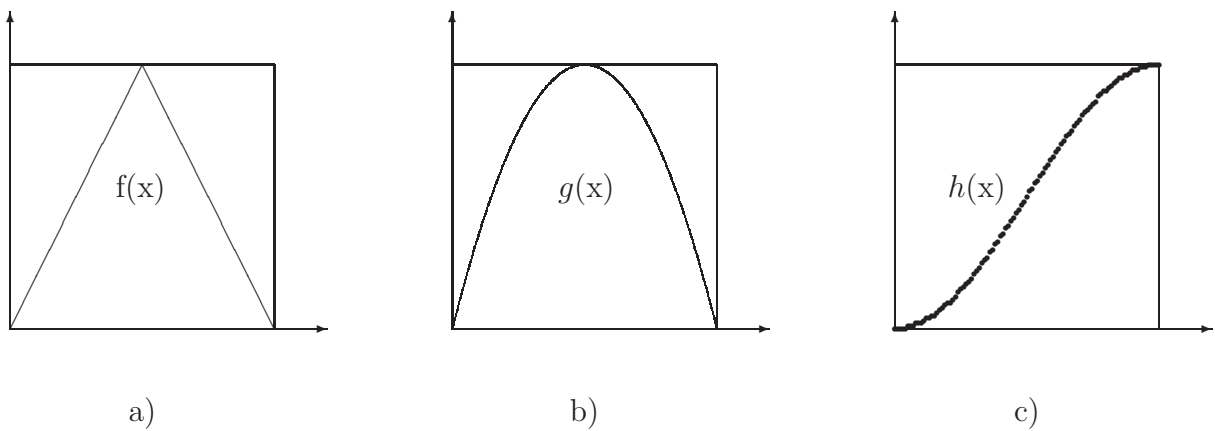

Figure~\ref{fig:1}a) contains the graph of $f$,
Figure~\ref{fig:1}b) contains the graph of $g$, the graph the
conjugacy $h$, which defines the conjugacy of $f$ and $g$, i.e. is
the solution of the functional equation $h(f)=g(h)$, is given at
Figure~\ref{fig:1}c).

During of one of the lectures, V. Fedorenko mentioned, that the
same reasonings, which were used for finding the conjugacy of $f$
and $g$, can be applied for finding the conjugacy $h_v:\, [0,\,
1]\rightarrow [0,\, 1]$ of $f$, and $f_v$, where
$$
f_v(x) = \left\{\begin{array}{ll} \frac{x}{v},& x\leqslant v;\\
 \frac{1-x}{1-v},&
x>v,
\end{array}\right.
$$ and $v\in (0,\, 1)$ is a parameter.
The finding of $h_v$ needs the solution of the system of linear
functional equations. Methods of solving of linear functional
equations are described, for instance, in~\cite{Pelukh-1974}.

This book was inspirit by multiple attempts to find the explicit
formulas of the conjugacy of $f$ and $f_v$. In spite of we was
failed with these attempts, we have obtained a list of results,
concerning with properties of this conjugacy and the semi
conjugacy of $f$ and $f_v$. These results are formulated in
Section~\ref{Sect-Main-Results}.

The conjugation of $f$ and $g$ appeared in the first time in the
middle of the 20-th century in the von Neumann's and S. Ulam's
collaborated works, which deal with generators of random numbers.
In the same time, the map $g$, which is know as a logistic map,
was discovered for the applications of mathematics by P. Verhulst
in the first half of 19-th century in the populational dynamics.
Also the one parametrical family of logistic maps was used in
1970-th by M. Feigenbaum in the discovery of universal constants,
which are known by his name.

The hardness of studying of the properties of the conjugacy $h_v$
of $f$ and $f_v$, follows, for instance, from the following its
property (see~\cite{Skufca}): the derivative $h_v'$ exists and
equals either $0$, or $\infty$. Moreover, $h_v'$ equals $0$ almost
everywhere in the cense of the Lebesgue measure of $[0,\, 1]$. The
last property yields that $h_v$ is non-differentiable on any
subinterval of $[0,\, 1]$.

The graph of $h_v$ for $v=3/4$ is given at Figure~\ref{fig:3}.

\begin{figure}[htbp]
\begin{center}
\begin{picture}(130,140)
\put(128,0){\line(0,1){128}} \put(0,128){\line(1,0){128}}
\put(0,0){\vector(1,0){140}} \put(0,0){\vector(0,1){140}}
\linethickness{0.3mm} \qbezier(0,0)(0,0)(1,17)
\qbezier(1,17)(1,17)(2,23) \qbezier(2,23)(2,23)(3,25)
\qbezier(3,25)(3,25)(4,30) \qbezier(4,30)(4,30)(5,32)
\qbezier(5,32)(5,32)(6,33) \qbezier(6,33)(6,33)(7,35)
\qbezier(7,35)(7,35)(8,41) \qbezier(8,41)(8,41)(9,42)
\qbezier(9,42)(9,42)(10,43) \qbezier(10,43)(10,43)(11,43)
\qbezier(11,43)(11,43)(12,44) \qbezier(12,44)(12,44)(13,46)
\qbezier(13,46)(13,46)(14,46) \qbezier(14,46)(14,46)(15,48)
\qbezier(15,48)(15,48)(16,54) \qbezier(16,54)(16,54)(17,56)
\qbezier(17,56)(17,56)(18,57) \qbezier(18,57)(18,57)(19,57)
\qbezier(19,57)(19,57)(20,57) \qbezier(20,57)(20,57)(21,58)
\qbezier(21,58)(21,58)(22,58) \qbezier(22,58)(22,58)(23,58)
\qbezier(23,58)(23,58)(24,59) \qbezier(24,59)(24,59)(25,60)
\qbezier(25,60)(25,60)(26,61) \qbezier(26,61)(26,61)(27,61)
\qbezier(27,61)(27,61)(28,62) \qbezier(28,62)(28,62)(29,64)
\qbezier(29,64)(29,64)(30,64) \qbezier(30,64)(30,64)(31,66)
\qbezier(31,66)(31,66)(32,72) \qbezier(32,72)(32,72)(33,74)
\qbezier(33,74)(33,74)(34,75) \qbezier(34,75)(34,75)(35,75)
\qbezier(35,75)(35,75)(36,75) \qbezier(36,75)(36,75)(37,76)
\qbezier(37,76)(37,76)(38,76) \qbezier(38,76)(38,76)(39,76)
\qbezier(39,76)(39,76)(40,77) \qbezier(40,77)(40,77)(41,77)
\qbezier(41,77)(41,77)(42,77) \qbezier(42,77)(42,77)(43,77)
\qbezier(43,77)(43,77)(44,77) \qbezier(44,77)(44,77)(45,77)
\qbezier(45,77)(45,77)(46,77) \qbezier(46,77)(46,77)(47,77)
\qbezier(47,77)(47,77)(48,78) \qbezier(48,78)(48,78)(49,80)
\qbezier(49,80)(49,80)(50,81) \qbezier(50,81)(50,81)(51,81)
\qbezier(51,81)(51,81)(52,81) \qbezier(52,81)(52,81)(53,82)
\qbezier(53,82)(53,82)(54,82) \qbezier(54,82)(54,82)(55,82)
\qbezier(55,82)(55,82)(56,83) \qbezier(56,83)(56,83)(57,84)
\qbezier(57,84)(57,84)(58,85) \qbezier(58,85)(58,85)(59,85)
\qbezier(59,85)(59,85)(60,86) \qbezier(60,86)(60,86)(61,88)
\qbezier(61,88)(61,88)(62,88) \qbezier(62,88)(62,88)(63,90)
\qbezier(63,90)(63,90)(64,96) \qbezier(64,96)(64,96)(65,98)
\qbezier(65,98)(65,98)(66,99) \qbezier(66,99)(66,99)(67,99)
\qbezier(67,99)(67,99)(68,99) \qbezier(68,99)(68,99)(69,100)
\qbezier(69,100)(69,100)(70,100) \qbezier(70,100)(70,100)(71,100)
\qbezier(71,100)(71,100)(72,101) \qbezier(72,101)(72,101)(73,101)
\qbezier(73,101)(73,101)(74,101) \qbezier(74,101)(74,101)(75,101)
\qbezier(75,101)(75,101)(76,101) \qbezier(76,101)(76,101)(77,101)
\qbezier(77,101)(77,101)(78,101) \qbezier(78,101)(78,101)(79,101)
\qbezier(79,101)(79,101)(80,102) \qbezier(80,102)(80,102)(81,102)
\qbezier(81,102)(81,102)(82,102) \qbezier(82,102)(82,102)(83,102)
\qbezier(83,102)(83,102)(84,102) \qbezier(84,102)(84,102)(85,102)
\qbezier(85,102)(85,102)(86,102) \qbezier(86,102)(86,102)(87,102)
\qbezier(87,102)(87,102)(88,103) \qbezier(88,103)(88,103)(89,103)
\qbezier(89,103)(89,103)(90,103) \qbezier(90,103)(90,103)(91,103)
\qbezier(91,103)(91,103)(92,103) \qbezier(92,103)(92,103)(93,103)
\qbezier(93,103)(93,103)(94,103) \qbezier(94,103)(94,103)(95,103)
\qbezier(95,103)(95,103)(96,104) \qbezier(96,104)(96,104)(97,106)
\qbezier(97,106)(97,106)(98,107) \qbezier(98,107)(98,107)(99,107)
\qbezier(99,107)(99,107)(100,107)
\qbezier(100,107)(100,107)(101,108)
\qbezier(101,108)(101,108)(102,108)
\qbezier(102,108)(102,108)(103,108)
\qbezier(103,108)(103,108)(104,109)
\qbezier(104,109)(104,109)(105,109)
\qbezier(105,109)(105,109)(106,109)
\qbezier(106,109)(106,109)(107,109)
\qbezier(107,109)(107,109)(108,109)
\qbezier(108,109)(108,109)(109,109)
\qbezier(109,109)(109,109)(110,109)
\qbezier(110,109)(110,109)(111,109)
\qbezier(111,109)(111,109)(112,110)
\qbezier(112,110)(112,110)(113,112)
\qbezier(113,112)(113,112)(114,113)
\qbezier(114,113)(114,113)(115,113)
\qbezier(115,113)(115,113)(116,113)
\qbezier(116,113)(116,113)(117,114)
\qbezier(117,114)(117,114)(118,114)
\qbezier(118,114)(118,114)(119,114)
\qbezier(119,114)(119,114)(120,115)
\qbezier(120,115)(120,115)(121,116)
\qbezier(121,116)(121,116)(122,117)
\qbezier(122,117)(122,117)(123,117)
\qbezier(123,117)(123,117)(124,118)
\qbezier(124,118)(124,118)(125,120)
\qbezier(125,120)(125,120)(126,120)
\qbezier(126,120)(126,120)(127,122)
\qbezier(127,122)(127,122)(128,128) \linethickness{0.1mm}
\end{picture}
\end{center}
\caption{Graph of $h_{3/4}$} \label{fig:3}
\end{figure}
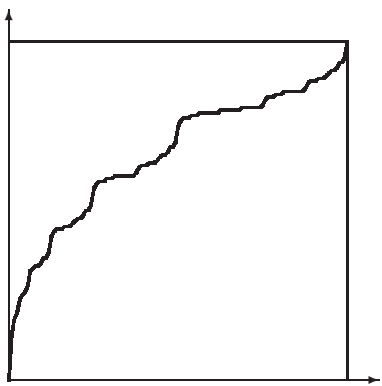

\newpage \section{The main results}\label{Sect-Main-Results}

In Section~\ref{sect:Vstup} we define the main notions such as
dynamical system, fixed and periodical point, trajectory and
orbit, topological conjugation and topological semiconjugation of
maps. We give examples of topologically conjugated $[0,\,
1]\rightarrow [0,\, 1]$ mappings and this leads us to the
following proposition.
\begin{proposition*}[Proposition~\ref{theor:ksklin-tpleqv},
also Lemma~1 in~\cite{UMZh}] Let $f,\, g:\, [0,\, 1]\rightarrow
[0,\, 1]$ be piecewise linear maps, which are topologically
conjugated via increasing piecewise linear homeomorphism $h$. If
$f(0)=0$, then $g(0)=0$ and $f'(0)=g'(0)$.
\end{proposition*}

In Section~\ref{Sect:IstorOgl} we give a historical review of
works, which deal with the study of iterations of one-dimensional
maps, studying of hat-maps, logistic map, E. L\'amerey's diagrams
and topological conjugation of one-dimensional maps
(one-dimensional dynamical systems).

Section~\ref{sect-Grupy} is devoted to the study of topological
conjugation of continuous $[0,\, 1]\rightarrow [0,\, 1]$ maps,
whose semigroup of iterations is a cyclic group call $C_n$. A map
$f$ with such semigroup of iterations satisfy the functional
equation \begin{equation} \label{eq-52} f^n = f,\end{equation} and
for $m,\, 2\leq m<n$ the equality $f^m=f$ does not hold. In
Section~\ref{sect-Grupy-1} we prove the following theorem:
\begin{theorem*}[Theorem~\ref{theor:7}, also Theorem~1
in~\cite{Yulia}] If the continuous maps $f:\, [0,\, 1]\rightarrow
[0,\, 1]$ satisfy~(\ref{eq-52}), then it also satisfy the equation
$$f^3 = f.$$
\end{theorem*}
This theorem is a generalization of the following Theorem.

\begin{theorem*}[Theorem~\ref{theor:Nathan},
also Theorem~2 from~\cite{Natanson}] If $g:\, [0,\, 1]\rightarrow
[0,\, 1]$ is a continuous function such that $g^p(x) =x$ for all
$x$, then $g^2(x)=x$ for all $x$. In particular, if $p$ is odd,
then $g(x)=x$ for all~$x$.
\end{theorem*}

We describe in Section~\ref{sect-Grupy-2} the graph of the
continuous $f:\, [0,\, 1]\rightarrow [0,\, 1]$, whose semigroup of
iterations is a finite group.

\begin{theorem*}[Theorem~\ref{theor:6}, also theorem~2
in~\cite{Yulia}] For a continuous maps $f:\, [0,\, 1]\rightarrow
[0,\, 1]$ the following statements are equivalent:

1) $f^{n}(x) =f(x)$ for all $x\in [0,\, 1]$ and every $n\in
\mathbb{N}$;

2) there exist numbers $a,\, b,\, 0\leq a\leq b\leq 1$ and the
continuous maps $g:\, [0,\, a]\rightarrow [a,\, b]$ and $h:\,
[b,\, 1] \rightarrow [a,\, b]$ such that $f$ can be represented as
$$
\begin{array}{cc}
f(x)=\left \{ \begin{array}{ll}
g(x), &0\leq x\leq a,\\
x,  &a\leq x \leq b,\\
h(x), &b\leq x\leq 1.
\end{array}\right.
\end{array}$$
\end{theorem*}

\begin{figure}[htbp]
\begin{minipage}[h]{\linewidth}
\begin{center}
\begin{picture}(75,75) \qbezier(0,0)(0,35)(0,70)
\qbezier(0,0)(35,0)(70,0) \qbezier(0,0)(35,35)(70,70)
\qbezier(0,70)(35,70)(70,70) \qbezier(70,0)(70,35)(70,70)
\qbezier[20](0,50)(35,50)(70,50)\qbezier[20](0,30)(35,30)(70,30)
\qbezier(70,0)(70,35)(70,70) \linethickness{0.7mm}
\qbezier(30,30)(40,40)(50,50) \linethickness{0.1mm}
\qbezier(30,30)(20,35)(15,45) \qbezier(0,33)(10,50)(15,45)
\qbezier(51,51)(53,45)(60,35) \qbezier(60,35)(62,40)(65,45)
\qbezier(70,35)(67,40)(65,45) \put(3,27){$g(x)$}
\put(50,25){$h(x)$}
\end{picture}
\end{center}
\end{minipage}
\caption{Graph of $f$}\label{fig:24}
\end{figure}
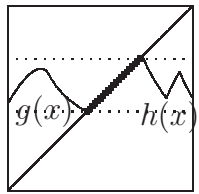

The graph of $f$ from Theorem~\ref{theor:6} is given at
Figure~\ref{fig:24}.

\begin{theorem*}[Theorem~\ref{theor:5},
also Theorem~3 in~\cite{Yulia} and Theorem~10 in~\cite{OdynAwt}]

For a continuous maps $f:\, [0,\, 1]\rightarrow [0,\, 1]$ the
following statements are equivalent:

1) $f^{3}(x) =f(x)$ for all $x\in I$;

2) there exist numbers $a,\, b$ with $0\leq a\leq b\leq 1$ and
maps continuous $g:\, [0,\, a]\rightarrow [a,\, b]$, $h:\, [b,\,
1] \rightarrow [a\, ,b]$ and $\varphi:\, [a,\, b]\rightarrow [a,\,
b]$ with the following properties. the graph of $\varphi$ is
symmetrical in the line $y=x$ and $\varphi([a,b])=[a,b]$. The maps
$f$ can be represented as follows:
$$
\begin{array}{cc}
 f(x)=\left \{ \begin{array}{ll}
    g(x),& 0\leq x\leq a,\\
   \varphi(x),&  a\leq x \leq b,\\
  h(x), & b\leq x\leq 1.
          \end{array}\right.
\end{array}
$$
\end{theorem*}

Section~\ref{sect-Grupy-3} is devoted to the description of
conjugated classes of maps, whose semigroup of iterations is a
finite groups and which have finitely many intervals of monotony.
Let $f,\, g:\, [0,\, 1]\rightarrow [0,\ 1]$ be continuous. Let
$h:\, [0,\, 1]\rightarrow [0,\ 1]$ be a homeomorphism such that
$g=h^{-1}(f(h))$, also $p_1,\, p_2,\, q_1$ and $q_2$ be such
numbers that $f([0,\, 1])=[p_1,\, q_1]$, and $g([0,\, 1])=[p_2,\,
q_2]$. Let $a_1,\ldots,\, a_n$ and $b_1,\ldots,\, b_m$ be
extremums of $f$ and $g$ correspondingly. Points $0$, $1$ and ends
of intervals of fixed points are also considered as extremums.
\begin{definition*} Vectors $(v_1,\ldots,\, v_k)$ and
$(w_1,\ldots,\, w_k)$ are caller co-ordered if for every $i,\, j$
the inequality $v_i\leq v_j$ is equivalent to $w_i\leq w_j$.
\end{definition*}

\begin{notation*}
Denote the following vectors:

\noindent $ v_f = (f(a_1),\, \ldots,\, f(a_n))$,

\noindent $\widetilde{v}_f = (f(a_1),\, f^2(a_1),\ldots,\,
f(a_n),\, f^2(a_n)) $,

\noindent $ w_g = (g(b_1),\, \ldots,\, g(b_m))$, and

\noindent  $\widetilde{w}_g = (g(b_1),\, g^2(b_1),\, \ldots,\,
g(b_m),\, g^2(b_m)). $
\end{notation*}
We prove the following two theorems.
\begin{theorem*}[Theorem~\ref{th:f2=f3}, also Lemma~13
in~\cite{ADM-2007}] Continuous maps $f,\, g:\, [0,\, 1]\rightarrow
[0,\, 1]$ are conjugated via increasing homeomorphism if only if
$m=n$ numbers of end points of $f([0,\, 1])$ and $g([0,\, 1])$
coincide and vectors $\widetilde{v}_f$ and $\widetilde{w}_g$ are
co-ordered.
\end{theorem*}
\begin{theorem*}[Theorem~\ref{theor:8}, see also Sect.~4
in~\cite{ADM-2007}] Maps $f$ and $g$ are conjugated via decreasing
homeomorphism if and only if $m=n$, numbers of end points of
$f([0,\, 1])$ and $g([0,\, 1])$ coincide and vectors
$\widetilde{v}_f$ and $w_g^*$ are co-ordered where
$$ w_g^*=(g(b_m),\, g^2(b_m),\, \ldots,\, g(b_1),\, g^2(b_1)).$$
\end{theorem*}

In Section~\ref{sect-Pobudowa} we consider the topological
conjugation of continuous maps $f,\, g:\, [0,\, 1]\rightarrow
[0,\, 1]$, given by
\begin{equation}
\label{eq:50} f(x) = \left\{\begin{array}{ll}
2x,& 0\leq x< 1/2;\\
2-2x,& 1/2 \leqslant x\leqslant 1,
\end{array}\right.
\end{equation}
and
\begin{equation}
\label{eq:51}
g(x) = \left\{\begin{array}{ll} g_l(x),& x\leqslant v;\\
g_r(x),& x>v,
\end{array}\right.
\end{equation}
where $v\in (0,\, 1)$ is fixed, $g_l(0)=g_r(1)=0$, $g_l(v)
=\lim\limits_{x\rightarrow v-}g_r(x) = 1$ and functions $g_l,\,
g_r$ are monotone and continuous. The problem about the
conjugation of $f$ and $g$ is stated at first at~\cite{Ulam-1964}
via the following theorem
\begin{theorem*}[Theorem~\ref{Theor:9}, also Appendix 1,
\S 3 in~\cite{Ulam-1964}] Let $f$ be a function of the
form~(\ref{eq:50}) and $g$ be a convex function of the
form~(\ref{eq:51}). Consider the integer trajectory of $1$ under
the action of $f$, i.e. the smallest set $M_f$ such that $1\in
M_f$ and $f(x)\in M_f$ is equivalent to $x\in M_f$. The necessary
and sufficient condition of $f$ and $g$ be conjugated is
combinatorial equivalence of $M_f$ and $M_g$ together with that
$\overline{M}_g = [0,\, 1]$.
\end{theorem*}

We prove Theorem~\ref{Theor:9} in details. Its proof is
constructive and is based on the following notions and reasonings.
\begin{definition*}[Definition~\ref{def:01}] Let $f$ be
of the form~(\ref{eq:50}), and the function $g$ be of the
form~(\ref{eq:51}). For every $x^*\in [0,\, 1]$ say the value
$y^*$ of the homeomorphism $h$ at $x^*$ to be
\textbf{conditionally found} if the following statement holds: if
$h$ is a topological conjugation of $f$ and $g$, then
$h(x^*)=y^*$.\end{definition*}

We use the word ``conditionally'' to notice that in the time when
we consider this conditionally found value of the conjugation, the
question about the existence of the conjugation is still opened.
\begin{proposition*}[Proposition~\ref{lema:25}] Let $f$
be of the form~(\ref{eq:50}) and $g$ be of the form~(\ref{eq:51}).

If the topological conjugation $h$  of $f$ and $g$ is
conditionally found at some point $x^*$ and equals $y^*$ then for
every $\widetilde{x}$ such that $f(\widetilde{x}) =x^*$, the value
at $\widetilde{x}$ is also found as follows:

1. If $\widetilde{x}\leq 1/2$ then $h(\widetilde{x}) =
g_l^{-1}(y^*)$;

2. If $\widetilde{x}>1/2$ then $h(\widetilde{x}) = g_r^{-1}(y^*)$.
\end{proposition*}

We use the following notations when prove Theorem~\ref{Theor:9}.
Denote by $A_n,\ n\geqslant 1$ the set of all points of $[0,\, 1]$
such that
$$ f^{n}(A_n) = 0.$$
Denote by $B_n,\ n\geqslant 1$ the set of all points of $[0,\, 1]$
such that
$$ g^{n}(B_n) = 0.$$
We obtain the following description of $A_n$.
\begin{proposition*}[Proposition~\ref{lema:An},
also Lemma~4 in~\cite{Fedorenko-2014}] $ \displaystyle{A_n =
\left\{0,\, \frac{1}{2^{n-1}};\ldots ;\frac{2^{n-1}-1}{2^{n-1}},\,
1\right\}.} $
\end{proposition*}
The density of $$ \mathcal{A} = \bigcup\limits_{n=1}^\infty A_n
$$ lets to reduce the proof of Theorem~\ref{Theor:9} to the proof of
of the density of $$\mathcal{B} = \bigcup\limits_{n=1}^\infty
B_n$$ in $[0,\, 1]$. We prove the following theorem
\begin{theorem*}[Theorem~\ref{theor:10}, also Lemma~1
in~\cite{Visnyk}] Let functions $f,\, g:\, [0,\, 1]\rightarrow
[0,\, 1]$ be defined by~(\ref{eq:50}) and~(\ref{eq:51}). If a
homeomorphism $h:\, [0,\, 1]\rightarrow [0,\, 1]$ satisfies the
functional equation $$h(f(x)) = g(h(x)),$$ then it increase and
$h(A_n) = B_n$.
\end{theorem*}
Section~\ref{sect-isnuv} is devoted to the study of topological
conjugation of the maps $f$, given by~(\ref{eq:50}) and the maps
$f_v:\, [0,\, 1]\rightarrow [0,\, 1]$, which is dependent on the
parameter $v\in (0,\, 1)$ and is given by formula
\begin{equation}\label{eq:53}f_v(x) =
\left\{\begin{array}{ll} \frac{x}{v},& x\leqslant
v;\\
 \frac{1-x}{1-v},&
x>v.
\end{array}\right.
\end{equation} In other words, we find the homeomorphism
$h:\, [0,\, 1]\rightarrow [0,\, 1],$ which satisfy the functional
equation
\begin{equation}\label{eq:58} h(f) = f_v(h).
\end{equation}
We prove the following theorem.
\begin{theorem*}[Theorem~\ref{theor:homeom-jed}, see
also Sect. 2.2 in~\cite{Fedorenko-2014}] For every $v\in (0,\, 1)$
the functional equation~(\ref{eq:58}) has a solution in the class
od homeomorphisms $h:\, [0,\, 1]\rightarrow [0,\, 1]$.
Furthermore, this solution is the unique and it increase.
\end{theorem*}

In Section~\ref{Pobudowa-2} we prove the following theorem.
\begin{theorem*}[Theorem~\ref{theor:11}] Let function
$f:\, [0,\, 1]\rightarrow [0,\, 1]$ be given by
formula~(\ref{eq:50}). For every $x_0\in [0,\, 1]$ and every
$\varepsilon>0$ there is a maps $g:\, [0,\, 1]\rightarrow [0,\,
1]$ with the following properties:

1. $g$ is unimodal;

2. $g(x)=f(x)$ for every $x\in [0,\, 1]\setminus
(x_0-\varepsilon,\, x_0+\varepsilon)$;

3. $f$ and $g$ are not topologically conjugated.
\end{theorem*}

Section~\ref{sect-dyffer} is devoted to the differentiability of
the homeomorphism $h:\, [0,\, 1]\rightarrow [0,\, 1]$, which
satisfies the functional equation~(\ref{eq:58}).

This study is motivated by Theorem~\ref{theor:13}.
\begin{theorem*}[Theorem~\ref{theor:13}, also
Proposition 2 at~\cite{Skufca}] The derivative of the
homeomorphism $h:\, [0,\, 1]\rightarrow [0,\, 1]$, which is the
solutions of the functional equation~(\ref{eq:58}), equals 0
almost everywhere and the unique its finite values is
0.\end{theorem*} Theorem~\ref{theor:13} is formulated
in~\cite[Proposition~2]{Skufca} a bit differently: the derivative
of $h:\, [0,\, 1]\rightarrow [0,\, 1]$, which is a solutions
of~(\ref{eq:58}), exists almost everywhere and equals 0
everywhere, where it \textbf{exists}. Nevertheless, it follows
from the proof in~\cite{Skufca}, that authors mean that existing
of the derivative is also its finiteness. Following~\cite[chap.
92, 101]{Fihtengoltz}, we will assume that function is
differentiable at a point if and only if the limit of ratios of
its grows over the grows of the argument exists. We will not
additionally assume that this limit is finite. In the same time,
Theorem~\ref{theor:13} mans only in that the derivative $h'(x)$
can not be finite except $0$. Another part of the Theorem follows
from the Lebesgue theorem about the derivative of monotone
function.
\begin{theorem*}[Theorem~\ref{theor:14}, also Lebesgue
Theorem (see~\cite{Lebeg}, or \S 1.2 in~\cite{Riss})] Every
monotone on the interval function has finite derivative almost
everywhere on this interval.
\end{theorem*}
Denote elements of $A_n$ by $\alpha_{n,k}$ and assume that
$\alpha_{n,k_1}<\alpha_{n,k_2}$ for all $k_1<k_2$. By
Proposition~\ref{lema:An}, $\alpha_{n,k}=\frac{1}{2^{n-1}}$. Also
denote elements of $B_n$ by $\beta_{n,k}$ and assume that
$\beta_{n,k_1}<\beta_{n,k_2}$ for all$k_1<k_2$. For every $n\geq
1$ denote by $h_n$ the piecewise linear maps such that all its
breaking points belong to $A_n$ and for every $k,\, 0\leq k\leq
2^{n-1}$ the equality $$ h_n(\alpha_{n,k})=\beta_{n,k}
$$ holds.
We find limits of derivatives $h_n'(x)$ for $x\in (0,\,
1)\backslash \mathcal{A}$ at Section~\ref{sect-dyffer-1}. Let
binary decomposition of $x$ be
\begin{equation}\label{eq:59}
x = 0,x_1x_2\ldots x_k\ldots\, .
\end{equation}
For a number $x$ of the form~(\ref{eq:59}) denote by $x_0 = 0$ and
for every $i\geq 2$ denote
\begin{equation}\label{eq:60}
\alpha_i(x) =\left\{
\begin{array}{ll} 2v& \text{if } x_{i-1}=x_{i-2},\\
2(1-v)& \text{if } x_{i-1}\neq x_{i-2}.
\end{array}\right.
\end{equation}
We prove the following theorems
\begin{theorem*}[Theorem~\ref{lema:32}, also Lema~7
in~\cite{Visnyk}] For every $n\geq 2$ and $x\not\in A_n$ of the
form~(\ref{eq:6}) he equality
$$h_n'(x)=\prod\limits_{i=2}^n\alpha_i(x)$$ holds, where
$\alpha_i(x)$ is defined by~(\ref{eq:54}).
\end{theorem*}
\begin{theorem*}[Theorem~\ref{theor:12}, also Lema 14
in~\cite{Visnyk}] 1. If $v<1/2$ then for every $x\in \mathcal{A}$
the limit $\lim\limits_{n\rightarrow \infty}h_n'(x) =0$ holds.

2. If $v>1/2$ then for every $x\in \mathcal{A}$ the limit
$\lim\limits_{n\rightarrow \infty}h_n'(x) =\infty$ holds.

3. For every $v\in (0,\, 1)\setminus \{1/2\}$ limits
$\lim\limits_{n\rightarrow \infty}\min\limits_{x\in (0,\,
1)\setminus \mathcal{A}}h_n'(x) =0$ and $\lim\limits_{n\rightarrow
\infty}\max\limits_{x\in (0,\, 1)\setminus \mathcal{A}}h_n'(x)
=\infty$ hold.
\end{theorem*}
The following observation follows from Theorems~\ref{theor:13}
and~\ref{theor:12}. Let $\lambda$ be the Lebesgue measure on the
interval $[0,\, 1]$. Denote by $\mathcal{A}^0$ the set, where the
derivative of $h$ equals to $0$ and denote by $\mathcal{B}$ the
set, where the derivative of $h$ equals to infinity. Then
$\lambda(\mathcal{A}^0) =1$, $\lambda(\mathcal{B}^0) =0$ and $h$
is non-differentiable on $[0,\, 1]\backslash (\mathcal{A}^0\cup
\mathcal{B}^0)$. It is evident, that the derivative of the inverse
function $h^{-1}$ equals $\infty$ on $h(\mathcal{A}^0)$ and this
derivative equals $0$ on $h(\mathcal{B}^0)$. Now, it follows from
Theorems~\ref{theor:13} and~\ref{theor:12} that
$\lambda(h(\mathcal{A}^0)) =0$ and $\lambda(h(\mathcal{B}^0)) =1$.
These properties of $h$ show how complicated it is.

Section~\ref{sect-dyffer-2} is devoted to values of the
homeomorphism $h:\, [0,\, 1]\rightarrow [0,\, 1]$, which
satisfies~(\ref{eq:58}). We give the proof of
Theorem~\ref{theor:13} as a corollary of Theorem~\ref{lema:32}.
Also we prove the following theorem.
\begin{theorem*}[Theorem~\ref{theor:15}, also
Lemmas~15-16 in~\cite{Visnyk} and Theorem~2 in~\cite{Studii}] Let
$x_0\in [0,\, 1]\cap \mathbb{Q}$. Then the derivative $h'(x_{0} )$
exists. More then this, if $v<1/2$ the $h'(x_0)=\infty $ and if
$v>1/2$ then $h'(x_{0} )=0$.
\end{theorem*}
Theorem~\ref{theor:15} was proved in~\cite{Visnyk} for $x_0\in A$
and later generalized at~\cite{Studii} for the case $x\in
\mathbb{Q}$. Theorem~\ref{theor:15} can be considered as
generalization of Theorem~\ref{theor:12} to rational points set.
It follows from Theorem~\ref{theor:13}, that
Theorem~\ref{theor:12} can not be generalized to all all $x_0\in
[0,\, 1]$ such that limit $\mathop{\lim }\limits_{n\to \infty }
h_{n}^{'}(x_0)$ exists. The following statement holds.
\begin{proposition*}[Proposition~\ref{Prop:4}] For every
$v\neq 1/2$ there exists $x_0\in [0,\, 1]$ such that $h'(x_0)=0$
and one of the following statements holds:

1. the limit $\lim\limits_{n\to \infty } h_{n}^{'} (x_{0} )$ does
not exists;

2. the limit $\lim\limits_{n\to \infty } h_{n}^{'} (x_{0} )$
exists, but equals infinity $\infty$.
\end{proposition*}

We construct at Section~\ref{subs-excell} the formula in terms of
electronic tables, which let to find values at $A_{n+1}$ of the
homeomorphic solution of~\ref{subs-excell}. We reformulate
Proposition~\ref{lema:25} as follows.
\begin{proposition*}[Proposition~\ref{prop:1}] Let $h:\,
[0,\, 1]\rightarrow [0,\, 1]$ be the conjugation of maps $f$ and
$f_v$, which are defined by~(\ref{eq:50}) and~(\ref{eq:53}) i.e.
$h$ is a solution of the functional equation~(\ref{eq:58}). Then
the following implications hold.

1. If $x=0$, then $h(x) =0$.

2. If $x=1$, then $h(x)=1$.

3.If $x\leq \frac{1}{2}$, then $$h(x) = v\cdot h(2x)$$ and the
value $h(2x)$ appears to be defined earlier.

4. If $x>\frac{1}{2}$, then $$h(x) = 1 - (1-v)\cdot h(-2x+2)$$ and
the value $h(-2x+2)$ appears to be defined earlier.
\end{proposition*}
Proposition~\ref{prop:1} let us to prove the following theorem.
\begin{theorem*}[Theorem~\ref{theor:16}, also Sect.~3.1
in~\cite{ADM-2016}] The value of conjugation $h$ of maps $f$ and
$f_v$, which are defined by~(\ref{eq:50}) and~(\ref{eq:53}) at the
set $A_{n+1}$ can be found via the following way:

Put $n$ into ``C1''.

Put $v$ into ``D1''.

Put 0 into ``A1'' and the formula A1+1/(2$\widehat{\,}$C\$1) into
``A2''.

Put the following formula into ``B1''.

\noindent IF(A1=0;\ 0;\ IF(A1=1,\, 1; IF(A1<=0,5;\

D\$1*INDIRECT( CONCATENATE("B";\ 2*A1*2$\widehat{\,}$ C\$1+1));

1-(1-D\$1)*INDIRECT(

CONCATENATE("B";\ -2*A1*2$\widehat{\,}\,$C\$1 +1 +2
$\widehat{\,}\,$(1+C\$1)))))).

Copy the values in columns $A$ and $B$ down till the line number
$2^{n}+1$.\end{theorem*} The Figure~\ref{fig:16} is prepared with
the use of the formula from Theorem~\ref{theor:16}. It contains
the values of the conjugation in $A_7$, which correspond to $v=
0.55,\, 0.6,\, 0.7,\, 0.75,\, 0.8,\, 0.85,\, 0.9$ and $0.95$.

\begin{figure}[htbp]
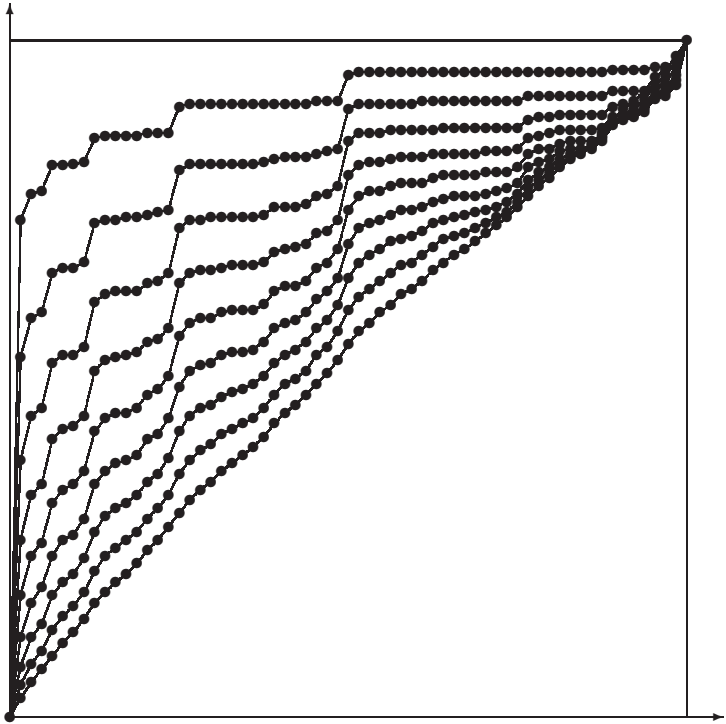

\begin{minipage}[h]{0.9\linewidth}
\begin{center}
\end{center}
\end{minipage}
\caption{Approximation $h_7$ for different values of
$v$}\label{fig:16}
\end{figure}

Smaller values of $v$ correspond to ``graphs'', which are closer
to $y=x$. All points of graphs are obtained via electronic tables.
It is clear from the picture, that maps, which correspond to
$v=0.55,\, 0.6,\, 0.65$ are indeed continuous. Nevertheless, the
continuity of the maps, which corresponds to $v=0.95$ is not so
evident.

We reduce in Section~\ref{sect-funct-rivn} the problem of the
conjugation of $f$ and $f_v$, which are defined by
formulas~(\ref{eq:50}) and~(\ref{eq:53}) to the solving of a
system of functional equations. We prove the following theorem in
Section~\ref{subcst-3-1} we prove the following theorem.
\begin{theorem*}[Theorem~\ref{theor:17}] The system of
functional equations
\begin{equation}\label{eq:67}
\left\{ \begin{array}{llllll} h(2x) = \displaystyle{\frac{1}{v}\,
h(x)} & x\leq 1/2 & (\theequation a)
\\
h(2-2x) = \displaystyle{\frac{1-h(x)}{1-v}} & x>1/2 &
(\theequation b)
\end{array}\right.
\end{equation}
has the unique solution $h:\, [0,\, 1]\rightarrow [0,\, 1]$ and it
is the solution of functional equation~(\ref{eq:58}) i.e. it is a
conjugation of maps $f$ and $f_v$, which are defined by
formulas~(\ref{eq:50}) and~(\ref{eq:53}).
\end{theorem*}
Section~\ref{subcst-3-2} is devoted to the study of the properties
of solutions of~(\ref{eq:67}). Section~\ref{subsusbs-formuly}
contains a general methods of solving of linear functional
equations. These methods are described in details
in~\cite{Pelukh-1974}.

We try to use the general methods of solving linear functional
equations, which are presented in Section~\ref{subsubs-Pidst} and
to write the general solution of~(\ref{eq:67}). The general
solution of the functional equation~(\ref{eq:67}a) is
\begin{equation}\label{eq:68}
h(x) = x^{-\log_2v}\omega_1(\log_2x),
\end{equation} where $\omega_1(x)$ is an arbitrary function with
period 1. If we plug the function $h$ of the form~(\ref{eq:68})
into the functional equation~(\ref{eq:67}b), then obtain
\begin{equation}\label{eq:69}
\begin{array}{c}
(1-v)(1-x)^{-\log_2v}\omega_1(\log_2(1-x)) =\\ \\
 = v(1-x^{-\log_2v}\omega_1(\log_2x)).
\end{array}
\end{equation}
We prove the following property of the equation~(\ref{eq:69}).
\begin{proposition*}[Proposition~\ref{zauv:2}] If
consider the functional equation~(\ref{eq:69}) as given on the
whole real line, the the unknown function $h$ appears to be
constant.
\end{proposition*}
Proposition~\ref{zauv:2} does not contradict to
Theorem~\ref{theor:homeom-jed} because of the following remark.
\begin{note*}[Remark~\ref{note:13}] If one solve the
equation~(\ref{eq:69}) for the unknown function $h: [0,\,
1]\rightarrow [0,\, 1]$, defined by formula~(\ref{eq:68}) then
reasonings from the proof of the Proposition~\ref{zauv:2} would be
incorrect.
\end{note*}
Also it is obtained un Section~\ref{subsusbs-formuly} that the
solution of the equation~(\ref{eq:67}b) is of the foem
\begin{equation}\label{eq:70}
\begin{array}{l}
h(x) = \frac{1}{2-v} +\left|x-\frac{2}{3}\right|^{-\log_2(1-v)}
\times
\\
\times\ \cdot \left\{
\begin{array}{ll}
\omega^+\left( \log_2\left|x-\frac{2}{3}\right|\right) &
x>\frac{2}{3};\\
\\
\omega^-\left( \log_2\left|x-\frac{2}{3}\right|\right) & x<
\frac{2}{3}.
\end{array}\right.
\end{array}
\end{equation}
for functions $\omega^+$ and $\omega^-$, which satisfy the
following relations:
\begin{equation} \label{eq:71} \left\{
\begin{array}{l}\omega^-(t+1) = -\omega^+(t)\\
\omega^+(t+1) = -\omega^-(t).
\end{array}\right.
\end{equation}

It follows from restrictions ~(\ref{eq:71}) that functions
$\omega^+$ and $\omega^-$ are periodical with period~2. If one
plug~(\ref{eq:70}) into the functional equation~(\ref{eq:67}a),
then obtain
\begin{equation}\label{eq:72}
\begin{array}{l}
\begin{array}{l}
\frac{v}{2-v} +v\left|2x-\frac{2}{3}\right|^{-\log_2(1-v)} \times
\\
\times\ v\cdot \left\{
\begin{array}{ll}
\omega^+\left( \log_2\left|2x-\frac{1}{3}\right|\right) &
x>\frac{2}{3};\\
\\
\omega^-\left( \log_2\left|2x-\frac{1}{3}\right|\right) & x<
\frac{2}{3}
\end{array}\right.
\end{array} = \\
\begin{array}{l}
=\frac{1}{2-v} +\left|x-\frac{2}{3}\right|^{-\log_2(1-v)} \times
\\
\times\ \cdot \left\{
\begin{array}{ll}
\omega^+\left( \log_2\left|x-\frac{2}{3}\right|\right) &
x>\frac{2}{3};\\
\\
\omega^-\left( \log_2\left|x-\frac{2}{3}\right|\right) & x<
\frac{2}{3}
\end{array}\right.
\end{array}
\end{array}
\end{equation} and the unknown functions $\omega^+$ and $\omega^-$
would satisfy the relations~(\ref{eq:71}). Functional
equations~(\ref{eq:69}) and~(\ref{eq:72}) are not linear and
standard methods of solving of functional equations can not be
applied to~(\ref{eq:69}) and~(\ref{eq:72}). The complicatedness of
equations~(\ref{eq:69}) and~(\ref{eq:72}) can be explained by the
properties of the conjugation $h$, which are stated by
Theorems~\ref{theor:13} and~\ref{theor:15}. We use in the
Section~\ref{subsubs-exper} the numerical values of $h:\, [0,\,
1]\rightarrow [0,\, 1]$, which is a solutions of the functional
equation~(\ref{eq:58}). These values are found by the formula from
Theorem~\ref{theor:16}. We use these values for studying the
properties of the function $\omega_1$ from the
formula~(\ref{eq:68}) and the function
$$ \omega_2(x) = \left\{
\begin{array}{ll}
\omega^+\left( \log_2\left|x-\frac{2}{3}\right|\right) &
x>\frac{2}{3};\\
\\
\omega^-\left( \log_2\left|x-\frac{2}{3}\right|\right) & x<
\frac{2}{3},
\end{array}\right.
$$ which in
fact appears in the formula~(\ref{eq:70}). In other words,
$\omega_2(x)$ is obtained from the equality
$$ h(x) = \frac{1}{2-v} +\left|x-\frac{2}{3}\right|^{-\log_2(1-v)}
\cdot \omega_2(x).
$$ Also remind that it still follows
from~(\ref{eq:71}) that $\omega_2$ is periodical with period 2.
The Figure~\ref{fig-19} contains the graph of $\omega_1(x)$ for
$v=3/4$ and $x\in [0,\, 1]$. The interval $[0,\, 1]$ is taken
since $\omega_1$ is periodical with period 1.
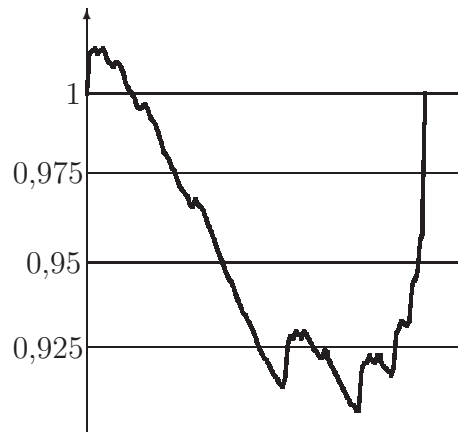
\begin{figure}[htbp]
\begin{minipage}[h]{0.9\linewidth}
\begin{center}
\begin{picture}(140,160)
\put(0,0){\vector(0,1){160}}

\put(0,64){\line(1,0){140}} \put(0,32){\line(1,0){140}}
\put(0,98){\line(1,0){140}} \put(0,128){\line(1,0){140}}

\linethickness{0.4mm}

\qbezier(0,128)(0,128)(1,143) \qbezier(1,143)(1,143)(3,145)
\qbezier(3,145)(3,145)(4,143) \qbezier(4,143)(4,143)(6,145)
\qbezier(6,145)(6,145)(7,143) \qbezier(7,143)(7,143)(8,140)
\qbezier(8,140)(8,140)(10,138) \qbezier(10,138)(10,138)(11,140)
\qbezier(11,140)(11,140)(13,138) \qbezier(13,138)(13,138)(14,135)
\qbezier(14,135)(14,135)(15,131) \qbezier(15,131)(15,131)(17,128)
\qbezier(17,128)(17,128)(18,126) \qbezier(18,126)(18,126)(19,123)
\qbezier(19,123)(19,123)(20,122) \qbezier(20,122)(20,122)(22,124)
\qbezier(22,124)(22,124)(23,122) \qbezier(23,122)(23,122)(24,119)
\qbezier(24,119)(24,119)(26,116) \qbezier(26,116)(26,116)(27,113)
\qbezier(27,113)(27,113)(28,110) \qbezier(28,110)(28,110)(29,106)
\qbezier(29,106)(29,106)(31,103) \qbezier(31,103)(31,103)(32,100)
\qbezier(32,100)(32,100)(33,99) \qbezier(33,99)(33,99)(34,96)
\qbezier(34,96)(34,96)(35,93) \qbezier(35,93)(35,93)(37,90)
\qbezier(37,90)(37,90)(38,89) \qbezier(38,89)(38,89)(39,86)
\qbezier(39,86)(39,86)(40,85) \qbezier(40,85)(40,85)(41,88)
\qbezier(41,88)(41,88)(42,86) \qbezier(42,86)(42,86)(44,84)
\qbezier(44,84)(44,84)(45,81) \qbezier(45,81)(45,81)(46,78)
\qbezier(46,78)(46,78)(47,76) \qbezier(47,76)(47,76)(48,73)
\qbezier(48,73)(48,73)(49,70) \qbezier(49,70)(49,70)(50,67)
\qbezier(50,67)(50,67)(51,65) \qbezier(51,65)(51,65)(52,62)
\qbezier(52,62)(52,62)(53,59) \qbezier(53,59)(53,59)(55,56)
\qbezier(55,56)(55,56)(56,53) \qbezier(56,53)(56,53)(57,50)
\qbezier(57,50)(57,50)(58,48) \qbezier(58,48)(58,48)(59,45)
\qbezier(59,45)(59,45)(60,44) \qbezier(60,44)(60,44)(61,42)
\qbezier(61,42)(61,42)(62,40) \qbezier(62,40)(62,40)(63,38)
\qbezier(63,38)(63,38)(64,35) \qbezier(64,35)(64,35)(65,32)
\qbezier(65,32)(65,32)(66,30) \qbezier(66,30)(66,30)(67,28)
\qbezier(67,28)(67,28)(68,27) \qbezier(68,27)(68,27)(69,25)
\qbezier(69,25)(69,25)(70,23) \qbezier(70,23)(70,23)(71,21)
\qbezier(71,21)(71,21)(72,20) \qbezier(72,20)(72,20)(73,18)
\qbezier(73,18)(73,18)(74,17) \qbezier(74,17)(74,17)(75,20)
\qbezier(75,20)(75,20)(76,33) \qbezier(76,33)(76,33)(77,36)
\qbezier(77,36)(77,36)(78,35) \qbezier(78,35)(78,35)(79,38)
\qbezier(79,38)(79,38)(80,37) \qbezier(80,37)(80,37)(81,36)
\qbezier(81,36)(81,36)(81,35) \qbezier(81,35)(81,35)(82,38)
\qbezier(82,38)(82,38)(83,37) \qbezier(83,37)(83,37)(84,35)
\qbezier(84,35)(84,35)(85,33) \qbezier(85,33)(85,33)(86,31)
\qbezier(86,31)(86,31)(87,30) \qbezier(87,30)(87,30)(88,28)
\qbezier(88,28)(88,28)(89,28) \qbezier(89,28)(89,28)(90,31)
\qbezier(90,31)(90,31)(91,30) \qbezier(91,30)(91,30)(91,28)
\qbezier(91,28)(91,28)(92,26) \qbezier(92,26)(92,26)(93,24)
\qbezier(93,24)(93,24)(94,22) \qbezier(94,22)(94,22)(95,20)
\qbezier(95,20)(95,20)(96,18) \qbezier(96,18)(96,18)(97,16)
\qbezier(97,16)(97,16)(98,14) \qbezier(98,14)(98,14)(99,12)
\qbezier(99,12)(99,12)(100,10) \qbezier(100,10)(100,10)(101,10)
\qbezier(101,10)(101,10)(102,8) \qbezier(102,8)(102,8)(103,8)
\qbezier(103,8)(103,8)(103,11) \qbezier(103,11)(103,11)(104,24)
\qbezier(104,24)(104,24)(105,26) \qbezier(105,26)(105,26)(106,26)
\qbezier(106,26)(106,26)(107,29) \qbezier(107,29)(107,29)(107,28)
\qbezier(107,28)(107,28)(108,27) \qbezier(108,27)(108,27)(109,26)
\qbezier(109,26)(109,26)(110,29) \qbezier(110,29)(110,29)(111,29)
\qbezier(111,29)(111,29)(111,27) \qbezier(111,27)(111,27)(112,25)
\qbezier(112,25)(112,25)(113,24) \qbezier(113,24)(113,24)(114,23)
\qbezier(114,23)(114,23)(115,22) \qbezier(115,22)(115,22)(115,21)
\qbezier(115,21)(115,21)(116,24) \qbezier(116,24)(116,24)(117,37)
\qbezier(117,37)(117,37)(118,39) \qbezier(118,39)(118,39)(119,42)
\qbezier(119,42)(119,42)(120,41) \qbezier(120,41)(120,41)(121,40)
\qbezier(121,40)(121,40)(122,42) \qbezier(122,42)(122,42)(123,55)
\qbezier(123,55)(123,55)(124,58) \qbezier(124,58)(124,58)(124,57)
\qbezier(124,57)(124,57)(125,60) \qbezier(125,60)(125,60)(126,72)
\qbezier(126,72)(126,72)(127,75) \qbezier(127,75)(127,75)(127,87)
\qbezier(127,87)(127,87)(128,128)

\put(-28,28){0,925} \put(-23,60){0,95} \put(-28,94){0,975}
\put(-8,124){1}

\end{picture}
\end{center}
\end{minipage}
\caption{The graph of $\omega_1(x)$ for $v=3/4$}\label{fig-19}
\end{figure}

The figure~\ref{fig-20} contains the graph of $\omega_2(x)$ for
$v=3/4$ and $x\in [0,\, 2]$. The interval $[0,\, 2]$ is take,
since $\omega_2$ is periodical with period~2.

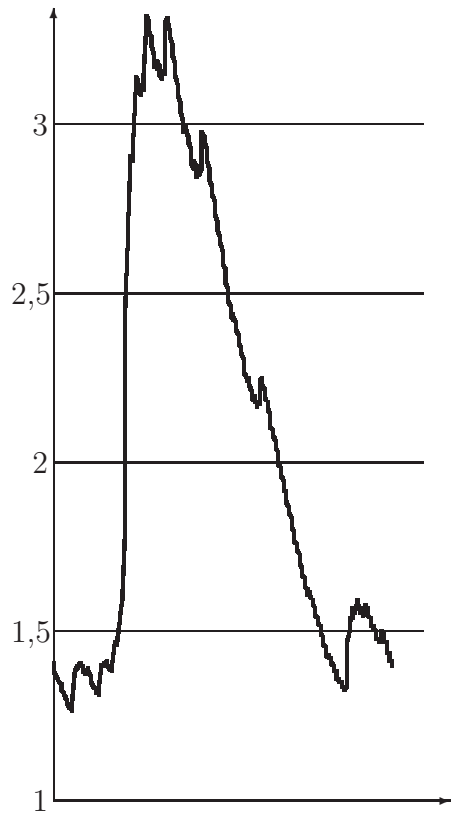
\begin{figure}[htbp]
\begin{minipage}[h]{0.9\linewidth}
\begin{center}
\begin{picture}(150,300)

\put(0,0){\vector(0,1){300}} \put(0,0){\vector(1,0){150}}

\put(0,128){\line(1,0){140}} \put(0,64){\line(1,0){140}}
\put(0,256){\line(1,0){140}} \put(0,192){\line(1,0){140}}

\linethickness{0.4mm}

\qbezier(0,52)(0,52)(0,49) \qbezier(0,49)(0,49)(1,47)
\qbezier(1,47)(1,47)(2,45) \qbezier(2,45)(2,45)(3,44)
\qbezier(3,44)(3,44)(3,42) \qbezier(3,42)(3,42)(4,40)
\qbezier(4,40)(4,40)(5,38) \qbezier(5,38)(5,38)(5,37)
\qbezier(5,37)(5,37)(6,35) \qbezier(6,35)(6,35)(7,34)
\qbezier(7,34)(7,34)(7,37) \qbezier(7,37)(7,37)(8,49)
\qbezier(8,49)(8,49)(9,51) \qbezier(9,51)(9,51)(9,50)
\qbezier(9,50)(9,50)(10,52) \qbezier(10,52)(10,52)(11,51)
\qbezier(11,51)(11,51)(11,49) \qbezier(11,49)(11,49)(12,48)
\qbezier(12,48)(12,48)(13,50) \qbezier(13,50)(13,50)(13,49)
\qbezier(13,49)(13,49)(14,47) \qbezier(14,47)(14,47)(14,45)
\qbezier(14,45)(14,45)(15,43) \qbezier(15,43)(15,43)(16,42)
\qbezier(16,42)(16,42)(16,41) \qbezier(16,41)(16,41)(17,40)
\qbezier(17,40)(17,40)(17,42) \qbezier(17,42)(17,42)(18,51)
\qbezier(18,51)(18,51)(19,52) \qbezier(19,52)(19,52)(19,51)
\qbezier(19,51)(19,51)(20,53) \qbezier(20,53)(20,53)(20,52)
\qbezier(20,52)(20,52)(21,50) \qbezier(21,50)(21,50)(22,49)
\qbezier(22,49)(22,49)(22,50) \qbezier(22,50)(22,50)(23,59)
\qbezier(23,59)(23,59)(23,60) \qbezier(23,60)(23,60)(24,59)
\qbezier(24,59)(24,59)(24,60) \qbezier(24,60)(24,60)(25,68)
\qbezier(25,68)(25,68)(25,69) \qbezier(25,69)(25,69)(26,76)
\qbezier(26,76)(26,76)(27,102) \qbezier(27,102)(27,102)(27,185)
\qbezier(27,185)(27,185)(28,210) \qbezier(28,210)(28,210)(28,215)
\qbezier(28,215)(28,215)(29,239) \qbezier(29,239)(29,239)(29,244)
\qbezier(29,244)(29,244)(30,242) \qbezier(30,242)(30,242)(30,247)
\qbezier(30,247)(30,247)(31,269) \qbezier(31,269)(31,269)(31,274)
\qbezier(31,274)(31,274)(32,272) \qbezier(32,272)(32,272)(32,269)
\qbezier(32,269)(32,269)(33,267) \qbezier(33,267)(33,267)(33,271)
\qbezier(33,271)(33,271)(34,269) \qbezier(34,269)(34,269)(34,273)
\qbezier(34,273)(34,273)(35,293) \qbezier(35,293)(35,293)(35,297)
\qbezier(35,297)(35,297)(36,295) \qbezier(36,295)(36,295)(36,291)
\qbezier(36,291)(36,291)(37,289) \qbezier(37,289)(37,289)(37,286)
\qbezier(37,286)(37,286)(38,282) \qbezier(38,282)(38,282)(38,278)
\qbezier(38,278)(38,278)(39,277) \qbezier(39,277)(39,277)(39,280)
\qbezier(39,280)(39,280)(40,278) \qbezier(40,278)(40,278)(40,275)
\qbezier(40,275)(40,275)(41,273) \qbezier(41,273)(41,273)(41,276)
\qbezier(41,276)(41,276)(42,275) \qbezier(42,275)(42,275)(42,277)
\qbezier(42,277)(42,277)(42,294) \qbezier(42,294)(42,294)(43,296)
\qbezier(43,296)(43,296)(43,295) \qbezier(43,295)(43,295)(44,291)
\qbezier(44,291)(44,291)(44,289) \qbezier(44,289)(44,289)(45,286)
\qbezier(45,286)(45,286)(45,282) \qbezier(45,282)(45,282)(46,279)
\qbezier(46,279)(46,279)(46,277) \qbezier(46,277)(46,277)(46,274)
\qbezier(46,274)(46,274)(47,271) \qbezier(47,271)(47,271)(47,267)
\qbezier(47,267)(47,267)(48,264) \qbezier(48,264)(48,264)(48,261)
\qbezier(48,261)(48,261)(49,257) \qbezier(49,257)(49,257)(49,254)
\qbezier(49,254)(49,254)(49,253) \qbezier(49,253)(49,253)(50,255)
\qbezier(50,255)(50,255)(50,254) \qbezier(50,254)(50,254)(51,251)
\qbezier(51,251)(51,251)(51,249) \qbezier(51,249)(51,249)(52,247)
\qbezier(52,247)(52,247)(52,244) \qbezier(52,244)(52,244)(52,241)
\qbezier(52,241)(52,241)(53,239) \qbezier(53,239)(53,239)(53,242)
\qbezier(53,242)(53,242)(54,240) \qbezier(54,240)(54,240)(54,238)
\qbezier(54,238)(54,238)(54,236) \qbezier(54,236)(54,236)(55,238)
\qbezier(55,238)(55,238)(55,237) \qbezier(55,237)(55,237)(56,239)
\qbezier(56,239)(56,239)(56,251) \qbezier(56,251)(56,251)(56,253)
\qbezier(56,253)(56,253)(57,251) \qbezier(57,251)(57,251)(57,249)
\qbezier(57,249)(57,249)(58,247) \qbezier(58,247)(58,247)(58,244)
\qbezier(58,244)(58,244)(58,242) \qbezier(58,242)(58,242)(59,239)
\qbezier(59,239)(59,239)(59,238) \qbezier(59,238)(59,238)(59,235)
\qbezier(59,235)(59,235)(60,232) \qbezier(60,232)(60,232)(60,229)
\qbezier(60,229)(60,229)(61,227) \qbezier(61,227)(61,227)(61,224)
\qbezier(61,224)(61,224)(61,222) \qbezier(61,222)(61,222)(62,219)
\qbezier(62,219)(62,219)(62,218) \qbezier(62,218)(62,218)(62,216)
\qbezier(62,216)(62,216)(63,213) \qbezier(63,213)(63,213)(63,211)
\qbezier(63,211)(63,211)(64,208) \qbezier(64,208)(64,208)(64,206)
\qbezier(64,206)(64,206)(64,203) \qbezier(64,203)(64,203)(65,201)
\qbezier(65,201)(65,201)(65,198) \qbezier(65,198)(65,198)(65,196)
\qbezier(65,196)(65,196)(66,194) \qbezier(66,194)(66,194)(66,191)
\qbezier(66,191)(66,191)(66,189) \qbezier(66,189)(66,189)(67,187)
\qbezier(67,187)(67,187)(67,185) \qbezier(67,185)(67,185)(67,183)
\qbezier(67,183)(67,183)(68,182) \qbezier(68,182)(68,182)(68,184)
\qbezier(68,184)(68,184)(68,183) \qbezier(68,183)(68,183)(69,181)
\qbezier(69,181)(69,181)(69,180) \qbezier(69,180)(69,180)(69,178)
\qbezier(69,178)(69,178)(70,176) \qbezier(70,176)(70,176)(70,174)
\qbezier(70,174)(70,174)(70,173) \qbezier(70,173)(70,173)(71,171)
\qbezier(71,171)(71,171)(71,169) \qbezier(71,169)(71,169)(72,167)
\qbezier(72,167)(72,167)(72,165) \qbezier(72,165)(72,165)(72,163)
\qbezier(72,163)(72,163)(72,162) \qbezier(72,162)(72,162)(73,160)
\qbezier(73,160)(73,160)(73,159) \qbezier(73,159)(73,159)(73,160)
\qbezier(73,160)(73,160)(74,160) \qbezier(74,160)(74,160)(74,158)
\qbezier(74,158)(74,158)(74,157) \qbezier(74,157)(74,157)(75,155)
\qbezier(75,155)(75,155)(75,154) \qbezier(75,154)(75,154)(75,152)
\qbezier(75,152)(75,152)(76,151) \qbezier(76,151)(76,151)(76,153)
\qbezier(76,153)(76,153)(76,152) \qbezier(76,152)(76,152)(77,150)
\qbezier(77,150)(77,150)(77,149) \qbezier(77,149)(77,149)(77,151)
\qbezier(77,151)(77,151)(78,150) \qbezier(78,150)(78,150)(78,151)
\qbezier(78,151)(78,151)(78,159) \qbezier(78,159)(78,159)(79,160)
\qbezier(79,160)(79,160)(79,159) \qbezier(79,159)(79,159)(79,158)
\qbezier(79,158)(79,158)(79,157) \qbezier(79,157)(79,157)(80,155)
\qbezier(80,155)(80,155)(80,153) \qbezier(80,153)(80,153)(80,152)
\qbezier(80,152)(80,152)(81,151) \qbezier(81,151)(81,151)(81,150)
\qbezier(81,150)(81,150)(81,148) \qbezier(81,148)(81,148)(82,146)
\qbezier(82,146)(82,146)(82,145) \qbezier(82,145)(82,145)(82,143)
\qbezier(82,143)(82,143)(82,141) \qbezier(82,141)(82,141)(83,140)
\qbezier(83,140)(83,140)(83,139) \qbezier(83,139)(83,139)(83,138)
\qbezier(83,138)(83,138)(84,136) \qbezier(84,136)(84,136)(84,135)
\qbezier(84,135)(84,135)(84,133) \qbezier(84,133)(84,133)(85,132)
\qbezier(85,132)(85,132)(85,130) \qbezier(85,130)(85,130)(85,128)
\qbezier(85,128)(85,128)(85,127) \qbezier(85,127)(85,127)(86,126)
\qbezier(86,126)(86,126)(86,124) \qbezier(86,124)(86,124)(86,123)
\qbezier(86,123)(86,123)(87,121) \qbezier(87,121)(87,121)(87,120)
\qbezier(87,120)(87,120)(87,119) \qbezier(87,119)(87,119)(87,117)
\qbezier(87,117)(87,117)(88,117) \qbezier(88,117)(88,117)(88,116)
\qbezier(88,116)(88,116)(88,114) \qbezier(88,114)(88,114)(88,113)
\qbezier(88,113)(88,113)(89,111) \qbezier(89,111)(89,111)(89,110)
\qbezier(89,110)(89,110)(89,109) \qbezier(89,109)(89,109)(90,107)
\qbezier(90,107)(90,107)(90,106) \qbezier(90,106)(90,106)(90,105)
\qbezier(90,105)(90,105)(90,103) \qbezier(90,103)(90,103)(91,102)
\qbezier(91,102)(91,102)(91,101) \qbezier(91,101)(91,101)(91,99)
\qbezier(91,99)(91,99)(91,98) \qbezier(91,98)(91,98)(92,97)
\qbezier(92,97)(92,97)(92,96) \qbezier(92,96)(92,96)(92,95)
\qbezier(92,95)(92,95)(93,93) \qbezier(93,93)(93,93)(93,92)
\qbezier(93,92)(93,92)(93,91) \qbezier(93,91)(93,91)(93,90)
\qbezier(93,90)(93,90)(94,88) \qbezier(94,88)(94,88)(94,87)
\qbezier(94,87)(94,87)(94,86) \qbezier(94,86)(94,86)(94,85)
\qbezier(94,85)(94,85)(95,84) \qbezier(95,84)(95,84)(95,83)
\qbezier(95,83)(95,83)(95,82) \qbezier(95,82)(95,82)(95,81)
\qbezier(95,81)(95,81)(96,80) \qbezier(96,80)(96,80)(96,79)
\qbezier(96,79)(96,79)(96,78) \qbezier(96,78)(96,78)(96,79)
\qbezier(96,79)(96,79)(97,79) \qbezier(97,79)(97,79)(97,78)

\qbezier(97,78)(97,78)(97,77) \qbezier(97,77)(97,77)(98,76)
\qbezier(98,76)(98,76)(98,75) \qbezier(98,75)(98,75)(98,74)

\qbezier(98,74)(98,74)(99,73) \qbezier(99,73)(99,73)(99,72)
\qbezier(99,72)(99,72)(99,71) \qbezier(99,71)(99,71)(99,70)
\qbezier(99,70)(99,70)(100,69) \qbezier(100,69)(100,69)(100,68)

\qbezier(100,68)(100,68)(100,67) \qbezier(100,67)(100,67)(101,66)
\qbezier(101,66)(101,66)(101,65) \qbezier(101,65)(101,65)(101,64)
\qbezier(101,64)(101,64)(101,63) \qbezier(101,63)(101,63)(102,62)
\qbezier(102,62)(102,62)(102,61) \qbezier(102,61)(102,61)(102,60)
\qbezier(102,60)(102,60)(102,59) \qbezier(102,59)(102,59)(103,58)
\qbezier(103,58)(103,58)(103,57)

\qbezier(103,57)(103,57)(103,56) \qbezier(103,56)(103,56)(103,55)
\qbezier(103,55)(103,55)(104,54)

\qbezier(104,54)(104,54)(104,55) \qbezier(104,55)(104,55)(104,54)
\qbezier(104,54)(104,54)(105,54) \qbezier(105,54)(105,54)(105,53)

\qbezier(105,53)(105,53)(105,52) \qbezier(105,52)(105,52)(106,51)

\qbezier(106,51)(106,51)(106,50) \qbezier(106,50)(106,50)(106,49)
\qbezier(106,49)(106,49)(107,49) \qbezier(107,49)(107,49)(107,48)
\qbezier(107,48)(107,48)(107,47) \qbezier(107,47)(107,47)(107,46)

\qbezier(107,46)(107,46)(108,45) \qbezier(108,45)(108,45)(108,46)

\qbezier(108,46)(108,46)(108,45) \qbezier(108,45)(108,45)(109,45)
\qbezier(109,45)(109,45)(109,44) \qbezier(109,44)(109,44)(109,43)

\qbezier(109,43)(109,43)(110,43)

\qbezier(110,43)(110,43)(110,42)

\qbezier(110,42)(110,42)(111,43)

\qbezier(111,43)(111,43)(111,47) \qbezier(111,47)(111,47)(111,60)
\qbezier(111,60)(111,60)(112,64)

\qbezier(112,64)(112,64)(112,68) \qbezier(112,68)(112,68)(112,69)
\qbezier(112,69)(112,69)(112,68) \qbezier(112,68)(112,68)(113,69)
\qbezier(113,69)(113,69)(113,72) \qbezier(113,72)(113,72)(113,73)
\qbezier(113,73)(113,73)(113,72) \qbezier(113,72)(113,72)(114,72)
\qbezier(114,72)(114,72)(114,71) \qbezier(114,71)(114,71)(114,72)
\qbezier(114,72)(114,72)(114,71) \qbezier(114,71)(114,71)(114,72)
\qbezier(114,72)(114,72)(115,75) \qbezier(115,75)(115,75)(115,76)
\qbezier(115,76)(115,76)(115,75)

\qbezier(115,75)(115,75)(115,74) \qbezier(115,74)(115,74)(116,73)

\qbezier(116,73)(116,73)(116,72) \qbezier(116,72)(116,72)(116,71)
\qbezier(116,71)(116,71)(116,72) \qbezier(116,72)(116,72)(117,71)

\qbezier(117,71)(117,71)(117,70) \qbezier(117,70)(117,70)(117,71)
\qbezier(117,71)(117,71)(117,70) \qbezier(117,70)(117,70)(118,71)
\qbezier(118,71)(118,71)(118,74)

\qbezier(118,74)(118,74)(118,73) \qbezier(118,73)(118,73)(119,73)
\qbezier(119,73)(119,73)(119,72) \qbezier(119,72)(119,72)(119,71)

\qbezier(119,71)(119,71)(119,70) \qbezier(119,70)(119,70)(120,69)

\qbezier(120,69)(120,69)(120,68) \qbezier(120,68)(120,68)(120,67)
\qbezier(120,67)(120,67)(120,66) \qbezier(120,66)(120,66)(121,66)
\qbezier(121,66)(121,66)(121,65)

\qbezier(121,65)(121,65)(122,64) \qbezier(122,64)(122,64)(122,63)

\qbezier(122,63)(122,63)(122,62) \qbezier(122,62)(122,62)(122,61)
\qbezier(122,61)(122,61)(123,61)

\qbezier(123,61)(123,61)(123,60)

\qbezier(123,60)(123,60)(124,60)

\qbezier(124,60)(124,60)(124,63) \qbezier(124,63)(124,63)(124,64)
\qbezier(124,64)(124,64)(124,63) \qbezier(124,63)(124,63)(125,63)
\qbezier(125,63)(125,63)(125,62)

\qbezier(125,62)(125,62)(125,61) \qbezier(125,61)(125,61)(125,60)
\qbezier(125,60)(125,60)(126,60) \qbezier(126,60)(126,60)(126,59)
\qbezier(126,59)(126,59)(126,58)

\qbezier(126,58)(126,58)(126,57) \qbezier(126,57)(126,57)(126,56)
\qbezier(126,56)(126,56)(127,56) \qbezier(127,56)(127,56)(127,55)

\qbezier(127,55)(127,55)(127,54) \qbezier(127,54)(127,54)(127,53)
\qbezier(127,53)(127,53)(128,53) \qbezier(128,53)(128,53)(128,52)
\qbezier(128,52)(128,52)(128,51)

\put(-8,252){3}

\put(-8,124){2}

\put(-8,-4){1} \put(-16,60){1,5} \put(-16,188){2,5}

\end{picture}
\end{center}
\end{minipage}
\caption{The graph of $\omega_2(x)$ for $v=3/4$}\label{fig-20}
\end{figure}

Consider examples of ``simple'' maps $\omega$ but such that the
function $h$, which is defined by~(\ref{eq:68}) is invertible and
consider the maps $\widetilde{f}_v$, which is defined by
commutativity of the diagram
\begin{equation}\label{eq:73}
\begin{CD}
[0,\, 1] @>f >> & [0,\, 1]\\
@V_{h} VV& @VV_{h}V\\
[0,\, 1] @>\widetilde{f}_v>>& [0,\, 1].
\end{CD}\end{equation}
We notice, that if for the invertible maps $h$ of the
form~(\ref{eq:68}) the diagram~(\ref{eq:73}) is commutative, then
for $x\hm{\in} [0,\, v]$ the equality $$ \widetilde{f}_v(x) =
\frac{x}{v}$$ holds. The simplest case for $\omega_1$ is that when
it is constant. The maps $\widetilde{f}_v$ for $v=3/4$, which is
defined by~(\ref{eq:73}) for the maps $h$ with constant $\omega_1$
is given at Figure~\ref{fig-32}.
\begin{figure}[htbp]
\begin{minipage}[h]{0.9\linewidth}
\begin{center}
\begin{picture}(100,120)
\put(0,0){\vector(0,1){120}} \put(0,0){\vector(1,0){120}}

\put(0,0){\line(3,4){75}}

\qbezier(75,100)(75,100)(78,96) \qbezier(78,96)(78,96)(81,91)
\qbezier(81,91)(81,91)(83,87) \qbezier(83,87)(83,87)(86,82)
\qbezier(86,82)(86,82)(88,77) \qbezier(88,77)(88,77)(90,72)
\qbezier(90,72)(90,72)(91,67) \qbezier(91,67)(91,67)(93,62)
\qbezier(93,62)(93,62)(94,57) \qbezier(94,57)(94,57)(96,52)
\qbezier(96,52)(96,52)(97,47) \qbezier(97,47)(97,47)(97,42)
\qbezier(97,42)(97,42)(98,37) \qbezier(98,37)(98,37)(99,32)
\qbezier(99,32)(99,32)(99,26) \qbezier(99,26)(99,26)(99,21)
\qbezier(99,21)(99,21)(100,16) \qbezier(100,16)(100,16)(100,11)
\qbezier(100,11)(100,11)(100,6) \qbezier(100,6)(100,6)(100,0)

\qbezier[32](75,100)(87.5,50)(100,0)

\end{picture}
\end{center}
\end{minipage}
\caption{Graph of $\widetilde{f}_v$ for $v =3/4$}\label{fig-32}
\end{figure}
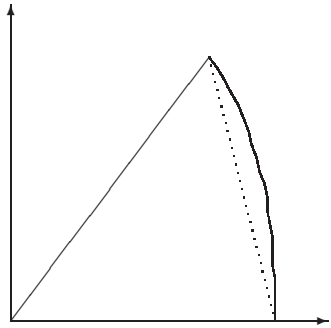

The Figure~\ref{fig-33} contains the graph of $\widetilde{f}_v$
for $v=3/4$, which is defined by~(\ref{eq:73}) for the maps $h$,
defined by~(\ref{eq:68}) if $\omega_1(\log_2x)$ is continuous,
whose graph at $[1/2,\, 1]$ is consisted of two parts of
linearity. We construct $\omega_1(\log_2x)$ in such a way that
$h(3/4)$ be equal to the right pre-image of $v$ under the action
of $f_v$. Notice, that if $x$ runs through $[1/2,\, 1]$, then
$\log_2x$ runs through the interval $[-1,\, 0]$, whose length is
1, which is the interval of periodicity of $\omega_1$.

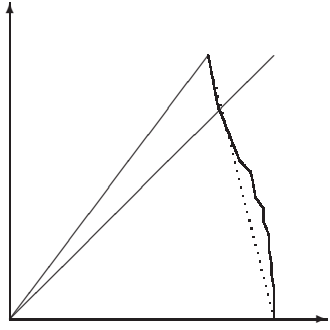
\begin{figure}[htbp]
\begin{minipage}[h]{0.9\linewidth}
\begin{center}
\begin{picture}(100,120)
\put(0,0){\vector(0,1){120}} \put(0,0){\vector(1,0){120}}

\put(0,0){\line(3,4){75}} \qbezier[32](75,100)(87.5,50)(100,0)

\qbezier(75,100)(75,100)(76,95) \qbezier(76,95)(76,95)(77,90)
\qbezier(77,90)(77,90)(78,84) \qbezier(78,84)(78,84)(79,80)
\qbezier(79,80)(79,80)(81,75) \qbezier(81,75)(81,75)(83,70)
\qbezier(83,70)(83,70)(85,65) \qbezier(85,65)(85,65)(87,60)
\qbezier(87,60)(87,60)(91,56) \qbezier(91,56)(91,56)(92,52)
\qbezier(92,52)(92,52)(93,46) \qbezier(93,46)(93,46)(96,42)
\qbezier(96,42)(96,42)(96,37) \qbezier(96,37)(96,37)(98,32)
\qbezier(98,32)(98,32)(98,27) \qbezier(98,27)(98,27)(99,21)
\qbezier(99,21)(99,21)(99,17) \qbezier(99,17)(99,17)(100,11)
\qbezier(100,11)(100,11)(100,6) \qbezier(100,6)(100,6)(100,0)

\put(0,0){\line(1,1){100}}
\end{picture}
\end{center}
\end{minipage}
\caption{Graph of $\widetilde{f}_v$ for $v =3/4$}\label{fig-33}
\end{figure}

These examples may be generalized as follows. Take an arbitrary
$n$ and use the maps $h_n$, which moves points of $A_n$ to $B_n$
to find the values of $\omega$ on the set $A_n\cap
\left[\frac{1}{2},\, 1\right]$ such that equality
$\widetilde{f}_v(A_n) = B_n$ hold. Then consider $\omega$ to be
linear at all other points and periodical with period $1$.
Consider the maps
\begin{equation}\label{eq:74} \widetilde{h}_n =
x^{-\log_2v}\omega_n(\log_2x)\end{equation} as an approximation of
$h$. If the constructed $\widetilde{h}_n$ would be invertible,
then there exists the unique $\widetilde{f}_n$, such that diagram
\begin{equation}\label{eq:75}
\begin{CD}
A @>f >> & A\\
@V_{\widetilde{h}_n} VV& @VV_{\widetilde{h}_n}V\\
B @>\widetilde{f}_n>>& B
\end{CD}\end{equation} would be commutative. This
$\widetilde{f}_n$ can be given via $$ \widetilde{f}_n =
\widetilde{h}_n(f(\widetilde{h}^{-1}_n)).
$$
For example, with the use of these notations the
Figure~\ref{fig-33} contains the graph $\widetilde{f}_2$.
Nevertheless, it is possible, that $\widetilde{h}_n$ would not be
invertible. Then there will not be $\widetilde{f}_n$ such that
diagram~(\ref{eq:73}) would be commutative. Maps $\widetilde{h}_n$
are not monotone for $v\approx 0$. Their graphs for different $v$
are given at Figure~\ref{fig-34}.

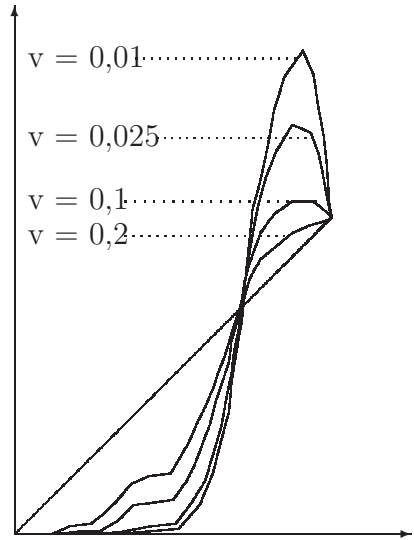
\begin{figure}[htbp]
\begin{minipage}[h]{0.9\linewidth}
\begin{center}
\begin{picture}(150,205)
\put(0,0){\vector(0,1){200}} \put(0,0){\vector(1,0){150}}

\put(5,177){v = 0,01}

\qbezier[20](46,180)(76,180)(106,180)

\put(5,147){v = 0,025}

\qbezier[20](54,150)(82,150)(110,150)

\put(5,123){v = 0,1}

\qbezier[20](42,126)(78,126)(114,126)

\put(5,110){v = 0,2}

\qbezier[20](42,113)(72,113)(102,113)

\qbezier(0,0)(0,0)(7,0) \qbezier(0,0)(0,0)(7,0)
\qbezier(0,0)(0,0)(10,0)    \qbezier(0,0)(0,0)(12,0)
\qbezier(7,0)(7,0)(14,0)    \qbezier(7,0)(7,0)(16,0)
\qbezier(10,0)(10,0)(21,0)  \qbezier(12,0)(12,0)(25,0)
\qbezier(14,0)(14,0)(21,3)  \qbezier(16,0)(16,0)(24,1)
\qbezier(21,0)(21,0)(31,0)  \qbezier(25,0)(25,0)(37,0)
\qbezier(21,3)(21,3)(29,4)  \qbezier(24,1)(24,1)(32,1)
\qbezier(31,0)(31,0)(41,1)  \qbezier(37,0)(37,0)(50,1)
\qbezier(29,4)(29,4)(34,8)  \qbezier(32,1)(32,1)(39,5)
\qbezier(41,1)(41,1)(52,3)  \qbezier(50,1)(50,1)(62,2)
\qbezier(34,8)(34,8)(39,13) \qbezier(39,5)(39,5)(45,11)
\qbezier(52,3)(52,3)(61,4)  \qbezier(62,2)(62,2)(70,11)
\qbezier(39,13)(39,13)(44,19)   \qbezier(45,11)(45,11)(53,12)
\qbezier(61,4)(61,4)(67,12) \qbezier(70,11)(70,11)(75,23)
\qbezier(44,19)(44,19)(51,22)   \qbezier(53,12)(53,12)(60,13)
\qbezier(67,12)(67,12)(72,20)   \qbezier(75,23)(75,23)(78,35)
\qbezier(51,22)(51,22)(59,23)   \qbezier(60,13)(60,13)(65,20)
\qbezier(72,20)(72,20)(75,30)   \qbezier(78,35)(78,35)(81,46)
\qbezier(59,23)(59,23)(62,28)   \qbezier(65,20)(65,20)(68,26)
\qbezier(75,30)(75,30)(78,40)   \qbezier(81,46)(81,46)(82,57)
\qbezier(62,28)(62,28)(66,35)   \qbezier(68,26)(68,26)(71,34)
\qbezier(78,40)(78,40)(80,51)   \qbezier(82,57)(82,57)(84,70)
\qbezier(66,35)(66,35)(69,42)   \qbezier(71,34)(71,34)(74,42)
\qbezier(80,51)(80,51)(82,61)   \qbezier(84,70)(84,70)(86,81)
\qbezier(69,42)(69,42)(72,48)   \qbezier(74,42)(74,42)(76,50)
\qbezier(82,61)(82,61)(83,69)   \qbezier(86,81)(86,81)(87,94)
\qbezier(72,48)(72,48)(75,55)   \qbezier(76,50)(76,50)(78,56)
\qbezier(83,69)(83,69)(85,79)   \qbezier(87,94)(87,94)(88,103)
\qbezier(75,55)(75,55)(77,61)   \qbezier(78,56)(78,56)(81,65)
\qbezier(85,79)(85,79)(86,89)   \qbezier(88,103)(88,103)(90,123)
\qbezier(77,61)(77,61)(80,69)   \qbezier(81,65)(81,65)(82,73)
\qbezier(86,89)(86,89)(88,100)  \qbezier(90,123)(90,123)(93,134)
\qbezier(80,69)(80,69)(82,76)   \qbezier(82,73)(82,73)(84,80)
\qbezier(88,100)(88,100)(89,108) \qbezier(93,134)(93,134)(96,148)
\qbezier(82,76)(82,76)(84,82) \qbezier(84,80)(84,80)(86,89)
\qbezier(89,108)(89,108)(92,124) \qbezier(96,148)(96,148)(98,160)
\qbezier(84,82)(84,82)(87,90) \qbezier(86,89)(86,89)(87,96)
\qbezier(92,124)(92,124)(95,135) \qbezier(98,160)(98,160)(102,173)
\qbezier(87,90)(87,90)(89,97) \qbezier(87,96)(87,96)(89,103)
\qbezier(95,135)(95,135)(99,145)
\qbezier(102,173)(102,173)(109,183) \qbezier(89,97)(89,97)(93,104)
\qbezier(89,103)(89,103)(93,114) \qbezier(99,145)(99,145)(105,155)
\qbezier(109,183)(109,183)(113,174)
\qbezier(93,104)(93,104)(99,109) \qbezier(93,114)(93,114)(98,121)
\qbezier(105,155)(105,155)(112,152)
\qbezier(113,174)(113,174)(115,161)
\qbezier(99,109)(99,109)(105,114)
\qbezier(98,121)(98,121)(105,126)
\qbezier(112,152)(112,152)(115,144)
\qbezier(115,161)(115,161)(117,150)
\qbezier(105,114)(105,114)(112,117)
\qbezier(105,126)(105,126)(113,126)
\qbezier(115,144)(115,144)(117,134)
\qbezier(117,150)(117,150)(118,141)
\qbezier(112,117)(112,117)(119,119)
\qbezier(113,126)(113,126)(119,121)
\qbezier(117,134)(117,134)(119,123)
\qbezier(118,141)(118,141)(119,125)
\qbezier(119,119)(119,119)(120,120)
\qbezier(119,121)(119,121)(120,120)
\qbezier(119,123)(119,123)(120,120)
\qbezier(119,125)(119,125)(120,120)
\qbezier(120,120)(120,120)(0,0) \qbezier(120,120)(120,120)(0,0)
\qbezier(120,120)(120,120)(0,0) \qbezier(120,120)(120,120)(0,0)
\end{picture}
\end{center}
\end{minipage}
\caption{Graphs of $\widetilde{h}_3$ for different $v$}
\label{fig-34}
\end{figure}

Notice as a comment to the Figure~\ref{fig-34} that all these maps
satisfy the functional equation~(\ref{eq:67}a) $$ h(2x) =
\frac{1}{v}\ h(x),
$$ i.e. $h$ repeats its graph on each interval of
the form $\left[\frac{1}{2^{k+1}},\, \frac{1}{2^k}\right]$, but
the graph is $v^k$-times compressed. We prove the following
proposition.
\begin{proposition*}[Proposition~\ref{lema:12}] For
every $n\in \mathbb{N}$ there exists $v_0 \in (0,\, 1)$ such that
for every $v\in (0,\, v_0)$ the maps $\widetilde{h}_n(x)$ is
non-monotone on $\left[ \frac{2^{n-1}-1}{2^{n-1}},\, 1\right]$.
\end{proposition*}
Now define in the previous way the solution $h$, which is
determined by~(\ref{eq:70}) and is obtained from ~(\ref{eq:67}b).
Graph of $\widetilde{f}_v$, which is determined by commutative
diagram~(\ref{eq:73}) for the maps $h$ of the form~(\ref{eq:70}),
if $\omega_2$  is constant is given at Figure~\ref{fig-35}.

\begin{figure}[htbp]
\begin{minipage}[h]{0.9\linewidth}
\begin{center}
\begin{picture}(100,120)
\put(0,0){\vector(0,1){120}} \put(0,0){\vector(1,0){120}}

\qbezier(0,0)(0,0)(2,3) \qbezier(2,3)(2,3)(5,9)
\qbezier(5,9)(5,9)(8,14) \qbezier(8,14)(8,14)(10,18)
\qbezier(10,18)(10,18)(12,22) \qbezier(12,22)(12,22)(15,28)
\qbezier(15,28)(15,28)(17,31) \qbezier(17,31)(17,31)(20,36)
\qbezier(20,36)(20,36)(23,41) \qbezier(23,41)(23,41)(25,45)
\qbezier(25,45)(25,45)(28,49) \qbezier(28,49)(28,49)(31,54)
\qbezier(31,54)(31,54)(34,58) \qbezier(34,58)(34,58)(37,62)
\qbezier(37,62)(37,62)(40,66) \qbezier(40,66)(40,66)(42,69)
\qbezier(42,69)(42,69)(45,72) \qbezier(45,72)(45,72)(48,74)
\qbezier(48,74)(48,74)(50,76) \qbezier(50,76)(50,76)(53,78)
\qbezier(53,78)(53,78)(56,79) \qbezier(56,79)(56,79)(59,80)
\qbezier(59,80)(59,80)(62,80) \qbezier(62,80)(62,80)(65,82)
\qbezier(65,82)(65,82)(67,84) \qbezier(67,84)(67,84)(69,86)
\qbezier(69,86)(69,86)(71,89) \qbezier(71,89)(71,89)(72,92)
\qbezier(72,92)(72,92)(74,96) \qbezier(74,96)(74,96)(75,99)
\qbezier(75,99)(75,99)(75,100)

\qbezier[32](0,0)(37.5,50)(75,100)

\qbezier(75,100)(87.5,50)(100,0)

\end{picture}
\end{center}
\end{minipage}
\caption{Graph of $\widetilde{f}_v$ for
$v=\frac{3}{4}$}\label{fig-35}
\end{figure}
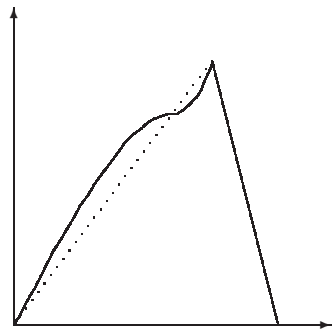

Let $\omega^+$ and $\omega^-$ be continuous, piecewise linear with
the smallest possible braking points such that $h(1/4)=v^2$. Then
the graph of $\widetilde{f}_v$, which is defined by
diagram~(\ref{eq:73}) for the maps $h$ of the form~(\ref{eq:70})
is given at Figure~\ref{fig-36}.

\begin{figure}[htbp]
\begin{minipage}[h]{0.9\linewidth}
\begin{center}
\begin{picture}(100,120)
\put(0,0){\vector(0,1){120}} \put(0,0){\vector(1,0){120}}

\put(56.25,75){\circle*{5}}

\qbezier(75,100)(87.5,50)(100,0)

\qbezier[32](0,0)(37.5,50)(75,100)

\qbezier(0,0)(0,0)(1,3) \qbezier(1,3)(1,3)(3,6)
\qbezier(3,6)(3,6)(5,10) \qbezier(5,10)(5,10)(7,14)
\qbezier(7,14)(7,14)(9,17) \qbezier(9,17)(9,17)(11,21)
\qbezier(11,21)(11,21)(13,24) \qbezier(13,24)(13,24)(15,28)
\qbezier(15,28)(15,28)(17,31) \qbezier(17,31)(17,31)(19,35)
\qbezier(19,35)(19,35)(21,38) \qbezier(21,38)(21,38)(23,42)
\qbezier(23,42)(23,42)(25,45) \qbezier(25,45)(25,45)(27,49)
\qbezier(27,49)(27,49)(30,52) \qbezier(30,52)(30,52)(32,55)
\qbezier(32,55)(32,55)(35,58) \qbezier(35,58)(35,58)(39,61)
\qbezier(39,61)(39,61)(42,64) \qbezier(42,64)(42,64)(45,66)
\qbezier(45,66)(45,66)(49,69) \qbezier(49,69)(49,69)(52,72)
\qbezier(52,72)(52,72)(56,75) \qbezier(56,75)(56,75)(59,78)
\qbezier(59,78)(59,78)(64,80) \qbezier(64,80)(64,80)(68,82)
\qbezier(68,82)(68,82)(71,85) \qbezier(71,85)(71,85)(72,90)
\qbezier(72,90)(72,90)(74,95) \qbezier(74,95)(74,95)(75,100)
\end{picture}
\end{center}
\end{minipage}
\caption{Graph of $\widetilde{f}_v$ for
$v=\frac{3}{4}$}\label{fig-36}
\end{figure}
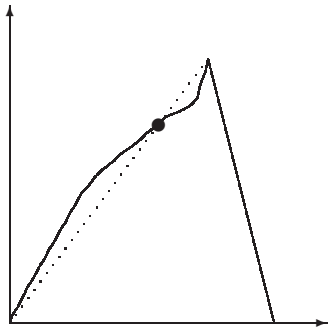

The Figure~\ref{fig-37} the result putting the graph from the
previous example the the just constructed one (the dots are used
for the first graph).
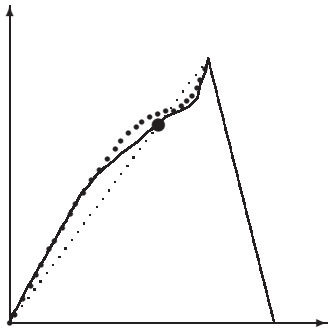
\begin{figure}[htbp]
\begin{minipage}[h]{0.9\linewidth}
\begin{center}
\begin{picture}(100,120)
\put(0,0){\vector(0,1){120}} \put(0,0){\vector(1,0){120}}

\put(56.25,75){\circle*{5}}

\qbezier(75,100)(87.5,50)(100,0)

\qbezier[32](0,0)(37.5,50)(75,100)

\qbezier(0,0)(0,0)(1,3) \qbezier(1,3)(1,3)(3,6)
\qbezier(3,6)(3,6)(5,10) \qbezier(5,10)(5,10)(7,14)
\qbezier(7,14)(7,14)(9,17) \qbezier(9,17)(9,17)(11,21)
\qbezier(11,21)(11,21)(13,24) \qbezier(13,24)(13,24)(15,28)
\qbezier(15,28)(15,28)(17,31) \qbezier(17,31)(17,31)(19,35)
\qbezier(19,35)(19,35)(21,38) \qbezier(21,38)(21,38)(23,42)
\qbezier(23,42)(23,42)(25,45) \qbezier(25,45)(25,45)(27,49)
\qbezier(27,49)(27,49)(30,52) \qbezier(30,52)(30,52)(32,55)
\qbezier(32,55)(32,55)(35,58) \qbezier(35,58)(35,58)(39,61)
\qbezier(39,61)(39,61)(42,64) \qbezier(42,64)(42,64)(45,66)
\qbezier(45,66)(45,66)(49,69) \qbezier(49,69)(49,69)(52,72)
\qbezier(52,72)(52,72)(56,75) \qbezier(56,75)(56,75)(59,78)
\qbezier(59,78)(59,78)(64,80) \qbezier(64,80)(64,80)(68,82)
\qbezier(68,82)(68,82)(71,85) \qbezier(71,85)(71,85)(72,90)
\qbezier(72,90)(72,90)(74,95) \qbezier(74,95)(74,95)(75,100)

\put(0,0){\circle*{2}} \put(2,3){\circle*{2}}
\put(5,9){\circle*{2}} \put(8,14){\circle*{2}}
\put(10,18){\circle*{2}} \put(12,22){\circle*{2}}
\put(15,28){\circle*{2}} \put(17,31){\circle*{2}}
\put(20,36){\circle*{2}} \put(23,41){\circle*{2}}
\put(25,45){\circle*{2}} \put(28,49){\circle*{2}}
\put(31,54){\circle*{2}} \put(34,58){\circle*{2}}
\put(37,62){\circle*{2}} \put(40,66){\circle*{2}}
\put(42,69){\circle*{2}} \put(45,72){\circle*{2}}
\put(48,74){\circle*{2}} \put(50,76){\circle*{2}}
\put(53,78){\circle*{2}} \put(56,79){\circle*{2}}
\put(59,80){\circle*{2}} \put(62,80){\circle*{2}}
\put(65,82){\circle*{2}} \put(67,84){\circle*{2}}
\put(69,86){\circle*{2}} \put(71,89){\circle*{2}}
\put(72,92){\circle*{2}} \put(74,96){\circle*{2}}
\end{picture}
\end{center}
\end{minipage}
\caption{The putting of one graph onto another}\label{fig-37}
\end{figure}

Section~\ref{sect-Javni-fornuly} is devoted to explicit formulas
of the homeomorphic solution $h:\, [0,\, 1]\rightarrow [0,\, 1]$
of the equation~(\ref{eq:58}). In Section~\ref{subs-h-worse} we
prove the following theorem.
\begin{theorem*}[Theorem~\ref{theor:18}, also Sect. 3.2
in~\cite{ADM-2016}] The homeomorphic solution $h:\, [0,\,
1]\rightarrow [0,\, 1]$ of the equation~(\ref{eq:58}) can be
expressed by formula
$$
h(x) = \lim\limits_{n \rightarrow
\infty}\beta\left(\displaystyle{\left[2^nx\right],\, n}\right) =
\lim\limits_{n \rightarrow \infty}
\sum\limits_{t=1}^{[2^nx]}\frac{1}{\zeta_{n,t}},
$$ where
\begin{equation}\label{eq:81}
\begin{array}{c} \zeta_{n,k} = \prod\limits_{t=1}^n\left(
\frac{2}{1-v}\cdot\left\{ \left\{\displaystyle{\left\{
\frac{k}{2^{n-t}}\right\}}/2\right\}+ \left\{\displaystyle{\left\{
\frac{k}{2^{n-t+1}}\right\}}/2\right\}\right\} +\right.\\
+\left.\frac{2}{v}\cdot\left\{\left\{\displaystyle{\left\{
\frac{k}{2^{n-t}}\right\}}/2\right\}+ \left\{\displaystyle{\left\{
\frac{k}{2^{n-t+1}}\right\}}/2\right\}
+\frac{1}{2}\right\}\right).\end{array}\end{equation}
\end{theorem*}
\begin{note*}[Remark~\ref{note:2}] In spite of the
formula for $h(x)$ is quit complicated and contains a limit, the
value of $h$ is defined at any point. The existence of the limit
follows from that the formula is obtained from the same
reasonings, which where made in the proof of
Theorem~\ref{theor:homeom-jed}.
\end{note*}
\begin{note*}[Remark~\ref{note:3}] Notice, that the
formula from Theorem~\ref{theor:18} has the following properties:

1. In the case when $x\in A$ the finding of $h(x)$ via this
formula leads to $n$ summand for the approx.imation $[2^nx]\in
A_n$. The case $x\in A$  means that ``approximated'' values of
$h(x)$ stabilize after finite number of steps on the exact value.

2. The stabilization via the obtained formula, i.e. the equality
$$ \sum\limits_{t=1}^{[2^nx]}\frac{1}{\zeta_{n,t}} =
\sum\limits_{t=1}^{[2^{n+1}x]}\frac{1}{\zeta_{n+1,t}}
$$ for
some $n\in \mathbb{N}$ means that $x\in A_n$ and the obtained
approximation is the exact value of $h(x)$.
\end{note*}
The first of mentioned properties from the Remark~\ref{note:3} can
be considered as a disadvantage of the formula. We obtain in
Section~\ref{subs-h-better} another formula for $\beta_{n,k}$,
which is free of this disadvantage. We prove the following
theorem.
\begin{theorem*}[Theorem~\ref{theor:19}, also
Section~3.3 in~\cite{ADM-2016}] The homeomorphic solution $h:\,
[0,\, 1]\rightarrow [0,\, 1]$ of the equation~(\ref{eq:58}) can be
expressed by the following formula
$$ h(x) = \sum\limits_{i=1}^{\infty}\frac{(2^{i-1}
x)((-1)^{[-\{ \log_2 [2^ix]\}]}-1)}{\zeta_{i,\,
[2^{i+[-\log_2(2^{i+1}x)]+1} x]}},
$$ where $\zeta_{n,k}$ is expressed by~(\ref{eq:81}).
\end{theorem*}
\begin{note*}[Remark~\ref{note:4}] Notice, that if $x\in
\mathcal{A}$ then the formula from Theorem~\ref{theor:19} contains
only finitely many summands.
\end{note*}

We consider in Section~\ref{sect:KuskowoLin} the conjugation of
maps $f:\, [0,\, 1]\rightarrow [0,\, 1]$ of the form~(\ref{eq:50})
and a continuous maps $g: \, [0,\, 1]\rightarrow [0,\, 1]$ of the
form
\begin{equation} \label{eq:90} g(x) = \left\{\begin{array}{ll}
g_l,& \text{if } 0\leq x< v,\\
g_r,& \text{if } v \leqslant x\leqslant 1,
\end{array}\right.
\end{equation} where
$g_l(0)=g_r(1)=0$, $g_r(v)=1$ and functions $g_l$ and $g_r$ are
monotone and piecewise linear. We study the conditions of the
existence and properties of the homeomorphism $h:\, [0,\,
1]\rightarrow [0,\, 1]$ (if it exists), which is the solution of
the functional equation
\begin{equation}\label{eq:91} h(f) = g(h).
\end{equation}
In Section~\ref{sect:KuskowoLin-2} we prove the following theorem.
\begin{theorem*}[Theorem~\ref{theor-dyfer-lin}, also
Theorem~1 in~\cite{UMZh}] Let the maps $f$, which is given
by~(\ref{eq:50}), be topologically conjugated a piecewise linear
$g$ which maps $[0,\, 1]$ onto itself and let $h$ be a
homeomorphism such that equality~(\ref{eq:91}) holds. If $h$ is
continuously differentiable on $(\alpha,\, \beta)$ for some $0\leq
\alpha<\beta\leq 1$, then $h$ is piecewise linear.
\end{theorem*}
The Proposition~\ref{theor:ksklin-tpleqv} is an example of
restrictions for $g$, which should be satisfied, if it is
conjugated to $f$ via piecewise linear homeomorphism. The
following proposition can be considered as one more restriction
for $g$
\begin{proposition*}[Proposition~\ref{lema:36}] If a
maps $g:\, [0,\, 1]\rightarrow [0,\, 1]$ of the form~(\ref{eq:90})
is topologically conjugated to $f:\, [0,\, 1]\rightarrow [0,\, 1]$
of the form~(\ref{eq:50}), then $g$ only two fixed points: one of
which is $0$ and the other belongs to $(v,\, 1)$.
\end{proposition*}
Proposition~\ref{lema:36} does not assume that conjugation is
piecewise linear, but this Proposition is used for the results of
the next subsection. In Section~\ref{sect:KuskowoLin-3} we
consider the conjugation of $f:\, [0,\, 1]\rightarrow [0,\, 1]$ of
the form~(\ref{eq:50}) and the maps $g:\, [0,\, 1]\rightarrow
[0,\, 1]$ of the form~(\ref{eq:90}) via piecewise linear
homeomorphism $h:\, [0,\, 1]\rightarrow [0,\, 1]$, which satisfy
the functional equation~(\ref{eq:91}). The main result of the
Section~\ref{sect:KuskowoLin-3} is the following two theorems.
\begin{theorem*}[Theorem~\ref{theor:21} also Theorem~2
in~\cite{UMZh}] For an arbitrary $v\in (0,\, 1)$ and an increasing
piecewise linear maps $g:\, [0,\, v]\to [0,\, 1]$ such that
$g(0)=0$, $g(v)=1$ and $g'(0)=2$ there exists and the unique its
continuation $\widetilde{g}:\, [0,\, 1]\to [0,\, 1]$, which is
topologically conjugated to $f$ of the form~(\ref{eq:50}) via
piecewise linear homeomorphism.
\end{theorem*}
\begin{theorem*}[Theorem~\ref{theor:22} also Theorem~3
in~\cite{UMZh}] For arbitrary $v\in (0,\, 1)$ and decreasing $g:\,
[v,\, 1]\mapsto [0,\, 1]$ such that $g(v)=1$, ${g(1)=0}$ and
$(g^2)'(x_0) = 4$, where $x_0$ is a fixed points of $g$, there
exists and the unique continuation of $\widetilde{g}$ to $[0,\,
1]$, which is topologically conjugated to $f$ of the
form~(\ref{eq:50}) via piecewise linear homeomorphism.
\end{theorem*}
In the proof of Theorem~\ref{theor:22} the following analogue of
Proposition~\ref{theor:ksklin-tpleqv} is used.
\begin{proposition*}[Proposition~\ref{theor:23} also
Lemma~13 in~\cite{UMZh}] Let $x_0$ be a fixed point of piecewise
linear unimodal maps $g$, which is conjugated with $f$ of the
form~(\ref{eq:50}) vis piecewise linear homeomorphism. Then there
exists $\varepsilon>0$ such that for every $x\in
(x_0-\varepsilon,\, x_0+\varepsilon)\backslash \{x_0\}$ the
equality $(g^2)'(x_0)=4$ holds, where $g^2$ as in general means
the second iteration of $g$.
\end{proposition*}
It is assumed in Theorem~\ref{Theor:9}, that the function $g$ of
the form~(\ref{eq:90}), which is conjugated to $f$ of the
from~(\ref{eq:50}), is convex. Nevertheless, convexity is not used
in the proof of this Theorem. In
Section~\ref{arch-sect-Non-convex} we will use in details the
techniques from the proof of Theorems~\ref{theor:21}
and~\ref{theor:22} for obtaining the example of non-convex $g$,
which is conjugated to $f$. The graph of this $g$ is given at
Figure~\ref{fig:27}.

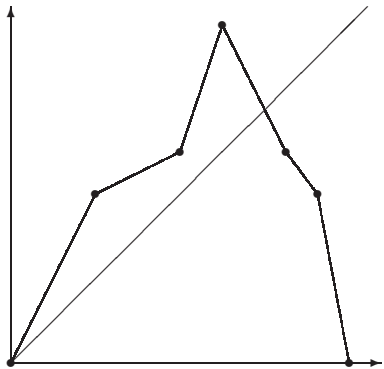
\begin{figure}[htbp]
\begin{center}
\begin{center}
\begin{picture}(140,135)
\put(0,0){\vector(1,0){140}} \put(0,0){\vector(0,1){135}}
\put(0,0){\line(1,1){135}}

\put(0,0){\circle*{3}} \Vidr{0}{0}{32}{64}
\put(32,64){\circle*{3}} \VidrTo{64}{80} \put(64,80){\circle*{3}}
\VidrTo{80}{128} \put(80,128){\circle*{3}} \VidrTo{104}{80}
\put(104,80){\circle*{3}} \VidrTo{116}{64}
\put(116,64){\circle*{3}} \VidrTo{128}{0} \put(128,0){\circle*{3}}

\end{picture}
\end{center}
\end{center} \caption{Graph of $g$}\label{fig:27}
\end{figure}

We consider in Section~\ref{sect:type-linear} the types of
linearity of the piecewise linear $g$ of the form~(\ref{eq:90}),
which is conjugated to $f$ of the from~(\ref{eq:50}) via piecewise
linear $h:\, [0,\ 1] \rightarrow [0,\, 1]$.

\begin{definition*}[Definition~\ref{def:03}]
Let for the piecewise linear $g$ of the form~(\ref{eq:g}) be
topologically conjugated to $f$ of the form~(\ref{eq:89}). Let the
number of pieces of linearity of $g_1$ and $g_2$ be $p$ and $q$
correspondingly. Call the pair $(p,\, q)$ the \textbf{type of
piecewise linearity} of $g$.
\end{definition*}

\begin{definition*}[Definition~\ref{def:04}]
If for a pair $(p,\, q)$ there exists a maps $g$ of the
form~(\ref{eq:g}) with type of piecewise linearity $(p,\, q)$,
then call this type \textbf{admissible}. If follows from
Proposition~\ref{theor:ksklin-tpleqv} and~\ref{theor:23} that pair
$(1,\, q)$ is admissible only if $q=1$ and the pair $(p,\, 1)$ is
admissible only if $p=1$. In both cases the equality $g=f$ holds.
\end{definition*}

We prove the following theorem in Section~\ref{sect:type-linear}.

\begin{theorem*}[Theorem~\ref{theor:24}]
1. For any $p\geq 2$ and $q\geq 2$ the type of linearity $(p,\,
q)$ is admissible.

2. A type of linearity $(p,\, 1)$ and $(1,\, q)$ is admissible
only if it is $(1,\, 1)$. In this case the maps $g$ coincides with
$f$.
\end{theorem*}

In Section~\ref{sect:SemiConj} we consider the problem of semi
conjugation of maps $f$ and $f_v$, which are given by
formulas~(\ref{eq:50}) and~(\ref{eq:53}). Precisely, we consider
continuous, but not necessary invertible solutions $\eta:\, [0,\,
1]\rightarrow [0,\, 1]$ of the functional equation
\begin{equation}\label{eq:03}\eta(f) = f_v(\eta).
\end{equation}
In Section~\ref{sect:SemiConj-1} we prove the following
proposition.
\begin{proposition*}[Proposition~\ref{prop:2}] Each
continuous monotone solution $\eta:\, [0,\, 1]\rightarrow [0,\,
1]$ of the functional equation~(\ref{eq:03}) is a conjugation of
$f$ and $f_v$ of the form~(\ref{eq:50}) and~(\ref{eq:53}).
\end{proposition*}
Let $h:\, [0,\, 1]\rightarrow [0,\, 1]$ be the homeomorphism,
which satisfies~(\ref{eq:58}) and $\eta:\, [0,\, 1]\rightarrow
[0,\, 1]$ be an arbitrary continuous solution of the
equation~(\ref{eq:03}). Consider the commutative diagram
\begin{equation}
\label{eq:10} \xymatrix{ [0,\, 1] \ar@/_3pc/@{-->}_{\xi}[dd]
\ar^{f}[r]
\ar_{\eta}[d] & [0,\, 1] \ar^{\eta}[d] \ar@/^3pc/@{-->}^{\xi}[dd]\\
[0,\, 1] \ar^{f_v}[r] \ar^{h^{-1}}[d] & [0,\, 1] \ar_{h^{-1}}[d]\\
[0,\, 1] \ar^{f}[r] & [0,\, 1] }\end{equation} and denote $\xi =
h^{-1}(\eta)$. We notice in Section~\ref{sect:SemiConj-2} that it
follows from the commutativity of~(\ref{eq:10}) that the problem
of finding continuous solutions of~(\ref{eq:03}) is equivalent to
the problem of finding continuous solutions of the functional
equation
\begin{equation}\label{eq:31}\xi(f) = f(\xi).
\end{equation}
Some continuous solutions of the equation~(\ref{eq:31}) can be
easily found, which is done in the Remark~\ref{note:14}.
\begin{note*}[Remark~\ref{note:14}] The following
functions $\xi$ satisfy the functional equation~(\ref{eq:31}).

1. $\xi(x)=x$ for all $x\in [0,\, 1]$;

2. $\xi$ is a constant, which is one of fixed points of $f$;

3. $\xi$ is an arbitrary iteration of $f$.
\end{note*}
The following two propositions show that if the solution
of~(\ref{eq:31}) is in some cense ``good'' on a subinterval of
$[0,\, 1]$ then it is piecewise ``good'' on the whole $[0,\, 1]$.

\begin{proposition*}[Proposition~\ref{theorNeStale}] If
the continuous solution $\xi$ of the functional
equation~(\ref{eq:31}) is constant on some interval $M =
[\alpha,\, \beta]$, then it is piecewise constant on $[0,\, 1]$.
\end{proposition*}
\begin{proposition*}[Proposition~\ref{theorNePriama}]
If the graph of the continuous solution $\xi$ of the functional
equation~(\ref{eq:31}) is a line segment on sone set $M =
[\alpha,\, \beta]$ then $\xi$ is piecewise linear on $[0,\, 1]$.
\end{proposition*}
We consider in Section~\ref{sect:SemiConj-3} the continuous
solutions $\xi:\, [0,\, 1]\rightarrow [0,\, 1]$ of the functional
equation~(\ref{eq:31}). The main result of this Section is the
following theorem.
\begin{theorem*}[Theorem~\ref{theor:01}] Each continuous
solution $\xi:\, [0,\, 1]\rightarrow [0,\, 1]$ of the
equation~(\ref{eq:31}) is piecewise linear.
\end{theorem*}
Section~\ref{pr-Sect-monotone} is devoted to the properties of the
piecewise linear solution $\xi:\, [0,\, 1]\rightarrow [0,\, 1]$ of
the functional equation~(\ref{eq:31}). We prove the following
facts.
\begin{proposition*}[Proposition~\ref{pr-Theor-1}] If a
continuous solution $\xi$ of the functional equation~(\ref{eq:31})
is monotone on some $M = [\alpha,\, \beta]$ then $\xi$ is
piecewise monotone on $[0,\, 1]$.
\end{proposition*}
\begin{theorem*}[Theorem~\ref{theor:02}]
If the function $\xi:\, [0,\, 1]\rightarrow [0,\, 1]$, which is a
solution of a functional equation~(\ref{eq:31}), is monotone on an
interval $M\subset [0,\, 1]$, then $\xi$ is piecewise linear on
$[0,\, 1]$.
\end{theorem*}

Theorems~\ref{theor:01} and~\ref{theor:02} make natural the
consideration of piecewise linear solutions $\xi:\, [0,\,
1]\rightarrow [0,\, 1]$ of the functional equation~(\ref{eq:31}).
Such its solutions are considered in
Section~\ref{sect:SemiConj-4}. We prove the following theorem
there.
\begin{theorem*}[Theorem~\ref{pr-theor-2}] Let $\xi$ be
piecewise linear solution of the functional
equation~(\ref{eq:08}). Then $\xi$ is one of the following
functions

\noindent 1. $\xi(x)=x_0$ for all $x\in [0,\, 1]$, where $x_0=0$,
or $x_0=2/3$;

\noindent 2. for some $k\in \mathbb{N}$
$$\xi(x)=\displaystyle{\frac{1 - (-1)^{[kx]}}{2}
+(-1)^{[kx]}\{kx\}},$$ where $\{\cdot \}$ denotes fractional part
and $[\cdot ]$ denotes integer part. More then this for any $k\in
\mathbb{N}$ the function $\xi(x)$ of the form above is a solution
of~(\ref{eq:08}).
\end{theorem*}

In Section~\ref{pr-sect-length} we consider the length of the
graph of the homeomorphic solution $h:\, [0,\, 1]\rightarrow [0,\,
1]$ of the functional equation~(\ref{eq:58}). In
subsection~\ref{pr-sect-length-1} we prove the following Theorem.
\begin{theorem*}[Theorem~\ref{pr-Theor-3}] Let $h:\,
[0,\, 1]\rightarrow [0,\, 1]$ be a conjugation of $f$ and $f_v$,
defined by~(\ref{eq:50}) and~(\ref{eq:53}) correspondingly, i.e.
$h$ is a solution of~(\ref{eq:58}).

Let $h_n$ are piecewise linear approximations of $h$, whose
breaking points belong to $A_n$ and which coincide with $h$ on
$A_n$. Define $l_n$ the length of the graph of $h_n$. Then
$$\lim\limits_{n\rightarrow \infty}l_n =2$$ for all $v\neq 1/2$.
\end{theorem*}
From another hand, we use Theorem~\ref{lema:32} in
Subsection~\ref{pr-sect-length-1} to find  the explicit formulas
of $l_n$ and prove the following proposition.
\begin{proposition*}[Proposition~\ref{prop:3}] The
following equality holds
\begin{equation}
\label{eq:36} l_{n+1}(v) =
\frac{1}{2^{n}}\cdot\sum\limits_{k=0}^nC_n^k \cdot
\sqrt{1+2^{2n}v^{2k}(1-v)^{2(n-k)}} \end{equation} for the length
$l_{n+1}(v)$ of $h_{n+1}$.
\end{proposition*}
Proposition~\ref{prop:3} together with Theorem~\ref{pr-Theor-3}
give the following corollary.
\begin{theorem*}[Theorem~\ref{theor:03}] For every $v\in
(0,\, 1)\backslash \{ 0.5\}$ the limit $ \lim\limits_{n\rightarrow
\infty}l_n(v)=2$ holds, where $l_n(v)$ are defined
by~(\ref{eq:36}).
\end{theorem*}
In Section~\ref{pr-sect-length-2} we prove the
Proposition~\ref{prop:3} as a combinatorial fact with the use of
probability reasonings. This is, in fact, an alternative proof of
Theorem~\ref{pr-Theor-3}.

In Section~\ref{sect-KuskLin} we consider maps $\xi:\,
A_n\rightarrow [0,\, 1]$ for a given $n\geq 1$. We call such $\xi$
\textbf{admissible}, if the equality
\begin{equation}
\label{eq:105} \xi(f(x)) = f(\xi(x)) \end{equation} for all $x\in
A_n$ and $f$ is of the form~(\ref{eq:50}). Notice, that since
$f(A_n)\subseteq A_n$ for all $n$, then~(\ref{eq:105}) is defined
for all $x\in A_n$. In Section~\ref{sect-KuskLin-1} we prove the
following theorem.
\begin{theorem*}[Theorem~\ref{theor:04}] There is one to
one correspondence between admissible self-semi conjugations
$\xi:\, A_n  \rightarrow [0,\, 1]$ and maps $\widetilde{\xi}:\,
\bigcup\limits_{i=1}^n\mathcal{B}_i \rightarrow
\bigcup\limits_{i=1}^n\mathcal{B}_i$ with the following
properties:

(1) For any $m,\, 1\leq m\leq n$, the inclusion
$\widetilde{\xi}(\mathcal{B}_m) \subseteq \mathcal{B}_m$ holds.

(2) For any $m,\, 2\leq m\leq n$ the equality $$
\widetilde{\xi}(j_1,\ldots,\, j_m) = (i_1,\ldots,\, i_m)
$$ yields $$
\widetilde{\xi}(j_1,\ldots,\, j_{m-1}) = (i_1,\ldots,\, i_{m-1}).
$$

(3) If the equality $$ \widetilde{\xi}(j_1,\ldots,\, j_n) =
(i_1,\ldots,\, i_n)
$$ holds for some $i_1,\ldots,\, i_n,\, j_1,\ldots,\, j_n$, then
for any $k,\, 1\leq k\leq n$ the equality $i_k=0$ yields
$j_k=i_0$, where $i_0\in \{0,\, 1\}$ is a fixed number.
\end{theorem*}
\begin{corollary*}[Corollary~\ref{corol:1}] For any
$n\geq 1$ the number of admissible self-semi conjugations $\xi:\
A_n \rightarrow [0,\, 1]$ is
$$\sum\limits_{k=0}^n2^{k+1}\cdot C_n^k$$.
\end{corollary*}
In Section~\ref{sect-KuskLin-2} we consider continuable maps
$\xi_n:\, A_n \rightarrow [0,\, 1]$. We call the maps $\xi_n:\,
A_n \rightarrow [0,\, 1]$ \textbf{continuable}, if the there is a
self-semiconjugation $h:\, [0,\, 1]\rightarrow [0,\, 1]$ of $f$
(i.e. $h$ is surjective and continuous, but not necessary
invertible), which coincides with $\xi_n$ on $A_n$. In this case
the maps $h$ can be considered as continuous surjective
continuation of $\xi_n$. By Theorems~\ref{theor:01}
and~\ref{pr-theor-2} if $\xi$ is continuable, then either
$\xi(x)=2/3$ for all $x\in A_n$, or $\xi(0)=0$. From the
definition of admissible maps and from the definition of $A_n$
obtain that $\xi(0)=0$ yields $\xi(A_n)\subseteq A_n$. We prove
the following theorem.
\begin{theorem*}[Theorem~\ref{theor-main-dopust}] 1. For
every $x\in A_n\setminus A_{n-1}$ and for every $y\in A_n$ there
exists a continuable $\xi: A_n\rightarrow A_n$ such that
$\xi(x)=y$.

2. Let $\xi_1,\, \xi_2:\, A_n \rightarrow A_n$ be continuable
self-semiconjugations of $f$ of the form~(\ref{eq:50}) and
$\xi_1(x) = \xi_2(x)$ for some $x\in A_n\setminus A_{n-1}$. Then
$\xi_1(x)=\xi_2(x)$ for all $x\in A_n$.
\end{theorem*}
\begin{corollary*}[Corollary~\ref{corol:6}] For every
$n\geq 1$ there are $2^{n-1}$ continuable self-semi conjugations
of $f$ of the form~(\ref{eq:50}).
\end{corollary*}

\newpage
\section{Introduction}\label{sect:Vstup}

\subsection{The main definitions}

As integer numbers in arithmetics is a mathematical tool of the
description of some objects, the dynamical system in dynamical
systems theory is a tool of the description of those objects which
are changed dependently on time. As an example of dynamical
systems is an experiment in Newtonian mechanics. When a system of
points is not under acting of any forces and velocities and
coordinates of all the points are known, then they (i.e.
coordinates and velocities of all points) can be found at any
other time. Dynamical systems have the following property. The
further extension of events in dynamical system is dependent only
on its state and independent on external factors, for example is
independent on previous states of a system or ways with which the
system has come to its now state. The properties mentioned above
can be easily be formalized (translated into mathematical
terminology), which let give the strict definition. Let $X$ be a
set of possible states of a system (i.e. in Newtonian mechanics it
is a set of arrays of the length 6 with triples of point
coordinates triples of its velocities coordinates). Let $T$ be a
set of possible values of a variable ``time'' (for example
$\mathbb{R}^+$) but in any way $T$ is a semi group i.e.  its
elements can be added one to another. Consider a maps
\begin{equation} f:\, X\times T \rightarrow X,
\end{equation} whose acting is that an state $f(x,\,
t)$ of a system is that its state, where is comes after the time
$t$ if $x$ is its the former state. Maps $f$ is obviously have
some obvious properties: 1) For every $x\in X$ the equality
$f(x,\, 0) =x$ hold i.e. during the time $0$ the state of the the
system does not change. 2) If the system which at the very
beginning (at the time $0$) was at the state $x_0$ and after the
time $t_1$ it appeared itself in the state $x_1 = f(x,\, t_1)$
then after the time $t_2$ it will appear in the state $f(x_1,\,
t_2) = f(f(x_0,\, t_1),\, t_2).$ Nevertheless the sate of the
system after the time $t_1$ does not influence the state in which
it appeared itself after the time $t_1+t_2$ after being in the
state $x_0$  i.e. fact of defining of the state of the system does
not influence its further extension. So, the equality
$$f(x_0,\, t_1+t_2) = f(f(x_0,\, t_1),\, t_2)$$ should hold.

\begin{definition}
\textbf{Dynamical system} is a triple $(X, T, f)$, where $X$ is a
set, $T$ is an additive semigroup (i.e. $0 \in T, \forall s,t \in
T, s + t = t + s \in T$) and $f$ is a maps which acts $f : X
\times T \to X$ such that the following properties hold:

1) $f(x,0) \equiv x,\ \forall x \in X$;

2) $f(f(x,s),t) = f(x,s+t),\ \forall x \in X, \forall s,t \in T$.
\end{definition}

\begin{definition}
Consider a point $x_0\in X$ of a phase space. The set $A\subset X$
is called an \textbf{orbit} of a point $x_0$ if for any point
$a\in A$ there exists a time $t\in T$ such that after this time
the point $x_0$ goes into the point $a$, i.e. the equality
$a=f(x_0,\, t)$ holds.\index{Orbit}.
\end{definition}

\begin{definition}
Consider a point $x_0\in X$ of a phase space. A function $N:\, T
\rightarrow X$ is called a \textbf{trajectory} of the point $x_0$
if $N(t) = f(x,\, t)$\index{Trajectory}.
\end{definition}

Note that in the case if dynamical system is a cascade (i.e. if
the time may be considered as a natural systems set) then
trajectory of a point $x_0$ is s sequence which is defined with
the recurrent equality $x_{k+1} = f(x_k)$ and the orbit of a point
$x_0$ is the value set of its trajectory.

\begin{definition}
A function $t \to f(x,t)$ is called a \textbf{motion} of a point
$x$, i.e. the motion of a point $x$ is called a function which for
any value of time corresponds the position of $x$ at this time.
\end{definition}

The impotent property of trajectories of phase space points which
appears (or not appears) during considering dynamical systems is
returning of the point into itself or into its neighborhood (in
the case if dynamical system is considered on either metric or
topological phase space).

\begin{definition}
Consider a point $x_0\in X$ of a phase space. This point is called
a \textbf{fixed point} of the dynamical system $(X,\, T,\, f)$ if
for arbitrary $t\in T$ the equality $f(x_0,\, t) = x_0$
holds.\index{Fixed point}.
\end{definition}

\begin{definition}
Consider a point $x_0\in X$ of a phase space. This point is called
a \textbf{periodic point} of the dynamical system $(X,\, T,\, f)$
\textbf{with period} $t_0$ if for arbitrary $t\in T$ the equality
$f(x_0,\, t_0) = x_0$ holds and for any $t<t_0$ the equality
$f(x_0,\, t) = x_0$ does not hold.\index{Periodic
point}\index{Point period}
\end{definition}

To make possible the defining the dynamical system it is necessary
to demand that elements of the set $T$ be comparable i.e. that
inequality $t<t_0$ make sense.

\begin{definition}\label{def:iteration}
Consider the maps $f$ of some set $A$ into itself. The $n$-th
\textbf{iteration} of $f$ for arbitrary non negative integer $n$
is called the maps $f^{n}(x)$ which is defined as follows:
$f^0(x)=x$, $f^k(k) = f(f^{k-1}(x))$ for arbitrary $k\in
\mathbb{N}.$\index{Maps iteration}
\end{definition}

\begin{definition}
For a set $A$ let $f:\, A\rightarrow A$ be a map. For every
$x_0\in A$ the sequence $\{x_n\},\, x\geq 0$ such that
$x_{n+1}=f(x_n)$ for every $n\geq 0$ is called the
\textbf{trajectory} of $x_0$ under the action of $f$.
\end{definition}

\begin{definition}
Let a map $f:\, A\rightarrow A$ be defined on a set $A$. For any
$x_0\in A$ the set $\{x_n\},\, x\in \mathbb{Z}$ such that
$x_{n+1}=f(x_n)$ for any $n\in \mathbb{Z}$ is called the
\textbf{integer trajectory} of $x_0$ under $f$.
\end{definition}

For an arbitrary $n\geq 1$ denote the $n$-th iteration of a maps
$f$ by $f^n$, i.e. $$ f^n(x) = \underbrace{f(f(\ldots (f}_{n\text{
times}}(x))\ldots )).
$$ We will use this notation not only for the maps $f$, but
also for those maps which are denoted in any other way.

\subsection{The notion about topological equivalence}

George Birkgoff is an American mathematician who lived in the
first half of 20-th century. He is one of founders of dynamical
systems theory and he has formulated the final problem of
dynamical systems theory in the following way: ``qualitative
determine of all possible trajectories types and stating the
interconnections between them''(see.~\cite{Birk-1999}, p.~194).

The necessity of considering the equivalent maps naturally implies
from the final problem the dynamical systems theory stated by G.
Birkgoff. For possibility of talking about trajectories types it
is necessary to study ourself to point out those systems which are
equal prom the point of view of those questions which are stated
in the dynamical systems theory and these questions after all
technical moments need to determine all trajectory types and to
state the interconnections between them.

For the case of dynamical systems, which are defined on the
interval $I=[a,\, b]$ the topological conjugation may by
considered as the changing of the scope on $I$. Let a map $g$ be
topologically conjugated to $f:\, I\rightarrow I$.

Assume, that we have a spring, whose length equals to the length
of $I$ and this spring is graduated, i.e. the numbers, which
correspond to $I$ are written uniformly on the it. Then ends of
the spring are fixed, and some its parts are stretched and some
are griped without knots, kinks and self intersections, i.e. the
spring is sketched and griped without taking away from the line,
where it was at the very beginning.

If the sketched and griped spring is graduated again, then we
obtain the monotone continuous increasing map $h:\, I\rightarrow
I$. This $h$ defines the topological conjugation of then maps $f$
and $g$. For every ``old'' point $x\in I$ the maps $h$ sets the
``new'' point $h(x)\in I$, which corresponds to the new
graduation.

Consider the maps $f$ as not interval $I$ into itself maps, but
the spring into itself maps. Then different graduations of the
spring give different interval $I\rightarrow I$ maps, and in these
terms, our new graduation defines the maps $g$, which is
conjugated to $f$ via $h$.

The notion of topological conjugation can be introduced not only
to interval maps, but for every dynamical systems. Assume that $h$
is a homeomorphism (i.e. continuous and invertible). Nevertheless,
interval homeomorphism can be as increasing, as decreasing.
Returning to spring as an interpretation of topological
conjugation of the interval maps, we should let to change the ends
of the spring.

The exact definition of the topological equivalence is following.
\begin{definition}\label{def:topol-equiv}
A maps $f$ of a set $A$ into itself is called
\textbf{topologically conjugated} to a maps $g$ of a set $B$ into
itself if there exists a homeomorphism $h$ of the set $A$ into the
set $B$ such that the following diagram is commutative.
\begin{equation}\label{eq:42}\begin{CD}
A @>f>> & A\\
@V_hVV& @VV_hV\\
B @>g>>& B
\end{CD}\end{equation}\end{definition}

Notice, that topological conjugation is an equivalence relation,
whence topologically conjugated maps are also called
\textbf{topologically equivalent}.
\begin{definition}
A maps $f$ of a set $A$ into itself is called
\textbf{topologically semi conjugated} (or \textbf{topologically
semi equivalent}) to a maps $g$ of a set $B$ into itself if there
exists a continuous surjective $h:\, A\rightarrow B$ such
that~(\ref{eq:42}) is commutative.
\end{definition}

The following theorem holds.

\begin{theorem}
For using the notations of the definition~\ref{def:topol-equiv}
the maps $f$ if topologically equivalent to the maps $g$. Let
$a\in A$ be a fixed point of $f$. The $h(a)$ is a fixed point of
$g$.
\end{theorem}

\begin{proof}
The commutativity of the diagram from the
definition~\ref{def:topol-equiv} yields that $g(h(a)) = h(f(a))$.
Taking into attention that $a$ is a foxed point of $f$ obtain that
$$ g(h(a)) = h(a),
$$ which finishes the proof.
\end{proof}

In the same manner we prove the following theorem.

\begin{theorem}
For using the notations of the definition~\ref{def:topol-equiv}
the maps $f$ is topologically equivalent to the maps $g$. Let
$a\in A$ be a periodic point of the maps $f$ with period $n$. Then
$h(a)$ is a periodic point of the maps $g$ with period $n$.
\end{theorem}

\begin{theorem}
For using the notations of the definition~\ref{def:topol-equiv}
let the maps $f$ be topologically equivalent to the maps $h$. For
each point $a\in A$ the equality of sets $$ h(f^{-1}(a)) =
g^{-1}(h(a))
$$ holds.
\end{theorem}

\begin{proof}
Consider an arbitrary point $x\in f^{-1}(a)$. The conditions of
theorem yield that the following diagram is commutative.
$$\begin{CD}
x @>f>> & f(a)\\
@V_hVV& @VV_hV\\
h(x) @>g>>& h(a)
\end{CD}$$

From the commutativity of the diagram obtain that $h(x)\in
g^{-1}(h(a))$ and this inclusion gives the following sets
inclusion$$ h(f^{-1}(a)) \subseteq g^{-1}(h(a)).
$$

Now consider an arbitrary point $y\in g^{-1}(h(a))$. Note that as
$h$ is a homeomorphism then there exists a maps $h^{-1}$ and so
that point $h^{-1}(y)$ is determined. With using the point
$h^{-1}(y)$ get that following diagram is commutative.
$$\begin{CD}
h^{-1}(y) @>f>> & f(a)\\
@V_hVV& @VV_hV\\
y @>g>>& h(a).
\end{CD}$$

The commutativity of the diagram yields that $y\in h(f^{-1}(a))$
and this inclusion gives the sets inclusion
$$ h(f^{-1}(a)) \supseteq g^{-1}(h(a)).
$$
The last finishes the proof.
\end{proof}

\begin{theorem}
For using the notations of the definition~\ref{def:topol-equiv}
the maps $f$ be topologically equivalent to $g$. Then for every
$n\geq 2$ the diagram
$$\begin{CD}
A @>f^n>> & A\\
@V_hVV& @VV_hV\\
B @>g^n>>& B
\end{CD}$$ commutes. Here the power signs mean the correspond
iteration of the maps.
\end{theorem}

\begin{proof}
Commutative diagram~(\ref{eq:42}) can be continued to the right as
follows
$$
\begin{CD}
A @>f >> & A @>f >> \ldots @>f >> A\\
@V_{h} VV& @VV_{h}V && @VV_{h}V \\
B @>g>> & B @>g >> \ldots @>g >> B,\\
\end{CD}
$$ and the obtained diagram is also commutative. If necessary
functional equation is obtained is consider the passes from the
left top angle by external sides of the rectangle.
\end{proof}

The topological equivalence of interval maps can be also
considered as their graph transformation.

Let it is known that for a maps $f$ and for a point $x_1\in I$ the
equality $f(x_1)=x_2\in I$ holds. From the graph definition this
means that the point $(x_1,\, x_2)$ belongs to the graph of the
function $f$. The commutative diagram from the
definition~\ref{def:topol-equiv} yields the equality $g(h(x_1)) =
h(x_2)$.
\begin{definition}
Call the \textbf{graduation} of any given interval $AB$ of a line
call an one to one correspondence between its points and points of
some real numbers line interval $[a,\, b]$ such that number $a$
corresponds to point $A$, number $b$ corresponds to point $B$ and
for every triple of points $C_1,\, C_2,\, C_3$ such that $C_2$ is
between $C_1$ and $C_3$ corresponds a triple of real numbers such
that those which corresponds $C_3$ id also between images of $C_1$
and $C_3$.
\end{definition}

\begin{definition}
If a graduation of an interval $AB$ has the property that for
subintervals of $AB$ which have equal lengthes correspond the
pairs of numbers with the same differences then call such
graduation a \textbf{continuous graduation}.
\end{definition}

Consider the graph of the maps $f$ which is defined on the
interval $I$ and consider this graph to be plotted in uniformly
graduated Cartesian coordinate plane. The note that we can
understand the in the following way the statement that a point of
the graph of $f$ with coordinates $(x_1,\, x_2)$ corresponds to
some points of the graph of maps $g$ with coordinates $g(h(x_1)) =
h(x_2)$. The graph of the maps $f$ (i.e. set o plane points) we
will consider as a graph of the maps $g$ but with usage of the
following non uniform graduation. For the points of uniformly
graduated interval $I$ of $x$-axis there is a correspondence
$x\leftrightarrow h(x)$ and for uniformly graduated interval $I$
of the $y$-axis there is correspondence which is defined with the
rule $y\leftrightarrow h(y)$. With the use of such graduation that
set of points of the plane which was the graph of the maps $f$
with uniform graduation will become the graph of the maps $g$
under the described new graduation.

\newpage

\subsection{Examples of topological equivalent maps}\label{subcst-1-5}

\begin{example}\label{exmpl-pobud-equiv}
Consider a maps $f$ of an $I=[0,\, 1]$ which is determined with
the formula
$$ f(x) = \left\{\begin{array}{ll}  2x,& x\leqslant 1/2;\\
 2-2x,&
x>1/2.
\end{array}\right.
$$ and find a new maps $g$ which is topologically equivalent to
it and is defined with the following maps $h$
$$ h(x) = \left\{\begin{array}{ll}  1,5x& x\leqslant 1/2;\\
 0,5x+0,5&
x>1/2.
\end{array}\right.
$$ from the definition of topological equivalence.
In this case during constructing the maps $g$ we will use
transformation of graphs.

The graph of the maps $f$ is given at the picture~\ref{fig-1}a).
\end{example}

\begin{proof}[Core of the example:] Plot the graph of
the maps $f$ and also plot the grid which will describe the
uniform graduation (see pict.~\ref{fig-1}b)).

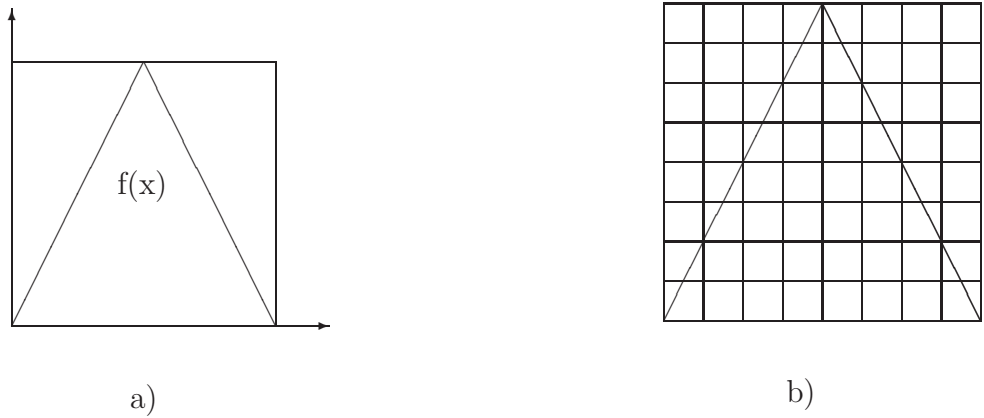
\begin{figure}[htbp]
\begin{minipage}[h]{0.49\linewidth}
\vskip 5mm \begin{center}
\begin{picture}(100,125)
\put(100,0){\line(0,1){100}} \put(0,100){\line(1,0){100}}
\put(0,0){\vector(1,0){120}} \put(0,0){\vector(0,1){120}}
\put(0,0){\line(1,2){50}} \put(50,100){\line(1,-2){50}}
\put(40,50){f(x)}
\end{picture}
\end{center}
\centerline{a)}
\end{minipage}
\hfill
\begin{minipage}[h]{0.49\linewidth}
\begin{center}
\begin{picture}(100,135)
\put(120,0){\line(0,1){120}} \put(0,120){\line(1,0){120}}
\put(0,0){\line(1,0){120}} \put(0,0){\line(0,1){120}}
\put(0,0){\line(1,2){60}} \put(60,120){\line(1,-2){60}}
\put(60,120){\line(1,-2){60}}

\put(15,0){\line(0,1){120}} \put(30,0){\line(0,1){120}}
\put(45,0){\line(0,1){120}} \put(60,0){\line(0,1){120}}
\put(75,0){\line(0,1){120}} \put(90,0){\line(0,1){120}}
\put(105,0){\line(0,1){120}}

\put(0,15){\line(1,0){120}} \put(0,30){\line(1,0){120}}
\put(0,45){\line(1,0){120}} \put(0,60){\line(1,0){120}}
\put(0,75){\line(1,0){120}} \put(0,90){\line(1,0){120}}
\put(0,105){\line(1,0){120}}
\end{picture}
\end{center}
\centerline{ b)}
\end{minipage}
\caption{Graphs of $f$} \label{fig-1}
\end{figure}

Now with out changing the graph (with out changing the set of
points which is the graph) we will change the graduation. Use as
at former plot 7 vertical and 7 horizontal lines bet plot them not
uniformly but as it is shown on the plot.

For example as $h(0,5) = 0,75$ then point 0.75 of $x$-axis
corresponds to the middle of the real line segment $[0,\, 1]$.
After this both left and right part of the interval will be
graduated uniformly each. Plot vertical lines which corresponds to
graduated values $\frac{1}{8},\ldots,\, \frac{7}{8}$. Naturally
that in this case right hand part of the $x$-axis will be divided
into 2 parts with one additional vertical line and left hand part
will be divided to 6 parts with five additional vertical lines.
\begin{figure}[htbp]
\begin{minipage}[h]{0.49\linewidth}
\begin{center}
\begin{picture}(100,135)
\put(120,0){\line(0,1){120}} \put(0,120){\line(1,0){120}}
\put(0,0){\line(1,0){120}} \put(0,0){\line(0,1){120}}
\put(0,0){\line(1,2){60}} \put(60,120){\line(1,-2){60}}
\put(60,120){\line(1,-2){60}}

\put(10,0){\line(0,1){120}} \put(20,0){\line(0,1){120}}
\put(30,0){\line(0,1){120}} \put(40,0){\line(0,1){120}}
\put(50,0){\line(0,1){120}} \put(60,0){\line(0,1){120}}
\put(90,0){\line(0,1){120}}

\put(7,-15){$\frac{1}{8}$} \put(17,-15){$\frac{2}{8}$}
\put(27,-15){$\frac{3}{8}$} \put(37,-15){$\frac{1}{2}$}
\put(47,-15){$\frac{5}{8}$}

\put(57,-15){$\frac{3}{4}$} \put(87,-15){$\frac{7}{8}$}

\put(0,10){\line(1,0){120}} \put(0,20){\line(1,0){120}}
\put(0,30){\line(1,0){120}} \put(0,40){\line(1,0){120}}
\put(0,50){\line(1,0){120}} \put(0,60){\line(1,0){120}}
\put(0,90){\line(1,0){120}}
\end{picture}
\end{center}
$ $\\
\centerline{a)}
\end{minipage}
\hfill
\begin{minipage}[h]{0.49\linewidth}
\begin{center}
\begin{picture}(100,135)
\put(120,0){\line(0,1){120}} \put(0,120){\line(1,0){120}}
\put(0,0){\line(1,0){120}} \put(0,0){\line(0,1){120}}
\put(0,0){\line(1,2){60}}

\qbezier(60,120)(70,60)(80,0) \put(80,0){\line(0,1){120}}

\put(10,0){\line(0,1){120}} \put(20,0){\line(0,1){120}}
\put(30,0){\line(0,1){120}} \put(40,0){\line(0,1){120}}
\put(50,0){\line(0,1){120}} \put(60,0){\line(0,1){120}}
\put(70,0){\line(0,1){120}}

\put(7,-15){$\frac{1}{8}$} \put(17,-15){$\frac{2}{8}$}
\put(27,-15){$\frac{3}{8}$} \put(37,-15){$\frac{1}{2}$}
\put(47,-15){$\frac{5}{8}$} \put(57,-15){$\frac{3}{4}$}
\put(67,-15){$\frac{7}{8}$} \put(77,-15){$1$}

\put(-8,6.5){$\frac{1}{8}$} \put(-18,16.5){$\frac{2}{8}$}
\put(-8,26.5){$\frac{3}{8}$} \put(-18,36.5){$\frac{4}{8}$}
\put(-8,46.5){$\frac{5}{8}$} \put(-18,56.5){$\frac{6}{8}$}
\put(-8,86.5){$\frac{7}{8}$}

\put(0,10){\line(1,0){120}} \put(-10,20){\line(1,0){130}}
\put(0,30){\line(1,0){120}} \put(-10,40){\line(1,0){130}}
\put(0,50){\line(1,0){120}} \put(-10,60){\line(1,0){130}}
\put(0,90){\line(1,0){120}}

\put(30,60){\circle*{4}} \put(70,60){\circle*{4}}
\put(60,120){\circle*{4}} \put(21,62){A} \put(71,62){C}
\put(55,125){B}
\end{picture}
\end{center}
 $ $\\ \centerline{ b)}
\end{minipage}
\caption{Construction of the conjugated map}\label{fig-2}
\end{figure}
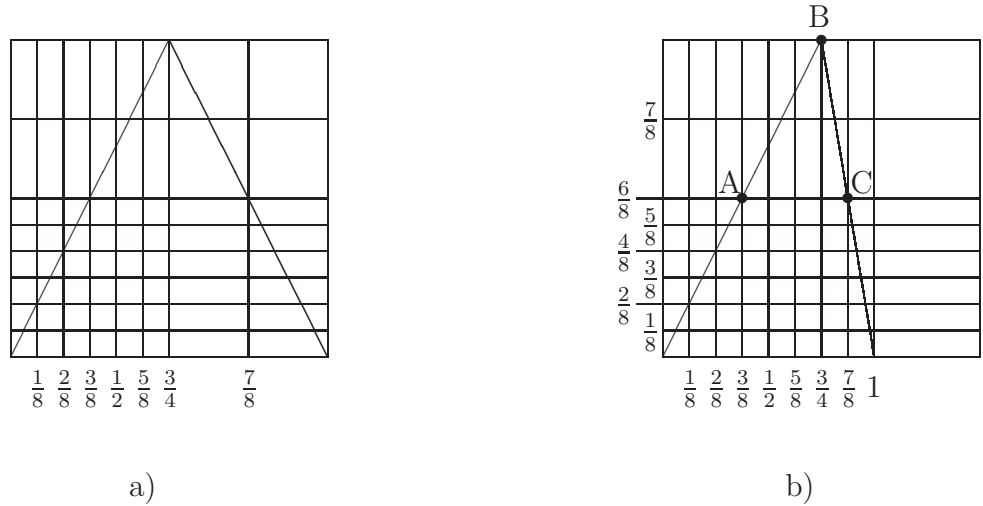

For obtaining the natural form of the graph of the maps $g$
squeeze the obtained picture (i.e. squeeze the picture with
considering it as a geometrical figure which is composed with
vertical lines) to make the graduation uniform. The new form of
the plot is presented of the picture (see fig.~\ref{fig-2}a). Now
repeat the same with the $y$-axis and obtain the graph of the maps
$g$. For doing this just for convenience note points $A,\, B,\, C$
on the picture. Points $A$ and $C$ stay as they are but points $B$
will move vertically down without changing its horizontal position
(see fig.~\ref{fig-2}b).
\begin{figure}[htbp]
\begin{minipage}[h]{0.49\linewidth}
\begin{center}
\begin{picture}(100,135)
\put(120,0){\line(0,1){120}} \put(0,120){\line(1,0){120}}
\put(0,0){\line(1,0){120}} \put(0,0){\line(0,1){120}}
\put(0,0){\line(1,2){30}}

\qbezier(70,60)(75,30)(80,0) \put(80,0){\line(0,1){120}}

\put(30,60){\line(3,2){30}} \put(60,80){\line(1,-2){10}}

\put(10,0){\line(0,1){120}} \put(20,0){\line(0,1){120}}
\put(30,0){\line(0,1){120}} \put(40,0){\line(0,1){120}}
\put(50,0){\line(0,1){120}} \put(60,0){\line(0,1){120}}
\put(70,0){\line(0,1){120}}

\put(7,-15){$\frac{1}{8}$} \put(17,-15){$\frac{2}{8}$}
\put(27,-15){$\frac{3}{8}$} \put(37,-15){$\frac{1}{2}$}
\put(47,-15){$\frac{5}{8}$} \put(57,-15){$\frac{3}{4}$}
\put(67,-15){$\frac{7}{8}$} \put(77,-15){$1$}

\put(-8,6.5){$\frac{1}{8}$} \put(-18,16.5){$\frac{2}{8}$}
\put(-8,26.5){$\frac{3}{8}$} \put(-18,36.5){$\frac{4}{8}$}
\put(-8,46.5){$\frac{5}{8}$} \put(-18,56.5){$\frac{6}{8}$}
\put(-8,66.5){$\frac{7}{8}$} \put(-18,76.5){$1$}

\put(0,10){\line(1,0){120}} \put(-10,20){\line(1,0){130}}
\put(0,30){\line(1,0){120}} \put(-10,40){\line(1,0){130}}
\put(0,50){\line(1,0){120}} \put(-10,60){\line(1,0){130}}
\put(0,70){\line(1,0){120}} \put(-10,80){\line(1,0){130}}

\put(30,60){\circle*{4}} \put(70,60){\circle*{4}}
\put(60,80){\circle*{4}}

\put(21,62){A} \put(71,62){C} \put(51,83){B}

\end{picture}
\end{center}
 $ $\\
\centerline{a)}
\end{minipage}
\hfill
\begin{minipage}[h]{0.49\linewidth}
\begin{center}
\begin{picture}(100,125)
\put(100,0){\line(0,1){100}} \put(0,100){\line(1,0){100}}
\put(0,0){\line(1,0){120}} \put(0,0){\line(0,1){120}}
\put(0,0){\line(1,1){100}}

\put(0,0){\circle*{4}} \put(37.5,75){\circle*{4}}
\put(75,100){\circle*{4}} \put(87.5,75){\circle*{4}}
\put(100,0){\circle*{4}}

\put(0,0){\line(1,2){37.5}} \put(37.5,75){\line(3,2){37.5}}
\put(75,100){\line(1,-2){12.5}}

\qbezier(87.5,75)(93.75,37.5)(100,0)

\end{picture}
\end{center}
 $ $\\
\centerline{ b)}
\end{minipage}
\caption{Construction of the conjugated map} \label{fig-3}
\end{figure}
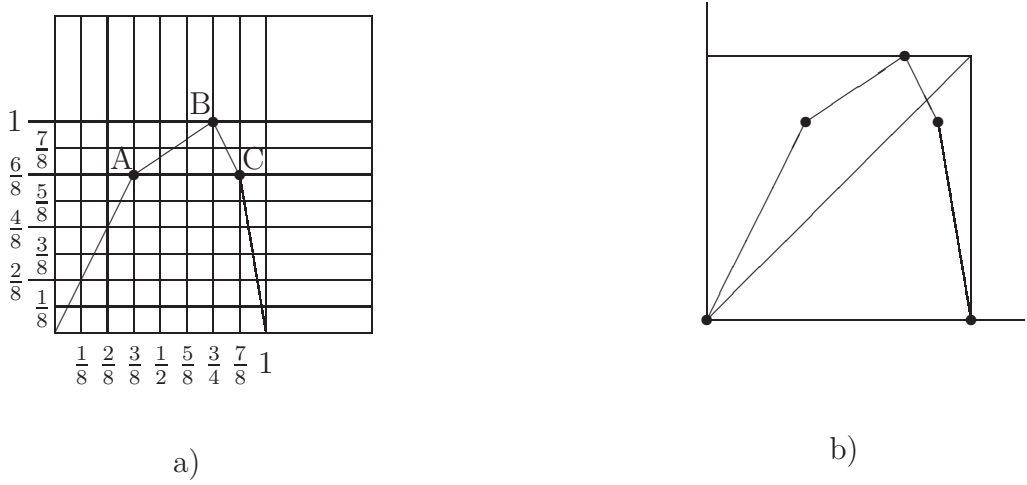

The graph which is obtained in such a way is the graph of maps
$g$.\end{proof}

\begin{example}\label{exmpl-pobud-analit}
Consider the maps $f$ and $h$ such as in the
example~\ref{exmpl-pobud-equiv}. Construct the maps $g$ which is
topologically equivalent to $f$ but use the definition of
topological equivalence and possibility of an analytical
representation of $f$ and $h$ and take into attention the
commutativity of a diagram from the
definition~\ref{def:topol-equiv}.
\end{example}

The analytical representation of this maps looks as
$$
\begin{array}{ll}
h(x) = \left\{\begin{array}{ll}  1,5x& x\leqslant 1/2;\\
 0,5x+0,5&
x>1/2.
\end{array}\right.& h^{-1}(x) = \left\{\begin{array}{ll}
2/3x& x\leqslant 3/4;\\
 2x -1& x>3/4.
\end{array}\right.
\end{array}
$$

Since for $x<3/4$ the condition $h^{-1}(x)<0,5$ hold then for
these values of $x$ the condition $f(h^{-1}(x)) = 2h^{-1}(x) =
4/3x$ also holds. It is obvious that $f(h^{-1}(x))<1/2$ for
$x<3/8$. For these values of $x$ the equality $h(f(h^{-1}(x))) =
3/2f(h^{-1}(x)) = 2x$ holds. whence the graph of the maps $g$
passes through the points $(0,\, 0)$ and $(3/8,\, 3/4)$. For $x\in
[3/8,\, 3/4)$ then $f(h^{-1}(x)) \geqslant 1/2$ and whence
$h(f(h^{-1}(x))) = 0,5f(h^{-1}(x))+0,5 = 0,5\cdot 4/3x + 0,5 =
2/3x+0,5$. This yields that the graph of the maps $g$ passes
through points $(3/8,\, 3/4)$ and $(3/4,\, 1)$. If $x\geqslant
3/4$, then $h^{-1}(x)=2x+0,5\geqslant 1/2$, whence $f(h^{-1}(x)) =
2-2h^{-1}(x) = 2-2(2x-1) = 4-4x$. For $x<7/8$ the inequality
$4-4x>1/2$ holds whence $h(f(h^{-1}(x))) = 0,5(f(h^{-1}(x))) +0,5
= 0,5(4-4x) +0,5 =-2x +2,5$. So the graph of the maps $g$ passes
through points $(3/4,\, 1)$ and $(7/8,\, 3/4)$. If $x\geqslant
7/8$ then the inequality $4-4x<1/2$ holds whence $h(f(h^{-1}(x)))
= 2(f(h^{-1}(x))) = 2(4-4x) = 8-8x$. So the graph of maps $g$ will
pass through points $(7/8,\, 3/4)$ and $(1,\, 0)$. The graph of
constructed maps $g$ is presented of the plot~\ref{fig-3}. The
calculations which are presented above in this example are a bit
huge and there are a lot of possibilities for technical mistakes
if use of this method. More then this, the presenter method does
not give possibility to catch the mistakes if they would really be
made. Nevertheless the presented example is useful because lets to
pay attention to some features of topological equivalence and we
will discuss them in the proposition~\ref{theor:ksklin-tpleqv}
below.

\begin{proposition}\label{theor:ksklin-tpleqv}
Consider continuous piecewise linear maps $f$ which maps the
interval $[0,\, 1]$ into itself such that $f(0)=0$. Consider also
increasing continuous piecewise linear maps $h$ which maps the
interval $[0,\, 1]$ into itself and defines the topological
equivalence of maps $f$ and $g$. Then $g(0)=0$ and $g'(0)=f'(0)$.
\end{proposition}

The example~\ref{exmpl-pobud-analit} and correspond huge
calculations illustrates the theorem. Nevertheless it is easy to
prove the proposition in general case.
\begin{proof}
Let in the neighborhood of its fixed points $0$ the maps $f$ is of
the form $f(x)=ax$ and analytical form of the maps $h$ in the
neighborhood of this points is of the form $h(x)=bx$.

Then there is some neighborhood of $0$ such that analytical for of
$h^{-1}$ in it looks as $h^{-1}(x)=1/bx$ whence in the
intersection of these neighborhoods one have
$$g(x) = h(f(h^{-1}(x))) = 1/b\cdot a\cdot bx = ax,
$$ which is necessary.
\end{proof}

\begin{lemma}
If a maps $g_1$ is topologically equivalent to a maps $g_2$ and
maps $g_2$ is topologically equivalent to a maps $g_3$ then the
maps $g_1$ is topologically equivalent to the maps $g_3$.
\end{lemma}

\begin{proof}
This lemma is the corollary of the definition of the topological
equivalence which can be shown with the following commutative
diagram $$
\begin{CD}
[0,\, 1] @>g_1 >> & [0,\, 1]\\
@V_{p_1} VV& @VV_{p_1}V\\
[0,\, 1] @>g_2>> & [0,\, 1]\\
@V_{p_2} VV& @VV_{p_2}V\\
[0,\, 1] @>g_3>> & [0,\, 1]
\end{CD} $$ In this case the proposition of the lemma and
definition of the topological equivalence yield one from another.
\end{proof}

\newpage

\section{The Historical review}\label{Sect:IstorOgl}

In the study of the population theory Piere Verhulst proposed
in~\cite{Verhulst-1} the differential equation
\begin{equation}\label{ogl-eq:15} \frac{Mdp}{pdt} = m-np,
\end{equation} which leads to the difference equation
$$M p_{k+1} = p_k(m-np_k).
$$
It was P.Verhulst, who called the \underline{\emph{logistic
curve}} the graph of the solution of the differential
equation~(\ref{ogl-eq:15}). He explained that the name
``logistic'' came from that finding the value of this curve needs
a lot of mathematical calculations and logists in ancient Greece
were those people, who were doing calculations. Also P.~Verhulst
studied the modelling of population at~\cite{Verhulst}.

Due to Verhulst, the following maps $\widetilde{f}_\lambda:\,
\mathbb{R}\rightarrow \mathbb{R}$, which is dependent on a
parameter $\lambda$ is called the \underline{\emph{logistic map}},
where
\begin{equation}\label{ogl-eq:13}
g_\lambda(x) = \lambda x(1-x).
\end{equation}
Notice, that if $\lambda\in [0,\, 4]$, then $g_\lambda([0,\,
1])\subseteq [0,\, 1]$ and is interesting from the point of
Dynamical Systems Theory point of view for $x\in [0,\, 1]$.

According to~\cite[p. 226]{Kautz}, Birkchoff was the first who
robust chaos  in iterated map in his 1932 paper~\cite{Birkchoff}
``Sur quelques courbes ferm\'ees remarquables''. The Oxford
University mathematicians Theodore Chaundy (1889 - 1971) and Eric
Philips were the first to explore the logistic map as a function
of the growth parameter $\lambda$ (see~\cite{Chaudy}). In 1936
they reported that, in the limit of the large time, $X_i$
approaches $0$ for $0< \lambda< 1$, approaches $1-1/\lambda$ for
$1<\lambda<3$, and ``oscillates finitely'' for $3<\lambda<4$. In
1970-s two groups studied periodical oscillations in this regime:
Nicholas Metropolis (born 1915), Myron Stein and Pual Stein at Los
Alamos National Laboratory and Robert May (later Lord May, born
1936) at Princeton University.

For $0 \leq\lambda \leq 1$, all orbits converge to $0$. For $1 <
\lambda \leq 3$, all orbits starting at $x_0 > 0$ converge to $1 -
1/\lambda$. For $ 3< \lambda \leq 1 + \sqrt{6}$, the orbits
converge to a cycle of period $2$.

For $\lambda > 1 + \sqrt{6}$, the system goes through a whole
sequence of period doubling. Let he values $\lambda_n$ denote the
parameter for which the $n$-th period doubling occurs. Then
$\lambda_n$ obey the law $$ \lim\limits_{n\rightarrow
\infty}\frac{\lambda_n-\lambda_{n-1}}{\lambda_{n+1}-\lambda_n} =
\delta = 4,669\ldots\, ,
$$ where $\delta$
is called the Feigenbaum constant. In 1978,
Feigenbaum~\cite{Feigenbaum-78} (also see~\cite{Feigenbaum-79}) as
well as Coullet and Tresser~\cite{Coullet-Tresser-78}
independently outlined an argument showing that such period
doubling cascades should be observable for a large class of
systems, and that the constant $\delta$ is universal. For
instance, it also appears in the two-dimensional H\'enon map $$
\left\{ \begin{array}{ll} x_{n+1} = 1 -\lambda
x_n^2 +y_n\\
y_{n+1} = bx_n & 0<b<1.
\end{array}\right.
$$

Rigorous proofs of these properties were later worked out by
Collet, Eckmann, Koch, Lanford and others in~\cite{Collet}.

Consider the difference logistic equation in the form
\begin{equation}\label{ogl-eq:2a}x_{n+1} = \lambda x_n(1-x_n).
\end{equation}
M. Ranferi Guti\'errez, M.A. Reyes, and H.C. Rosu
(see~\cite{Ranferi} and~\cite[p. 918]{Wolfram}) say, that nowadays
evident solutions of~(\ref{ogl-eq:2a}) are known only for
$\lambda=-2,\, 2$ and $4$. These solutions can be written in the
form
$$ x_n = \frac{1}{2}\left( 1- v_\lambda(r^n
v_\lambda^{-1}(1-2x_0))\right)
$$ and correspond $g_r$ are as follows:
$$
v_{-2}(x) = 2\cos\left(\frac{1}{3}(\pi -\sqrt{3}x)\right);
$$
$$
v_2(x) = e^x;
$$$$
v_4(x)=\cos x.
$$

The logistic equation has since been applied to a wide range of
phenomena including spread of technological change,
innovations~\cite{Rogers} new product diffusion within
markets~\cite{Bass-1969} diffusion of social change~\cite{Coleman}
and diffusion of epidemics~\cite{Ross}.

We have found at~\cite{Kautz}, that it was John Herschel, who
obtained at first the topological conjugation of logistic and tent
map. R. Kautz write the following:  \emph{As it happens, that
$\lambda = 4$} (i.e. logistic equation $x_{n+1} = 4x_n(1-x_n)$)
\emph{is mathematically especially simple in spite of being
chaotic. Its simplicity first became apparent in the work of the
English mathematician John Herschel (1792 - 1871), the son of
astronomer William Herschel (1738-1822). In 1814, Herschel showed
that i-th iterate of the map can be expressed as
$$ X_i = \frac{1}{2}(1-\cos(2^i\theta)),
$$ where $$
\theta = \cos^{-1}(1-2X_0).
$$}

Nevertheless, studying of his work~\cite{Herschel}, which is given
at the bibliography of correspond section of~\cite{Kautz}, we have
made a conclusion, that it is not so. The only thing, which
W.~Herschel does at this work correspondingly to our interest is
the following.

He considers the iterations of different functions and states the
problem of finding the explicit formula $g^n$ for the $n-th$
iteration  function $g(x)=2x^2-1$. He finds the formula
\begin{equation}\label{ogl-eq:19}
g^n(x) = \frac{1}{2}\left( \left(x + \sqrt{x^2-1} \right)^{2^n} +
\left(x - \sqrt{x^2-1} \right)^{2^n} \right).
\end{equation}

After obtaining such a solution, Herschel solves the functional
equation
$$
\varphi^n(x) = g(x)
$$ for $f(x)=2x^2 -1$. He just plug $\frac{1}{n}$ instead of $n$
into~(\ref{ogl-eq:19}) and get the answer $$ \varphi(x) =
\frac{1}{2}\left( \left(x + \sqrt{x^2-1} \right)^{\sqrt[n]{2}} +
\left(x - \sqrt{x^2-1} \right)^{\sqrt[n]{2}} \right).
$$
After this he writes: ``we may here observe, that any one of of
the $n$ values of $\sqrt[n]{2}$ will equally afford a satisfactory
value of $\varphi(x)$''. Also Herschel repeats his
formula~(\ref{ogl-eq:19}) for iterations of $g$ in 1820
at~\cite[p. 169]{Herschel-2}.

George Bool at his ``A Treatice on the Calculus of finite
differences'' see~\cite[p. 170, ex. 11]{Bool} considers the
difference equation $u_{n+1} - 2u_n^1 + 1 =0$ and solves it as
$u_n = \cos 2^nx$.

C. Babbage was the first, who in fact used the idea of topological
conjugateness. He has paid his attention to the functional
equation \begin{equation}\label{ogl-eq:18} \psi^n(x) = x
\end{equation} and mentioned that for any solution $t$ of this
equation and for any arbitrary function $\varphi$, the function
\begin{equation}\label{ogl-eq:20}
F(x) = \varphi^{-1}(t(\varphi(x))) \end{equation} would also be
solution. He wrote this at his own part, called ``Examples of the
Solutions of Functional Equations'' of the book~\cite{Herschel-2}
by J.F.W. Herschel.

Ritt also mentions this Babbage's work at~\cite{Ritt}  and
cites~\cite{Herschel-2}. Also, Ritt writes there, that it was
Babbage, who has made the first attempt to find the general
solution of~(\ref{ogl-eq:18}). Talking about the general solution
of~(\ref{ogl-eq:18}), Ritt uses as he says ``well known''
periodical transformation $$ w(x) = \frac{\alpha +\beta x}{\gamma
+\delta x},
$$ where $$
\delta = -\, \frac{\beta^2 -2\beta\gamma \cos\frac{2k\pi}{n}
+\gamma^2}{2\alpha\left( 1+\cos\frac{2k\pi}{n}\right)},
$$ $k$ being any integer prime to $n$. With giving
this transformation, Ritt mentions Bool's Book ``Calculus of
finite Differences'', whose the first edition was in 1860. Also
Ritt criticizes the Babbage statement, that for every solution $t$
and $F$ of~(\ref{ogl-eq:18}) there exists an invertible function
$\varphi$ such that
$$F(x) = \varphi^{-1}(w(\varphi(x))).$$
Suppose, says Ritt, that $\varphi(a)=b$ for some $a$ and $b$. Then
for any $n\in \mathbb{N}$ the equality $\varphi(F^n(a))=w^n(b)$.
If some $F^p(a)$ and $F^p(a')$ coincide, then so should $F^p(a)$
and $F^p(a')$, than it may give a contradiction with $\varphi$
being a one to one function. Then Ritt goes Further and say
$\varphi$ being defined on some interval $(a,\, a')$ by a function
$\Phi$. Then the function $w(\Phi(F^{-1}))$ defines $\varphi$ on
$[F(a),\, F(a')]$ and, continuing, $w^k(\Phi((F^{-1})^k))$ defines
$\varphi$ on $[F^k(a),\, F^k(a')]$ for all $k\in \mathbb{N}$.

We see, that the mathematicians of 19-th century paid attention to
the functional equation of the form $t^2(x) = x$ and to its
generalization $t^n(x) = x$. The following theorem is proved
in~\cite[Theor. 2]{Natanson}.

\begin{theorem}\label{theor:Nathan}
If $t:\, [0,\, 1]\rightarrow [0,\, 1]$ is a continuous function
such that $t^p(x) =x$ for all $x$, then $t^2(x)=x$ for all $x$. In
particular, if $p$ is odd, then $t(x)=x$ for all $x$.
\end{theorem}

Nowadays it is natural to think about such equations with the use
so called Lamerey diagrams, which were discovered almost in the
same time.

Nevertheless, we can say, that it was Babbage, who has done the
first attempt of graphical interpretation of the solution of the
functional equation~(\ref{ogl-eq:18}) at his~\cite{Babbage}. It is
the following (see Figure~\ref{ogl-fig:12}).

\begin{figure}
\begin{center}
\includegraphics[width=.5\linewidth]{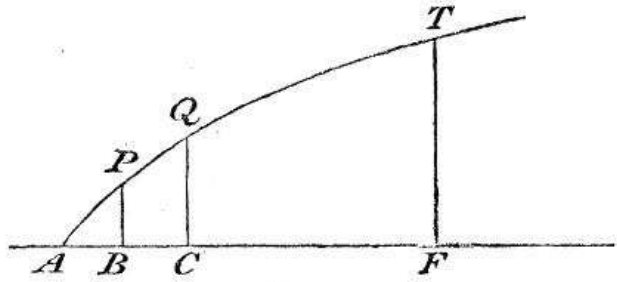}
\end{center}
\caption{A picture from the work of Babbage} \label{ogl-fig:12}
\end{figure}

He wrote: \emph{Required the nature of the curve, such that taking
any point $B$ in the abscissa, and drawing the ordinate $BP$ if we
make $AC$ another abscissa equal $BP$ the preceding ordinate, and
if we continue this $n$ times, then the $n$-th ordinate may be
equal to the first abscissa. If $AB = x$ and the equation of the
curve is $y=\psi(x)$. $PB = y =\psi(x)$ and $AC = PB = \psi(x)$,
and $QC = \psi^2(x)$, and generally the $n$-th ordinate $TF$ is
equal to $\psi^n(x)$, hence $\psi^n(x) =x$, which is the equation
whose solution has been just found}.

We will introduce below the history of discovering of the Lamerey
diagram (see~\cite{Days}).

In 1914 Pincherle considered a problem of the convergence to a
fixed point for trajectories of dynamical systems which are
defined by interval into itself maps (see~\cite{Pincherle}). In
other words he considered an increasing maps $f$ of some interval
into itself and proved that the trajectory of each point i.e. the
sequence, given by the equation
\begin{equation}\label{eq-l1-1}
x_{n+1} = f(x_n),
\end{equation} converges to a solution of $f(x)=x$.

One of the earliest application of this method in the context of
complex dynamics occurs in 1897 in Lameray's work~\cite[pp.
315-318]{Lemerai}. The roots of this method, however, go back much
farther and seem to have their origins in the work of Adrien-Marie
Legendre and \'Evariste Galois. Writing in the Bulletin des
Sciences Mathematiques in 1830, Galois remarked, (see~\cite[p.
413]{Galois}): ``\emph{Legendre was the first to notice that, when
an algebraic equation was written in, the form $\varphi(x)=x$,
where $\varphi$ is a function in $x$ which increases along with
$x$, it is easy to find the root of this equation if for a nearby
a. smaller than the. root, $\varphi(a)>a$, or for nearby a. larger
than the root, $\varphi(a)<a$. \\ To show this one draws the curve
$y=\varphi(x)$ and the line $y=x$. Given an abscissa $= a$,
suppose, to fix ideas, that $\varphi(a)>a$.  I say that it will be
easy to obtain the nearby root which is larger than $a$. In fact
the roots of the equation $\varphi(x)=x$ are nothing but the
values of the intersection points of the line and of the curve,
and it is clear that one approaches the intersection point by
substituting $\varphi(a)$ for $a$.  One will find a closer and
closer value as one assumes $\varphi(a)$, $\varphi(\varphi(a)) =
\varphi^2(a)$, $\varphi^3(a)$ and so on. [1830, p. 413]}''

\begin{figure}
\begin{center}
\includegraphics[width=.95\linewidth]{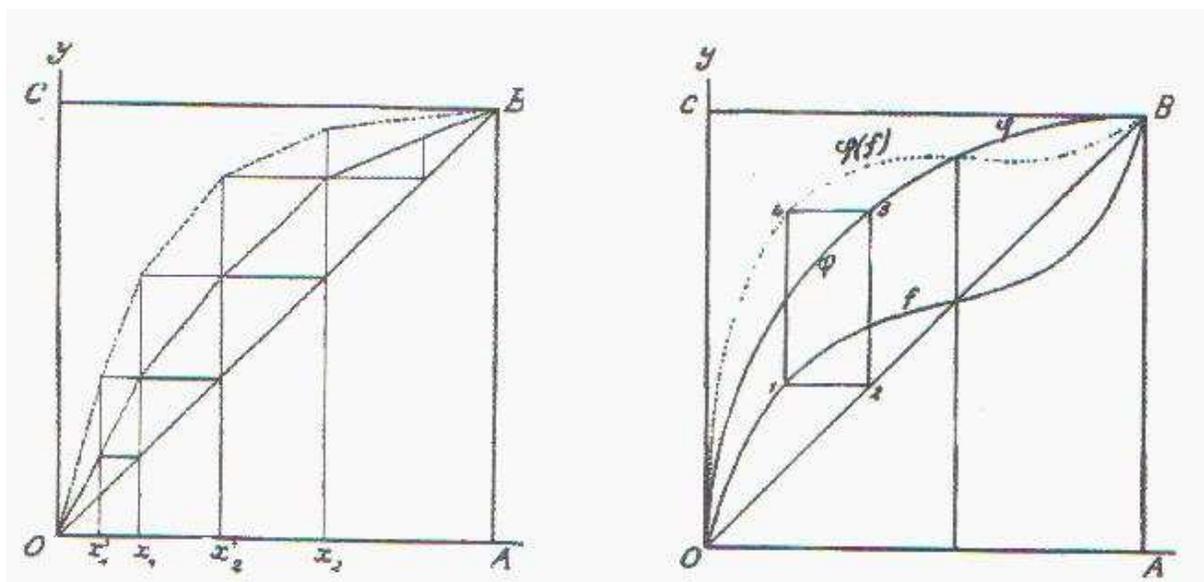}
\end{center}
\caption{A Picture from Pinchrle's work} \label{ogl-fig:11}
\end{figure}

The work to which Galois refers is Legendre, which employs a
method similar to graphical iteration to solve $\varphi(x) = x$
when $\varphi$ is increasing for $x > 0$ (see~\cite[p.
32]{Legendre}). Joseph Fourier used a method very close to what
Pincherle depicts in Figure~\ref{ogl-fig:11} as a tool for the
solution of equations (see~\cite{Fourier}).

The idea of graphical illustration of a topological conjugateness
is given at~\cite{Skufca}. This representation for the maps $f$,
$\widetilde{f}$ and the homeomorphism $\widetilde{h}$ mentioned
above, are given at Figure~\ref{ogl-fig:8}. For graphs are given
at this picture, i.e. the left top quarter contains ``a proper
graph'' of $f$; the right top quarter contains the graph of
$\zeta$, but the $x$-axis goes up and $y$-axis goes right. Due to
this, the composition $y=\zeta(f(x))$ is a maps of points from the
left horizontal segment (of the $x$-axis of the graph of $f$) to
points of the right horizontal segment (of the $y$-axis of the
graph of $\zeta$). In the same way graphs of the bottom part of a
picture are organized.
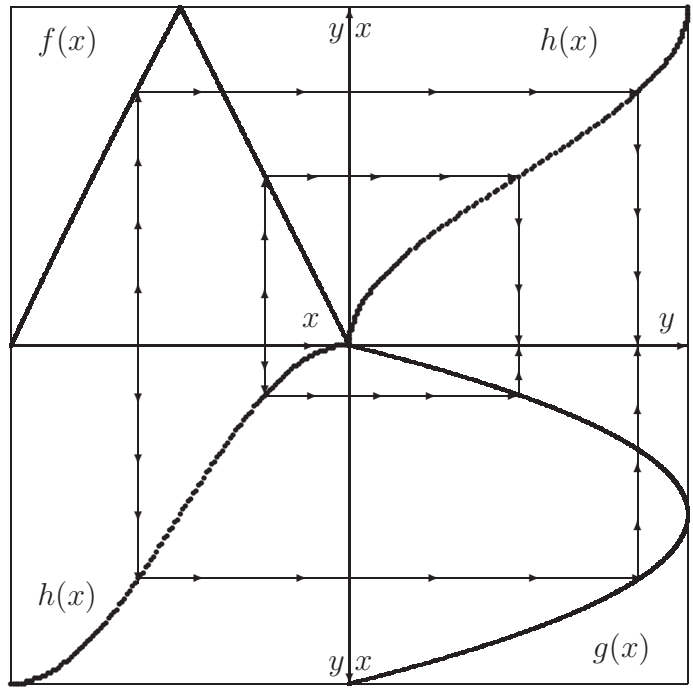
\begin{figure}[htbp]
\begin{center}
\begin{picture}(260,260)
\put(0,0){\line(1,0){256}} \put(0,0){\line(0,1){256}}
\put(0,256){\line(1,0){256}} \put(256,0){\line(0,1){256}}
\put(0,128){\line(1,0){128}} \put(0,128){\vector(1,0){115}}
\put(128,128){\vector(1,0){128}} \put(128,128){\vector(0,1){128}}
\put(128,128){\vector(0,-1){128}} \linethickness{0.4mm}
\qbezier(128,128)(384,64)(128,0) \qbezier(0,128)(32,196)(64,256)
\qbezier(64,256)(96,192)(128,128) \linethickness{0.1mm}
\put(110,135){$x$} \put(120,245){$y$} \put(130,245){$x$}
\put(245,135){$y$} \put(120,5){$y$} \put(130,5){$x$}
\put(10,240){$f(x)$} \put(220,10){$g(x)$} \put(10,30){$h(x)$}
\put(200,240){$h(x)$} \put(128,128){\circle*{2}}
\put(128,129){\circle*{2}} \put(128,130){\circle*{2}}
\put(128,131){\circle*{2}} \put(128,132){\circle*{2}}
\put(128,133){\circle*{2}} \put(129,134){\circle*{2}}
\put(129,135){\circle*{2}} \put(129,136){\circle*{2}}
\put(130,137){\circle*{2}} \put(130,138){\circle*{2}}
\put(130,139){\circle*{2}} \put(131,140){\circle*{2}}
\put(131,141){\circle*{2}} \put(132,142){\circle*{2}}
\put(132,143){\circle*{2}} \put(133,144){\circle*{2}}
\put(133,145){\circle*{2}} \put(134,146){\circle*{2}}
\put(135,147){\circle*{2}} \put(136,148){\circle*{2}}
\put(136,149){\circle*{2}} \put(137,150){\circle*{2}}
\put(138,151){\circle*{2}} \put(139,152){\circle*{2}}
\put(140,153){\circle*{2}} \put(141,154){\circle*{2}}
\put(142,155){\circle*{2}} \put(143,156){\circle*{2}}
\put(144,157){\circle*{2}} \put(145,158){\circle*{2}}
\put(146,159){\circle*{2}} \put(147,160){\circle*{2}}
\put(148,161){\circle*{2}} \put(149,162){\circle*{2}}
\put(150,163){\circle*{2}} \put(151,164){\circle*{2}}
\put(153,165){\circle*{2}} \put(154,166){\circle*{2}}
\put(155,167){\circle*{2}} \put(156,168){\circle*{2}}
\put(158,169){\circle*{2}} \put(159,170){\circle*{2}}
\put(160,171){\circle*{2}} \put(162,172){\circle*{2}}
\put(163,173){\circle*{2}} \put(165,174){\circle*{2}}
\put(166,175){\circle*{2}} \put(168,176){\circle*{2}}
\put(169,177){\circle*{2}} \put(170,178){\circle*{2}}
\put(172,179){\circle*{2}} \put(173,180){\circle*{2}}
\put(175,181){\circle*{2}} \put(176,182){\circle*{2}}
\put(178,183){\circle*{2}} \put(180,184){\circle*{2}}
\put(181,185){\circle*{2}} \put(183,186){\circle*{2}}
\put(184,187){\circle*{2}} \put(186,188){\circle*{2}}
\put(187,189){\circle*{2}} \put(189,190){\circle*{2}}
\put(190,191){\circle*{2}} \put(192,192){\circle*{2}}
\put(194,193){\circle*{2}} \put(195,194){\circle*{2}}
\put(197,195){\circle*{2}} \put(198,196){\circle*{2}}
\put(200,197){\circle*{2}} \put(201,198){\circle*{2}}
\put(203,199){\circle*{2}} \put(204,200){\circle*{2}}
\put(206,201){\circle*{2}} \put(208,202){\circle*{2}}
\put(209,203){\circle*{2}} \put(211,204){\circle*{2}}
\put(212,205){\circle*{2}} \put(214,206){\circle*{2}}
\put(215,207){\circle*{2}} \put(216,208){\circle*{2}}
\put(218,209){\circle*{2}} \put(219,210){\circle*{2}}
\put(221,211){\circle*{2}} \put(222,212){\circle*{2}}
\put(224,213){\circle*{2}} \put(225,214){\circle*{2}}
\put(226,215){\circle*{2}} \put(228,216){\circle*{2}}
\put(229,217){\circle*{2}} \put(230,218){\circle*{2}}
\put(231,219){\circle*{2}} \put(233,220){\circle*{2}}
\put(234,221){\circle*{2}} \put(235,222){\circle*{2}}
\put(236,223){\circle*{2}} \put(237,224){\circle*{2}}
\put(238,225){\circle*{2}} \put(239,226){\circle*{2}}
\put(240,227){\circle*{2}} \put(241,228){\circle*{2}}
\put(242,229){\circle*{2}} \put(243,230){\circle*{2}}
\put(244,231){\circle*{2}} \put(245,232){\circle*{2}}
\put(246,233){\circle*{2}} \put(247,234){\circle*{2}}
\put(248,235){\circle*{2}} \put(248,236){\circle*{2}}
\put(249,237){\circle*{2}} \put(250,238){\circle*{2}}
\put(251,239){\circle*{2}} \put(251,240){\circle*{2}}
\put(252,241){\circle*{2}} \put(252,242){\circle*{2}}
\put(253,243){\circle*{2}} \put(253,244){\circle*{2}}
\put(254,245){\circle*{2}} \put(254,246){\circle*{2}}
\put(254,247){\circle*{2}} \put(255,248){\circle*{2}}
\put(255,249){\circle*{2}} \put(255,250){\circle*{2}}
\put(256,251){\circle*{2}} \put(256,252){\circle*{2}}
\put(256,253){\circle*{2}} \put(256,254){\circle*{2}}
\put(256,255){\circle*{2}} \put(256,256){\circle*{2}}
\put(0,0){\circle*{2}} \put(1,0){\circle*{2}}
\put(2,0){\circle*{2}} \put(3,0){\circle*{2}}
\put(4,0){\circle*{2}} \put(5,0){\circle*{2}}
\put(6,1){\circle*{2}} \put(7,1){\circle*{2}}
\put(8,1){\circle*{2}} \put(9,2){\circle*{2}}
\put(10,2){\circle*{2}} \put(11,2){\circle*{2}}
\put(12,3){\circle*{2}} \put(13,3){\circle*{2}}
\put(14,4){\circle*{2}} \put(15,4){\circle*{2}}
\put(16,5){\circle*{2}} \put(17,5){\circle*{2}}
\put(18,6){\circle*{2}} \put(19,7){\circle*{2}}
\put(20,8){\circle*{2}} \put(21,8){\circle*{2}}
\put(22,9){\circle*{2}} \put(23,10){\circle*{2}}
\put(24,11){\circle*{2}} \put(25,12){\circle*{2}}
\put(26,13){\circle*{2}} \put(27,14){\circle*{2}}
\put(28,15){\circle*{2}} \put(29,16){\circle*{2}}
\put(30,17){\circle*{2}} \put(31,18){\circle*{2}}
\put(32,19){\circle*{2}} \put(33,20){\circle*{2}}
\put(34,21){\circle*{2}} \put(35,22){\circle*{2}}
\put(36,23){\circle*{2}} \put(37,25){\circle*{2}}
\put(38,26){\circle*{2}} \put(39,27){\circle*{2}}
\put(40,28){\circle*{2}} \put(41,30){\circle*{2}}
\put(42,31){\circle*{2}} \put(43,32){\circle*{2}}
\put(44,34){\circle*{2}} \put(45,35){\circle*{2}}
\put(46,37){\circle*{2}} \put(47,38){\circle*{2}}
\put(48,40){\circle*{2}} \put(49,41){\circle*{2}}
\put(50,42){\circle*{2}} \put(51,44){\circle*{2}}
\put(52,45){\circle*{2}} \put(53,47){\circle*{2}}
\put(54,48){\circle*{2}} \put(55,50){\circle*{2}}
\put(56,52){\circle*{2}} \put(57,53){\circle*{2}}
\put(58,55){\circle*{2}} \put(59,56){\circle*{2}}
\put(60,58){\circle*{2}} \put(61,59){\circle*{2}}
\put(62,61){\circle*{2}} \put(63,62){\circle*{2}}
\put(64,64){\circle*{2}} \put(65,66){\circle*{2}}
\put(66,67){\circle*{2}} \put(67,69){\circle*{2}}
\put(68,70){\circle*{2}} \put(69,72){\circle*{2}}
\put(70,73){\circle*{2}} \put(71,75){\circle*{2}}
\put(72,76){\circle*{2}} \put(73,78){\circle*{2}}
\put(74,80){\circle*{2}} \put(75,81){\circle*{2}}
\put(76,83){\circle*{2}} \put(77,84){\circle*{2}}
\put(78,86){\circle*{2}} \put(79,87){\circle*{2}}
\put(80,88){\circle*{2}} \put(81,90){\circle*{2}}
\put(82,91){\circle*{2}} \put(83,93){\circle*{2}}
\put(84,94){\circle*{2}} \put(85,96){\circle*{2}}
\put(86,97){\circle*{2}} \put(87,98){\circle*{2}}
\put(88,100){\circle*{2}} \put(89,101){\circle*{2}}
\put(90,102){\circle*{2}} \put(91,103){\circle*{2}}
\put(92,105){\circle*{2}} \put(93,106){\circle*{2}}
\put(94,107){\circle*{2}} \put(95,108){\circle*{2}}
\put(96,109){\circle*{2}} \put(97,110){\circle*{2}}
\put(98,111){\circle*{2}} \put(99,112){\circle*{2}}
\put(100,113){\circle*{2}} \put(101,114){\circle*{2}}
\put(102,115){\circle*{2}} \put(103,116){\circle*{2}}
\put(104,117){\circle*{2}} \put(105,118){\circle*{2}}
\put(106,119){\circle*{2}} \put(107,120){\circle*{2}}
\put(108,120){\circle*{2}} \put(109,121){\circle*{2}}
\put(110,122){\circle*{2}} \put(111,123){\circle*{2}}
\put(112,123){\circle*{2}} \put(113,124){\circle*{2}}
\put(114,124){\circle*{2}} \put(115,125){\circle*{2}}
\put(116,125){\circle*{2}} \put(117,126){\circle*{2}}
\put(118,126){\circle*{2}} \put(119,126){\circle*{2}}
\put(120,127){\circle*{2}} \put(121,127){\circle*{2}}
\put(122,127){\circle*{2}} \put(123,128){\circle*{2}}
\put(124,128){\circle*{2}} \put(125,128){\circle*{2}}
\put(126,128){\circle*{2}} \put(127,128){\circle*{2}}
\put(128,128){\circle*{2}} \put(48,128){\vector(0,1){24}}
\put(48,128){\vector(0,1){48}} \put(48,128){\vector(0,1){72}}
\put(48,128){\vector(0,1){96}} \put(48,224){\vector(1,0){25}}
\put(48,224){\vector(1,0){64}} \put(48,224){\vector(1,0){152}}
\put(48,224){\vector(1,0){114}} \put(48,224){\vector(1,0){190}}
\put(237,224){\vector(0,-1){96}} \put(237,224){\vector(0,-1){24}}
\put(237,224){\vector(0,-1){48}} \put(237,224){\vector(0,-1){72}}
\put(48,128){\vector(0,-1){22}} \put(48,128){\vector(0,-1){45}}
\put(48,128){\vector(0,-1){67}} \put(48,128){\vector(0,-1){90}}
\put(48,40){\vector(1,0){191}} \put(48,40){\vector(1,0){25}}
\put(48,40){\vector(1,0){64}} \put(48,40){\vector(1,0){114}}
\put(48,40){\vector(1,0){152}} \put(96,128){\vector(0,1){64}}
\put(96,128){\vector(0,1){22}} \put(96,128){\vector(0,1){44}}

 \put(237,40){\vector(0,1){88}}
\put(237,40){\vector(0,1){44}} \put(237,40){\vector(0,1){22}}
\put(237,40){\vector(0,1){66}}

\put(96,192){\vector(1,0){96}} \put(96,192){\vector(1,0){20}}
\put(96,192){\vector(1,0){44}} \put(96,192){\vector(1,0){70}}
\put(96,128){\vector(0,-1){19}} \put(96,128){\vector(0,-1){10}}

\put(192,192){\vector(0,-1){64}} \put(192,192){\vector(0,-1){20}}
\put(192,192){\vector(0,-1){45}}

\put(96,109){\vector(1,0){96}} \put(96,109){\vector(1,0){20}}
\put(96,109){\vector(1,0){45}} \put(96,109){\vector(1,0){65}}

\put(192,109){\vector(0,1){19}} \put(192,109){\vector(0,1){10}}

\end{picture}
\end{center}
\caption{Graphs, which illustrate topological conjugateness}
\label{ogl-fig:8}
\end{figure}

Independently on Verhulst studying of populational dynamics, the
logistic map $g(x)=4x(1-x)$ appeared in the study of random
generators in~\cite{Ulam} -- an abstract dedicated to the Summer
Meeting of the AMS in 1947. There was announced the fact, that for
almost all $x\in \mathbb{R}$ (in the sense of Lebesgue measure)
after finite number of steps, iterations of $x$ belong to $[0,\,
1]$ and are ``randomly'' (uniformly) distributed in $[0,\, 1]$.
Nevertheless, J. von Neumann showed, that the
function~(\ref{ogl-eq:13}) can not be used as random numbers
generator (see~\cite{Neumann}). In fact, he invented in this work
the topological conjugation of $g$ and the hat map $f$
\begin{equation}\label{ogl-eq:f} f(x) = \left\{\begin{array}{ll}
2x,& \text{if }\ 0\leq x< 1/2,\\
2-2x,& \text{if }\ 1/2 \leqslant x\leqslant 1.
\end{array}\right.
\end{equation}
J. von Neumann suggested to consider the correspondence $x_i =
\sin^2\pi\alpha_i$ and pay attention to the sequence
$\{\alpha_i\}$ if $\{x_i\}$ is generated by $g$ as random
generator. He has written, that in this case the equality
$\alpha_{i+1} = 2\alpha_i$ (modulo 1) will hold. After this, von
Neumann concluded, that in some sense the trajectory our number
$x$ will be as random as many random numbers of correspond
$\alpha$ we will take at the very beginning, whence the generator
can not be used ``in a real world'', possibly being good ``in
mathematical world''. Also von Neumann mentioned, that if
$\alpha_i$-th are uniformly distributed, then correspond $x_i$-th
would be distributed with the probability distribution
$(2\pi)^{-1}[x(x-1)]^{1/2}\, dx$. This result is close to the
invariant measures theory, we will mention below.

Stanislaw Ulam introduced at~\cite[Appendix 1]{Ulam-1964} the
homeomorphism $$ h = \frac{2}{\pi}\sin^{-1}(\sqrt{x}),
$$ such that $$
f(x) = h(g_\lambda(h^{-1}(x))
$$ for $\lambda = 4$.

Characterizing invariant measures for explicit nonlinear dynamical
systems is a fundamental problem which connects dynamical theory
with statistics and statistical mechanics. In some cases, it would
be desirable to to characterize ergodic invariant measures for
simple chaotic dynamical systems. However, in the cases of chaotic
dynamical systems, such attempts to obtain explicit invariant
measures have rarely been made.

Stanislaw Ulam also proposed the way of constructing a conjugation
of piecewise linear maps. He proved the following
Theorem~(see~\cite[p. 460 (53)]{Ulam-1964}). Let $f$ be
broken-linear function, of equation~(\ref{ogl-eq:f}). Let $t(x)$
be a convex function on $[0,\, 1]$ which transforms the interval
into itself, and such that $t(0) = t(1) = 0$. For some $p$ in the
interval, we must have $t(p)=1$; by convexity, there is only one
such point. Consider the lower tree of $1$ (i.e. all integer
trajectories of $1$). The necessary and sufficient condition that
$t(x)$ be conjugate to $f$ is is that tree combinatorially the
same as that generated by $1$ under $f(x)$, and closure of this
points be the whole interval, i.e. that the tree is dense in
$[0,\, 1]$.

O. Rechard has used this result to finding the invariant measure
of $f$ (see~\cite{Rechard}). Let $\tau$ be a transformation (not
one to one) of the space $X$ (or, for example, of the interval
$I=[0,\, 1]$) into itself. If the complete inverse image of every
measurable set is itself measurable, then $\tau$ is called a
\underline{\emph{measurable transformation}}, and if, in addition,
$m(\tau^{-1}(A)) =0$, whenever $m(A)=0$, the transformation will
be said to be \underline{\emph{non-singular}}. Let $\mathcal{X}$
be a class of measurable subsets of~$X$.

Assuming $\tau$ to be a measurable transformation, a finite
measure $m^*$ defined for sets $\mathcal{X}$ is
\underline{\emph{invariant}} under $\tau$ if $m^*(\tau^{-1}(A)) =
m^*(A)$ for every measurable set $A$.

If $\tau$ be a measurable, non-singular transformation of the
space $X$ (or, for example, of the interval $I=[0,\, 1]$) into
itself, and $v$ is any real valued integrable function on $X$,
then $\mu_v(A) =\int\limits_{\tau^{-1}(A)}v\, dm$ is a
finite-valued countably additive set function on $\mathcal{X}$
which is absolutely continuous with respect to $m$. Consequently,
by the Radon-Nikodym theorem, there exists an integrable function
$w$ on $X$ such that $\mu_v(A) = \int\limits_Aw\, dm$ for every
measurable set $A$. Denoting by $L(X,\, m)$ the space of functions
integrable over $X$ with respect to $m$, it is easy to see that
the transformation $T$ of $L(X,\, m)$ into itself, defined by $Tv
= w$ is additive and homogeneous and transforms non-negative
functions into non-negative functions. In addition, $$
\int\limits_XTv\, dm  = \int\limits_{\tau^{-1}(X)}v\, dm =
\int\limits_Xv\, dm,
$$ and $
||Tv ||\leq ||v|| $ for $v\in L(X,\, m)$.

Thus, O. Rechard defines the function $$ \phi (x) = h'(x) =
\frac{1}{\pi \sqrt{x}\, \sqrt{1-x}}
$$ for $\tau(x)=g(x)=4x(1-x)$ and shows, that for $0\leq x\leq 1/2$ the equality $$
\int\limits_{0}^{4x(1-x)}\phi(t)\, dt = h(\tau(x)) = f(h(x))
$$ hods. If $0\leq x\leq 1/2$, then $$
f(h(x)) = 2h(x) = \int\limits_{0}^x h'(t)\, dt +
\int\limits_{0}^{1-x}h'(t)\, dt;
$$ if $1/2\leq x\leq 1$, then $$
f(h(x)) = 2(1-h(x)) = \int\limits_{x}^1 h'(t)\, dt +
\int\limits_{0}^{1-x}h'(t)\, dt.
$$ This implies, that the function $h'$ is invariant under the
transformation $T$. Thus, the measure $$ m^{*}(A) =
\frac{1}{\pi}\int\limits_A\frac{dx}{\sqrt{x}\cdot\sqrt{1-x}},
$$ is invariant under $\tau$.

Katsura and Fukuda studied a dynamical system, which is generated
by the maps \begin{equation}\label{ogl-eq:17}x \rightarrow
\frac{4x(1-x)(1-lx)}{(1-lx^2)^2} \end{equation} for $0\leq l<1$.
Evidently, it is a generalization of a logistic map for $\lambda
=4$ (see~\cite{Katsura}, \cite{Umeno}). It was showed that
invariant measure for~(\ref{ogl-eq:17}) can be given by its
density as
$$\rho(x) = \frac{1}{2K(l)\sqrt{(x(1-x)(1-lx)}},
$$ where $K(l)$ is the elliptic integral of the first kind given
by $K(l)=\int_0^1\frac{du}{\sqrt{(1-u^2)(1-lu^2)}}$
(see~\cite{Umeno} and~\cite{Umeno-2}).

Dynamical properties of $g_\lambda$ for different values of
$\lambda$ such as stability of fixed points, the period doubling,
chaos appearing etc. were also studied at~\cite{May}. It was May,
who studied at first the period doubling phenomenon of the
logistic map.

The topological conjugation of piecewise linear maps was also
studied at~\cite{Block}. A continuous map $v$ of a compact
interval to itself is \underline{\emph{linear Markov}}, if it is
piecewise linear, and the set $P$ of all $v^k(x)$, where $k\geq 0$
and $x$ is endpoint of a linear piece, is finite. Denote elements
of $P$ as $P =\{ x_1<\ldots<x_N\}$. Let $P^* =\{ 1,\ldots,\, N\}$
be the indexing set of a certain $v$-invariant subset $P$. Denote
$v^*:\, P^* \rightarrow P^*$ by $v^*(i)=j$ if $v(x_i)=x_j$. Also
denote $^*v:\, P^* \rightarrow P^*$ by $^*v(i) = N+1-v^*(N+1-i)$.
Let $Q = \{y_1<\ldots\, <y_M\}$ and $w^*,\, ^*w\, Q^*\rightarrow
Q^*$ be the corresponding objects for a linear Markov $w$.
According to~\cite[Theorem 2.6]{Block}, linear Markov maps $v$ and
$w$ are topologically conjugate if and only if $v^* =w^*$, or $v^*
= ^*w$.

The review of the results on the iterations of interval maps,
which were known at the middle of 20-th century, is given
at~\cite{Collet}.

One of chapters of~\cite[\S II.6]{Collet} is devoted to iterations
of unimodal interval maps. Due to~\cite[\S II.1]{Collet} the maps
$f:\, [0\, 1]\rightarrow [0,\, 1]$ is called unimodal, if it
satisfies the following properties:

(1) $f$ us continuous;

(2) $f(1/2)=1$;

(3) $f$ increase on $[0,\, 1/2]$ and decrease on $[1/2,\, 1]$.

Notice, that interval maps, which are considered in~\cite{Collet},
act on $[-1,\, 1]$. Due to generality of the style, we reformulate
the definitions from~\cite{Collet} to maps, which act on $[0,\,
1]$. For instance, the (2) in~\cite{Collet} is written as $f(0)=1$
and (3) is considers intervals $[-1,\, 0]$ and $[0,\, 1]$. We will
reformulate below the definitions from~\cite{Collet} in the same
manner.

Let $f:\, [0,\, 1]\rightarrow [0,\, 1]$ be unimodal map such that
$f(1/2)=1$. For any $x\in [0,\, 1]$ construct the sequence
$\underline{I}(x)$ of symbols $L,\, R,\, C$ as follows.

(1) $\underline{I}(x)$ is either infinite sequence of $L$ and $R$,
or is a finite sequence of $L$ and $R$, which is followed by
infinite sequence of $C$. We will denote by $\underline{I}_j(x)$
the $j$-th symbol of the mentioned sequence.

(2) If $f^j(x)\neq 1/2$ for all $j\neq 0$ then
$\underline{I}_j(x)=L$ for $f^j(x)< 1/2$ and
$\underline{I}_j(x)=R$ for $f^j(x)> 1/2$.

(3) If $f^k(x)=1/2$ for some $k$, then denote by $j$ the smallest
such $k$, then set $\underline{I}_j(x)=C$. Also set
$\underline{I}_l(x)=L$, if $0\leq l<j$ and $f^l(x)<1/2$. Moreover,
set $\underline{I}_l(x)=R$, if $0\leq l<j$ and $f^l(x)>1/2$.

The sequence $\underline{I}(x)$ is called the itinerary of the
point $x$. The sequence $\underline{I}(1/2)$ is called the
kneading sequence of the $f$.

The calculus of itineraries is scattered in the literature and
usually presented only in a circumstantial context. The most
systematic account in
Derrida-Gervous-Pomerau~\cite{Der-Ger-Pon-78}
and~\cite{Der-Ger-Pon-79}, but some precursory use can be found in
Metropolis-Stein-Stein~\cite{Mtr-St-St-73}. Some lecture notes
from Lanford~\cite{Lanford} can be useful to study of this
calculus~\cite[\S II.2, p. 81]{Collet}.

It is noticed in~\cite[\S II.6]{Collet} that if for a preserving
orientation homeomorphism $h:\, [0,\, 1]\rightarrow [0,\, 1]$, and
for the map $g:\, [0,\, 1]\rightarrow [0,\, 1]$ the equality
$$ f = h^{-1}(g(h)),
$$ holds,
then $g$ is also unimodal and has maximum point at $x_0=h(1/2)$.
If $x_0=1/2$, then $\underline{I}_f(1/2) =\underline{I}_g(1/2)$,
where, naturally, $\underline{I}_f(1/2)$ is a kneading  sequence
of $f$, and $\underline{I}_g(1/2)$ is a kneading  sequence $g$.

The following question appeared in the first time
at~\cite{Milnor-Thurston-77}: Is it true, that the kneading
sequence of the unimodal function defines it up to topological
conjugation. Also it is shown at~\cite{Milnor-Thurston-77}, that
maps with the same kneading sequences are semi conjugated, i.e.
there exists a surjective $h:\, [0,\, 1]\rightarrow [0,\, 1]$ such
that $h(f) = f(h)$.

Some properties of topologically conjugated $S$-unimodal maps are
given at~\cite[\S II.6]{Collet}. Start at first with the
definition of a $S$-unimodal map.

The unimodal map $f:\, [0,\, 1]\rightarrow [0,\, 1]$ is called
$\mathcal{C}^1$-unimodal, if it is 1 time differentiable and and
$f'(x)\neq 0$ for $x\neq 1/2$.

Let $f$ be $\mathcal{C}^1$-unimodal and let $P$ be a periodical
orbit with period $n$. This point is called stable, if
$|f'(x)|\leq 1$ for every $x\in P$, where $f'$ is the derivative.
If follows from the formula of the derivative of a composite
function, that the value $f'(x)$ if the same for all the points
$x\in P$, thence the definition is independent on $x$. The
importance of stable periodical orbits of dynamical system is
explained by the following observation. If $P$ is a stable
periodical orbit of $f$ with period $n$, then there exists an
neighborhood $U$ of $x\in P$, such that $\lim\limits_{m\rightarrow
\infty}f^{mn}(y)=x$ for all $y\in U$ excepted, possibly, the case,
when $|f'(x)|=1$, which will be discussed later. Whence, if the
trajectory is stable, then a lot of points have similar habitation
under $f^m$ for $m\rightarrow \infty$. The periodical point $P$ is
super stable, if $f'(x)=0$ for $x\in P$.

The following example is given at~\cite[p. 429]{Layek}. Consider
the maps $g(x)=\lambda x(1-x)$ for $x\in [0,\, 1],\, \lambda\geq
0$. Evidently, $g'(x)=\lambda(1-2x)$. The fixed point of $g$ is
fixed if and only if $g'(x)=0$, whence $x=1/2$ should be a fixed
point. It is so only in the case, when $\lambda=2$. We can
similarly study the super stability of periodical points with
period $2$. If the cycle $p \stackrel{g}{\longrightarrow}
q\stackrel{g}{\longrightarrow} p$ is super stable, then $x=1/2$
belongs to this cycle. Plugging $x=1/2$ into $g^2(x)=x$ as an
equation for $\lambda$ obtain that $\lambda=2,\, 1\pm \sqrt{5}$.
But $\lambda=2$ corresponds to the fixed point $1/2$ and $\lambda
= 1-\sqrt{5}<0$. The value $\lambda=1+\sqrt{5}$ corresponds to the
stable cycle $\{ \frac{1}{2},\, \frac{1+\sqrt{5}}{4}\}$.

Consider below the question on how any periodical orbit may be for
a unimodal map. This problem was stated at first by Julia in 1918
at his~\cite{Julia-1918}. He has showed that some unimodal maps,
which are the restrictions of functions, which are analytical on
$[0,\ ,1]$, can have more then one stable periodical orbit. His
theory deals with the maps $w(x)=1-\mu x^2,\, 0<\mu \leq 2$.
Nevertheless, the natural discovery was made by Singer in 1978
in~\cite{Singer-1978}, when he separated the case of the negative
Schwartz derivative as one, when the situation becomes simpler.

Let $f\in\mathcal{C}^3$. The Schwartz derivative of $f$ is the
expression $$ Sf(x) = \frac{f'''(x)}{f'(x)} -
\frac{3}{2}\left(\frac{f''(x)}{f'(x)}\right)^2.
$$

A maps $f:\, [0,\, 1]\rightarrow [0,\, 1]$ is called $S$-unimodal,
if the following conditions hold.

(S1) $f$ is $\mathcal{C}^1$-unimodal;

(S2) $f\in \mathcal{C}^3$;

(S3) $Sf(x)<0$ for all $x\in [0,\, 1]$. Also we admit
$Sf(x)=-\infty$ for $x=1/2$.

(S4) $f$ maps $J(f) = [f(1),\, 1]$ onto itself, i.e. $f([f(1/2),\,
1]) = [f(1/2),\, 1]$

(S5) $f''(1/2)<0$.

\begin{theorem}[Theorem~II.4.1 in~\cite{Collet}]
If $f$ satisfy the conditions (S1), (S2) and (S3), then every
stable periodical point attracts at least one of the points $0,\,
1/2,\, 1$, i.e. ends of the interval and the critical point.
\end{theorem}

\begin{corollary}[Corollary II.4.2 in~\cite{Collet}]
If a maps $f$ satisfies the conditions (S1) ... (S4), then it has
at most one stable fixed orbit on $[0,\, f(1)]$. If $1/2$ is not
attracted to the stable periodical orbit, then $f$ has no stable
periodical orbits on $[f(1),\, 1]$.
\end{corollary}

\begin{corollary}[Corollary II.4.3 in~\cite{Collet}]
There exist $S$-unimodal functions without stable periodical
orbits.
\end{corollary}

As we have already mentioned, the question about the quantity of
stable orbits was appeared at first at~\cite{Julia-1918}. As far
as it is known, the role of the negative Schwartz derivative with
the question of the quantity of fixed points was studied the first
by Singer at~\cite{Singer-1978}. The connection of the Schwartz
derivative and analytical functions is described
at~\cite{Hille-1976}. The role of the cross ratio in this question
is found by Guckenheimer at~\cite{Guckenheimer-1979}.

Let maps $f,\, g:\, [0,\, 1]\rightarrow [0,\, 1]$ be
non-topological conjugated, but such that $\underline{I}_f(1)
=\underline{I}_g(1)$.

\begin{figure}[htbp]
\begin{minipage}[h]{0.49\linewidth}
\begin{center}
\begin{picture}(100,120)
\put(100,0){\line(0,1){100}} \put(0,100){\line(1,0){100}}
\put(0,0){\vector(0,1){120}} \put(0,0){\vector(1,0){120}}

\Vidr{0}{0}{100}{100}

\Vidr{0}{50}{50}{100} \VidrTo{95}{85} \VidrTo{100}{50}

\put(87,87){\circle*{4}}

\qbezier[30](87,87)(87,60)(87,33)

\put(83,25){$x_f$}
\end{picture}
\end{center}
\centerline{a. Graph of $f$}
\end{minipage}
\hfill
\begin{minipage}[h]{0.49\linewidth}
\begin{center}
\begin{picture}(100,120)
\put(100,0){\line(0,1){100}} \put(0,100){\line(1,0){100}}
\put(0,0){\vector(0,1){120}} \put(0,0){\vector(1,0){120}}

\Vidr{0}{0}{100}{100}

\Vidr{0}{50}{50}{100} \VidrTo{70}{55} \VidrTo{100}{50}

\put(65.5,65.5){\circle*{4}}

\qbezier[20](65.5,65.5)(65.5,49.5)(65.5,33.5)

\put(61.5,25.5){$x_g$}

\end{picture}
\end{center}
\centerline{b. Graph of $g$}
\end{minipage}
\caption{Maps with the same itineraries of 1}\label{fig:19}
\end{figure}
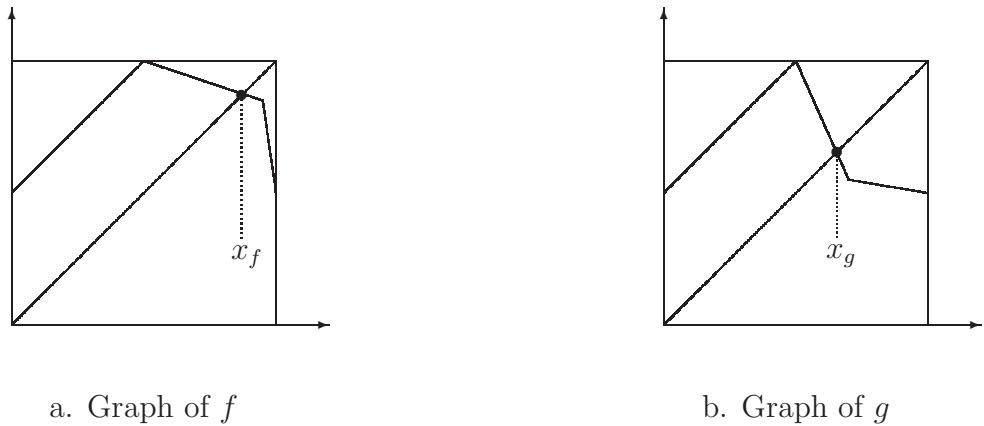

Maps $f$ and $g$ on the Figure~\ref{fig:19} have the mentioned
properties. These maps are piecewise linear, have two braking
points each and $f(0)=f(1)=g(0)=g(1)=1/2$, $f(1/2)=g(1/2)=1$.

Evidently, $f$ and $g$ are not topologically conjugated, because
the fixed point $x_f$ of $f$ is attracting, but the fixed point
$x_g$ of $g$ is repelling. In the same time, $\underline{I}_f(1) =
\underline{I}_g(0) = RC$.

In the same time, in the following  theorems, the topological
conjugation follows from the equality of itineraries with some
additional conditions.

\begin{theorem}[Theorem II.6.1 in~\cite{Collet}]\label{theor:20}
Suppose that $f$ and $g$ are $S$-unimodal and $f$ has no stable
periodic point. If $\underline{I}_f(0) = \underline{I}_g(0)$ and
$\underline{I}_f(1) = \underline{I}_g(1)$ then $f$ and $g$ are
topologically conjugate.
\end{theorem}

The following theorem can be considered as a variant or
Theorem~\ref{theor:20}.

\begin{theorem}[Theorem II.6.1.A in~\cite{Collet}]
Suppose that $f$ and $g$ are $S$-unimodal and $f$ has no stable
periodic point. If $\underline{I}_f(1) = \underline{I}_g(1)$ then
$f|_{[f(1),\, 1]}$ and $g|_{[g(1),\, 1]}$ are topologically
conjugate through a homeomorphism $\widehat{h}:\, [f(1),\, 1]$
onto $[g(1),\, 1]$.
\end{theorem}

\begin{theorem}[Theorem II.6.3 in~\cite{Collet}]
Let $f$ and $g$ be $S$-unimodal and assume that
$\underline{I}_f(1) = \underline{I}_g(1)$.

1. If $\underline{I}_f(1)$ is finite, then $f|_{J_f}$ and
$g|_{J_g}$ are topologically conjugate.

2. If $\underline{I}_f(1)$ is infinite and periodic of period $n$,
i.e. $\underline{I}_f(1) = D^{\infty}$ with $|D|=n$, then there
are two possibilities:

(a) If $|D|$ is odd then $f|_{J_f}$ and $g|_{J_g}$ are
topologically conjugate if and only if their stable periodic
orbits have the same period (which is $n$ or $2n$).

(b) if $|D|$ is even, then $f|_{J_f}$ and $g|_{J_g}$ are
topologically conjugate if their stable periodic orbits (which
have period $n$) are both stable from one side or stable from both
sides.

3. If $\underline{I}_f(1)$ is finite but not periodic, then
$f|_{J_f}$ and $g|_{J_g}$ are topologically conjugate.
\end{theorem}

Guckenheimer has proved at~\cite{Guckenheimer-1979}, that any
$C^1$ unimodal map of the interval is semiconjugate to a quadratic
map and that the semi-conjugacy is strictly monotone in the
backward orbit of the turning point. The prove of this result uses
the assumption that their Schwarzian derivative of the map is
negative.

That a quadratic map is described by a very simple mathematical
formula is not very useful for the understanding of its dynamics
because this property is not preserved under iteration: the $n$-th
iterate of the map is a polynomial of degree $2n$. Singer made the
following fundamental observation at~\cite{Singer-1978}: if a map
has negative Schwarzian derivative then all of its iterates also
have this property. Independently Allwright~\cite{Allwright-1978}
observed something similar. Furthermore, quadratic maps turn out
to have negative Schwarzian derivative.

In the same paper, Singer proved that such maps have a finite
number of attracting periodic orbits, if they have a finite number
of turning points. This, because each of these orbits must attract
at least one critical point or one boundary point. Later
Guckenheimer showed~\cite{Guckenheimer-1979}, for unimodal maps
with negative Schwarzian derivative, that any interval whose
points have the same itinerary must be contained in the basin of
the unique attracting periodic orbit. In particular, if the map
has no attracting periodic orbit, the backward orbit of its
turning point is dense.

Milnor and Thurston proved~\cite{Milnor-Thurston-77}, that a
continuous, piecewise monotone map with positive topological
entropy is semiconjugate to a continuous, piecewise linear map
with constant slope and with the same entropy. This result is the
following theorem.

\begin{theorem}
Assume that $f:\, [0,\, 1]\rightarrow [0,\, 1]$ is a continuous,
piecewise (strictly) monotone map with positive topological
entropy $h_t(f)$ and let $s = exp(h_t(f))$. Then there exists a
continuous, piecewise linear map $T:\, [0,\, 1] \rightarrow [0,\,
1]$ with slope $\pm s$, and a continuous, monotone increasing map
$w:\, [0,\, 1] \rightarrow [0,\, 1]$ which is a semi-conjugacy
between $f$ and $T$, i.e. $$ w(f) = T(w).$$
\end{theorem}

Essentially this result was already proved by Parry
at~\cite{Parry-1966}.

The proof of this theorem gives also a very important relationship
between the lap numbers $l(f^n)$ and the kneading invariants. The
definition of lap numbers is the following.

\begin{definition}
Let $f: [0,\, 1] \rightarrow [0,\, 1]$ be a continuous piecewise
monotone map. The lap number $l(f)$ of $f$, is the number of
maximal intervals on which f is monotone. In other words, $l(f) -
1$ is the number of turning points of $f$.
\end{definition}

\newpage
\section{Maps, whose
semigroup of iterations is a finite group}\label{sect-Grupy}

Consider the pair of topologically conjugated maps
$$ y(x) = \left\{\begin{array}{ll}
2x,& x< 1/2;\\
2-2x,& x\geqslant 1/2,
\end{array}\right.
$$ and
$$ \widetilde{y}(x) = 4x(1-x).
$$
We have paid crucial attention to these two maps in
Section~\ref{Sect:IstorOgl}. Maps $y$ and $\widetilde{y}$ are
representors of the following families of maps. For every $v\in
(0,\, 1)$ consider the piecewise linear maps $y_v(x)$, whose graph
consists on two line segments, which connect points with
coordinates $(0,\, 0)$, $(1/2,\, v)$ and $(1,\, 0)$. Also for
every $\alpha\in [0,\, 4]$ consider the maps
$\widetilde{y}_\alpha(x) = \alpha\, x(1-x)$. Consider both $y_v$
and $\widetilde{y}_\alpha$ as maps of $[0,\, 1]$ into itself.

The maps $y_v$ for $v=1/2$ is such that $y_v^2 = y_v$. This
example leads to the problem of the description of all $f: [0,\,
1] \rightarrow [0,\, 1]$ such that
\begin{equation}\label{eq:39}f^n = f,\end{equation}
where $n$ is fixed.

We will assume in this section that $n$ from the
equality~(\ref{eq:39}) is the smallest possible. With the use of
the algebraic notion of representation we may note, that the maps
$f$, which satisfies~(\ref{eq:39}), defines the exact
representation of the cyclic group with $n$ elements. In general,
iterations of the map form the cyclic semigroup with respect to
compositions. This semigroup will be a group $C_n$ if and only if
the map satisfies~(\ref{eq:39}).

\subsection{Groups which are exactly represented with iterations of continuous
interval maps}\label{sect-Grupy-1}

Notice, that it follows from the equality~(\ref{eq:39}) that
cardinalities of orbits of $f$ are uniformly bounded, or, more
precisely, for every $x_0\in I$ its orbit has not mote then $n$
elements.
\begin{lemma}\label{lema:27} \label{int-per-bilsh2} If
$f \in C^0 (I,I)$ has a periodical point of period $m, \ m
> 2$ then orbits are not uniformly bounded.
\end{lemma}
\begin{proof}Let $B$ be a periodical point of period
$m, \ m
> 2$ and let it consists of $\beta_1 < \beta_2 < ... < \beta_m$.
Denote by $\beta_0 = \max \{\beta \not\in B \ | \ f (\beta)
> \beta\}$ (the set $\{\beta \in B \ | \ f (\beta) > \beta\}$ is
non-empty). Since $\beta_0 \neq \beta_m$, then the set $\{\beta
\in B \ | \ \beta > \beta_0\}$ is also non-empty. Let $\beta_1 =
\min \{\beta \in B \ | \ \beta > \beta_0\}$. Since $m > 2$ then
either $f (\beta_0) \neq \beta_1$ or $f (\beta_1) \neq \beta_0$.
Assume that $f (\beta_0) \neq \beta_1$. From the definitions of
$\beta_0$ and $\beta_1$ obtain that $B \cap (\beta_0, \beta_1) =
\emptyset$. Whence either $f (\beta_1) \leq \beta_0 < \beta_1 < f
(\beta_0)$, or $f [\beta_0, \beta_1] \supset [\beta_0, \beta_1]$.
From this property obtain that there exists a sequence $(y_i, \ i
= 1, \ldots )$ such that1) $\beta_0 < y_1 < y_3 < ... < y_{2k+1} <
y_{2k+3} < ... < y_{2k+2} < y_{2k} < ... < y_4 < y_2 < \beta_1,$

2) $f (y_{i+1}) = y_i, \ \ i = 1, 2, \ldots$

3) $f (y_1) = \beta_1.$

Whence for every $i, \ card (orb (y_i) )= i + m$, whence
cardinalities of orbits are not uniformly bounded.\end{proof}

\begin{theorem}\label{theor:7}If a
continuous maps $f$ of interval into itself
satisfies~(\ref{eq:39}), then it satisfies the equality
\begin{equation}\label{eq:40}f^3 =
f.\end{equation}\end{theorem}

\begin{proof}If for the maps $f$ the
equality~(\ref{eq:40}) does not hold, but the
equality~(\ref{eq:39}) holds, then there exists $x_0\in [0,\, 1]$
such that $f^3(x_0)\neq f(x_0)$, but $f^n(x_0)=f(x_0)$. It means
that the point $f(x_0)$ is periodical with period 2 for the maps
$f$. It follows from Lemma~\ref{int-per-bilsh2} then cardinalities
of orbits for the maps $f$ are not uniformly bounded. This
contradicts to equality~(\ref{eq:39}) which means that all these
cardinalities are not grater than~$n$.\end{proof}

\subsection{Graph of the maps with finite group of
iterations}\label{sect-Grupy-2}

We will describe graphs of the maps of the interval into itself,
whose semigroup of iterations is a finite group. Let $f$
satisfies~(\ref{eq:39}). Consider cases when $f^2 = f$ and $f^2
\neq f$. Describe the graph of $f$ in each of these cases. Let
$f^2 =f$. Since $f$ is continuous, then the image of $I$ under $f$
is some continuous interval call $[a,\, b]$, i.e. $f(I) = [a,\,
b]$. For every $x_0\in I$ the condition $f^2 = f$ yields that
$f(x_0)$ is a fixed point of $f$, because
\begin{equation}\label{eq:41}f(f(x_0)) = f(x_0).\end{equation}
From the other hand, the condition $f(I)=[a,\, b]$ yields that for
every $x_0\in [a,\, b]$ there exists $x^*$ such that $f(x^*)=x_0$.
Not it follows from~(\ref{eq:41}) that $[a,\, b]$ is a fixed
points set of $f$. There reasonings prove the following theorem.

\begin{theorem}\label{theor:6}For a maps
$f\in C^{0}(I, I)$ the following properties are equivalent:

1) $f^{n}(x) =f(x)$ for all $x\in I$ and every $n$;

2) there exist real numbers $a$ and $b$ and maps $g\in
C^{0}([0,a],[a,b])$, $h\in C^{0}([b,1],[a,b])$ such that $0\leq
a\leq b\leq 1$ and the maps $f$ can be represented as
follows\begin{equation}\label{theor:2-5}
\begin{array}{cc}
 f(x)=\left \{ \begin{array}{c}
    g(x), ~~0\leq x\leq a,\\
   ~~x,  ~~~~a\leq x \leq b,\\
  h(x), ~~b\leq x\leq 1.
          \end{array}\right.
  \end{array}
\end{equation}\end{theorem}

\begin{figure}[htbp]
\begin{minipage}[h]{\linewidth}
\begin{center}
\begin{picture}(75,75) \qbezier(0,0)(0,35)(0,70)
\qbezier(0,0)(35,0)(70,0) \qbezier(0,0)(35,35)(70,70)
\qbezier(0,70)(35,70)(70,70) \qbezier(70,0)(70,35)(70,70)
\qbezier[20](0,50)(35,50)(70,50)\qbezier[20](0,30)(35,30)(70,30)
\qbezier(70,0)(70,35)(70,70) \linethickness{0.7mm}
\qbezier(30,30)(40,40)(50,50) \linethickness{0.1mm}
\qbezier(30,30)(20,35)(15,45) \qbezier(0,33)(10,50)(15,45)
\qbezier(51,51)(53,45)(60,35) \qbezier(60,35)(62,40)(65,45)
\qbezier(70,35)(67,40)(65,45) \put(3,27){$g(x)$}
\put(50,25){$h(x)$}
\end{picture}
\end{center}

\end{minipage}
\caption{Graph of $f$}\label{fig:25}
\end{figure}
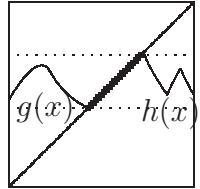

The graph of $f$ from Theorem~\ref{theor:6} is given at
Figure~\ref{fig:25}.

The case when the maps $f$ satisfies the equality $f^3=f$, but not
$f^2=f$ can be described by the following equivalent conditions.

\begin{theorem}\label{theor:5}
For a maps $f\in C^{0}(I, I)$ the following conditions are
equivalent:

1) $f^{3}(x) =f(x)$ for all $x\in I$;

2) There exist reals $a$ and $b$ maps $g\in C^{0}([0,a],[a,b])$,
$h\in C^{0}([b,1],[a,b])$ and a maps $\varphi\in
C^{0}([a,b],[a,b])$, such that the graph of $\varphi$ is
symmetrical in the line $y=x$, and $\varphi([a,b])=[a,b]$ and the
maps $f$ may be represented as follows.
\begin{equation}\label{theot:3-6}
\begin{array}{cc}
 f(x)=\left \{ \begin{array}{c}
    g(x), ~~0\leq x\leq a,\\
   ~~\varphi(x),  ~~~~a\leq x \leq b,\\
  h(x), ~~b\leq x\leq 1.
          \end{array}\right.
\end{array}
\end{equation}\end{theorem}

We will need two technical lemmas for the proof of this theorem.
\begin{lemma}\label{fix-interval}If cardinalities of
orbits of $f\in C(I,I)$ are uniformly bounded, then $Fix(f)$ is a
closed interval.\end{lemma}

\begin{proof}Consider the maps $g=f^2$. It follows
from Lemma~\ref{lema:27} that $Per(f)=Fix(g)$. If the set of fixed
points of $g$ contains only one point, then lemma is trivial.
Otherwise denote by $\beta$ and arbitrary points of $Fix(g)$ and
$I_{\beta}$ the maximal interval of fixed points of $g$, which
contains $\beta$. Assume that $Fix(g)\neq I_\beta$. Let
$\beta_1\in Fix(g)\setminus I_\beta.$ Denote $I_\beta = [a,b]$ and
assume that $\beta_1<a.$ Since $Fix(g)$ is closed, then without
lose of generality assume that the interval $(\beta_1,\, a)$ does
not contain any fixed points of $g$. whence either $g(x)>x$ for
all $x\in (\beta_1,a)$ or $g(x)<x$ for all $x\in (\beta_1,a)$.
Assume that $g(x)>x$ (see Fig~\ref{fig:28}).
\begin{figure}[htbp]
\begin{minipage}[h]{\linewidth}
\begin{center}
\begin{picture}(75,75)

\qbezier(0,0)(0,35)(0,70) \qbezier(0,0)(35,0)(70,0)
\qbezier(0,0)(35,35)(70,70) \qbezier(0,70)(35,70)(70,70)
\qbezier(70,0)(70,35)(70,70) \linethickness{0.7mm}
\qbezier(45,45)(55,55)(60,60) \linethickness{0.1mm}
\qbezier(15,15)(25,80)(45,45) \qbezier[20](45,45)(45,20)(45,0)
\qbezier[20](60,60)(60,30)(60,0) \qbezier[7](15,15)(15,8)(15,0)
\put(18,2){$\beta_1$} \put(50,5){$I_\beta$}

\end{picture}
\end{center}
\end{minipage}
\caption{Graph of $f$}\label{fig:28}
\end{figure}
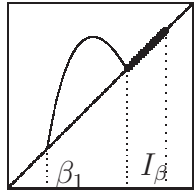

Let $y_1\in (\beta_1,a)$. Since $g(\beta_1,y_1)\supset
[\beta_1,y_1]$ then there exists a sequence $( y_1,\
i\geq 1\,)$ such that\\ 1) $\beta_1<\cdots <y_{i+1}<y_i<\cdots <y_2<y_1<a$\\
2) $g(y_{i+1})=y_i,\ i\geq 1$. Whence, we see that
$card(orb(y_i))=i$ for every $i\geq 1$. This means that in the
case $I_\beta \neq Fix(g)$ the cardinalities of orbits of $g$ are
not uniformly bounded. This prove lemma.\end{proof}

\begin{lemma}\label{lema:28}If $\{x_1,\, x_2\}$ is a
orbit of period 2 of the continuous maps $f$ of hte interval $I$,
then there is a fixed points $x_0$ of $f$ between $x_1$ and
$x_2$.\end{lemma}

\begin{proof}Without loss of generality assume that
$x_1<x_2$. Consider the function $h(x)=f(x)-x$. Then $h(x_1)>0$
and $h(x_2)<0$. Now lemma follows from the known theorem about
middle points of a continuous function.\end{proof}

\begin{proof}[Proof of Theorem~\ref{theor:5}] Noticen
that each maps, which satisfies the conditions of
Theorem~\ref{theor:6} also satisfies the both conditions of
Theorem~\ref{theor:5}.

Evidently, the condition 2. yields the condition 1. Prove the
converse implication. In the same way as in the proof of
Theorem~\ref{theor:6}, denote $f(I) = [a,\, b]$. Consider the maps
$g=f^2$. It follows from the equality $f^3 = f$ that $g^2 =g$. If
follows from Theorem~\ref{theor:6}, applied to the maps $g$, that
for every $x\in [a,\, b]$ the equality $g(x)=x$ holds, which means
that $f^2(x)=x$. Show that the set of fixed points of $f$ consists
of one point. Otherwise by Lemma~\ref{fix-interval} the fixed
points of $f$ is some interval $[a_1,\, b_1]$. From the
construction of  $a$ and $b$ obtain that $[a_1,\, b_1]\subseteq
[a,\, b]$. If $a_1=a$, and $b_1=b$ then $f$ satisfies part 2. of
Theorem~\ref{theor:6}. Assume that $a_1 > a$. Then it follows from
continuity of $f$ that there exists $\varepsilon$ such that
$f(x)<b_1$ for all $x\in (a_1-\varepsilon,\, a_1)$. Consider as
arbitrary point $x_0\in (a_1-\varepsilon,\, a_1)$. It follows from
Lemma~\ref{fix-interval} and the construction of points $a,\, b,\,
a_1$ and $b_1$ that $x_0$ is a periodical point of period 2 of the
maps $f$, whence $x_1 =f(x_0)\in (a,\, a_1)$ and $f(x_1)=x_0$. If
follows from Lemma~\ref{lema:28} that there is a fixed point $x_0$
of $f$ between points $x_0$ and $x_1$. This contradicts to
Lemma~\ref{fix-interval} because this means that $Fix(f)$ is not
an interval. The case $b_1<b$ should be considered analogically,
whence the set of fixed points of $f$ is consisted of the unique
number. Denote it by $c$. Whence, if the function $f$ does not
satisfy $f^2 =f$ then it's set of fixed points consists of the
unique point $c\in [a,\, b]$. If the function $f$ is no monotone
on $[a,\, b]$ then it would contradict to that every point of this
interval is either fixed, or periodical. Show that $f$ decrease on
$[a,\, b]$. Fix an arbitrary point $\widetilde{x}\in (a,\, c)$. If
$f(\widetilde{x})<c$ then by Lemma~\ref{lema:28} there exists a
fixed point of $f$ between $\widetilde{x}$ and $f(\widetilde{x})$,
which contradicts to the uniqueness of the fixed point. whence,
$f(\widetilde{x})>c$. The monotonicity of $f$ on $[a,\, b]$
together with $f(\widetilde{x})>c$ and $f(c)=c$ means that $f$
decrease on $[a,\, b]$. Since $f$ decrease on $[a,\, b]$ and
$[a,\, b]$ is a fixed points set of $f^2$ then $\{a,\, b\}$ is a
periodical orbit of the period 2.\end{proof}

\subsection{Topological conjugation of maps with finite group of
iterations}\label{sect-Grupy-3}

Let $f,\, g,\, h$ be continuous maps $[0,\, 1]\rightarrow [0,\,
1]$ and iterations of $f$ and $g$ for a finite group. Let $h$ be
invertible and $p_1,\, p_2,\, q_1$ and $q_2$ be such points that
$f([0,\, 1])=[p_1,\, q_1]$ and $g([0,\, 1])=[p_2,\, q_2]$. Let
$a_1,\ldots ,a_n,\, b_1,\ldots b_m$ be extremums of $f$ and $g$
correspondingly. Consider points $0$ and $1$ as extremums. Notice,
that end-points of intervals of fixed points of $f^2$ and $g^2$
are necessarily extremums. Assume that $g=h^{-1}(f(h))$.
\begin{definition}Call vectors $(v_1,\ldots,\, v_k)$
and $(w_1,\ldots,\, w_k)$ \textbf{equivalently ordered}, if for
every $i,\, j$ the equality $v_i\leq v_j$ is equivalent to
$w_i\leq w_j$.\end{definition}

\begin{notation}Denote the following
vectors.\noindent $ v_f = (f(a_1),\, \ldots,\, f(a_n))$,

\noindent $\widetilde{v}_f = (f(a_1),\, f^2(a_1)\ldots,\,
f(a_n),\, f^2(a_n) $,

\noindent $ w_g = (g(b_1),\, \ldots,\, g(b_m))$,

\noindent  $\widetilde{w}_g = (g(b_1),\, g^2(b_1),\, \ldots,\,
g(b_m),\, g^2(b_m)). $ \end{notation}

\subsubsection{Idempotent maps with increasing
conjugation}

Assume that $f$ and $g$ are idempotent maps and $h$ is an
increasing conjugation.
\begin{lemma}The equalities $h(p_2)=p_1$ and
$h(q_2)=q_1$ hold.\end{lemma}

\begin{proof}Consider $x_0\in Fix(g)$. Assume that
$h(x_0)>q_1$. Then $f(h(x_0))<h(x_0)$, because $f(h(x_0))\in
f([0,\, 1])=[p_1,\, q_1]$ and $h(x_0)>q_1$. Since $h$ increase,
then $h^{-1}$ also increase whence
$h^{-1}(f(h(x_0)))<h^{-1}(h(x_0))$. But the last equality means
that $g(x_0)<x_0$, which contradicts to that $x_0$ is a fixed
point of $g$. The case $h(x_0)<p_1$ should be considered
analogically. We have that $h([p_2,\, q_2])$ belongs to $[p_1,\,
q_1]$. Since it follows from $g=h^{-1}(f(h))$ that
$f=h(g(h^{-1}))$ then the same reasonings give that
$h^{-1}([p_1,\, q_1])\subseteq [p_2,\, q_2]$. The last finishes
the proof.\end{proof}

\begin{lemma}\label{extremumy-v-extremumy}$m=n$ and
for arbitrary $i,\, 1\leq i\leq n$ the equality $h(b_i)=a_i$ holds
and vectors $v_f$ and $w_g$ are equivalently ordered.\end{lemma}

\begin{proof}This lemma follows from that composition
of monotone functions is a monotone function. Consider an
arbitrary $i\in [1,\, n-1]$ and consider the interval of
monotonicity of $f$, call $[a_i,\, a_{i+1}]$. Let $f$ increase on
this interval. Show that in this case $g$ would increase on
$[h^{-1}(a_i),\, h^{-1}(a_{i+1})]$. For arbitrary $x_1,\, x_2\in
[a_1,\, a_{i+1}],\, x_1<x_2$ we have $h(x_1)<h(x_2)$; since $f$
increase on $[a_i,\, a_{i+1}]$, then $f(h(x_1))<f(h(x_2))$, whence
$h^{-1}(f(h(x_1)))<h^{-1}(f(h(x_2)))$, which proves the
monotonicity of $g$ on $[h^{-1}(a_i),\, h^{-1}(a_{i+1})]$. If the
maps $f$ decrease on $[a_i,\, a_{i+1}]$ then the prove of
degreasing of $g$ on $[h^{-1}(a_i),\, h^{-1}(a_{i+1})]$ is
analogical to the previous. Consider points $b_i=h^{-1}(a_i),\,
i=1,\ldots ,n$. It follows from the previous, that these points
are extremums of $g$ and for every $i=1,\ldots ,n$ the character
of extremum (minimum, of maximum) of $b_i$ coincides with one of
$a_i$. Let for some $r,\, s=1,\ldots, n$ the equality $g(b_r)\leq
g(b_s)$ holds, i.e. $h^{-1}(f(h(b_r)))\leq h^{-1}(f(h(b_s)))$.
Take $h$ from the both sides of the inequality and it would follow
from the increasing of $h$ that $f(h(b_r))\leq f(h(b_s))$. Since
$b_r=h^{-1}(a_r)$ and $b_s=h^{-1}(a_s)$ thn if follows from the
last equality that $f(a_r)\leq f(a_s)$ which is was necessary to
prove.\end{proof}

\begin{lemma}\label{spriajenist:constructyvno}Let
$m=n$, vectors $v_f$ and $w_g$ be equivalently ordered and numbers
of extremums, which are end-points of intervals $[p_1,\, q_1]$ and
$[p_2,\, q_2]$ coincide. Then $f$ and $g$ are
conjugated.\end{lemma}

\begin{proof}Since for every $i\in [1,\, n]$ the
equality $a_i=h(b_i)$ holds, then plugging $b_i$ into
$g=h^{-1}(f(h))$ obtain $g(b_i)=h^{-1}(f(a_i))$, i.e. the graph of
$h$ pathes through the point $(g(b_i),\, f(a_i))$. Notice, that
since the condition $f(a_i)\leq f(a_j)$ is equivalent to
$g(b_i)\leq g(b_j)$ then the obtained restriction for $h$ does not
contradict to its monotonicity. More then this, since $h(p_2)=p_1$
and $h(q_2)=q_1$ then the restriction on $h$ is the restriction
only on the interval $[p_2,\, q_2]$. Take $h$ to be arbitrary
increasing on $[p_2,\, q_2]$ and passing through the mentioned
points. For example, take $h$ to be piecewise linear. Consider an
arbitrary $i=1,\ldots ,n$ such that $[a_i,\, a_{i+1}]$ is not
$[p_1,\, q_1]$. Consider an arbitrary $x_0\in [b_i,\, b_{i+1}]$.
Then the condition $g(x_0)=h^{-1}(f(h(x_0)))$ is equivalent to
$h(g(x_0))=f(h(x_0))$. Since $h$ is already defined on $[p_2,\,
q_2]$ then $h(g(x_0))$ is already defined. Since $x_0\in [b_i,\,
b_{i+1}]$ then $h(x_0)\in [a_i,\, a_{i+1}]$. Since $f$ increase on
$[a_i,\, a_{i+1}]$, then it has an inverse $f_i$, which is defined
on $[\min(f(a_i),\, f(a_{i+1})),\, \max(f(a_i),\, f(a_{i+1}))]$.
Whence, the equality $h(x_0)=f_i^{-1}(h(g(x_0)))$ is necessary for
the equality $h(g(x_0))=f(h(x_0))$.

For every $i=1,\ldots ,n$ define the maps $h$ on $[b_i,\,
b_{i+1}]$ by the formula $h=f_i^{-1}(h(g))$. It follows from the
construction of $h$, that if defined the conjugation of $f$
and~$g$.\end{proof}

\subsubsection{Generators of $C_2$ with increasing
conjugation}

Let $f$ and $g$ be maps, whose iterations for the group $C_2$ each
and $h$ be increasing homeomorphism.
\begin{lemma}If the equality $g = h^{-1}(f(h))$
holds, then the maps $h$ moves end-points of $[p_2,\, q_2]$ to
end-points of $[p_1,\, q_1]$.\end{lemma}

\begin{proof}If follows from $g = h^{-1}(f(h))$ that
$g^2 = h^{-1}(f^2(h))$.

Assume that $x_0\in Fix(g^2)$ and $h(x_0)>q_1$. Then
$f^2(h(x_0))<h(x_0)$, because $f^2([0,\, 1])=[p_1,\, q_1]$.
Applying $h^{-1}$ to both sides of the inequality obtain
$g^2(x_0)=h^{-1}(f^2(h(x_0)))< h^{-1}(h(x_0))=x_0$, which
contradicts to $x_0\in Fix(g^2)$. The analogical consideration of
the case $h(x_0)<p_1$ yields that $h([p_2,\, q_2])\subset [p_1,\,
q_1]$. It follows from the equality $f=h(g(h^{-1}))$ that $f^2 =
h(g^2(h^{-1}))$, whence $h^{-1}([p_1,\, q_1])\subset [p_2,\,
q_2]$. Obtain from this that $h([p_2,\, q_2])=([p_1,\, q_1])$,
which means that $h(p_2)=p_1$ and $h(q_2)=q_1$.\end{proof}

\begin{lemma}\label{monotinni-f-g}If $[p_1,\,
q_1]=[p_2,\, q_2]=[0,\, 1]$, then maps $f$ and $g$ are conjugated
via the increasing homeomorphism.\end{lemma}

\begin{proof}Let $x_0^f$ and $x_0^g$ be fixed points
of $f$ and $g$ correspondingly. Define the new maps $f_1$ and
$g_1$ as follows. $f_1(x)=f(x)$ for $x\leq x_0^f$ and $f_1(x)=x$
for $x>x_0^f$; also define $g_1(x)=g(x)$ for $x\leq x_0^g$ and
$g_1(x)=x$ for $x>x_0^g$.

Construct the increasing maps $h$, which defines the conjugation
of $f_1$ and $g_1$. The the graph of $h$ passes through the point
$(x_0^g,\, x_0^f)$ and be defined arbitrary on $[x_0^g,\, 1]$.
Then for arbitrary $x\in [x_0^g,\, 1]$ we have $h(x)\in [x_0^f,\,
1]$, whence $f_1(h(x))=h(x)$, which means that $h^{-1}(f_1(h(x)))=
h^{-1}(h(x)) =x$ and $g_1=h^{-1}(f_1(h))$ on $[x_0^g,\, 1]$. Since
$g_1$ is monotone on $[0,\, x_0^g]$, then for the conjugateness of
$f_1$ and $g_1$ it is enough to take $h=f_1^{-1}(h(g_1))$ on
$[0,\, x_0^g]$, where $h$ at the right hand side of defined
earlier, because $g_1^{-1}([0,\, x_0^g])=[x_0^g,\, 1]$.

Show that the map $h$, which is constructed in this way, defines
the conjugation of maps $f$ and $g$. The equality
$g(x)=h^{-1}(f(h(x)))$ for all $x\in [0,\, x_0^g]$ follows from
the construction. Since there are compositions of monotone
functions from the left and from the right of the equality
$g_1(x)=h^{-1}(f_1(h(x)))$, then write the the equality of there
inverses and obtain $g_1^{-1}(x)=h^{-1}(f_1^{-1}(h(x)))$ for $x\in
[0,\, x_0^g]$. Since $f^2=id$ and $g^2=id$ then $g_1^{-1}(x)=g(x)$
and $f_1^{-1}(h(x))=f(h(x))$ for $x\in [x_0^g,\, 1]$, which means
that $g = h^{-1}(f(h))$. The last proves the Lemma.\end{proof}

\begin{lemma}\label{extremumy-v-extremumy2}Let the
maps $g=h^{-1}(f(h))$ be conjugated to $f$ via increasing $h$.
Then for arbitrary $i\in [1,\, n]$ the equality $h(b_i)=a_i$ holds
and vectors $v_f$ and $w_g$ are equivalently ordered.\end{lemma}

\begin{proof} Proof of this lemma is analogical to the
proof of Lemma~\ref{extremumy-v-extremumy}.
\end{proof}
\begin{lemma}\label{lema:29}Let $f$ and $g$ be
conjugated via $h$ and ${g=h^{-1}(f(h))}$. Then vectors
$\widetilde{v}_f$ and $\widetilde{w}_g$ are equivalently
ordered.\end{lemma}

\begin{proof}Plug the value $x=g(b_i)$ into the
equality $g(x)=h^{-1}(f(h(x)))$ and obtain
$g^2(b_i)=h^{-1}(f(h(g(b_i))))$. Since according to
Lemma~\ref{spriajenist:constructyvno} the graph of $h$ pathes
through the points $(g(b_i),\, f(a_i))$ for all $i,\, 1\leq i\leq
n$, then the obtained equality is equivalent to
$g^2(b_i)=h^{-1}(f(f(a_i)))$. Plugging the left and right part of
the obtained equality into $h$, obtain $h(g^2(b_i))=f^2(a_i)$,
which means that $h$ passes through points $(g^2(b_i),\,
f^2(a_i))$ for all $i\in [1,\, n]$. Whence, the lemma follows from
the monotonicity of $h$, and that it passes through $(g(b_i),\,
f(a_i))$ for all $i\in [1,\, n]$.\end{proof}

\begin{theorem}\label{th:f2=f3}The maps $f$
and $g$ are conjugated via the increasing homeomorphism if and
only if when $m=n$ numbers or extremums, which are end-points of
intervals $f([0,\, 1])$ and $g([0,\, 1])$ coincide and vectors
$\widetilde{v}_f$ and $\widetilde{w}_g$ are equivalently ordered.
\end{theorem}

\begin{proof}The necessity is proved in
Lemma~\ref{lema:29}.

Prove the conjugateness of $f$ and $g$.

Since vectors $\widetilde{v}_f$ and $\widetilde{w}_g$ are
equivalently ordered, then there exists the maps $h_1:\, [p_2,\,
q_2]\longrightarrow [p_1,\, q_1]$, which defined the conjugation
of $f$ and $g$ on the set of their periodical points and passes
through points $(g(b_i),\, f(a_i))$ and $(g^2(b_i),\, f^2(a_i))$
for all $i,\, 1,\ldots,\, n$. This maps should be constructed in
the same way as in the proof of Lema~\ref{monotinni-f-g}.

The maps $h$ should be constructed on the set $[0,\, 1]\setminus
[p_2,\, q_2]$ in the same manner as in the proof of
Lemma~\ref{spriajenist:constructyvno}.
\end{proof}

\subsubsection{Decreasing conjugation}

Describe the classes of conjugacy via decreasing homeomorphism of
continuous interval maps, whose iterations form a group. An
arbitrary decreasing maps $h$ can be represented as
$h(x)=1-h_1(x)$, where $h_1(x)=1-h(x)$ is increasing map. Whence
the conjugation vis increasing $h$ is a composition of
conjugations via $1-x$ and via increasing homeomorphism. The
action of $1-x$ on the graph of $f$ can be interpreted as sequent
symmetrical reflecting it in the line $x=1/2$ and symmetrical
reflecting in the line $y=1/2$.

\begin{theorem}\label{theor:8}
Maps $f$ and $g$ are conjugated via decreasing homeomorphism if
and only if $m=n$, numbers of extremums, which correspond to end
points of $f([0,\, 1])$ and $g([0,\, 1])$ coincide and vectors
$\widetilde{v}_f$ and $w_g^*$ are equally ordered, where
$$ w_g^*=(g(b_m),\, g^2(b_m),\, \ldots,\, g(b_1),\, g^2(b_1)).$$
\end{theorem}

\newpage

\section{Constructing of the conjugation}\label{sect-Pobudowa}

\subsection{Values of conjugation on the dense set}\label{sect-Pobudowa-1}

Consider continuous maps $f,\, g:\, [0,\, 1]\rightarrow [0,\, 1]$,
which are defined as follows.

\begin{equation} \label{eq:f} f(x) =
\left\{\begin{array}{ll}
2x,& 0\leq x< 1/2;\\
2-2x,& 1/2 \leqslant x\leqslant 1,
\end{array}\right.
\end{equation}and\begin{equation}\label{eq:f-v}
g(x) = \left\{\begin{array}{ll} g_l(x),& x\leqslant v;\\
g_r(x),& x>v,
\end{array}\right.
\end{equation}

\noindent where $v\in (0,\, 1)$ is fixed and functions $g_l,\,
g_r$ are continuous monotone such that $g_l(0)=g_r(1)=0$, $g_l(v)
=\lim\limits_{x\rightarrow v-}g_r(x) = 1$. The problem about
conjugateness of $f$ and $g$ were stated at first
in~\cite{Ulam-1964} in the following theorem.

\begin{theorem}\cite[Appendix 1, \S
3]{Ulam-1964}\label{Theor:9} Let $f$ be of the form~(\ref{eq:f})
and $g$ be a convex function of the form~(\ref{eq:f-v}). Consider
the whole pre image of $1$ under $f$, i.e. the such the smallest
set $M_f$ that $1\in M_f$ and $f(x)\in M_f$ yields $x\in M_f$.
Maps $f$ and $g$ are topologically conjugated if and only if sets
$M_f$ is combinatorially equivalent, with $M_g$ and additionally
$\overline{M}_g = [0,\, 1]$.
\end{theorem}

\begin{proof}[Proof from~\cite{Ulam-1964}] The
necessity is obvious, because the whole pre image of $1$ is the
the set of binary-rational numbers, i.e. those rational numbers,
whose denominator is a power of $2$.

We will prove the existence of conjugation constructively. Tale
$h(1/2)=v$. Then take $h(1/4)$ to be the smallest of numbers
$g^{-1}(v)$, and $h(3/4)$ be the greatest of $g^{-1}(v)$. Take
$h(1/8)$ the smallest of $g^{-1}(h(3/4))$ and so on. Continuing
this way, obtain the function which is defined on binary rational
numbers of $(0,\, 1)$. Define $h$ on the whole $[0,\, 1]$ by the
continuity. This definition by continuity is possible, because
$M_g$ is dense. The maps $h$ would obviously be monotone and its
continuity yields that there exists $h^{-1}$. The equality $h(f) =
g(h)$ follows from the construction.
\end{proof}

This section is devoted to the proof in details of
Theorem~\ref{Theor:9}.

Assume that $f$ and $g$ are topologically conjugated and there
exists a homeomorphism $h$ such that the diagram
\begin{equation}\label{eq:main-comut-Diagr}  \begin{CD}
[0,\, 1] @>f >> & [0,\, 1]\\
@V_{h} VV& @VV_{h}V\\
[0,\, 1] @>g>> & [0,\, 1].
\end{CD}
\end{equation} is commutative.
The commutativity of this diagram is equivalent to that $h$ is a
solution of the functional equation
\begin{equation}\label{eq:45} h(f(x)) = g(h(x)).
\end{equation}

We will obtain the existence of this homeomorphism later, but now
we will find some its properties in the assumption that the
homeomorphism exists. Precisely, we will find the values of $h$ on
the dense set of $[0,\, 1]$.

\begin{lemma}\label{lema:h(0)(1)}If the
homeomorphism $h$ defines the topological conjugation of $f$ and
$g$, then it increase, i.e. $h(0)=0$ and $h(1)=1$.
\end{lemma}

\begin{proof}Since $h$ is a homeomorphism then it is
either increase or decrease. Since $h$ maps the interval $[0,\,
1]$ onto itself then either $h(0)=0$ and $h(1)=1$ or $h(0)=1$ and
$h(1)=0$. Plug the value $x=0$ into the equality~(\ref{eq:45}) and
obtain $$ h(f(0)) = g(h(0)).
$$ Notice, that this plugging may illustrated more clearly by the
commutative diagram
$$ \xymatrix{
0 \ar^{f}[rr] \ar_{h}[d] && f(0) \ar^{h}[d]\\
h(0) \ar^{g}[rr] && g(h(0))\lefteqn{=h(f(0)),} }
$$
which is obtained from~(\ref{eq:main-comut-Diagr}) by plugging
$x=0$ into left top angle.

In any way, since $f(0)=0$ then equality $h(0)=g(h(0))$ holds,
i.e. $x=h(0)$ is a fixed point of $g$.

Nevertheless, the point $x=1$ is not fixed for $g$, but $x=0$ is.
Whence, $h(0)=0$, which is necessary.
\end{proof}

\begin{definition}\label{def:01}For
every point $x^*\in [0,\, 1]$ we would say that the value $y^*$ of
the homeomorphism $h$ at $x^*$ is \textbf{conditionally found}, id
the following statement holds. If $h$ defines a conjugation of $f$
and $g$, then $h(x^*)=y^*$. We use the word ``conditionally'' for
paying the attention to that the question about existence of $h$
is still open.
\end{definition}

For example, in terms of Definition~\ref{def:01},
Lemma~\ref{lema:h(0)(1)} states that the maps $h$ is conditionally
found at points $0$ and $1$. Show that the conditional value of
$h$ at $x^*=1/2$ exists and $h(1/2)=v$.

\begin{lemma}\label{lema:30}If $h$ is the
conjugation of $f$ and $g$, then $h(1/2)=v$.
\end{lemma}

\begin{proof}Lemma follows from the commutativity of
the diagram
$$
\begin{CD}
1/2 @>f >> & 1\\
@V_{h} VV& @VV_{h}V\\
h(1/2) @>g>>& 1,
\end{CD}$$ which
is obtained from~(\ref{eq:main-comut-Diagr}) by plugging $x=1/2$
into left top angle.
\end{proof}

For an arbitrary point $x^*$ where the value of $h$ is
conditionally found, consider the pre image $\widetilde{x}$ under
$f$, i.e. $f(\widetilde{x})=x^*$ and consider the diagram $$
\begin{CD}
\widetilde{x} @>f >> & x^*\\
@V_{h} VV& @VV_{h}V\\
h(\widetilde{x}) @>g>>& h(x^*).
\end{CD}$$
This commutative diagram let us to find the value
$h(\widetilde{x})$ under the assumption that $\widetilde{x}$,
$x^*$ and $h(x^*)$ are known.

If $x^*\neq 1$ then there are two choices for the pre image
$\widetilde{x}$ of $x^*$. Is we search $\widetilde{x}<1/2$, then
formulas~(\ref{eq:f}) give that $\widetilde{x} = x^*/2$. Searching
$\widetilde{x}>1/2$, obtain $\widetilde{x} = \frac{2-x^*}{2}$.

Let the new pre image $\widetilde{x}$ of $x^*$ is found as
$\widetilde{x} = x^*/2$, i.e. $\widetilde{x}<1/2$. Since by
Lemma~\ref{lema:h(0)(1)} the homeomorphism $h$ increase, then if
follows from Lemma~\ref{lema:30} that $h(\widetilde{x})<v$. Since
$h(\widetilde{x})<v$, then $g(h(\widetilde{x})) =
g_l(h(\widetilde{x}))$. From another hand, it follows from the
commutativity of diagram that $g_l(h(\widetilde{x}))= h(x^*)$,
whence $h(\widetilde{x}) = g_l^{-1}(h(x^*))$.

From the analogical reasonings give that if $\widetilde{x}$ s
found from the condition $\widetilde{x}>1/2$, then $h(x^*) =
g_r(h(\widetilde{x}))$, i.e. $h(\widetilde{x}) =
g_r^{-1}(h(x^*))$.

This construction proves the following lemma.

\begin{proposition}\label{lema:25}
If a homeomorphism $h$ is a conjugation of $f$ and $g$ and
$h(x^*)=y^*$ for some $x^*$ and $y^*$ then for any point
$\widetilde{x}$ such that $f(\widetilde{x}) =x^*$ the following
implications hold.

1. If $\widetilde{x}\leq 1/2$, then $h(\widetilde{x}) =
g_l^{-1}(y^*)$;

2. If $\widetilde{x}>1/2$, then $h(\widetilde{x}) =
g_r^{-1}(y^*)$.
\end{proposition}

Show that, starting from $x^*=1/2$, one may use
Lemma~\ref{lema:30} to obtain the set, which is dense in $[0,\,
1]$ such that $h$ would be conditionally found at each point of
this set.

\begin{lemma}\label{note:f}If the the binary
decomposition of the number $x\in [0,\, 1]$ is
$$ x = 0,\alpha_1\alpha_2\ldots\, $$ then the binary decomposition of
$f(x)$ is
$$ f(x)= \left\{
\begin{array}{ll}
0,\alpha_2\alpha_3\ldots \alpha_n\ldots, & \text{if }\alpha_1 =
0,\\
0,\overline{\alpha}_2\overline{\alpha}_3\ldots
\overline{\alpha}_n\ldots, & \text{if }\alpha_1 = 1,
\end{array}\right.
$$ where $\overline{\alpha}_i = 1 - \alpha_i$.
\end{lemma}

\begin{proof}Lemma follows from
formulas~(\ref{eq:f}) for the function $f$.
\end{proof}

\begin{example}\label{ex:04}Find the pre images of
$1/2$ under $f$.
\end{example}

\begin{proof}[Deal of the example] The binary
decomposition of $1/2$ is $0.1$. It follows from
Lemma~\ref{note:f}, that its pre images are $x_1=0.01=1/4$ and
$x_2=0.10(1) = 0.11=1/2+1/4=3/4$.

Pre images of $x_1$ and $x_2$ under $f$ also can be obtained by
Lemma~\ref{note:f}. Pre images of $x_1$ under $f$ have binary
decompositions $x_{11}=0.001=1/8$ and
$x_{12}\hm{=}0.110(1)=0.111=1/2+1/4+1/8 = 7/8$.

Pre images of $x_2$ under $f$ have the binary decompositions
$x_{21}=0.011=1/4+1/8 = 3/8$ and $x_{22}=0.100(1)= 0.101 = 1/2 +
1/8  = 5/8$.

Whence the set of pre images of $1/2$ is the set $\{ 0,\, 1/2,\,
1\}$. The set of its pre images is $\{k/4,\, 0\leq k\leq 4\}$ and
its pre images set is $\{k/8,\, 0\leq k\leq 8\}$.
\end{proof}

The obtained property of pre images of $1/2$ can be generalized as
follows.

\begin{notation}Denote with $A_n,\ n\geqslant 1$ the
set of all those points of the interval $[0,\, 1]$ such that
$$ f^{n}(A_n) = 0.
$$\end{notation}

\begin{proposition}\label{lema:An}$$ A_n
= \left\{0,\, \frac{1}{2^{n-1}};\ldots
;\frac{2^{n-1}-1}{2^{n-1}},\, 1\right\}.
$$
\end{proposition}

\begin{proof}Show by induction at first that the
every $n\geq 1$ the set $A_n$ consists of $2^{n-1}+1$ elements.
The base of induction for $n=1$ is obvious. Since for every $n>1$
each of points of $A_n$ except $x_0=1$ has two pre images, then
the cardinality of $A_{n+1}$ is $2\cdot(2^{n-1}+1-1)+1 = 2^n+1$.

Show that $2^{n-1}+1$ of $A_n$ are those elements, which are
mentioned in the Statement. Use the inductive reasonings again.
The base for $n=1$ is clear. Consider
$$ A = \left\{0,\, \frac{1}{2^{n}};\ldots
;\frac{2^{n}-1}{2^{n}},\, 1\right\}
$$ and prove that $A = A_{n+1}$, assuming that the statement for
the previous $n$ is correct.

Consider an arbitrary $x_0$ of the form $x_0=\frac{p}{2^n}$, where
$p$ is as integer between 0 and $2^n$. Find $f(x_0)$. If $x_0\leq
1/2$, then $f(x_0)=2x_0 = \frac{p}{2^{n-1}}\in A_{n-1}$. If
$x_0>1/2$, then $f(x_0)\hm{=}2-2x_0 = \frac{2^n-p}{2^{n-1}} \in
A_{n-1}$.

The last finishes the proof.
\end{proof}

Show how to find the conditional values of $h$ at $A_n$ under the
assumption that values at $A_{n-1}$ are known.

Consider an arbitrary point $\widetilde{x}\in A_n$. It follows
from the definition of the sets $A_1,\ldots,\, A_n$ that
$\widetilde{x} =f(x_0)\in A_{n-1}$, whence the value $y_0 =
h(f(x_0))$ is found earlier. Now the value $h(\widetilde{x})$ can
be found by Proposition~\ref{lema:25}, dependently on whether
$\widetilde{x}<1/2$, or $\widetilde{x}>1/2$.

we can apply for $g$ the those reasonings concerning $f$, which
were used in the construction of the sets $A_n$.

\begin{notation}Denote by $B_n,\ n\geqslant 1$ the
set of all points $x\in [0,\, 1]$ such that
$$ g^{n}(x) = 0.
$$\end{notation}

The evident lemma holds.
\begin{lemma}\label{lema:B1}Independently on $v,\,
g_l$ and $g_r$ the equality $$ B_1 = \{ 0,\, 1\}
$$ holds.
\end{lemma}

\begin{theorem}\label{theor:10}If a
homeomorphism $h:\, [0,\, 1]\rightarrow [0,\, 1]$
satisfies~(\ref{eq:45}), then it increase and $h(A_n) = B_n$.
\end{theorem}

\begin{proof}Show that the cardinality of $B_n$
equals $2^{n-1}+1$.

Check with the mathematical induction that for any $n>1$ the graph
of $f_v^n$ is piecewise linear and is consisted of $2^n$ intervals
being linear on each of them such that the image of each of these
$2^n$ intervals is the whole interval $[0,\, 1]$. If $n=1$ the the
statement follows from the form the graph of  $f_v$. Let for
$B_{n-1}$ the theorem is correct. For each of the monotone
intervals $P$ of the maps $f_v^{n-1}$ one have $f_v^{n-1}(P) =
[0,\, 1]$. In this case the graph of the maps $f_v^n(P)$ in
consisted of two monotone intervals and the image of each of them
under the acting of $f_v^n$ will be the whole interval $[0,\, 1]$.
Since the interval $P$ is arbitrary obtain the statement the
necessary cardinality of $B_n$.

From the construction of $A_n$ and~(\ref{eq:45}) obtain that for
every $x\in A_n$ the equality $f_v^n(h(x)) =0$ holds, i.e.
$h(A_n)\subseteq B_n$. Now Theorem follows from that cardinalities
of $A_n$ and $B_n$ coincide.
\end{proof}

For every $n\geq 1$ denote by $\alpha_{n,k},\, \beta_{n,k},\,
0\leq k\leq 2^{n-1} $ the increasingly ordered elements of $A_n$
and $B_n$ correspondingly, i.e. for any $k_1<k_2$ inequalities
$\alpha_{n,k_1}<\alpha_{n,k_2}$ and $\beta_{n,k_1}<\beta_{n,k_2}$
hold.

Denote $\mathcal{A}=\bigcup _{n\ge 1} A_{n} $ and
$\mathcal{B}=\bigcup _{n\ge 1} B_{n} $. For every $n\geq 1$ denote
$h_n$ the piecewise linear maps, all whose braking points belong
to $A_n$ and such that for every $k,\, 0\leq k\leq 2^{n-1}$ the
equality $$ h_n(\alpha_{n,k})=\beta_{n,k}
$$ holds.

\begin{lemma}\label{lema:31}For every $n\geq 1$ and
$k,\, 0\leq k\leq 2^{n-1}$ the equality
\begin{equation}\label{eq:49}
h_n(f(\alpha_{n,k})) = g(h_n(\alpha_{n,k}))
\end{equation} holds.
\end{lemma}

\begin{proof}Notice, that $\alpha_{n,2^{n-2}} =
\frac{1}{2}$. Indeed, consider the sets $A_{n-1}$, $A_n^- =
A_{n,k}\cap [0,\, 1/2)$ and $A_n^+ = A_{n,k}\cap (1/2,\, 1]$. The
set $A_n^-$ is obtained from $A_{n-1}$ by taking smaller pre
images under $f$ (i.e. pre images under the maps $x\mapsto 2x$)
and $A_n^+$ is obtained from $A_{n-1}$ by taking greater pre
images under $f$ (i.e. pre images under $x\mapsto 2-2x$). Whence
$\# A_n^- = \#A_n^+$ and $1/2$ is the middle point of $A_n$.

Notice that if $\alpha_{n,k}\leq 1/2$, then $f(\alpha_{n,k}) =
\alpha_{n-1,k}$. This follows from the monotone increasing of $f$
on $[0,\, 1/2]$ and that $f(A_n^-) = A_{n-1}$.

Notice, that if $\alpha_{n,k}\geq 1/2$, then $f(\alpha_{n,k}) =
\alpha_{n-1,2^{n-2}-k}$. This follows from monotone decreasing of
$f$ on $[0,\, 1/2]$ and that $f(A_n^+) = A_{n-1}$.

Analogically obtain that $\beta_{n,2^{n-2}}=v$. More then this, if
$\beta_{n,k}\leq v$, then $g(\beta_{n,k}) = \beta_{n-1,k}$ and if
$\beta_{n,k}\geq v$, then $g(\beta_{n,k}) =
\beta_{n-1,2^{n-2}-k}$.

Consider two cases: when $\alpha_{n,k}\leq 1/2$ and when
$\alpha_{n,k}> 1/2$.

Assume that $\alpha_{n,k}\leq 1/2$. Then~(\ref{eq:49}) follows
from the following chain of equalities:
$$ h_n(f(\alpha_{n,k})) =
h_n(\alpha_{n-1,k}) = \beta_{n-1,k} = g(\beta_{n,k}) =
g(h(\alpha_{n,k})).
$$

Assume that $\alpha_{n,k}\geq 1/2$. Then~(\ref{eq:49}) follows
from the following chain of equalities:
$$ h_n(f(\alpha_{n,k})) =
h_n(\alpha_{n-1,2^{n-2}-k}) = \beta_{n-1,2^{n-2}-k} =
g(\beta_{n,k}) = g(h(\alpha_{n,k})).
$$\end{proof}

For arbitrary $x_0\in [0,\, 1]$ denote $h^+(x_0) =
\varlimsup\limits_{n\rightarrow \infty}h_n(x_0)$ and $h^-(x_0)=
\varliminf\limits_{n\rightarrow \infty}h_n(x_0)$.

\begin{lemma}If for some $x_0\in (0,\, 1)$ the
inequality holds $h^+(x_0)\neq h^-(x_0)$, then maps $f$ and $g$
are not conjugated. In this case $\overline{\mathcal{B}} \neq
[0,\, 1]$.
\end{lemma}

\begin{proof}Assume that $h$ is a conjugation of $f$
and $g$, which satisfies the functional equation~(\ref{eq:45}).

Let the binary decomposition of $x_0$ be
$$ x_0 = 0,x_1x_2\ldots\, x_n\ldots\, .
$$
For every $k\in \mathbb{N}$ denote by $x_k^-$ the number, whose
binary decomposition is
$$ x_k^- = x_0 = 0,x_1x_2\ldots\, x_k
$$ and denote $$x_k^+ =x_k^-+\frac{1}{2^k}.$$

Since $\{ x_k^-,\, x_k^+\}\subset A_k$, then $h(x_k^-)=h_k(x_k^-)$
and $h(x_k^+)=h_k(x_k^+)$.

Since by Lemma~\ref{lema:h(0)(1)} the homeomorphism $h$ increase
then $h^+(x_0)>h^-(x_0).$

By Theorem~\ref{Theor:9} for every $k$ homeomorphism $h_k$
increase and $$ h(x_k^-) = h_k(x_k^-) \leq h^-(x_0) < h^+(x_0)
\leq h_k(x_k^+) = h(x_k^+),
$$ whence $h(x_k^+)-h(x_k^-)\geq h^+(x_0) - h^-(x_0)$. But the
last inequality contradicts to continuality of $h$ at $x_0$.

The fact that $\mathcal{B}$ is not dense in $[0,\, 1]$ follows
from that $\mathcal{B}\cap (h^-(x_0),\, h^+(x_0)) = \emptyset$.
\end{proof}

\begin{lemma}If the set $\mathcal{B}$ is dense in
$[0,\, 1]$, then there exists a homeomorphism $h:\, [0,\,
1]\rightarrow [0,\, 1]$, which satisfies the functional
equation~(\ref{eq:45}).
\end{lemma}

\begin{proof}For every $x\in [0,\, 1]$ define $h(x)$
as
$$ h(x) = \lim\limits_{n\rightarrow \infty}h_n(x).
$$
The existence of the limit and the continuity of $h$ follows from
the density of $\mathcal{B}$ in $[0,\, 1]$.

The monotonicity of $h$ follows from that for every $n$
homeomorphism $h_n$ increase.

By Lemma~\ref{lema:31} equality~(\ref{eq:45}) holds for all $x\in
\mathcal{A}$. The equality~(\ref{eq:45}) for $x\in [0,\,
1]\backslash \mathcal{A}$ follows from continuity of the function
$h(f(x))-g(h(x))$, which is a composition of continuous ones.
\end{proof}

\subsection{Existing of the conjugation}\label{sect-isnuv}

In this section we continue the study of the problem on the
topological conjugacy of maps, which were considered in
Section~\ref{sect-Pobudowa} above.

\begin{equation}\label{eq:46} f(x) = \left\{\begin{array}{ll}
2x,& 0\leq x< 1/2;\\
2-2x,& 1/2 \leqslant x\leqslant 1.
\end{array}\right.
\end{equation}

Instead of the maps $g$ consider more precise one $f_v: [0,\,
1]\rightarrow [0,\, 1]$, which depends on $v\in (0,\, 1)$ and if
defined by the following formulas
\begin{equation} \label{eq:47}f_v(x) =
\left\{\begin{array}{ll} \frac{x}{v},& x\leqslant v;\\
 \frac{1-x}{1-v},&
x>v.
\end{array}\right.
\end{equation}

\begin{figure}[htbp]
\begin{minipage}[h]{0.49\linewidth}
\begin{center}\begin{picture}(100,125)
\put(100,0){\line(0,1){100}} \put(0,100){\line(1,0){100}}
\put(0,0){\vector(1,0){120}} \put(0,0){\vector(0,1){120}}
\put(0,0){\line(1,2){50}} \put(50,100){\line(1,-2){50}}
\put(40,50){f(x)}
\end{picture}\end{center}
\centerline{a)}
\end{minipage}
\hfill
\begin{minipage}[h]{0.49\linewidth}
\begin{center}\begin{picture}(100,125)
\put(100,0){\line(0,1){100}} \put(0,100){\line(1,0){100}}
\put(0,0){\vector(1,0){120}} \put(0,0){\vector(0,1){120}}
\put(0,0){\line(3,4){75}} \put(75,100){\line(1,-4){25}}
\put(50,50){$f_v(x),$} \put(45,30){$v=3/4$}
\end{picture}\end{center}
 \centerline{ b)}
\end{minipage}
\caption{}\label{fig:4}
\end{figure}

Denote by $\mathcal{A}$, $\mathcal{B}$, $A_n$ and $B_n$ the same
sets and denote by $h_n$ the same maps, which were considered in
Section~\ref{sect-Pobudowa-1}, but now use the function $f_v$
instead of $g$.

Remind that the sets $A_n$ and $B_n$ are such that $f^n(A_n) =
f_v^n(B_n) = 0$ and by Theorem~\ref{theor:10} for a homeomorphism
$h$ the equality

\begin{equation}\label{eq:48}h(f) = f_v(h)
\end{equation}
implies that $h(A_n) = B_n$.

\subsubsection{ Find the values of conjugation at the
dense set}

Conditional values of $h$ at points of $A_2 =\{ 0,\, 1/2;\,1 \}$
follow from Lemmas~\ref{lema:h(0)(1)} and~\ref{lema:30},
precisely: $h(0)=0$, $h(1/2)=v$, $h(1)=1$. Correspond three points
of the graph of $h$ for $v=3/4$ are given at Figure~\ref{fig:5}a).

\begin{example}\label{ex:05}Find the conditional
values of $h$ at $A_3$.
\end{example}

\begin{proof}[Deal of the example] As conditional
values of $h$ at $A_2$ are found in Lemmas~\ref{lema:h(0)(1)}
and~\ref{lema:30}, it is sufficient to find $h$ at $A_3\setminus
A_2 = \{ 1/4,\, 3/4\}$.

For the point $1/4$ by i. 1 of Proposition~\ref{lema:25} one have
that $h(1/4)=vh(1/2).$ Since by Lemma~\ref{lema:30} $h(1/2)=v$,
then $h(1/4)=v^2$.

For the points $3/4$ by i. 2 of Proposition~\ref{lema:25} have
that $h(3/4)=1-(1-v)h(1/2) = 1-v(1-v)$.

Values of $h$ on $A_3$, if this homeomorphism exists, are given at
Figure~\ref{fig:5}b).
\end{proof}

\begin{figure}[htbp]
\begin{minipage}[h]{0.23\linewidth}
\begin{center}
\begin{picture}(80,80) \put(64,0){\line(0,1){64}}
\put(0,64){\line(1,0){64}} \put(0,0){\vector(1,0){80}}
\put(0,0){\vector(0,1){80}} \put(0,0){\circle*{2}}
\put(32,51){\circle*{2}} \put(64,64){\circle*{2}}
\end{picture}
\centerline{a)}\end{center}
\end{minipage}
\hfill
\begin{minipage}[h]{0.23\linewidth}
\begin{center}
\begin{picture}(80,80)
\put(64,0){\line(0,1){64}} \put(0,64){\line(1,0){64}}
\put(0,0){\vector(1,0){80}} \put(0,0){\vector(0,1){80}}
\put(0,0){\circle*{2}} \put(16,41){\circle*{2}}
\put(32,51){\circle*{2}} \put(48,54){\circle*{2}}
\put(64,64){\circle*{2}}
\end{picture}
\centerline{b)}\end{center}
\end{minipage}
\hfill
\begin{minipage}[h]{0.23\linewidth}
\begin{center}
\begin{picture}(80,80)
\put(64,0){\line(0,1){64}} \put(0,64){\line(1,0){64}}
\put(0,0){\vector(1,0){80}} \put(0,0){\vector(0,1){80}}
\put(0,0){\circle*{2}} \put(8,33){\circle*{2}}
\put(16,41){\circle*{2}} \put(24,43){\circle*{2}}
\put(32,51){\circle*{2}} \put(40,53){\circle*{2}}
\put(48,54){\circle*{2}} \put(56,56){\circle*{2}}
\put(64,64){\circle*{2}}
\end{picture}
\centerline{c)}\end{center}
\end{minipage}
\hfill
\begin{minipage}[h]{0.23\linewidth}
\begin{center}
\begin{picture}(80,80)
\put(64,0){\line(0,1){64}} \put(0,64){\line(1,0){64}}
\put(0,0){\vector(1,0){80}} \put(0,0){\vector(0,1){80}}
\put(0,0){\circle*{2}} \put(4,26){\circle*{2}}
\put(8,33){\circle*{2}} \put(12,34){\circle*{2}}
\put(16,41){\circle*{2}} \put(20,43){\circle*{2}}
\put(24,43){\circle*{2}} \put(28,45){\circle*{2}}
\put(32,51){\circle*{2}} \put(36,53){\circle*{2}}
\put(40,53){\circle*{2}} \put(44,53){\circle*{2}}
\put(48,54){\circle*{2}} \put(52,55){\circle*{2}}
\put(56,56){\circle*{2}} \put(60,57){\circle*{2}}
\put(64,64){\circle*{2}}
\end{picture}
\centerline{d)}\end{center}
\end{minipage}
\hfill \caption{} \label{fig:5}
\end{figure}

\begin{example}\label{ex:06}Find conditionally
values of $h$ on $A_4$.
\end{example}

\begin{proof}[Deal of the example] In the same manner
as in Example~\ref{ex:05}, find all the conditional values of $h$
only in the set $A_4\setminus A_3 = \{ 1/8,\, 3/8,\, 5/8,\, 7/8\}$
and use the data from Example~\ref{ex:05} and
Lemmas~\ref{lema:h(0)(1)} and~\ref{lema:30}.

For points $1/8$ and $3/8$, which are less than $1/2$, we have by
i. 1 of Proposition~\ref{lema:25} that $h(1/8)=vh(1/4)$. Since the
conditional value $h(1/4)$ is already found in
Example~\ref{ex:05}, then $h(1/8)=v^3$. Analogically,
$h(3/8)=vh(3/4)=v-v^2(1-v)$.

For points $5/8$ and $7/8$, which are grater than $1/2$, we have
by i. 1 of Proposition~\ref{lema:25} that
$h(5/8)=1-(1-v)h(f(5/8))\hm{=}1-(1-v)h(3/4)=1-(1-v)(1-v(1-v))$.
Analogically $h(7/8)=1-(1-v)h(1/4)1-v^2(1-v)$.

The conditional values of $h$ on $A_3$ are given at
Figure~\ref{fig:5}c).
\end{proof}

In the same manner as  it was done at Example~\ref{ex:06}, we may
find the conditional values of $h$ on $A_5$ and on $A_n$ for
$n>5$. Conditional values of $h$ on $A_5$ are given at
Figure~\ref{fig:5}d).

\subsubsection{Density of pre images}\label{Sect:DynVlsl}

By Theorem~\ref{Theor:9} the conjugation of $f$ and $f_v$ follows
from the density of $\mathcal{B}$ in $[0,\, 1]$, whence we will
concentrate on the proving of this fact. We will need the
following technical lemmas.

\begin{lemma}\label{lema:vlastfv1}Assume
for some numbers $a,\, b \in [0,\, 1]$ the increasing maps $g$
such that $g(a)=0$ and $g(b)=1$ is given and its graph on $[a,\,
b]$ is a line segment. Then the graph of $s=f_v(g)$ is a braking
line, consisted of two line segments such that $s(a) \hm{=}
s(b)=0$ and $s(t)=1$, where $$ t = a +v(b-a).
$$
\end{lemma}

\begin{proof}The tangent of $g$ is $k=\frac{1}{b-a}.$
The tangent of $s$ of the left segment of monotonicity is
$\frac{k}{v}$.

In this case the value of $t$ can be found from $s(t)=1$ as
follows. $t = a + \frac{v}{k} \hm{=}a +v(b -a).$
\end{proof}

\begin{lemma}\label{lema:vlastfv2}Assume
for some numbers $a,\, b \in [0,\, 1]$ the increasing maps $g$
such that $g(a)=1$ and $g(b)=0$ is given and its graph on $[a,\,
b]$ is a line segment. Then the graph of $s=f_v(g)$ is a braking
line, consisted of two line segments such that $s(a) \hm{=}
s(b)=0$ and $s(t)=1$, where $$ t = b-v(b-a).
$$
\end{lemma}

\begin{proof}The tangent of $g$ is
$k=\frac{-1}{\beta-\alpha}.$ The tangent of $s$ on the right
segment of monotonicity is $\frac{k}{v}$.

In this case the value $t$ can be found from $s(t)=1$ as follows.
$ t = \beta + \frac{v}{k} = \beta - v(\beta - \alpha).$
\end{proof}

The following corollary follows from Lemmas~\ref{lema:vlastfv1}
and~\ref{lema:vlastfv2}.

\begin{corollary}\label{corol:2}Let for some $a,\, b
\in [0,\, 1]$ the monotonic linear $g:\, [a,\, b]\rightarrow [0,\,
1]$ is given. Then the graph of $s=f_v(g)$ is breaking line, which
is consisted of two segments and for the extremum $t$ the
following equality of sets
$$ \left\{
\frac{t-a}{b-a},\, \frac{b-t}{b-a}\right\} = \{ v,\, 1-v\}
$$ holds.
\end{corollary}

\begin{lemma}\label{lema:Bszczilna}The set $B$ is
dance in~$[0,\, 1]$.
\end{lemma}

\begin{proof}For every $n\geqslant 1$ consider the
maximum $d_n$ of distances between points of the set $B_n$.

to prove the density $B$ is the same as to prove that
$$ \lim\limits_{n\rightarrow \infty}d_n = 0. $$ Consider
two arbitrary neighbor points $\beta_{i,n},\, \beta_{i+1;n}$ of
$B_n$.

Notice that $\beta_{n,k}$ are extremums of $f_v^{n-1}$.

It follows from Corollary~\ref{corol:2} that
$$ \left\{
\frac{\beta_{n+1,2k+1}-\beta_{n+1,2k}}{\beta_{n,k+1}-\beta_{n,k}},\,
\frac{\beta_{n+1,2k+2}-\beta_{n+1,2k+1}}{\beta_{n,k+1}-\beta_{n,k}}\right\}\in
\{v,\, 1-v\}.
$$

Whence the following bound
$$d_{n+1}\leqslant \max\{ v,\,
1-v\}\cdot d_n$$ holds for $d_n$ and $d_{n+1}$. It means the
density of $\mathcal{B}$ in $[0,\, 1]$ because $0< \max\{ v,\,
1-v\} <1$.
\end{proof}

The proof of Lemma~\ref{lema:Bszczilna} contains the proof of the
following Lemma.

\begin{lemma}\label{lema:34}Let $n>1$
and $0=\beta_0,\ldots \beta_{2^{n-1}}=1$ be points of $B_n$. Then
for every $i$ such that $\beta_i\in B_n\setminus B_{n-1}$ the
following statements hold:

1. $i\neq 0$, $i\neq 2^{n-1}$.

2. $\beta_{i-1}\in B_{n-1}$, $\beta_{i+1}\in B_{n-1}$.

3. $\left\{ \frac{\beta_i-\beta_{i-1}}{\beta_{i+1}-\beta_{i-1}},\,
\frac{\beta_{i+1}-\beta_{i}}{\beta_{i+1}-\beta_{i-1}} \right\} =
\{ v,\, 1-v \}$.
\end{lemma}

\begin{corollary}\label{corol:3}For
$n\in \mathbb{N}$ and $t\geq 2$ consider two neighbor points
$\beta_i,\, \beta_{i+1}$ of $B_{n+t}$, which does not belong to
$B_n$. Let $a,\, b\in B_n$ such neighbor points of $B_n$ that
$\beta_i,\, \beta_{i+1}\in (a,\, b)$. Then $$ (\min(v,\, 1-v))^t
(b-a) \leq \beta_{i+1} - \beta_i \leq (\max(v,\, 1-v))^t (b-a).
$$
\end{corollary}

The following theorem follows from Lemma~\ref{lema:Bszczilna} and
Theorem~\ref{Theor:9}.

\begin{theorem}\label{theor:homeom-jed}For
every $v\in (0,\, 1)$ the functional equation~(\ref{eq:48}) has a
solution in the class of homeomorphisms $h:\, [0,\, 1]\rightarrow
[0,\, 1]$ and this solution of unique and increasing.
\end{theorem}

\subsection{Example of non-conjugated maps}\label{Pobudowa-2}

Let the maps $f$ be defined by formulas~(\ref{eq:f}) and $g$ be
piecewise linear $[0,\, 1]\rightarrow [0,\, 1]$, whose graph
passes through points $(0,\, 0)$, $(1/2,\, 1)$, $(2/3,\, 2/3)$,
$(8/9,\, 5/9)$ and $(1,\, 0)$.

\begin{figure}[htbp]
\begin{minipage}[h]{0.3\linewidth}
\begin{center}
\begin{picture}(110,110)
\put(0,0){\line(1,1){110}} \put(0,0){\line(1,2){45}}
\put(0,0){\vector(1,0){110}} \put(0,0){\vector(0,1){110}}

\put(45,90){\line(1,-2){15}} \put(60,60){\line(2,-1){10}}
\qbezier(70,55)(70,55)(90,0)

\put(0,0){\circle*{3}} \put(45,90){\circle*{3}}
\put(60,60){\circle*{3}} \put(70,55){\circle*{3}}
\put(90,0){\circle*{3}}

\qbezier[25](60,60)(75,30)(90,0)

\end{picture}
\centerline{a) Graph of $g$}\end{center}
\end{minipage}
\hfill
\begin{minipage}[h]{0.3\linewidth}
\begin{center}
\begin{picture}(110,110)
\put(0,0){\line(1,1){110}} \put(0,0){\vector(1,0){110}}
\put(0,0){\vector(0,1){110}}

\put(0,0){\circle*{3}} \put(22.5,90){\circle*{3}}
\put(30,60){\circle*{3}} \put(35,55){\circle*{3}}
\put(45,0){\circle*{3}} \put(55,55){\circle*{5}}
\put(70,70){\circle*{5}} \put(74,90){\circle*{3}}
\put(90,0){\circle*{3}}

\Vidr{0}{0}{22.5}{90} \VidrTo{30}{60} \VidrTo{35}{55}
\VidrTo{45}{0} \VidrTo{55}{55} \VidrToBold{70}{70} \VidrTo{74}{90}
\VidrTo{90}{0}

\end{picture}
\centerline{b) Graph of  $g^2$}\end{center}
\end{minipage}
\hfill
\begin{minipage}[h]{0.3\linewidth}
\begin{center}
\begin{picture}(110,110)
\put(0,0){\line(1,1){110}} \put(0,0){\vector(1,0){110}}
\put(0,0){\vector(0,1){110}}

\Vidr{0}{0}{45}{90} \VidrTo{90}{0}

\put(20,40){\circle{8}} \put(40,80){\circle*{3}}
\put(80,20){\circle*{3}}

\Vidr{20}{20}{20}{40} \VidrTo{40}{40} \VidrTo{40}{80}
\VidrTo{80}{80} \VidrTo{80}{20} \VidrTo{20}{20}

\end{picture}
\centerline{c) Graph of $f$}\end{center}
\end{minipage}
\caption{} \label{fig:9}
\end{figure}

The graph of $g$ is given at Figure~\ref{fig:9}a. In other words,
graphs of $f$ and $g$ coincide for $x\in [0,\, 2/3]$ but $g$ has a
break for $x\in [2/3,\, 1]$ in the time, when $f$ in linear for
$x\in [2/3,\, 1]$.

It is easy to prove that these $f$ and $g$ are not topologically
conjugated.

\begin{lemma}\label{theor:4}The maps $f$ which is
defined by formulas~(\ref{eq:f}), is not topologically conjugated
to piecewise linear $g:\, [0,\, 1]\rightarrow [0,\, 1]$, all whose
braking points are $(0,\, 0)$, $(1/2,\, 1)$, $(2/3,\, 2/3)$,
$(8/9,\, 5/9)$ and $(1,\, 0)$.
\end{lemma}

\begin{proof}Consider the second iteration of $g$.
All the breaking points of the graph of $g^2$ are $$ M_2 = \{
(0,0),\, (1/4,1),\, (1/3,2/3),\, (7/30, 11/30),$$ $$ (1/2, 1),\,
(11/30, 11/30),\, (7/9, 7/9),\, (9/11,\, 1),\, (1, 0) \}.
$$ This graph
is given on figure~\ref{fig:9}b. It is evidently that for every
$x\in (11/30,7/9)$ the equality $g^2(x)=x$ holds. If $\psi_1$ is a
homeomorphism, which defines the topological conjugacy of $f$ and
$g$ then it follows from the commutative diagram
$$ \xymatrix{
x \ar^{g^2}[r] \ar_{\psi_1^{-1}}[d] & x
\ar^{\psi_1^{-1}}[d]\\
\psi_1^{-1}(x) \ar^{f^2}[r] & \psi_1^{-1}(x) }
$$ that $\psi_1^{-1}(x)$ is a fixed point
of $f^2$. This means that either $\psi_1^{-1}$ is not a
homeomorphism, of $f^2$ has an interval of fixed points. This
contradiction finishes the proof.
\end{proof}

In spite that we have proved in Lemma~\ref{theor:4} that $f$ and
$g$ are not topologically conjugated, we may consider the
reasonings, which are analogical to those, which were used during
the construction of the conjugation of $f$ and $f_v$, given by
formulas~(\ref{eq:f}) and~(\ref{eq:f-v}). Assume that $\psi_1$ is
monotone maps such that the following diagram
\begin{equation}\label{eq:32}\begin{CD}
[0,\, 1] @>f >> & [0,\, 1]\\
@V_{\psi_1} VV& @VV_{\psi_1}V\\
[0,\, 1] @>g>> & [0,\, 1].
\end{CD}\end{equation} is commutative.

More then this, consider the family of $g_v$ instead of $g$ such
that reasonings from the proof of Lemma~\ref{theor:4} might be
repeated and $g_v$ be possible to considered as approximations of
$g$ for some specific $v$. In this time construct $g_v$ to be
non-conjugated to $f$ for any $v$.

Let for every $v\in [2/3,\, 1)$ the maps $g_v$ be piecewise linear
on $[0,\, 2/3]$ and all its breaking points be $(0,\, 0)$,
$(1/2,\, 1/2)$ and $(2/3,\, 2/3)$. Let the tangent of $g_v$ on
$(2/3,\, v)$ be equal to $1/2$ and let $g$ be linear on $(v,\, 1)$
such that $g_v(1)=0$. The would guarantee, that $g^2$ would have
tangent $1$ in some neighborhood of $2/3$.

\begin{figure}[htbp]
\begin{minipage}[h]{0.3\linewidth}
\begin{center}
\begin{picture}(140,140)
\put(128,0){\line(0,1){128}} \put(0,128){\line(1,0){128}}
\put(0,0){\vector(1,0){140}} \put(0,0){\vector(0,1){140}}

\put(0,0){\circle*{2}} \put(2,2){\circle*{2}}
\put(4,4){\circle*{2}} \put(6,6){\circle*{2}}
\put(8,8){\circle*{2}} \put(10,10){\circle*{2}}
\put(12,12){\circle*{2}} \put(14,14){\circle*{2}}
\put(16,16){\circle*{2}} \put(18,18){\circle*{2}}
\put(20,20){\circle*{2}} \put(22,23){\circle*{2}}
\put(24,25){\circle*{2}} \put(26,26){\circle*{2}}
\put(28,28){\circle*{2}} \put(30,30){\circle*{2}}
\put(32,32){\circle*{2}} \put(34,34){\circle*{2}}
\put(36,36){\circle*{2}} \put(38,38){\circle*{2}}
\put(40,39){\circle*{2}} \put(42,41){\circle*{2}}
\put(44,46){\circle*{2}} \put(46,48){\circle*{2}}
\put(48,49){\circle*{2}} \put(50,51){\circle*{2}}
\put(52,53){\circle*{2}} \put(54,55){\circle*{2}}
\put(56,57){\circle*{2}} \put(58,58){\circle*{2}}
\put(60,60){\circle*{2}} \put(62,62){\circle*{2}}
\put(64,64){\circle*{2}} \put(66,66){\circle*{2}}
\put(68,68){\circle*{2}} \put(70,70){\circle*{2}}
\put(72,71){\circle*{2}} \put(74,73){\circle*{2}}
\put(76,75){\circle*{2}} \put(78,77){\circle*{2}}
\put(80,79){\circle*{2}} \put(82,80){\circle*{2}}
\put(84,82){\circle*{2}} \put(86,90){\circle*{2}}
\put(88,92){\circle*{2}} \put(90,93){\circle*{2}}
\put(92,95){\circle*{2}} \put(94,97){\circle*{2}}
\put(96,98){\circle*{2}} \put(98,100){\circle*{2}}
\put(100,102){\circle*{2}} \put(102,104){\circle*{2}}
\put(104,105){\circle*{2}} \put(106,107){\circle*{2}}
\put(108,110){\circle*{2}} \put(110,112){\circle*{2}}
\put(112,113){\circle*{2}} \put(114,115){\circle*{2}}
\put(116,117){\circle*{2}} \put(118,119){\circle*{2}}
\put(120,121){\circle*{2}} \put(122,122){\circle*{2}}
\put(124,124){\circle*{2}} \put(126,126){\circle*{2}}
\put(128,128){\circle*{2}}

\qbezier(0,0)(0,0)(2,2) \qbezier(2,2)(2,2)(4,4)
\qbezier(4,4)(4,4)(6,6) \qbezier(6,6)(6,6)(8,8)
\qbezier(8,8)(8,8)(10,10) \qbezier(10,10)(10,10)(12,12)
\qbezier(12,12)(12,12)(14,14) \qbezier(14,14)(14,14)(16,16)
\qbezier(16,16)(16,16)(18,18) \qbezier(18,18)(18,18)(20,20)
\qbezier(20,20)(20,20)(22,23) \qbezier(22,23)(22,23)(24,25)
\qbezier(24,25)(24,25)(26,26) \qbezier(26,26)(26,26)(28,28)
\qbezier(28,28)(28,28)(30,30) \qbezier(30,30)(30,30)(32,32)
\qbezier(32,32)(32,32)(34,34) \qbezier(34,34)(34,34)(36,36)
\qbezier(36,36)(36,36)(38,38) \qbezier(38,38)(38,38)(40,39)
\qbezier(40,39)(40,39)(42,41) \qbezier(42,41)(42,41)(44,46)
\qbezier(44,46)(44,46)(46,48) \qbezier(46,48)(46,48)(48,49)
\qbezier(48,49)(48,49)(50,51) \qbezier(50,51)(50,51)(52,53)
\qbezier(52,53)(52,53)(54,55) \qbezier(54,55)(54,55)(56,57)
\qbezier(56,57)(56,57)(58,58) \qbezier(58,58)(58,58)(60,60)
\qbezier(60,60)(60,60)(62,62) \qbezier(62,62)(62,62)(64,64)
\qbezier(64,64)(64,64)(66,66) \qbezier(66,66)(66,66)(68,68)
\qbezier(68,68)(68,68)(70,70) \qbezier(70,70)(70,70)(72,71)
\qbezier(72,71)(72,71)(74,73) \qbezier(74,73)(74,73)(76,75)
\qbezier(76,75)(76,75)(78,77) \qbezier(78,77)(78,77)(80,79)
\qbezier(80,79)(80,79)(82,80) \qbezier(82,80)(82,80)(84,82)
\qbezier(84,82)(84,82)(86,90) \qbezier(86,90)(86,90)(88,92)
\qbezier(88,92)(88,92)(90,93) \qbezier(90,93)(90,93)(92,95)
\qbezier(92,95)(92,95)(94,97) \qbezier(94,97)(94,97)(96,98)
\qbezier(96,98)(96,98)(98,100) \qbezier(98,100)(98,100)(100,102)
\qbezier(100,102)(100,102)(102,104)
\qbezier(102,104)(102,104)(104,105)
\qbezier(104,105)(104,105)(106,107)
\qbezier(106,107)(106,107)(108,110)
\qbezier(108,110)(108,110)(110,112)
\qbezier(110,112)(110,112)(112,113)
\qbezier(112,113)(112,113)(114,115)
\qbezier(114,115)(114,115)(116,117)
\qbezier(116,117)(116,117)(118,119)
\qbezier(118,119)(118,119)(120,121)
\qbezier(120,121)(120,121)(122,122)
\qbezier(122,122)(122,122)(124,124)
\qbezier(124,124)(124,124)(126,126)
\qbezier(126,126)(126,126)(128,128)

\end{picture}
\centerline{a) $v=0,7$}\end{center}
\end{minipage}
\hfill
\begin{minipage}[h]{0.3\linewidth}
\begin{center}
\begin{picture}(140,140)
\put(128,0){\line(0,1){128}} \put(0,128){\line(1,0){128}}
\put(0,0){\vector(1,0){140}} \put(0,0){\vector(0,1){140}}

\put(0,0){\circle*{2}} \put(2,2){\circle*{2}}
\put(4,4){\circle*{2}} \put(6,6){\circle*{2}}
\put(8,8){\circle*{2}} \put(10,10){\circle*{2}}
\put(12,13){\circle*{2}} \put(14,14){\circle*{2}}
\put(16,16){\circle*{2}} \put(18,18){\circle*{2}}
\put(20,19){\circle*{2}} \put(22,24){\circle*{2}}
\put(24,26){\circle*{2}} \put(26,27){\circle*{2}}
\put(28,29){\circle*{2}} \put(30,30){\circle*{2}}
\put(32,32){\circle*{2}} \put(34,34){\circle*{2}}
\put(36,35){\circle*{2}} \put(38,37){\circle*{2}}
\put(40,38){\circle*{2}} \put(42,40){\circle*{2}}
\put(44,49){\circle*{2}} \put(46,50){\circle*{2}}
\put(48,51){\circle*{2}} \put(50,52){\circle*{2}}
\put(52,54){\circle*{2}} \put(54,56){\circle*{2}}
\put(56,58){\circle*{2}} \put(58,59){\circle*{2}}
\put(60,61){\circle*{2}} \put(62,62){\circle*{2}}
\put(64,64){\circle*{2}} \put(66,66){\circle*{2}}
\put(68,67){\circle*{2}} \put(70,69){\circle*{2}}
\put(72,70){\circle*{2}} \put(74,72){\circle*{2}}
\put(76,74){\circle*{2}} \put(78,76){\circle*{2}}
\put(80,77){\circle*{2}} \put(82,78){\circle*{2}}
\put(84,79){\circle*{2}} \put(86,96){\circle*{2}}
\put(88,97){\circle*{2}} \put(90,98){\circle*{2}}
\put(92,100){\circle*{2}} \put(94,101){\circle*{2}}
\put(96,102){\circle*{2}} \put(98,104){\circle*{2}}
\put(100,105){\circle*{2}} \put(102,106){\circle*{2}}
\put(104,108){\circle*{2}} \put(106,109){\circle*{2}}
\put(108,113){\circle*{2}} \put(110,114){\circle*{2}}
\put(112,115){\circle*{2}} \put(114,116){\circle*{2}}
\put(116,118){\circle*{2}} \put(118,120){\circle*{2}}
\put(120,122){\circle*{2}} \put(122,123){\circle*{2}}
\put(124,125){\circle*{2}} \put(126,126){\circle*{2}}
\put(128,128){\circle*{2}}

\qbezier(0,0)(0,0)(2,2) \qbezier(2,2)(2,2)(4,4)
\qbezier(4,4)(4,4)(6,6) \qbezier(6,6)(6,6)(8,8)
\qbezier(8,8)(8,8)(10,10) \qbezier(10,10)(10,10)(12,13)
\qbezier(12,13)(12,13)(14,14) \qbezier(14,14)(14,14)(16,16)
\qbezier(16,16)(16,16)(18,18) \qbezier(18,18)(18,18)(20,19)
\qbezier(20,19)(20,19)(22,24) \qbezier(22,24)(22,24)(24,26)
\qbezier(24,26)(24,26)(26,27) \qbezier(26,27)(26,27)(28,29)
\qbezier(28,29)(28,29)(30,30) \qbezier(30,30)(30,30)(32,32)
\qbezier(32,32)(32,32)(34,34) \qbezier(34,34)(34,34)(36,35)
\qbezier(36,35)(36,35)(38,37) \qbezier(38,37)(38,37)(40,38)
\qbezier(40,38)(40,38)(42,40) \qbezier(42,40)(42,40)(44,49)
\qbezier(44,49)(44,49)(46,50) \qbezier(46,50)(46,50)(48,51)
\qbezier(48,51)(48,51)(50,52) \qbezier(50,52)(50,52)(52,54)
\qbezier(52,54)(52,54)(54,56) \qbezier(54,56)(54,56)(56,58)
\qbezier(56,58)(56,58)(58,59) \qbezier(58,59)(58,59)(60,61)
\qbezier(60,61)(60,61)(62,62) \qbezier(62,62)(62,62)(64,64)
\qbezier(64,64)(64,64)(66,66) \qbezier(66,66)(66,66)(68,67)
\qbezier(68,67)(68,67)(70,69) \qbezier(70,69)(70,69)(72,70)
\qbezier(72,70)(72,70)(74,72) \qbezier(74,72)(74,72)(76,74)
\qbezier(76,74)(76,74)(78,76) \qbezier(78,76)(78,76)(80,77)
\qbezier(80,77)(80,77)(82,78) \qbezier(82,78)(82,78)(84,79)
\qbezier(84,79)(84,79)(86,96) \qbezier(86,96)(86,96)(88,97)
\qbezier(88,97)(88,97)(90,98) \qbezier(90,98)(90,98)(92,100)
\qbezier(92,100)(92,100)(94,101) \qbezier(94,101)(94,101)(96,102)
\qbezier(96,102)(96,102)(98,104) \qbezier(98,104)(98,104)(100,105)
\qbezier(100,105)(100,105)(102,106)
\qbezier(102,106)(102,106)(104,108)
\qbezier(104,108)(104,108)(106,109)
\qbezier(106,109)(106,109)(108,113)
\qbezier(108,113)(108,113)(110,114)
\qbezier(110,114)(110,114)(112,115)
\qbezier(112,115)(112,115)(114,116)
\qbezier(114,116)(114,116)(116,118)
\qbezier(116,118)(116,118)(118,120)
\qbezier(118,120)(118,120)(120,122)
\qbezier(120,122)(120,122)(122,123)
\qbezier(122,123)(122,123)(124,125)
\qbezier(124,125)(124,125)(126,126)
\qbezier(126,126)(126,126)(128,128)

\end{picture}
\centerline{b) $v=0,75$}\end{center}
\end{minipage}
\hfill
\begin{minipage}[h]{0.3\linewidth}
\begin{center}
\begin{picture}(140,140)
\put(128,0){\line(0,1){128}} \put(0,128){\line(1,0){128}}
\put(0,0){\vector(1,0){140}} \put(0,0){\vector(0,1){140}}

\put(0,0){\circle*{2}} \put(2,2){\circle*{2}}
\put(4,4){\circle*{2}} \put(6,7){\circle*{2}}
\put(8,8){\circle*{2}} \put(10,9){\circle*{2}}
\put(12,15){\circle*{2}} \put(14,15){\circle*{2}}
\put(16,16){\circle*{2}} \put(18,17){\circle*{2}}
\put(20,17){\circle*{2}} \put(22,29){\circle*{2}}
\put(24,29){\circle*{2}} \put(26,29){\circle*{2}}
\put(28,31){\circle*{2}} \put(30,31){\circle*{2}}
\put(32,32){\circle*{2}} \put(34,33){\circle*{2}}
\put(36,33){\circle*{2}} \put(38,35){\circle*{2}}
\put(40,35){\circle*{2}} \put(42,35){\circle*{2}}
\put(44,58){\circle*{2}} \put(46,58){\circle*{2}}
\put(48,58){\circle*{2}} \put(50,58){\circle*{2}}
\put(52,59){\circle*{2}} \put(54,61){\circle*{2}}
\put(56,61){\circle*{2}} \put(58,61){\circle*{2}}
\put(60,63){\circle*{2}} \put(62,63){\circle*{2}}
\put(64,64){\circle*{2}} \put(66,65){\circle*{2}}
\put(68,65){\circle*{2}} \put(70,67){\circle*{2}}
\put(72,67){\circle*{2}} \put(74,67){\circle*{2}}
\put(76,69){\circle*{2}} \put(78,70){\circle*{2}}
\put(80,70){\circle*{2}} \put(82,70){\circle*{2}}
\put(84,70){\circle*{2}} \put(86,115){\circle*{2}}
\put(88,115){\circle*{2}} \put(90,115){\circle*{2}}
\put(92,116){\circle*{2}} \put(94,116){\circle*{2}}
\put(96,116){\circle*{2}} \put(98,117){\circle*{2}}
\put(100,117){\circle*{2}} \put(102,117){\circle*{2}}
\put(104,117){\circle*{2}} \put(106,118){\circle*{2}}
\put(108,122){\circle*{2}} \put(110,122){\circle*{2}}
\put(112,122){\circle*{2}} \put(114,122){\circle*{2}}
\put(116,123){\circle*{2}} \put(118,125){\circle*{2}}
\put(120,125){\circle*{2}} \put(122,125){\circle*{2}}
\put(124,127){\circle*{2}} \put(126,127){\circle*{2}}
\put(128,128){\circle*{2}}

\qbezier(0,0)(0,0)(2,2) \qbezier(2,2)(2,2)(4,4)
\qbezier(4,4)(4,4)(6,7) \qbezier(6,7)(6,7)(8,8)
\qbezier(8,8)(8,8)(10,9) \qbezier(10,9)(10,9)(12,15)
\qbezier(12,15)(12,15)(14,15) \qbezier(14,15)(14,15)(16,16)
\qbezier(16,16)(16,16)(18,17) \qbezier(18,17)(18,17)(20,17)
\qbezier(20,17)(20,17)(22,29) \qbezier(22,29)(22,29)(24,29)
\qbezier(24,29)(24,29)(26,29) \qbezier(26,29)(26,29)(28,31)
\qbezier(28,31)(28,31)(30,31) \qbezier(30,31)(30,31)(32,32)
\qbezier(32,32)(32,32)(34,33) \qbezier(34,33)(34,33)(36,33)
\qbezier(36,33)(36,33)(38,35) \qbezier(38,35)(38,35)(40,35)
\qbezier(40,35)(40,35)(42,35) \qbezier(42,35)(42,35)(44,58)
\qbezier(44,58)(44,58)(46,58) \qbezier(46,58)(46,58)(48,58)
\qbezier(48,58)(48,58)(50,58) \qbezier(50,58)(50,58)(52,59)
\qbezier(52,59)(52,59)(54,61) \qbezier(54,61)(54,61)(56,61)
\qbezier(56,61)(56,61)(58,61) \qbezier(58,61)(58,61)(60,63)
\qbezier(60,63)(60,63)(62,63) \qbezier(62,63)(62,63)(64,64)
\qbezier(64,64)(64,64)(66,65) \qbezier(66,65)(66,65)(68,65)
\qbezier(68,65)(68,65)(70,67) \qbezier(70,67)(70,67)(72,67)
\qbezier(72,67)(72,67)(74,67) \qbezier(74,67)(74,67)(76,69)
\qbezier(76,69)(76,69)(78,70) \qbezier(78,70)(78,70)(80,70)
\qbezier(80,70)(80,70)(82,70) \qbezier(82,70)(82,70)(84,70)
\qbezier(84,70)(84,70)(86,115) \qbezier(86,115)(86,115)(88,115)
\qbezier(88,115)(88,115)(90,115) \qbezier(90,115)(90,115)(92,116)
\qbezier(92,116)(92,116)(94,116) \qbezier(94,116)(94,116)(96,116)
\qbezier(96,116)(96,116)(98,117) \qbezier(98,117)(98,117)(100,117)
\qbezier(100,117)(100,117)(102,117)
\qbezier(102,117)(102,117)(104,117)
\qbezier(104,117)(104,117)(106,118)
\qbezier(106,118)(106,118)(108,122)
\qbezier(108,122)(108,122)(110,122)
\qbezier(110,122)(110,122)(112,122)
\qbezier(112,122)(112,122)(114,122)
\qbezier(114,122)(114,122)(116,123)
\qbezier(116,123)(116,123)(118,125)
\qbezier(118,125)(118,125)(120,125)
\qbezier(120,125)(120,125)(122,125)
\qbezier(122,125)(122,125)(124,127)
\qbezier(124,127)(124,127)(126,127)
\qbezier(126,127)(126,127)(128,128)

\end{picture}
\centerline{c) $v=0,9$}\end{center}
\end{minipage}
\hfill \caption{} \label{fig:10}
\end{figure}

Is is easy to prove that if the diagram~(\ref{eq:32}) is
commutative for $f$, $g_v$ and a monotone $\psi_1$ such that
$\psi_1(\{0,\, 1\}) = \{0,\, 1\}$, then $\psi_1$ increase. The
values of $\psi_1$ at $A_7$ for $v=0,7,\, 0,75,\, 0,9$ are given
at Figure~\ref{fig:10}. It is seen from these Figures that
$\psi_1$ is discontinuous. The proper proof of the discontinuity
of $\psi_1$ is similar to the proof of Theorem~\ref{theor:4}.

Remind that $g\in [0,\, 1]\rightarrow [0,\, 1]$ is called
\underline{\emph{unimodal}}, if there exist $a\in (0,\, 1)$ such
that $g$ is monotone on $[0,\, a]$, also $g$ is monotone on $[a,\,
1]$, but $g$ is not monotone on $[0,\, 1]$.

\begin{theorem}\label{theor:11}For every
$x_0\in [0,\, 1]$ and every $\varepsilon>0$ there exists a maps
$g:\, [0,\, 1]\rightarrow [0,\, 1]$ with the following properties.

1. $g$ is unimodal;

2. $g(x)=f(x)$ for every $x\in [0,\, 1]\setminus
(x_0-\varepsilon,\, x_0+\varepsilon)$;

3. $f$ and $g$ are not topologically conjugated.
\end{theorem}

\begin{proof}If $x_0=1/2$, then take
$g(x_0)=1-\varepsilon/2$. After this take the maps $g$ to be
piecewise linear, whose graph passes trough points $(0,\, 0)$,
$(1/2-\varepsilon,\, 1-2\varepsilon)$, $(1/2,\, 1-\varepsilon)$,
$(1/2+\varepsilon,\, 1-2\varepsilon)$, $(1,\, 0)$. The constructed
$g$ would not be topologically conjugated to $f$, because the
every point has at last one pre-image under $f$ in the time, when
$1$ has no any pre images under $g$.

Assume that $x_0\neq 1/2$. Without losing the generality assume
that $1/2\not\in (x_0-\varepsilon,\, x_0+\varepsilon)$, because
otherwise just decrease $\varepsilon$.

Since the set of periodical points of $f$ is dense in $[0,\, 1]$,
there exists a periodical point $x^*\in (x_0-\varepsilon,\,
x_0+\varepsilon)$. Let $n$ be its period.

This means that $x^*$ would be a periodical point of $f^n$ and,
correspondingly, either $(f^n)'(x^*) = 2^n$, or $(f^n)'(x^*) =
-2^n$. The point $x^*$ also would be a fixed point of $f^{2n}$ and
$(f^{2n})'(x^*) = 2^{2n}$.

for the number $$\delta = \frac{\min\{ x^* - (x_0-\varepsilon),\,
(x_0+\varepsilon)-x^*\}}{2}$$ construct the maps $g$ as follows.

1. Tangent of $g$ on $(x^*-\delta,\, x^*+\delta)\cap [0,\, 1]$
equals $\frac{f'(x^*)}{2^{2n}}$;

2. $g(x)=f(x)$ for every $x\in [0,\, 1]\setminus
(x_0-\varepsilon,\, x_0+\varepsilon)$;

3. The maps $g$ is linear on each of two intervals of the set
$$\left((x_0-\varepsilon,\, x_0+\varepsilon)\cup [0,\,
1]\right)\setminus (x^*-\delta,\, x^*+\delta).$$

The neighborhood of  of the periodical point $x^*$ of period 3 is
given at Figure~\ref{fig:9}c.

The fact that $g$ would not be topologically conjugated to $f$ may
be proved in the same manner as in the proof of
Theorem~\ref{theor:4}.

If maps $f$ and $g$ are topologically conjugated, then so are (via
the same homeomorphism) $f^{2n}$ and $g^{2n}$. In the same time,
by construction $(g^{2n})(x)=x$ for all $x\in (x^*-\delta,\,
x^*+\delta)$ and $f^{2n}$ has no an interval of fixed points. This
contradiction finishes the proof.
\end{proof}

\newpage

\section{On the differentiability of conjugation}\label{sect-dyffer}

We continue in this section to consider the problem about the
topological conjugation of the maps $f,\ f_v:\, [0,\,
1]\rightarrow [0,\, 1]$, where $$ f(x) = \left\{\begin{array}{ll}
2x,& x< 1/2;\\
2-2x,& x\geqslant 1/2
\end{array}\right.
$$
and
$$
f_v(x) = \left\{\begin{array}{ll} \frac{x}{v},& x\leqslant v;\\
 \frac{1-x}{1-v},&
x>v.
\end{array}\right.
$$ for $v\in (0,\, 1)$ being a parameter.
By Theorem~\ref{theor:homeom-jed}, there exists and it is unique
the homeomorphism $h:\, [0,\, 1]\rightarrow [0,\, 1]$ such that
the following would commute.

\begin{equation}\label{eq:43} \begin{CD}
[0,\, 1] @>f >> & [0,\, 1]\\
@V_{h} VV& @VV_{h}V\\
[0,\, 1] @>f_v>>& [0,\, 1],
\end{CD}\end{equation}

This Section is devoted to the differentiability of the
homeomorphism $h$. This problem was inspirit by the following
result.

\begin{theorem}\label{theor:13}The
derivative of the homeomorphism $h$ such that the
diagram~(\ref{eq:43}) is commutative, exists almost everywhere and
equals to 0 everywhere it is finite.
\end{theorem}

This result is given at~\cite[Proposition~2]{Skufca}. It is
formulated there as follows: the derivative of $h$, which makes
the diagram~(\ref{eq:43}) commutative, exists almost everywhere
and equals 0 everywhere, where it \textbf{exists}. Nevertheless,
it is seen from the proof in~\cite{Skufca}, that under the
assumption about the existence of the derivative, authors mean
also the finiteness of the derivative. They consider an arbitrary
point $x\in (0, 1)$ and construct the sequence $k_{n}$ such that
$x\in I_{n} =\left[\frac{k_{n} }{2^{n} },\, \frac{k_n +1}{2^n }
\right]$ and numbers $p_{n} =h\left(\frac{k_{n} +1}{2^{n} }
\right)-h\left(\frac{k_n}{2^n} \right)$. They claim that if the
derivative $h'(x)$ exists and does not equal to 0, then
$\frac{p_{n+1} }{p_{n} } \to \frac{1}{2} $. But such reasonings
are correct only in the case when the derivative $h'(x)$ is also
finite.

Following~\cite[sect. 92, 101]{Fihtengoltz}, we will assume that
the derivative of a function of real argument is the limit of the
ratio of the its increment over the correspond increment of the
argument. The derivative (finite or unfinite) is said to exists if
and only if the mentioned limit exists and the derivative equals
to the value of the limit.

The interest of Theorem~\ref{theor:13} is because of Lebesgue
theorem on the differentiability of the monotone function.

\begin{theorem}\label{theor:14}Every
monotone function on the interval has finite derivative almost
everywhere.
\end{theorem}

Whence, from one hand by Theorem~\ref{theor:14} the derivative of
$h$ exists almost everywhere, but from another hand by
Theorem~\ref{theor:13} this derivative equals $0$ everywhere,
where it is finite.

Remind that we have considered in Section~\ref{sect-Pobudowa} the
sets $A_n$, $B_n$, $\mathcal{A}$ and~$\mathcal{B}$, which are
defined as follows.

$A_n,\ n\geqslant 1$ is the set of all points $x\in [0,\, 1]$ such
that
$$ f^{n}(x) = 0.$$

$B_n,\ n\geqslant 1$ is the set of all $x\in [0,\, 1]$ such that
$$ f_v^{n}(x) = 0. $$

For every $n\geq 1$ we have denoted by $\alpha_{n,k},\,
\beta_{n,k},\, 0\leq k\leq 2^{n-1} $ the increasingly ordered
elements of $A_n$ and $B_n$ correspondingly, i.e. for any
$k_1<k_2$ inequalities $\alpha_{n,k_1}<\alpha_{n,k_2}$ and
$\beta_{n,k_1}<\beta_{n,k_2}$ hold.

We have denoted $\mathcal{A}=\bigcup _{n\ge 1} A_{n} $ and
$\mathcal{B}=\bigcup _{n\ge 1} B_{n} $. For every $n\geq 1$ denote
$h_n$ the piecewise linear maps, all whose braking points belong
to $A_n$ and such that for every $k,\, 0\leq k\leq 2^{n-1}$ the
equality $$ h_n(\alpha_{n,k})=\beta_{n,k}
$$ holds.

By Theorem~\ref{theor:10} for every $n$ the conjugation $h$
coincides with $h_n$ on $A_n$.

\subsection{Limits of derivatives of approximations of the
conjugation}\label{sect-dyffer-1}

Consider as arbitrary point $x\in (0,\, 1)\backslash \mathcal{A}$
and find the limit of the sequence $\lim\limits_{n\rightarrow
\infty}h_n'(x)$. The condition $x\not\in \mathcal{A}$ guarantee
that for every $n\geq 1$ the limit $h_n'(x)$ exists.

Let the binary decomposition of $x$ be as follows.

\begin{equation}\label{eq:6} x = 0,x_1x_2\ldots x_k\ldots\,
.
\end{equation}

For the number $x$ of the form~(\ref{eq:6}), denote $x_0 = 0$ and
for every $i\geq 2$ denote

\begin{equation}\label{eq:54}
\alpha_i(x) =\left\{
\begin{array}{ll} 2v& \text{if } x_{i-1}=x_{i-2},\\
2(1-v)& \text{if } x_{i-1}\neq x_{i-2}.
\end{array}\right.
\end{equation}

\begin{theorem}\label{lema:32}For every
$n\geq 2$ and a number $x\not\in A_n$ of the form~(\ref{eq:6}),
the equality
$$h_n'(x)=\prod\limits_{i=2}^n\alpha_i(x)$$ holds, where
$\alpha_i(x)$ are defined by~(\ref{eq:54}).
\end{theorem}

\begin{proof}Consider the maps $f_v^n$. Its set of zeros
is $\beta_{n,k},\, 0\leq k\leq 2^{n-1}$. The set of zeros of
$f_v^{n-1}$ is $\beta_{n,2k},\, 0\leq k\leq 2^{n-2}$ and the set
of solution of the equation $f_v^{n-1}(x)=1$ is $\beta_{n,2k+1},\,
0\leq k\leq 2^{n-2}-1$.

Consider the maps $f_v^{n-1}$ for $x\in [\beta_{n, 2k},\,
\beta_{n,2k+2}]$. The graph of this maps passes through the points
$M(\beta_{n, 2k},\, 0)$, $Q(\beta_{n, 2k+1},\, 1)$ and
$N(\beta_{n, 2k+2},\, 0)$ as shown on Figure~\ref{fig:11}a.

For $x\in (\beta_{n,2k},\, \beta_{n,\, 2k+1})$ and the maps
$f_v^{n-1}$ it follows from Lemma~\ref{lema:vlastfv1} that the
graph of $f_v(f_v^{n-1})$ passes through points $M(\beta_{n,
2k},\, 0)$, $S(t_1,\, 1)$ and $K(\beta_{n, 2k+1},\, 0)$, where
$t_1 = \beta_{n, 2k} + v(\beta_{n, 2k+1}-\beta_{n, 2k})$, as it is
shown on the Figure~\ref{fig:11}b. Nevertheless, since
$f_v^{n}(t_1)=1$, then $f_v^{n+1}(t_1)=0$, i.e. $t_1\in B_{n+1}$,
whence
\begin{equation}\label{eq:4} \beta_{n+1,4k+1} = \beta_{n, 2k} +
v(\beta_{n, 2k+1}-\beta_{n, 2k}).
\end{equation}
In terms of notations of the Figure~\ref{fig:11}b the
equality~(\ref{eq:4}) means that
\begin{equation}\label{eq:8} \frac{PS}{PQ} =v.\end{equation}

\begin{figure}[htbp]
\begin{minipage}[h]{0.45\linewidth}
\begin{center}
\begin{picture}(140,120)
\put(0,100){\line(1,0){140}} \put(0,0){\line(1,0){140}}
\put(20,0){\line(1,2){50}} \put(70,100){\line(1,-2){50}}

\put(20,0){\line(0,1){100}} \put(120,0){\line(0,1){100}}

\put(20,0){\circle*{4}} \put(7,5){$M$} \put(120,0){\circle*{4}}
\put(123,5){$N$}

\put(70,100){\circle*{4}} \put(67,105){$Q$}
\put(20,100){\circle*{4}} \put(17,105){$P$}
\put(120,100){\circle*{4}} \put(117,105){$R$}
\end{picture}
\centerline{a. A part of the graph of $f_v^{n-1}$}\end{center}
\end{minipage}
\hfill
\begin{minipage}[h]{0.45\linewidth}
\begin{center}
\begin{picture}(140,120)
\put(0,100){\line(1,0){140}} \put(0,0){\line(1,0){140}}

\linethickness{0.3mm} \Vidr{20}{0}{70}{100} \VidrTo{120}{0}
\linethickness{0.1mm}

\Vidr{20}{0}{57.5}{100} \VidrTo{70}{0} \VidrTo{82.5}{100}
\VidrTo{120}{0}

\put(70,100){\circle*{4}} \put(67,105){$Q$}
\put(20,100){\circle*{4}} \put(17,105){$P$}
\put(120,100){\circle*{4}} \put(117,105){$R$}

\put(57.5,100){\circle*{4}} \put(54.5,105){$S$}
\put(82.5,100){\circle*{4}} \put(79.5,105){$T$}

\put(20,0){\circle*{4}} \put(7,5){$M$} \put(120,0){\circle*{4}}
\put(123,5){$N$}

\put(70,0){\circle*{4}} \put(75,5){$K$}
\end{picture}
\centerline{b. A part of the graph of $f_v^{n}$}\end{center}
\end{minipage}
\hfill \caption{Parts of the graphs of iterations of $f_v$}
\label{fig:11}
\end{figure}
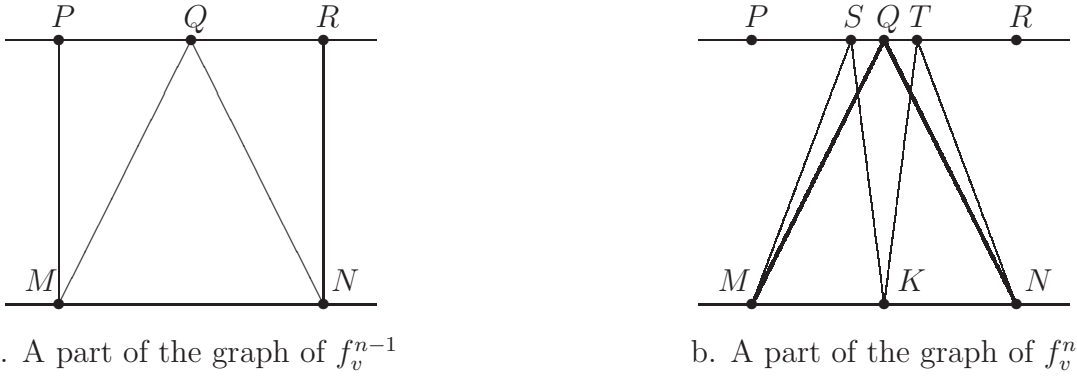

For $x\in (\beta_{n,2k+1},\, \beta_{n,2k+2})$ and the maps
$f_v^{n-1}$ is follows from Lemma~\ref{lema:vlastfv2} that graph
of $f_v(f_v^{n-1})$ passes through points $K(\beta_{n, 2k+1},\,
0)$, $T(t_2,\, 1)$ and $N(\beta_{n, 2k+2},\, 0)$, where $t_2 =
\beta_{n, 2k+1} + (1-v)(\beta_{n, 2k+2}-\beta_{n, 2k+1})$. Since
$f_v^{n}(t_2)=1$, then $f_v^{n+1}(t_2)=0$, i.e. $t_2\in B_{n+1}$,
whence
\begin{equation}\label{eq:5} \beta_{n+1,4k+3} = \beta_{n, 2k+1} + (1-v)(\beta_{n, 2k+2}-\beta_{n,
2k+1}).
\end{equation} In terms of notations
of the Figure~\ref{fig:11}b the equality~(\ref{eq:4}) means that
\begin{equation}\label{eq:9}\frac{QT}{RQ}
=1-v.\end{equation}

Denote by $\gamma_1$ and $\gamma_2$ correspondingly the the
coefficients of $h_n$ on intervals
$A_{n,2k}=(\frac{2k}{2^{n-1}},\, \frac{2k+1}{2^{n-1}})$ and
$A_{n,2k+1}=(\frac{2k+1}{2^{n-1}},\, \frac{2k+2}{2^{n-1}})$. It
means that $\gamma_1 = 2^{n-1}(\beta_{n,2k+1}-\beta_{n,2k})$ and
$\gamma_2 = 2^{n-1}(\beta_{n,2k+2}-\beta_{n,2k+1})$.

Let $x\in A_{n+1,4k}$. From~(\ref{eq:4}) obtain
$$h_{n+1}'(x) = 2^n(\beta_{n+1,4k+1} -\beta_{n+1,4k}) =
2^nv(\beta_{n,2k+1}-\beta_{n,2k}) = 2v\gamma_1.$$

Let $x\in A_{n+1,4k+1}$. From~(\ref{eq:4}) obtain
$$h_{n+1}'(x) = 2^n(\beta_{n+1,4k+2} -\beta_{n+1,4k+1}) = $$ $$=
2^n(\beta_{n,2k+1}-(\beta_{n, 2k} + v(\beta_{n, 2k+1}-\beta_{n,
2k}))) = 2(1-v)\gamma_1.$$

Let $x\in A_{n+1,4k+2}$. From~(\ref{eq:5}) obtain
$$h_{n+1}'(x) = 2^n(\beta_{n+1,4k+3} -\beta_{n+1,4k+2}) =
2^n(\beta_{n, 2k+2}-\beta_{n, 2k+1}) = 2(1-v)\gamma_2$$

Let $x\in A_{n+1,4k+3}$. From~(\ref{eq:5}) obtain
$$h_{n+1}'(x) = 2^n(\beta_{n+1,4k+4} -\beta_{n+1,4k+3}) =
2v\gamma_2$$

The maps $h_{n-1}$ divides $[0,\, 1]$ to $2^{n-2}$ equal
intervals. Each of these intervals is defines by $n-1$ the first
numbers of the binary decomposition of its points. More then this,
if the binary composition of $x$ is of the form~(\ref{eq:6}), then
the binary decomposition of a natural $k$ is $x_1\, \ldots\,
x_{n-2}$ (where, as in general, the first zeros should be
ignored).

The inclusion $x\in A_{n+1,4k}$ is equivalent to that~(\ref{eq:6})
is of the form $x = 0.x_1\, \ldots\, x_{n-2}\, 0\, 0\, \ldots\, $.
The inclusion $x\in A_{n+1,4k+1}$ is equivalent to
that~(\ref{eq:6}) is of the form $x = 0.x_1\, \ldots\, x_{n-2}\,
0\, 1\, \ldots\, $.

The inclusion $x\in A_{n+1,4k+2}$ is equivalent to
that~(\ref{eq:6}) is of the form $x = 0.x_1\, \ldots\, x_{n-2}\,
1\, 0\, \ldots\, $.

The inclusion $x\in A_{n+1,4k+3}$ is equivalent to
that~(\ref{eq:6}) is of the form $x = 0.x_1\, \ldots\, x_{n-2}\,
1\, 1\, \ldots\, $.

Now evident inductive reasonings finish the proof.
\end{proof}

The following theorem~\ref{lema:32} follows from the proved one.

\begin{figure}[htbp]
\begin{minipage}[h]{0.49\linewidth}
\begin{center}
\begin{picture}(100,120)
\put(100,0){\line(0,1){100}} \put(0,100){\line(1,0){100}}
\put(0,0){\vector(0,1){120}} \put(0,0){\vector(1,0){120}}
\put(50,75){\circle*{4}} \put(0,0){\circle*{6}}
\put(100,100){\circle*{6}} \put(25,56.25){\circle*{4}}
\put(75,81.25){\circle*{4}} \put(12.5,42.1875){\circle*{4}}
\put(37.5,60.9375){\circle*{4}} \put(62.5,79.6875){\circle*{4}}
\put(87.5,82.8125){\circle*{4}} \qbezier(0,0)(0,0)(12.5,42.1875)
\qbezier(12.5,42.1875)(12.5,42.1875)(25,56.25)
\qbezier(25,56.25)(25,56.25)(37.5,60.9375)
\qbezier(37.5,60.9375)(37.5,60.9375)(50,75)
\qbezier(50,75)(50,75)(62.5,79.6875)
\qbezier(62.5,79.6875)(62.5,79.6875)(75,81.25)
\qbezier(75,81.25)(75,81.25)(87.5,82.8125)
\qbezier(87.5,82.8125)(87.5,82.8125)(100,100)
\end{picture}
\end{center}
\centerline{a. The graph of $h_4$ for $v=3/4$}
\end{minipage}
\hfill
\begin{minipage}[h]{0.49\linewidth}
\begin{center}
\begin{picture}(100,120)
\put(100,0){\line(0,1){100}} \put(0,100){\line(1,0){100}}
\put(0,0){\vector(0,1){120}} \put(0,0){\vector(1,0){120}}
\put(0,0){\circle*{6}} \put(100,100){\circle*{6}}
\put(12.5,1.5625){\circle*{4}} \put(25,6.25){\circle*{4}}
\put(37.5,20.3125){\circle*{4}} \put(50,25){\circle*{4}}
\put(62.5,39.0625){\circle*{4}} \put(75,81.25){\circle*{4}}
\put(87.5,95.3125){\circle*{4}} \qbezier(0,0)(0,0)(12.5,1.5625)
\qbezier(12.5,1.5625)(12.5,1.5625)(25,6.25)
\qbezier(25,6.25)(25,6.25)(37.5,20.3125)
\qbezier(37.5,20.3125)(37.5,20.3125)(50,25)
\qbezier(50,25)(50,25)(62.5,39.0625)
\qbezier(62.5,39.0625)(62.5,39.0625)(75,81.25)
\qbezier(75,81.25)(75,81.25)(87.5,95.3125)
\qbezier(87.5,95.3125)(87.5,95.3125)(100,100)

\end{picture}
\end{center}
\centerline{b. The graph of $h_4$ for $v=1/4$}
\end{minipage}
\caption{Graphs of $h_4$}\label{fig:15}
\end{figure}

\begin{theorem}\label{theor:12}1. If
$v<1/2$ then for every $x\in \mathcal{A}$ the limit
$\lim\limits_{n\rightarrow \infty}h_n'(x) =0$ holds.

2. If $v>1/2$, then for every $x\in \mathcal{A}$ the limit
$\lim\limits_{n\rightarrow \infty}h_n'(x) =\infty$ holds.

3. For every $v\in (0,\, 1)\setminus \{1/2\}$ the following limits
$\lim\limits_{n\rightarrow \infty}\min\limits_{x\in (0,\,
1)\setminus \mathcal{A}}h_n'(x) =0$ and $\lim\limits_{n\rightarrow
\infty}\max\limits_{x\in (0,\, 1)\setminus \mathcal{A}}h_n'(x)
=\infty$ hold.
\end{theorem}

The i. 3 of Theorem is illustrated at Figure~\ref{fig:15}, where
are given graphs of $h_4$ for $v=3/4$ and $v=3/4$. For an
arbitrary $k\in \mathbb{N}$ consider the number, whose the binary
decomposition is
$$
x_0 = 0,\underbrace{10\, 10\, 10}_{k\, \text{ times}}.
$$ Then by Theorem~\ref{lema:32} we have that $$
h_{2k+1}^{'}(x_0) = (2(1-v))^{2k}.
$$ In the same time $$
\lim\limits_{n\rightarrow \infty}h_n(x_0) =
\lim\limits_{n\rightarrow \infty}(2v)^n.
$$

\subsection{Values of the derivative of conjugation}\label{sect-dyffer-2}

We will prove Theorem~\ref{theor:13} in this section, i.e. we will
prove the Proposition~2 from~\cite{Skufca}, which states that the
derivative of $h$ can be equal either $0$, or infinity. Also we
will find the values of the derivative of conjugation at rational
points.

Remind that non triviality of Theorem~\ref{theor:13} follows from
the Lebesgue Theorem (Theorem~\ref{theor:14}) about
differentiability of the monotonic function.

\begin{proof}[Proof of Theorem~\ref{theor:13}] Let
for the point $x\in (0,\, 1)$ the derivative $h'(x)=t$ exists.

For every $n\geq 1$ denote by $k$ as interval such that $x\in
[\alpha_{n,k},\, \alpha_{n,k+1})$.

The condition $h'(x)=t$ is equivalent to that
\begin{equation}\label{eq:55} h(x) -h(\alpha_{n,k}) =
t(x-\alpha_{n,k}) + m^{(1)}_n
\end{equation} and
\begin{equation}\label{eq:56} h(\alpha_{n,k+1}) -h(x) = t(\alpha_{n,k+1}-x) +
m^{(2)}_n,
\end{equation} where $m^{(1)}_n$ and $m^{(2)}_n$ are sequences,
which tend to 0 for $n\rightarrow \infty$.

By adding the equalities~(\ref{eq:55}) and~(\ref{eq:56}) obtain
that
\begin{equation}\label{eq:57} h(\alpha_{n,k+1})-h(\alpha_{n,k}) =
t(\alpha_{n,k+1} - \alpha_{n,k}) + m^{(3)}_n,
\end{equation} where $m^{(1)}_n = m^{(1)}_n + m^{(2)}_n$.

Nevertheless, equality~(\ref{eq:57}) can be rewritten as
$$ h_n'(x) = t+m^{(3)}_n,
$$ which means that $\lim\limits_{n\rightarrow \infty}h_n'(x_0)=t$.

Now theorem follows from parts 1 and 2 of Theorem~\ref{theor:12}.
\end{proof}

\begin{lemma}\label{lema:33}If $v<1/2$ then for
every $x_0 \in \mathcal{A}$ there exists the derivative $h'(x_0
)=\infty $ and if $v>1/2$ then for every $x_{0} \in \mathcal{A}$
there exists a derivative $h'(x_0)=0$.
\end{lemma}

\begin{proof}For $x_0\neq 0$ consider that the left
derivative $h'(x_{0} -0)$ exists and is equal to that, which is
stated in the Lemma (the condition of $x_0$ to be positive comes
from that the left derivative for $h'(x_{0} -0)$ is undefined for
$x_0=0$).

Since the sets $\{ A_{n},\, n>0\} $ are embedded, then there
exists $n_0>0$ such that $x_{0} \in A_{n} $ for all $n>n_0$.

For every $n>n_{0} $ denote by $\alpha _{n} $ the biggest element
of $A_{n} $, which is less than $x_{0} $. Since $A_{n}$ consists
of all rational numbers, whose denominators in the proper form are
divisors of $2^{n-1} $, then $\alpha _{n} =x_{0} -2^{-n+1} $. It
follows from the construction of $h_{n} $ that for every $n>n_{0}
$ the equality $h_{n}^{'} (x_{0} -0)=\frac{h(x_{0} )-h(\alpha _{n}
)}{x_{0} -\alpha _{n} } $ holds, because $h_{n} (x_{0} )=h(x_{0}
)$ and $h_{n} (\alpha _{n} )=h(\alpha _{n} )$. It follows from
continuity of $h$ that for every $x\in (\alpha _{n} ,\alpha _{n+1}
]$ the double following inequality $\frac{h(x_{0} )-h(\alpha
_{n+1} )}{x_{0} -\alpha _{n} } \leq \frac{h(x_{0} )-h(x)}{x_{0}
-x} \leq \frac{h(x_{0} )-h(\alpha _{n} )}{x_{0} -\alpha _{n+1} } $
holds. This inequality can be rewritten as $2h_{n+1}^{'} (x_{0}
)\leq \frac{h(x_{0} )-h(x)}{x_{0} -x} \leq \frac{h_{n}^{'} (x_{0}
)}{2} $.

Consider now an arbitrary increasing sequence $x_{k} $, which
tends to $x_{0}$ and prove the equality $\mathop{\lim
}\limits_{k\to \infty } \frac{h(x_{0} )-h(x_{k} )}{x_{0} -x_{k} }
=\mathop{\lim }\limits_{n\to \infty } h_{n}^{'} (x_{0} -0)$. For
every $k$ there exists an $n_k$ such that $x_{k} \in (\alpha
_{n_{k} } ,\alpha _{n_{k} +1} ]$. Then $2h_{n_{k} +1}^{'} (x_{0}
)\leq \frac{h(x_{0} )-h(x_{k} )}{x_{0} -x_{k} } \leq
\frac{2h_{n_{k} }^{'} (x_{0} )}{2} $. The last double inequality
proves Lemma independently on whether $v>0,5$ or $v<0,5$ and,
correspondingly, independently on the value of the limit
$\mathop{\lim }\limits_{n\to \infty } h_{n}^{'} (x_{0} )$.

The proof for the right derivative $h'(x_{0} +0)$ is analogical.
\end{proof}

\begin{theorem}\label{theor:15}Let $x_0\in
[0,\, 1]\cap \mathbb{Q}$. Then the derivative $h'(x_{0} )$ exists.
More then this, if $v<1/2$ then $h'(x_{0} )=\infty $ and if
$v>1/2$ then $h'(x_{0} )=0$.
\end{theorem}

\begin{proof}Because of Lemma~\ref{lema:33}, we can
restrict our consideration on the case $x_0\in \mathbb{Q}
\backslash \mathcal{A}$. Consider the binary decomposition of
$x_0$. Since $x_0\in \mathbb{Q} \backslash \mathcal{A}$, then
\begin{equation}\label{eq-1}x_0 = 0,x_1x_2\ldots\,
x_p(x_{p+1}\ldots x_{p+t}), \end{equation} where $x_{p+1}\ldots
x_{p_t}$ is a periodical part of $x_0$ and not all digits of this
periodical part equal to 1, because in this case $x_0\in
\mathcal{A}$.

For an arbitrary sequence$\{ s_n\}$, which converge  to $x_0$,
consider numbers
$$k(s_n)=\frac{h(x_0)-h(s_n)}{x_0-s_n}.$$

Consider an arbitrary $n\in \mathbb{N}$. For every $s$ denote by
$\alpha_{n_s}$ the closest from the left element $x_0$ of $A_s$
and denote by $\alpha_{n_s}^+$ the element from $A_s$., which is
the closest from the right to $x_0$. Choose the maximal $s_0$ such
that $s_n\in [\alpha_{n_{s_0}},\, \alpha_{n_{s_0}}^+]$. It follows
from the maximality of $s_0$ and $x_0\not\in \mathcal{A}$ that
numbers $x_0,\, s_n$ belong to different ``halves'' of the
interval $[\alpha_{n_{s_0}},\, \alpha_{n_{s_0}}^+]$, i.e. it
follows from the inclusion $x_0\in \left(\alpha_{n_{s_0}},\,
\displaystyle{\frac{\alpha_{n_{s_0}}+\alpha_{n_{s_0}}^+}{2}}\right)$
that $s_n\in
\left[\displaystyle{\frac{\alpha_{n_{s_0}}+\alpha_{n_{s_0}}^+}{2}},\,
\alpha_{n_{s_0}}^+\right]$. We will write $\alpha_n$ and
$\alpha_{n_s}^+$ instead of $\alpha_{n_s}$ and $\alpha_{n_s}^+$.

\begin{figure}[htbp]
\begin{minipage}[h]{0.48\linewidth}
\begin{center}
\begin{picture}(100,100)

\put(0,0){\line(0,1){96}} \put(0,0){\line(1,0){96}}

\put(0,24){\line(1,0){96}} \put(0,48){\line(1,0){96}}
\put(0,72){\line(1,0){96}} \put(0,96){\line(1,0){96}}

\put(24,0){\line(0,1){96}} \put(48,0){\line(0,1){96}}
\put(72,0){\line(0,1){96}} \put(96,0){\line(0,1){96}}

\put(0,0){\circle*{5}} \put(96,96){\circle*{5}}
\put(64,64){\circle*{5}} \put(-13,0){$A$}  \put(100,100){$A^+$}
\put(50,54){$X_0$}

\put(48,48){\circle{5}} \put(72,72){\circle{5}}

\end{picture}
\end{center}
\caption{}\label{fig:12}
\end{minipage}
\hfill
\begin{minipage}[h]{0.48\linewidth}
\begin{center}
\begin{picture}(100,100)

\put(0,0){\line(0,1){96}} \put(0,0){\line(1,0){96}}

\put(0,24){\line(1,0){96}} \put(0,48){\line(1,0){96}}
\put(0,96){\line(1,0){96}}

\put(24,0){\line(0,1){96}} \put(48,0){\line(0,1){96}}
\put(96,0){\line(0,1){96}}

\put(0,0){\circle*{5}} \put(96,96){\circle*{5}}
\put(64,64){\circle*{5}} \put(-13,0){$A$}  \put(100,100){$A^+$}
\put(68,68){$X_0$}

\put(48,48){\circle{5}}

\put(6,0){\line(0,1){48}} \put(12,0){\line(0,1){48}}
\put(18,0){\line(0,1){48}} \put(30,0){\line(0,1){48}}
\put(36,0){\line(0,1){48}} \put(42,0){\line(0,1){48}}

\put(0,6){\line(1,0){48}} \put(0,12){\line(1,0){48}}
\put(0,18){\line(1,0){48}} \put(0,30){\line(1,0){48}}
\put(0,36){\line(1,0){48}} \put(0,42){\line(1,0){48}}

\qbezier(48,0)(56,32)(64,64)

\qbezier(0,48)(12.5,50)(64,64)

\end{picture}
\end{center}
\caption{}\label{fig:13}
\end{minipage}
\end{figure}

The Figure~\ref{fig:12} contains points $A(\alpha_n,\,
h(\alpha_n))$, $B(\alpha_n^+,\, h(\alpha_n^+))$ and $X(x_0,\,
h(x_0))$ for the case $t=2$, $x_{k+1}=1$, $x_{k+2}=0$. Since
$t=2$, then the segment $x\in (\alpha_n,\, \alpha_n^+)$ is divided
by vertical lines to $2^t=4$ parts. Since $x_{k+1}=1$, then $x_0$
is more right than the middle of the interval $(\alpha_n,\,
\alpha_n^+)$ and since $x_{k+2}=0$, then $x_0$ is more left then
the middle of the interval $(\frac{\alpha_n+\alpha_n^+}{2},\,
\alpha_n^+)$. The intermediate vertical and horizontal lines at
Figure~\ref{fig:12} correspond to values from the sets $A_{s+t}$
and $B_{s+t}$.

Denote by $\widetilde{\alpha}_n$ and $\widetilde{\alpha}_n^+$ the
closest left to $x_0$ and the closest right to $x_0$ points of the
set $A_{p+(n+1)t+1}$. On the Figure~\ref{fig:12} points
$(\widetilde{\alpha}_n,\, h(\widetilde{\alpha}_n))$ and
$(\widetilde{\alpha}_n^+,\, h(\widetilde{\alpha}_n^+))$ are
denoted by circles.

Assume that for some $n$ the inclusion $s_n\in [\alpha_n,\,
\widetilde{\alpha}_n]$ holds. Find the bounds for $k(s_n)$. We
will use this bounds to prove the theorem for the left derivative
$h'_-(x_0)$. The prove the the right derivative $h'_+(x_0)$ is
analogical.

The figure~\ref{fig:13} contains lines, whose tangents bound
$k(s_n)$, because the points with coordinates $(s_n,\, h(s_n))$ is
somewhere in the shades rectangle in the case when $s_n\in
[\alpha_n,\, \widetilde{\alpha}_n]$.

Consider the intersection of the set $A_{s+2t}$ with the points of
the interval $(\widetilde{\alpha}_n,\, \widetilde{\alpha}_n^+)$
and consider points of $B_{s+2t}$, which correspond to this
intersection. Put these points to the Figure~\ref{fig:14} in the
manner, which was used for the Figure~\ref{fig:12}.

\begin{figure}[htbp]
\begin{center}
\begin{picture}(380,380)

\put(0,0){\line(0,1){376}} \put(0,0){\line(1,0){376}}

\put(0,96){\line(1,0){376}} \put(0,192){\line(1,0){376}}
\put(0,288){\line(1,0){376}} \put(0,376){\line(1,0){376}}

\put(192,216){\line(1,0){96}} \put(192,240){\line(1,0){96}}
\put(192,264){\line(1,0){48}}

\put(216,192){\line(0,1){96}} \put(240,192){\line(0,1){96}}
\put(264,192){\line(0,1){48}}

\put(96,0){\line(0,1){376}} \put(192,0){\line(0,1){376}}
\put(288,0){\line(0,1){376}} \put(376,0){\line(0,1){376}}

\put(256,256){\circle*{5}} \put(0,192){\circle*{5}}

\put(192,0){\circle*{5}} \put(4,196){R} \put(267,245){S}
\put(253,263){$X_0$} \put(230,267){D} \put(180,4){C}

\qbezier(0,192)(132,216)(264,240)
\qbezier(192,0)(216,132)(240,264)

\put(264,240){\circle*{5}} \put(240,264){\circle*{5}}

\put(0,0){\circle*{5}} \put(3,3){A}

\put(376,376){\circle*{5}} \put(360,360){$A^+$}

\put(192,192){\circle{5}} \put(288,288){\circle{5}}

\end{picture}
\end{center}
\caption{} \label{fig:14}
\end{figure}

Denote by $\widehat{\alpha}_n$ and $\widehat{\alpha}_n^+$ the
points of $A_{s+2t}$, which are the closest to $x_0$ from the left
and from the right correspondingly. Points $C,\, D,\, R,\, S$ on
Figure~\ref{fig:14} are of coordinates $C(\widetilde{\alpha}_n,\,
h(\alpha_n))$, $D(\widehat{\alpha}_n,\, h(\widehat{\alpha}_n^+))$,
$R(\alpha_n,\, h(\widetilde{\alpha}_n))$ and
$S(\widehat{\alpha}_n^+,\, h(\widehat{\alpha}_n))$. Evidently then
$k(s_n)$ is bounded by tangents of lines $CD$ and $RS$ which are
\begin{equation}\label{eq-2}
k_{CD} = \frac{h(\alpha_n^+)-h(\alpha_n)}{\widehat{\alpha}_n
-\widetilde{\alpha}_n}
\end{equation}
and
\begin{equation}\label{eq-3}
k_{RS} = \frac{h(\widehat{\alpha}_n)
-h(\widetilde{\alpha}_n)}{\alpha_n^+-\alpha_n}. \end{equation}
Find the lower bound for $k_{RS}$ and the upper bound for
$k_{CD}$.

Denote by $d$ the number whose the binary decomposition is
$$ d = x_{p+1}\ldots\, x_{p+t},
$$ where $x_{p+1},\, \ldots\, x_{p+t}$ are taken from~(\ref{eq-1}).
Then it follows from the construction of numbers $\alpha_n$,
$\alpha_n^+$, $\widetilde{\alpha}_{n}$,
$\widetilde{\alpha}_{n}^+$, $\widehat{\alpha}_{n}$ and
$\widehat{\alpha}_{n}^+$ that the following equalities hold.
$$\frac{\widetilde{\alpha}_n-\alpha_n}{\alpha_n^+-\alpha_n}
= \frac{d}{2^t},
$$

$$\frac{\widehat{\alpha}_n- \widetilde{\alpha}_n}
{\widetilde{\alpha}_n^+ -\widetilde{\alpha}_n} = \frac{d}{2^t}.
$$

It follows from these two equalities
that\begin{equation}\label{eq-4} \widehat{\alpha}_n-
\widetilde{\alpha}_n = \frac{d^2(\alpha_n^+-\alpha_n)}{2^{2t}},
\end{equation}

\begin{equation}
\label{eq-5} \begin{array}{l}\widehat{\alpha}_n^+ - \alpha_n =
(\widetilde{\alpha}_n-\alpha_n) + (\widehat{\alpha}_n -
\widetilde{\alpha}_n) + (\widehat{\alpha}_n^+ -
\widehat{\alpha}_n) =
\\
\\
\phantom{\widehat{\alpha}_n^+ -
\alpha_n}=\displaystyle{\frac{d(\alpha_n^+-\alpha_n)}{2^t} +
\frac{d^2(\alpha_n^+-\alpha_n)}{2^{2t}} +
\frac{\alpha_n^+-\alpha_n}{2^{2t}}}.
\end{array}
\end{equation}

If follows from Lemma~\ref{lema:34} that for every $i\in
\mathbb{N}$ the distances between the neighbor points of $B_i$ are
not equal. By Corollary~\ref{corol:3} the following bounds hold.
\begin{equation}\label{eq-6} d(\min\{v,\, 1-v\})^{2t}
\leq \frac{h(\hat{\alpha}_n) -
h(\widetilde{\alpha}_n)}{h(\alpha_n^+)-h(\alpha_n)} \leq
d(\max\{v,\, 1-v\})^{2t};
\end{equation}

\begin{equation}\label{eq-7}
\begin{array}{l}
\phantom{\leq \ }d(\min\{v,\, 1-v\})^{t} + (d+1)(\min\{v,\,
1-v\})^{2t} \leq \displaystyle{\frac{h(\hat{\alpha}_n^+) -
h(\alpha_n)}{h(\alpha_n^+)-h(\alpha_n)}} \leq
\\
\\
\leq d(\max\{v,\, 1-v\})^{t} + (d+1)(\max\{v,\, 1-v\})^{2t}.
\end{array}
\end{equation}

By formulas~(\ref{eq-2}), (\ref{eq-4}) and~(\ref{eq-7}) we have
\begin{equation}\label{eq-9} m_1(d)h'_{p+nt+1}(x_0) \leq k_{CD}
\leq M_1(d)h'_{p+nt+1}(x_0),
\end{equation} where

$$ m_1(d) = \frac{2^t(d(\min\{v,\, 1-v\})^{t} +
(d+1)(\min\{v,\, 1-v\})^{2t})}{d^2}
$$
and$$ M_1(d) = \frac{2^t(d(\max\{v,\, 1-v\})^{t} +
(d+1)(\max\{v,\, 1-v\})^{2t})}{d^2}.
$$

By formulas~(\ref{eq-3}), (\ref{eq-5}) and~(\ref{eq-6}) we have
\begin{equation}\label{eq-10}m_2(d)h'_{p+nt+1}(x_0) \leq k_{RS}
\leq M_2(d)h'_{p+nt+1}(x_0), \end{equation} where $$ m_2(d) =
\frac{d(\min\{v,\, 1-v\})^{2t}}{\displaystyle{\frac{d}{2^t} +
\frac{d^2}{2^{2t}} + \frac{1}{2^{2t}}}} = \frac{2^{2t}d(\min\{v,\,
1-v\})^{2t}}{2^td + d^2 + 1}.
$$
and $$M_2(d) = \frac{2^{2t}d(\max\{v,\, 1-v\})^{2t}}{2^td + d^2 +
1}
$$

Notice, that constants $m_1$, $m_2$, $M_1$ and $M_2$ are dependent
on $x_0$, which is independent on $n$. These constants are
independent on $d$.

Since the periodical part of $x_0$ can be written in different
ways (i.e. the first digit of the periodical part of $x_0$ can be
chosen differently), we have that this periodical part can be
given of the form
\begin{equation}\label{eq-8} \sigma^i(x_{p+1}\ldots\, x_{p+t}),
\end{equation} where $\sigma$ is a cyclic permutation of $t$
elements and $i$ is some its iteration. Define by $d_i$ the
natural number, whose the binary decomposition is the
sequence~(\ref{eq-8}). Clearly, that $d_i$ is a periodical
sequence of period $t$. Denote by $m_1(x_0) = \min\limits_{1\leq
i\leq t}m_1(\sigma^i(d_i))$, $m_2(x_0) = \min\limits_{1\leq i\leq
t}m_2(\sigma^i(d_i))$, $M_1(x_0) = \max\limits_{1\leq i\leq
t}M_1(\sigma^i(d_i))$ and $M_2(x_0) = \max\limits_{1\leq i\leq
t}M_2(\sigma^i(d_i)$. Then the following restrictions, which are
analogical to~(\ref{eq-9}) and~(\ref{eq-10}) would hold.
$$
m_1(x_0)h'_{p+nt+1}(x_0) \leq k_{CD} \leq
M_1(x_0)h'_{p+nt+1}(x_0), $$
$$
m_2(x_0)h'_{p+nt+1}(x_0) \leq k_{RS} \leq
M_2(x_0)h'_{p+nt+1}(x_0).
$$
Since $$ k_{RS} \leq k(s_n)\leq k_{CD},
$$ we have that
\begin{equation}\label{eq-11}
m_2(x_0)h'_{p+nt+1}(x_0) \leq k(s_n) \leq
M_1(x_0)h'_{p+nt+1}(x_0).
\end{equation}

Now Theorem follows from Theorem~\ref{lema:32} and the evident
remark that the limit $\lim\limits_{n\to \infty } h_{n}^{'} (x_{0}
)$ exists for all $x_0\in \mathbb{Q}$.
\end{proof}

Theorem~\ref{theor:15} can be considered as a generalization of
Theorem~\ref{theor:12} to the set of rational numbers. It follows
from Theorem~\ref{theor:13} that Theorem~\ref{theor:12} can not be
generalized to the set of all real numbers $x_0\in [0,\, 1]$ such
that the limit $\mathop{\lim }\limits_{n\to \infty }
h_{n}^{'}(x_0)$ exists.

The following proposition holds.

\begin{proposition}\label{Prop:4}For every
$v\neq 1/2$ there exists $x_0\in [0,\, 1]$ such that $h'(x_0)=0$
and one of the following conditions hold.

1. The limit $\lim\limits_{n\to \infty } h_{n}^{'} (x_{0} )$ does
not exist;

2. The limit $\lim\limits_{n\to \infty } h_{n}^{'} (x_{0} )$
exists, but equals $\infty$.
\end{proposition}

\begin{proof}Proposition follows from
Theorems~\ref{theor:13} and~\ref{theor:12} if consider $f_{1-v}$
instead of $f_v$.
\end{proof}

The following observation follows from Theorems~\ref{theor:13}
and~\ref{theor:12}. Let $\lambda$ be the Lebesgue measure on the
interval $[0,\, 1]$. Denote by $\mathcal{A}^0$ the set, where the
derivative of $h$ equals to $0$ and denote by $\mathcal{B}$ the
set, where the derivative of $h$ equals to infinity. Then
$\lambda(\mathcal{A}^0) =1$, $\lambda(\mathcal{B}^0) =0$ and $h$
is non-differentiable on $[0,\, 1]\backslash (\mathcal{A}^0\cup
\mathcal{B}^0)$. It is evident, that the derivative of the inverse
function $h^{-1}$ equals $\infty$ on $h(\mathcal{A}^0)$ and this
derivative equals $0$ on $h(\mathcal{B}^0)$. Now, it follows from
Theorems~\ref{theor:13} and~\ref{theor:12} that
$\lambda(h(\mathcal{A}^0)) =0$ and $\lambda(h(\mathcal{B}^0)) =1$.
These properties of $h$ show how complicated it is.

\newpage

\section{Constructing of the conjugation via electronic
tables}\label{subs-excell}

This section is devoted to the explicit formulas for the
conjugation $h$ of maps

\begin{equation}\label{eq:65} f(x) =
\left\{\begin{array}{ll}
2x,& x< 1/2;\\
2-2x,& x\geqslant 1/2
\end{array}\right.
\end{equation} and
\begin{equation} \label{eq:66}
f_v(x) = \left\{\begin{array}{ll} \frac{x}{v},& x\leqslant v;\\
 \frac{1-x}{1-v},&
x>v.
\end{array}\right.
\end{equation}
In other words we will find the explicit formulas for the
homeomorphism $h: [0,\, 1]\rightarrow [0,\, 1]$, which is the
unique solution of the functional equation
\begin{equation} \label{eq:78} h(f) = f_v(h).
\end{equation}

The existence and uniqueness of this $h$ is proved in
Theorem~\ref{theor:homeom-jed}

In fact, Proposition~\ref{lema:25} contains the way of
constructing of the conjugation $h:\, [0,\, 1]\rightarrow [0,\,
1]$ of maps $f$ and $f_v$ at points of the set
$$ A_n = \left\{ 0,\, \frac{1}{2^{n-1}},\ldots\,
\frac{2^{n-1}-1}{2^{n-1}},\, 1\right\}
$$ as a limit of piecewise linear maps $h_n:\, [0,\, 1]\rightarrow [0,\,
1]$, whose breaking points belong to the set $A_n$ and such that
$h(A_n) = B_n$, where $$ B_n = f_v^{-n}(0).
$$

Under the electronic table we mean the table, whose lines are
numbered by arabic numbers (1,\, 2,\, ...), and columns are
numbered by letters (``A'',\, ``B'',\, ``C''\, ...), and the
following changes of this table are allowed.

1. Put some number into some cell;

2. Put the formula into some cell. The formula may contain symbols
of arithmetical operations the most known mathematical functions
and names of another cells. Also formula may contain some specific
functions, which are specially deals with electronic tables (we
will mention these functions below);

3. To copy the formula from the cell into any fixed number of
cells in vertical and (or) horizontal direction. In these case the
general agreement on sell copying holds (we will explain this
agreement just below).

Now we will explain a bit these rules. Since lines and columns of
the table are numbered as they are, then for instance, the left
top cell is ``A1'' and it is above ``A2''. The cells to the right
from ``A1'' is ``B1'' and so on.

If a cells contains a formula, then this cell has two
``parameters'': the formula itself and the value of the formula.
The simplest example could be the formula ``A1+1'', which is put
into any cell except ``A1'' (for instance ``A2''). The value of
the cell would be a number, which is 1 more then ``A1''. It is
important (and it is the one of the main deals of the use of
electronic tables), that if the cell ``A1'' (in our case) will be
changed, then the value of our cell will be changed immediately
and automatically. This rule is transitive. It means, that (for
instance in our case) that changing of ``A1'' leads to changing
not only the value of ``A2'', but also all the cell, which contain
formulas, which use ``A2'', because ``A2'' would be changed
because of the change of ``A1''.

the agreement about copying the formulas is the following. If we
copy the formula $k$ cells down, then all the references to cells
(i.e. names of cells) will be changer by increasing the number of
lines of cells by $k$. The similar rule is if the cell is copying
up of horizontally. Clearly, in the case of copying horizontally,
the thing which is changes is numbers of columns of references to
cells in the formula. These rules can be not applied to the
reference, which contain the symbol ``\$'' before the name of a
column and (or) the number of line. For instance, the formula
``A1+1'' will be transformed to ``A2+1'' when copying one cell
down and it will be transformed to ``B1+1'' when copying one cell
wight. In the same time, the formula ``A\$1+1'' becomes ``B\$1+1''
when copying right and does not transformed when copying
vertically. The formula ``\$A\$1+1'' does not transforms under
copying at all.

Let is return to finding the values of the homeomorphism $h$. For
an arbitrary natural $n>1$ we will construct the table of values
of $h$ at all rational points of the form $\frac{k}{2^n}$, where
$0\leq k\leq 2^n$.

For the convenience of the further use, we will reformulate
Proposition~\ref{lema:25} as follows.

\begin{proposition}\label{prop:1}Let $h:\,
[0,\, 1]\rightarrow [0,\, 1]$ be the topological conjugation of
the maps $f$ and $f_v$, which are defined by
formulas~(\ref{eq:65}) and~(\ref{eq:66}), i.e. $h$ is the solution
of the functional equation~(\ref{eq:78}). Then the following
implications hold.

1. If $x=0$, then $h(x) =0$.

2. If $x=1$, then $h(x)=1$.

3. If $x\leq \frac{1}{2}$, then $$h(x) = v\cdot h(2x)$$ and the
value $h(2x)$ appears to be found earlier.

4. If $x>\frac{1}{2}$, then $$h(x) = 1 - (1-v)\cdot h(-2x+2)$$ and
the value $h(-2x+2)$ appears to be found earlier.
\end{proposition}

Put $n$ into ``C1''. We will assume that ``C1'' contains an
integer greater then 1 and will not check the correctness of the
data from ``C1''. Also put $v$ into ``D1'' with the same remark,
i.e. we will not check that ``D1'' contains a number and that this
number is between 0 and 1. The deal of the remark of the
correctness of data is that we will not check the correctness of
numbers from ``C1'' and ``D1'' in another formula. Further
formulas may appear (and will appear) incorrect in the case if one
put ``bad'' data into ``C1'' and ``D1''.

Construct at first the formulas for obtaining the points of the
set $A_{n+1}$ in the array ``A1:\, A$2^n$'', i.e. obtain the
increasing rational numbers from $[0,\, 1]$ with denominator $2^n$
in ``A1:\, A$2^n$''.

Put 0 into ``A1'' and the formula ``A1+1/(2$\widehat{\,}$C\$1)''
into ``A2''. Here the symbol~~$\widehat{\,}$~, naturally means
that powering. After copying the cell ``A2'' obtain the necessary
set in the first columns of the table.

Make some additional remarks for constructing formulas for ``Bi''.
The cell ``Ai'' contains the value $\frac{i-1}{2^n}$. This means
that the equality $$ i = Ai\cdot 2^n+1$$ holds.

If $Ai<\frac{1}{2}$, then the line of the table, which contains
the value $h(2x)$ is of the number $2i-1$, or, in terms of $Ai$,
its number is $$ 2i-1 = 2\cdot Ai\cdot 2^{n}+1.
$$

Whence, if $A_i<\frac{1}{2}$, then the value $Bi$ should be found
by formula
\begin{center}v$\cdot$ INDIRECT( CONCATENATE("B";\ 2$\cdot$
A1$\cdot$ $2^{n}$+1)),\end{center} where the function CONCATENATE
constructs the text line with the array of its arguments and the
function INDIRECT returns the value of the cell, whose name if its
the unique argument in the format $Ax$, where $A$ is a Latin
letter (a number of a column) and $x$ is a natural number (the
number of a line).

If $Ai\geq \frac{1}{2}$, then the line of the table, which
contains the value $h(-2x+2)$ has the number $2i-1$ or, in terms
of $Ai$, its number
$$ (2-2(Ai))\cdot 2^n +1.
$$

Whence the values of $h$ on $A_{n+1}$ can be found with the
formula, which is introduced in the following theorem.

\begin{theorem}\label{theor:16}The value
of the conjugation $h$ of maps $f$ and $f_v$, which are determined
by formulas~(\ref{eq:65}) and~(\ref{eq:66}), can be determined in
the set $A_{n+1}$ with the following way.

Put $n$ into ``C1''.

Put $v$ into ``D1''.

Put 0 into ``A1'' and the formula ``A1+1/(2$\widehat{\,}$C\$1)''
into ``A2''.

Put the following formula into ``B1''.

\noindent IF(A1=0;\ 0;\ IF(A1=1,\, 1; IF(A1<=0,5;\

D\$1*INDIRECT( CONCATENATE("B";\ 2*A1*2$\widehat{\,}$ C\$1+1));

1-(1-D\$1)*INDIRECT(

CONCATENATE("B";\ -2*A1*2$\widehat{\,}\,$C\$1 +1 +2
$\widehat{\,}\,$(1+C\$1)))))).

Copy the formulas in columns $A$ and $B$ down will till the line
number $2^{n}+1$.
\end{theorem}

The Figure~\ref{fig:6} contains the values of conjugation at $A_7$
for $v= 0.55,\, 0.6,\, 0.7,\, 0.75,\, 0.8,\, 0.85,\, 0.9$ and
$0.95$.

\begin{figure}[htbp]
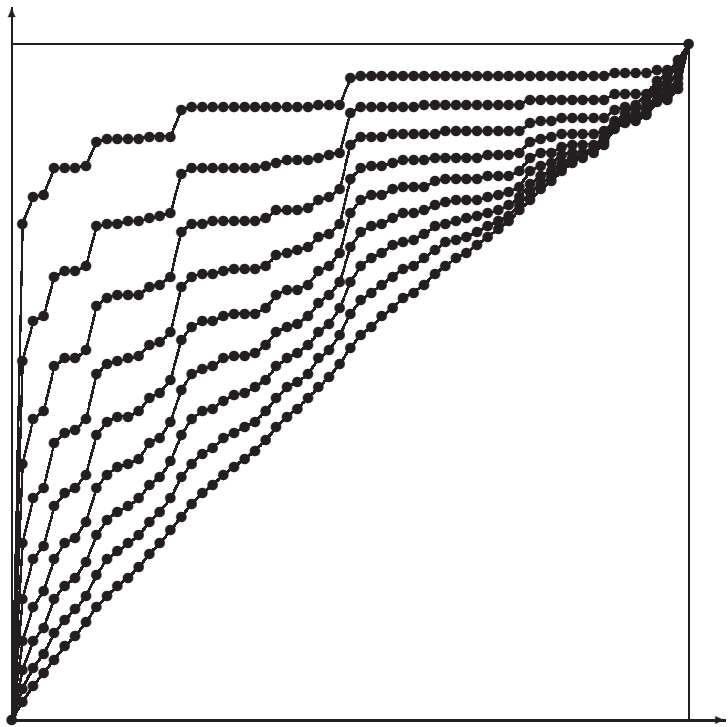

\begin{minipage}[h]{0.9\linewidth}
\begin{center}
\end{center}
\end{minipage}
\caption{Conjugation at $A_7$ for different $v$}\label{fig:6}
\end{figure}

Smaller $v$ correspond to graphs, which are closer to $y=x$. All
these graphs are obtained the formula above.

Remind that these graphs are approximations of the conjugation,
i.e. the monotone and continuous maps.

It seams from the picture that maps, which correspond to
$v=0.55,\, 0.6,\, 0.65$ are such as necessary in the time, when it
is ``hard to believe'' in the continuality of the maps, which
corresponds to $v=0.95$, because it seams to be ``evidently
discontinuous''. Nevertheless, it is also obvious that $h$ is
continuously dependent on $v$. In any was we have proved in
Theorem~\ref{theor:homeom-jed} that $h$ is continuous for every
$v\in (0,\, 1)$.

\newpage

\section{Functional equations which describe the topological
conjugation}\label{sect-funct-rivn}

We will try in this section to apply the methods of solving of
linear functional equations to finding the explicit formulas for
the conjugation of

\begin{equation} f(x) = \left\{\begin{array}{ll}
\label{eq:61}
2x,& x< 1/2;\\
2-2x,& x\geqslant 1/2
\end{array}\right.
\end{equation}and\begin{equation}
\label{eq:62}
f_v(x) = \left\{\begin{array}{ll} \frac{x}{v},& x\leqslant v;\\
 \frac{1-x}{1-v},&
x>v.
\end{array}\right.
\end{equation}

Let $h:\, [0,\, 1]\rightarrow [0,\, 1]$ be a homeomorphism such
that the diagram
\begin{equation} \label{eq:63}
\begin{CD}
[0,\, 1] @>f >> & [0,\, 1]\\
@V_{h} VV& @VV_{h}V\\
[0,\, 1] @>f_v>> & [0,\, 1]
\end{CD}
\end{equation} is commutative.

We obtain in this section the system of linear functional
equations, whose solution of the necessary homeomorphism $h$.

\begin{lemma}\label{lema:08}If a
homeomorphism $h:\, [0,\, 1]\rightarrow [0,\, 1]$ satisfies the
commutative diagram~(\ref{eq:63}), then it satisfies the
functional equation
\begin{equation}\label{eq:syst:line}
\left\{
\begin{array}{llllll}  h(2x) = \displaystyle{\frac{1}{v}\, h(x)} &
x\leq 1/2 & (\theequation a)
\\
h(2-2x) = \displaystyle{\frac{1-h(x)}{1-v}} & x>1/2 &
(\theequation b)\\
\end{array}\right.
\end{equation}\end{lemma}

\begin{proof}By Theorem~\ref{theor:homeom-jed} the
homeomorphism $h$ increase, i.e. $h(0)= 0$ and $h(1) = 1$.

Plug $x=1/2$ into the commutative diagram~(\ref{eq:63}) and obtain
$$
\begin{CD}
1/2 @>f >> & 1\\
@V_{h} VV& @VV_{h}V\\
h(1/2) @>f_v>>& 1,
\end{CD}$$
whence $h(1/2) = v$. Since $h$ is monotone for $x\in [0,\, 1/2]$
then it follows from the commutative diagram~(\ref{eq:63}) that
$$
\begin{CD}
[0,\, 1/2] @>f >> & [0,\, 1]\\
@V_{h} VV& @VV_{h}V\\
[0,\, v] @>f_v>>& [0,\, 1].
\end{CD}$$ Because of the form of equations of~(\ref{eq:61})
and~(\ref{eq:62}) the commutativity of the diagram is equivalent
to the functional equation~(\ref{eq:syst:line}a).

Functional equation~(\ref{eq:syst:line}b) is obtained in the same
manner if plug $x\in [1/2,\, 1]$ into the commutative
diagram~(\ref{eq:63}) and notice that $h(1/2) = v$.
\end{proof}

Section~\ref{sect-funct-rivn} is devoted to the solution of the
system of linear functional equations~(\ref{eq:syst:line}).

\subsection{The uniqueness of the solution of the system of
functional equations}\label{subcst-3-1}

We will prove in this section that the continuous solution of the
functional equation~(\ref{eq:syst:line}) is unique. It would
follow from Lemma~\ref{lema:08} and Theorem~\ref{theor:homeom-jed}
that this solution will be the conjugation of $f$ and $f_v$.

\begin{lemma}If the maps $h,\, [0,\, 1]\rightarrow
[0,\, 1]$ satisfies~(\ref{eq:syst:line}), then $h(0) =0$.
\end{lemma}

\begin{proof}Plug the value $x=0$
into~(\ref{eq:syst:line}a) and obtain $h(0) =0$.\end{proof}

\begin{lemma}\label{lema:01}If $h,\, [0,\,
1]\rightarrow [0,\, 1]$, which satisfies~(\ref{eq:syst:line}),
then $h\left(\frac{1}{2}\right) =v$.
\end{lemma}

\begin{proof}Plug $x=1/2$ into~(\ref{eq:syst:line}a)
and~(\ref{eq:syst:line}b), obtaining
$$ h(1) = \displaystyle{\frac{1}{v}\,
h(1/2)} = \displaystyle{\frac{1-h(1/2)}{1-v}},$$ whence $$h(1/2) =
v.
$$
\end{proof}

\begin{notation}We will say that the value of $h$ at
a point $x$ is \textbf{unambiguously defined} if there exists
$y\in [0,\, 1]$ such that for every continuous solution
$\widetilde{h}$ of the system of functional
equations~(\ref{eq:syst:line}) the equality $\widetilde{h}(x)=y$
holds.

For instance, the maps $h$ is unambiguously defined at points
$0,\, 1$ and $\frac{1}{2}$.
\end{notation}

\begin{lemma}\label{lema:02}If the maps $h$ is
unambiguously defined at a point $\widetilde{x}$ then it is
unambiguously defined at the integer trajectory of
$\widetilde{x}$.
\end{lemma}

\begin{proof}Let $\widetilde{x}$ be an arbitrary
point, where $h$ is unambiguously defined.

Show, that in this case the maps $h$ is unambiguously defined at
the point $f(\widetilde{x})$. Indeed, if $\widetilde{x}\leq
\frac{1}{2}$ then plugging of $x=\widetilde{x}$,
into~(\ref{eq:syst:line}a) we can find $h(2\widetilde{x}) =
\frac{1}{v}h(\widetilde{x})$ and this would mean that
$h(f\widetilde{x})$ is unambiguously defined.

If $\widetilde{x}> \frac{1}{2}$, then plugging of
$x=\widetilde{x}$ into~(\ref{eq:syst:line}b) we would find
$h(-2\widetilde{x} + 2) = \frac{1-h(\widetilde{x})}{1-v}$, which
means that $h(f\widetilde{x})$ is unambiguously defined.

Let $\widetilde{\widetilde{x}}$ be some pre image of
$\widetilde{x}$. In this case the fact that $h$ is unambiguously
defined at $\widetilde{x}$ can be proved in the same manner either
by plugging of $x = 2\widetilde{\widetilde{x}}$ into the equation
(\ref{eq:syst:line}a), or by plugging $x =
\frac{2-\widetilde{\widetilde{x}}}{2}$ into~(\ref{eq:syst:line}b).
\end{proof}

These results can be generalized in the following theorem.

\begin{theorem}\label{theor:17}The
system of functional equations~(\ref{eq:syst:line}) has a unique
continuous solution. This solution is the conjugation of $f$ and
$f_v$.
\end{theorem}

\begin{proof}Consider the conjugation $h$ of maps $f$
and $f_v$, which exists by Theorem~\ref{theor:homeom-jed}. By
Lemma~\ref{lema:08} $h$ is the solution of the
system~(\ref{eq:syst:line}).

Prove the uniqueness of the solution of~(\ref{eq:syst:line}).
Consider the sets $A_n$, such that $f^n(A_n) =0$, which were
introduced at Section~\ref{sect-Pobudowa-1}. By
Lemma~\ref{lema:02} the solution $h$ is unambiguously defined at
$\mathcal{A} = \bigcup\limits_{n=1}^\infty A_n$. Now the
uniqueness of the solution of~(\ref{eq:syst:line}) follows from
Proposition~(\ref{lema:An}), i.e. from the density
of~$\mathcal{A}$.
\end{proof}

\subsection{Explicit formulas for the solutions of functional
equations}\label{subcst-3-2}

Each from functional equations of~(\ref{eq:syst:line}) belongs to
the class of so called linear functional equations. Methods of
solving of linear functional equations are known, but such a
solution is determined up to an arbitrary function.

After finding the general solution of any of the equations
of~(\ref{eq:syst:line}) (up to arbitrary function) we may plug
this solution into another functional equation for obtaining the
new equation on the ``arbitrary function'' from the former
equation and the for finding this ``arbitrary function'' from the
new functional equation.

\subsubsection{General methods of solving of functional
equations.}\label{subsusbs-formuly}

Consider the functional equation
\begin{equation}\label{eq:15}h(A(x)) = B(x,\, h(x)), \end{equation}
where $h$ is unknown function and $A$ and $B$ are known functions,
which act from the real axes to real axes.

The following so called characteristic transformation $S:\,
\mathbb{R}^2 \rightarrow \mathbb{R}^2$ can be constructed
by~(\ref{eq:15}).
\begin{equation} \label{eq:19}
S:\left\{ \begin{array}{l} t \rightarrow A(t),\\
y\rightarrow B(t,\, y).
\end{array}\right.
\end{equation}

If a function $h$ satisfies~(\ref{eq:15}), then its graph (i.e.
the set $\{(x,\, h(x)) \}$) would be invariant under the
characteristic transformation $S$ and conversely, each set, which
is invariant under $S$, corresponds to some solution
of~(\ref{eq:15}). This let to reduce the problem of solving of
functional equation to the problem of finding the invariants of
its characteristic transformations, i.e. to finding of functions
$\varphi:\, \mathbb{R}^2 \rightarrow \mathbb{R}$, such that the
equality $$ \varphi(t,\, y) = \varphi(S(t,\, y)).
$$ holds.

\begin{notation}An equation of the form
\begin{equation}\label{eq:20}h(A(x)) = B(x)\cdot h(x) + C(x),
\end{equation} where
$A,\, B$ and $C$ are given functions $\mathbb{R}\rightarrow
\mathbb{R}$ is called a \textbf{linear functional equation}.
\end{notation}

\begin{notation}An equation of the form
\begin{equation} \label{eq:16} h(ax) = b\cdot h(t),
\end{equation} where $a$ and $b$ are constants is called a \textbf{linear
functional equation with constant coefficients}.
\end{notation}

For instance the equation~(\ref{eq:syst:line}a) is a linear
functional equation with constant coefficients, but the
equation~(\ref{eq:syst:line}b) is a linear functional equation,
but it is not with constant coefficients.

The characteristic transformation of the equation~(\ref{eq:16}) is
as follows.
\begin{equation}\label{eq:21} S:\left\{
\begin{array}{l}
t \rightarrow at,\\
y\rightarrow by
\end{array}\right.
\end{equation}

If $a>0$ and $b>0$ then consider $\mu = log_ab$ and notice that
the function
$$\varphi(t,\, y) =\frac{y}{t^{\log_ab}}$$ is
invariant under the characteristic transformation, since the
equality $$\frac{y}{t^{\log_ab}} = \frac{by}{(at)^{\log_ab}} $$
from the definition holds.

Since such invariant is known, it is naturally to find the
solution of~(\ref{eq:16}) in the form
\begin{equation} \label{eq:17}h(x) = \omega(x)\cdot x^{\log_ab},
\end{equation} where $\omega$ is unknown function.

\begin{note}Notice, that we could write the
equality~(\ref{eq:17}) for the equation~(\ref{eq:16}) without any
explanation and say that we want to find the solution
of~(\ref{eq:16}) in his form.
\end{note}

Plugging~(\ref{eq:17}) into~(\ref{eq:16}) obtain
$$ \omega(ax)\cdot (ax)^{\log_ab} = b\cdot \omega(x)\cdot
x^{\log_ab},
$$ i.e. $$
\omega(ax) = \omega(x).
$$

We can understand the obtained equality as the dependence of
$\omega$ on $\log_ax$, such that $\omega$ is periodical with
period 1, i.e. $h$ is a solution of the equation~(\ref{eq:15}) if
and only if it is of the form
\begin{equation}\label{eq:64} h(x) =
x^{\log_ab}\omega(\log_ax),\end{equation} where $\omega(x)$ is a
function with period 1.

If at least one of numbers $a$ and $b$ in~(\ref{eq:16}) is
negative, then the invariant of the characteristic transformation
would be $$ \varphi = \frac{y}{t^{\log_{|a|}|b|}}.
$$ Since the sign of $x$ is unknown, then it would be impossible
to find the solution in the form~(\ref{eq:17}) and its correspond
form is
\begin{equation}\label{eq:18} h(x) = |x|^{\log_{|a|}|b|}\cdot
\left\{
\begin{array}{ll}
\omega^+(\log_{|a|}x) & x>0;\\
\omega^-(\log_{|a|}|x|) & x<0,
\end{array}\right.
\end{equation} where $\omega^+$ and $\omega^-$ are some functions.

If plug the expression~(\ref{eq:18}) into the
equation~(\ref{eq:16}) then, dependently in the signs of $a$ and
$b$ obtain the following conditions for $\omega^+$ and $\omega^-$.
$$
\begin{array}{ll}\left\{ \begin{array}{l}
\omega^+(x+1) = \omega^+(x);\\
\omega^-(x+1) = \omega^-(x),
\end{array}\right. \text{ for }a>0,\ b>0.
\end{array}$$

$$
\begin{array}{ll}\left\{ \begin{array}{l}
\omega^+(x+1) = -\omega^+(x);\\
\omega^-(x+1) = -\omega^-(x),
\end{array}\right. \text{ for }a>0,\ b<0.
\end{array}
$$

$$
\begin{array}{ll}\left\{ \begin{array}{l}
\omega^-(x+1) = \omega^+(x);\\
\omega^+(x+1) = \omega^-(x),
\end{array}\right. \text{ for }a<0,\ b>0.
\end{array}
$$

$$
\begin{array}{ll}\left\{ \begin{array}{l}
\omega^-(x+1) = -\omega^+(x);\\
\omega^+(x+1) = -\omega^-(x),
\end{array}\right. \text{ for }a<0,\ b<0.
\end{array}
$$

\begin{note}Notice, that the first of the forth
conditions above is considered earlier, but we have presented it
for completeness.
\end{note}

We will present methods, which let to obtain the explicit solution
of the equation~(\ref{eq:syst:line}b) in spite of that it is not
linear functional equation.

Assume that the functional equation~(\ref{eq:15}) is such that
invariants of its characteristic transformation $S$, which is
defined by~(\ref{eq:19}), is not obvious.

\begin{lemma}\label{lema:11}For an arbitrary
invertible maps $H:\, \mathbb{R}^2 \rightarrow \mathbb{R}^2$
consider the following commutative diagram
$$\begin{CD}
\mathbb{R}^2 @>S>> & \mathbb{R}^2\\
@V_{H} VV& @VV_{H}V\\
\mathbb{R}^2 @>\widetilde{S}>>& \mathbb{R}^2.
\end{CD}$$ If $\widetilde{\varphi}:\, \mathbb{R}^2\rightarrow
\mathbb{R}$ is the invariant of the characteristic maps
$\widetilde{S}$, then $\varphi(t,\, y) =
\widetilde{\varphi}(H(S(t,\, y))$ is an invariant of $S$.
\end{lemma}

\begin{proof}It is enough to prove that $\varphi
(S(t,\, y)) = \varphi(t,\, y)$ for proving the lemma.

By the definition of $\varphi$ we have that
$$\varphi(S(t,\, y)) = \widetilde{\varphi}(H(S(t,\, y)).$$ It
follows from the commutativity of the diagram, obtain that $
H(S(t,\, y)) = \widetilde{S}(H(t,\, y)),$ whence $$
\widetilde{\varphi}(H(S(t,\, y)) =
\widetilde{\varphi}(\widetilde{S}(H(t,\, y))).$$

But since $\widetilde{\varphi}$ is an invariant of $\widetilde{S}$
then $$ \widetilde{\varphi}(\widetilde{S}(H(t,\, y))) =
\widetilde{\varphi}(H(t,\, y)),
$$ and again by the construction of $\varphi$
obtain that $$\widetilde{\varphi}(H(t,\, y))= \varphi(t,\, y),$$
which proves Lemma.
\end{proof}

\begin{notation}The maps $H$ from
Lemma~\ref{lema:11} is called the \textbf{change of variables}.
\end{notation}

We will show the applying of Lemma~\ref{lema:11} to the functional
equation~(\ref{eq:syst:line}b):
$$ h(2-2x) = \frac{-h(x)}{1-v} + \frac{1}{1-v}.
$$

Characteristic transformation of this equation is as follows
\begin{equation}\label{eq:22}
S:\left\{ \begin{array}{l} t \rightarrow 2-2t,\\
y\rightarrow \frac{-y}{1-v} + \frac{1}{1-v}.
\end{array}\right.
\end{equation}

Let us find a changing of variables $H$ such that the
characteristic transformation of the
equation~(\ref{eq:syst:line}b) become of the form~(\ref{eq:21}).

We will find the maps $H$ of the form
$$H(t,\, y) = (t - t_0,\, y-y_0)$$ for fixed $t_0$ and
$y_0$.

The maps $\widetilde{S}$ can be found from the formula
$$
\widetilde{S}(t,\, y) = H(S(H^{-1}(t,\, y))) =
$$
$$
=\left(2-2(t+t_0)-t_0,\, \frac{-(y+y_0)}{1-v} +
\frac{1}{1-v}-y_0\right).
$$

In the assumption that $\widetilde{S}$ has no ``free variables'',
i.e. saying that there should not be expressions,
which are independent on $t$ and $y$ obtain $$\left\{ \begin{array}{l} 2-3t_0 =0\\
\frac{-y_0+1}{1-v} - y_0 =0.
\end{array}\right.
$$

\begin{note}Notice that the point $(x_0,\, y_0)$
appeared to be a fixed point of~(\ref{eq:22}).
\end{note}

For the found $(t_0,\, x_0)$ the form transformation
$\widetilde{S}$ appears to be $$ \widetilde{S}:\left\{
\begin{array}{l} \widetilde{t}
\rightarrow -2\widetilde{t},\\
\widetilde{y}\rightarrow \frac{-\widetilde{y}}{1-v}.
\end{array}\right.
$$

The constructed transformation $\widetilde{S}$ is a characteristic
transformation of a linear functional equation with constant
coefficients. The solution of this equation can be easily obtained
in the form
$$
h(x) = \frac{1}{2-v} +\left|x-\frac{2}{3}\right|^{-\log_2(1-v)}
\times
$$$$ \times\ \cdot \left\{
\begin{array}{ll}
\omega^+\left( \log_2\left|x-\frac{2}{3}\right|\right) &
x>\frac{2}{3};\\
\\
\omega^-\left( \log_2\left|x-\frac{2}{3}\right|\right) & x<
\frac{2}{3},
\end{array}\right.
$$ where functions $\omega^+$ and $\omega^-$ satisfy the condition
$$\left\{
\begin{array}{l}\omega^-(t+1) = -\omega^+(t)\\
\omega^+(t+1) = -\omega^-(t).
\end{array}\right.
$$

\subsubsection{Plugging of the solutions of one functional
equations into another equation os the
system}\label{subsubs-Pidst}

The equation~(\ref{eq:syst:line}a) is a linear functional
equation. It is obtained from~(\ref{eq:16}) by plugging $a=2$ and
$b=\frac{1}{v}$. The formula for the solution
of~(\ref{eq:syst:line}a) os obtained by plugging of $a=2$ and
$b=\frac{1}{v}$ into~(\ref{eq:64}) and is

\begin{equation}\label{ex:h:first} h(x) =
x^{-\log_2v}\omega(\log_2x),
\end{equation} where $\omega(x)$ is an arbitrary function
with period 1. If plug~(\ref{ex:h:first})
into~(\ref{eq:syst:line}b), then obtain

$$(2-2x)^{-\log_2v}\omega(\log_2(2-2x)) = $$ $$ =
\frac{1-x^{-\log_2v}\omega(\log_2x)}{1-v}.$$

The periodicity of $\omega$ let us to rewrite this equation as
follows

\begin{equation}\label{eq:gliuk}
\begin{array}{c}
(1-v)(1-x)^{-\log_2v}\omega(\log_2(1-x)) =\\ \\
 = v(1-x^{-\log_2v}\omega(\log_2x)).
\end{array}
 \end{equation}
\begin{proposition}\label{zauv:2}
If consider the equation~(\ref{eq:gliuk}) as a functional equation
of the function, which is defied on the whole line, then it will
appear, that $h$ it is constant function.
\end{proposition}

\begin{proof}Define $t=1-x$ and obtain
$$(1-v)t^{-\log_2v}\omega(\log_2t)
= $$ $$ = v(1-(1-t)^{-\log_2v}\omega(\log_2(1-t))).$$

If write $x$ instead of $t$ then we can use~(\ref{eq:gliuk}) to
express $(1-x)^{-\log_2v}\omega(\log_2(1-x))$ from each of the
equations and equate them, obtaining
$$ \frac{v}{1-v}\left(1-x^{-\log_2v}\omega(\log_2x)\right) =
$$$$ = 1 -\frac{1-v}{v}\, x^{-\log_2v}\omega(\log_2x),$$ whence $$
x^{-\log_2v}\omega(\log_2x) = v.
$$ This equality mean that
the function $h$ is constant and has no inverse.
\end{proof}

\begin{note}\label{note:13} If solve
the equation~(\ref{eq:gliuk}) for the $h: [0,\, 1]\rightarrow
[0,\, 1]$, which is defined by~(\ref{ex:h:first}), then reasonings
from the Proposition~\ref{zauv:2} would be incorrect.
\end{note}

\begin{proof}[Explanation of the Remark.] The deal is
that equation~(\ref{eq:gliuk}) is obtained by plugging of the
solution of~(\ref{eq:syst:line}a) into~(\ref{eq:syst:line}b).

That is why, the substitution $t = 1-x$ is, if fact, the
substitution in~(\ref{eq:syst:line}a).

Nevertheless, the equation~(\ref{eq:syst:line}a) is obtained from
the commutativity of the diagram
$$\begin{CD}
[0,\, 1/2] @>f:\, x\mapsto 2x >> & [0,\, 1]\\
@V_{h} VV& @VV_{h}V\\
[0,\, v] @>f_v:\, x\mapsto x/v>>& [0,\, 1],
\end{CD}$$ which is defined only
for $x\in [0,\, 1/2]$. Since the substitution $t=1-x$ for $x\in
[0,\, 1/2]$ mens that $t\in [1/2,\, 1]$, then it leads to a
functional equation, which is defined for another set of arguments
and expressing of
$$(1-x)^{-\log_2v}\omega(\log_2(1-x))$$ from both
equations with further equating is incorrect.
\end{proof}

We have obtained in the Section~\ref{subsusbs-formuly} that the
solution of~(\ref{eq:syst:line}b) is of the form

\begin{equation}\label{ex:h-second}
\begin{array}{l}
h(x) = \frac{1}{2-v} +\left|x-\frac{2}{3}\right|^{-\log_2(1-v)}
\times
\\
\times\ \cdot \left\{
\begin{array}{ll}
\omega^+\left( \log_2\left|x-\frac{2}{3}\right|\right) &
x>\frac{2}{3};\\
\\
\omega^-\left( \log_2\left|x-\frac{2}{3}\right|\right) & x<
\frac{2}{3}.
\end{array}\right.
\end{array}
\end{equation}
for functions $\omega^+$ and $\omega^-$, such that
\begin{equation}\label{eq:14}\left\{
\begin{array}{l}\omega^-(t+1) = -\omega^+(t)\\
\omega^+(t+1) = -\omega^-(t).
\end{array}\right.
\end{equation}

In the same manner as in the previous section, we may consider the
obtained expressions of $h$ and try to find $\omega$.

It follows from~(\ref{eq:14}) that functions $\omega^+$ and
$\omega^-$ are periodical with period 2.

Plug the solution~(\ref{ex:h-second}) of the functional
equation~(\ref{eq:syst:line}b) into~(\ref{eq:syst:line}a) and
obtain

$$
\begin{array}{l}
\frac{v}{2-v} +v\left|2x-\frac{2}{3}\right|^{-\log_2(1-v)} \times
\\
\times\ v\cdot \left\{
\begin{array}{ll}
\omega^+\left( \log_2\left|2x-\frac{1}{3}\right|\right) &
x>\frac{2}{3};\\
\\
\omega^-\left( \log_2\left|2x-\frac{1}{3}\right|\right) & x<
\frac{2}{3}
\end{array}\right.
\end{array} =
$$$$
\begin{array}{l}
=\frac{1}{2-v} +\left|x-\frac{2}{3}\right|^{-\log_2(1-v)} \times
\\
\times\ \cdot \left\{
\begin{array}{ll}
\omega^+\left( \log_2\left|x-\frac{2}{3}\right|\right) &
x>\frac{2}{3};\\
\\
\omega^-\left( \log_2\left|x-\frac{2}{3}\right|\right) & x<
\frac{2}{3}
\end{array}\right.
\end{array}
$$for unknown
functions $\omega^+$ and $\omega^-$, which are connected by
expressions~(\ref{eq:14}).

The complicatedness of the obtained equation in comparison with
the former one is evident.

The problem on the invertibility of approximations of $h$, which
are obtained in the described way, is so complicated as in the
case, when the solution of the first equation was plugged into the
second one.

\subsection{Numerical experiments}\label{subsubs-exper}

Theorem~\ref{theor:16} contains the formula, which let ro
construct $h$ on $A_n$ in terms on electronic tables for arbitrary
$n$.

Let maps $h$ be defined by~(\ref{ex:h:first}). With the use of
known values of $h$ on $A_n$ for enough huge $n$, we can calculate
values of $\omega$ for $x\in [0,\, 1]$. Remind that $\omega$ is
periodical with period 1.

Graph of $\omega$ for $v=3/4$ and $x\in [0,\, 1]$ is given at
Figure~\ref{fig:17}.
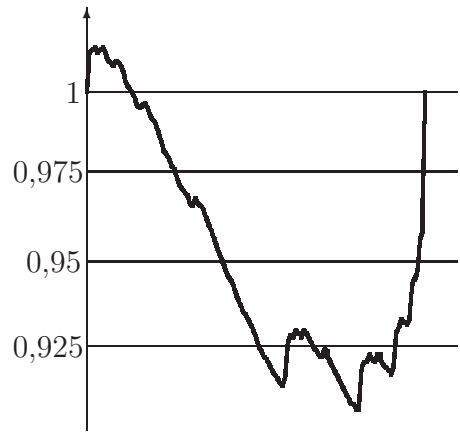
\begin{figure}[htbp]
\begin{minipage}[h]{0.9\linewidth}
\begin{center}
\begin{picture}(140,160)
\put(0,0){\vector(0,1){160}}

\put(0,64){\line(1,0){140}} \put(0,32){\line(1,0){140}}
\put(0,98){\line(1,0){140}} \put(0,128){\line(1,0){140}}

\linethickness{0.4mm}

\qbezier(0,128)(0,128)(1,143) \qbezier(1,143)(1,143)(3,145)
\qbezier(3,145)(3,145)(4,143) \qbezier(4,143)(4,143)(6,145)
\qbezier(6,145)(6,145)(7,143) \qbezier(7,143)(7,143)(8,140)
\qbezier(8,140)(8,140)(10,138) \qbezier(10,138)(10,138)(11,140)
\qbezier(11,140)(11,140)(13,138) \qbezier(13,138)(13,138)(14,135)
\qbezier(14,135)(14,135)(15,131) \qbezier(15,131)(15,131)(17,128)
\qbezier(17,128)(17,128)(18,126) \qbezier(18,126)(18,126)(19,123)
\qbezier(19,123)(19,123)(20,122) \qbezier(20,122)(20,122)(22,124)
\qbezier(22,124)(22,124)(23,122) \qbezier(23,122)(23,122)(24,119)
\qbezier(24,119)(24,119)(26,116) \qbezier(26,116)(26,116)(27,113)
\qbezier(27,113)(27,113)(28,110) \qbezier(28,110)(28,110)(29,106)
\qbezier(29,106)(29,106)(31,103) \qbezier(31,103)(31,103)(32,100)
\qbezier(32,100)(32,100)(33,99) \qbezier(33,99)(33,99)(34,96)
\qbezier(34,96)(34,96)(35,93) \qbezier(35,93)(35,93)(37,90)
\qbezier(37,90)(37,90)(38,89) \qbezier(38,89)(38,89)(39,86)
\qbezier(39,86)(39,86)(40,85) \qbezier(40,85)(40,85)(41,88)
\qbezier(41,88)(41,88)(42,86) \qbezier(42,86)(42,86)(44,84)
\qbezier(44,84)(44,84)(45,81) \qbezier(45,81)(45,81)(46,78)
\qbezier(46,78)(46,78)(47,76) \qbezier(47,76)(47,76)(48,73)
\qbezier(48,73)(48,73)(49,70) \qbezier(49,70)(49,70)(50,67)
\qbezier(50,67)(50,67)(51,65) \qbezier(51,65)(51,65)(52,62)
\qbezier(52,62)(52,62)(53,59) \qbezier(53,59)(53,59)(55,56)
\qbezier(55,56)(55,56)(56,53) \qbezier(56,53)(56,53)(57,50)
\qbezier(57,50)(57,50)(58,48) \qbezier(58,48)(58,48)(59,45)
\qbezier(59,45)(59,45)(60,44) \qbezier(60,44)(60,44)(61,42)
\qbezier(61,42)(61,42)(62,40) \qbezier(62,40)(62,40)(63,38)
\qbezier(63,38)(63,38)(64,35) \qbezier(64,35)(64,35)(65,32)
\qbezier(65,32)(65,32)(66,30) \qbezier(66,30)(66,30)(67,28)
\qbezier(67,28)(67,28)(68,27) \qbezier(68,27)(68,27)(69,25)
\qbezier(69,25)(69,25)(70,23) \qbezier(70,23)(70,23)(71,21)
\qbezier(71,21)(71,21)(72,20) \qbezier(72,20)(72,20)(73,18)
\qbezier(73,18)(73,18)(74,17) \qbezier(74,17)(74,17)(75,20)
\qbezier(75,20)(75,20)(76,33) \qbezier(76,33)(76,33)(77,36)
\qbezier(77,36)(77,36)(78,35) \qbezier(78,35)(78,35)(79,38)
\qbezier(79,38)(79,38)(80,37) \qbezier(80,37)(80,37)(81,36)
\qbezier(81,36)(81,36)(81,35) \qbezier(81,35)(81,35)(82,38)
\qbezier(82,38)(82,38)(83,37) \qbezier(83,37)(83,37)(84,35)
\qbezier(84,35)(84,35)(85,33) \qbezier(85,33)(85,33)(86,31)
\qbezier(86,31)(86,31)(87,30) \qbezier(87,30)(87,30)(88,28)
\qbezier(88,28)(88,28)(89,28) \qbezier(89,28)(89,28)(90,31)
\qbezier(90,31)(90,31)(91,30) \qbezier(91,30)(91,30)(91,28)
\qbezier(91,28)(91,28)(92,26) \qbezier(92,26)(92,26)(93,24)
\qbezier(93,24)(93,24)(94,22) \qbezier(94,22)(94,22)(95,20)
\qbezier(95,20)(95,20)(96,18) \qbezier(96,18)(96,18)(97,16)
\qbezier(97,16)(97,16)(98,14) \qbezier(98,14)(98,14)(99,12)
\qbezier(99,12)(99,12)(100,10) \qbezier(100,10)(100,10)(101,10)
\qbezier(101,10)(101,10)(102,8) \qbezier(102,8)(102,8)(103,8)
\qbezier(103,8)(103,8)(103,11) \qbezier(103,11)(103,11)(104,24)
\qbezier(104,24)(104,24)(105,26) \qbezier(105,26)(105,26)(106,26)
\qbezier(106,26)(106,26)(107,29) \qbezier(107,29)(107,29)(107,28)
\qbezier(107,28)(107,28)(108,27) \qbezier(108,27)(108,27)(109,26)
\qbezier(109,26)(109,26)(110,29) \qbezier(110,29)(110,29)(111,29)
\qbezier(111,29)(111,29)(111,27) \qbezier(111,27)(111,27)(112,25)
\qbezier(112,25)(112,25)(113,24) \qbezier(113,24)(113,24)(114,23)
\qbezier(114,23)(114,23)(115,22) \qbezier(115,22)(115,22)(115,21)
\qbezier(115,21)(115,21)(116,24) \qbezier(116,24)(116,24)(117,37)
\qbezier(117,37)(117,37)(118,39) \qbezier(118,39)(118,39)(119,42)
\qbezier(119,42)(119,42)(120,41) \qbezier(120,41)(120,41)(121,40)
\qbezier(121,40)(121,40)(122,42) \qbezier(122,42)(122,42)(123,55)
\qbezier(123,55)(123,55)(124,58) \qbezier(124,58)(124,58)(124,57)
\qbezier(124,57)(124,57)(125,60) \qbezier(125,60)(125,60)(126,72)
\qbezier(126,72)(126,72)(127,75) \qbezier(127,75)(127,75)(127,87)
\qbezier(127,87)(127,87)(128,128)

\put(-28,28){0,925} \put(-23,60){0,95} \put(-28,94){0,975}
\put(-8,124){1}

\end{picture}
\end{center}
\end{minipage}
\caption{Graph of $\omega_1(x)$ for $v=3/4$}\label{fig:17}
\end{figure}

\begin{note}Make the remark on the way, how the
graph on Figure~\ref{fig:17} was obtained.
\end{note}

\begin{proof}[Deal of the remark] Since $\omega$ is
periodical with period 1, then it is enough to find it on any
interval of the length 1.

If plug all numbers $x\in [1/2,\, 1]$ with the step, small enough,
into the equation~(\ref{ex:h:first}), then function $\log_2x$
would be found in all points (with correspond small step) of the
interval $[-1,\, 0]$.

In this case for every $x\in [0,\, 1]$ we may consider $\log_2x$
and $\omega(\log_2x) = h(x)\cdot x^{log_2v}$. This procedure will
give use the set of points, where the function $\omega$ is found.
\end{proof}

The fact, that $h$, which is a solution of~(\ref{eq:syst:line}),
is very complicated, follows from that $h$ is non-differentiable
of the dense set in $[0,\, 1]$. It follows from the form of the
formula~(\ref{ex:h:first}) that it is $\omega$, which is the term,
where complexity of $h$ comes from, because $x^{-\log_2v}$ is
differentiable everywhere.

In the same time, we present below the properties of invertible
maps $h$, which is of the form~(\ref{ex:h:first}).

\begin{lemma}\label{lema:09}If the invertible maps
$h:\, [0,\, 1]\rightarrow [0,\, 1]$ satisfies~(\ref{ex:h:first}),
then the following conditions hold.

1. $h$ increase;

2. For every $n\in \mathbb{N}$ the equality
$h\left(\frac{1}{2^n}\right) = v^n$.

Also for all integer $t$ the equality $\omega(t) =1$ holds.
\end{lemma}

\begin{proof}Prove first, that $\omega$ is bounded.
It is so, because $\omega$ is defined by its values, which are
obtained from the equation~(\ref{ex:h:first}) for $x\in [1/2,\,
1]$. But in this case the function $x^{-\log_2v}$ is bounded by
$v$.

Plug $x=0$ into~(\ref{ex:h:first}) and obtain the product of zero
function times bounded, whence $h(0)=0$. Whence, together with
that $h$ is invertible, means that $h$ increase.

If follows from $h(1)=1$ that plugging of $x=1$
into~(\ref{ex:h:first}) gives $1 = \omega(\log_2x)$. Since
$\omega$ is periodical with period 1, then for every $t\in
\mathbb{Z}$ the equality $\omega(t) = 1$ holds.

Plugging $x=\frac{1}{2^n}$ into~(\ref{ex:h:first}) gives
$$h\left(\frac{1}{2^n}\right) = v^n\,
\omega(\log_22^{-n}) = v^n.$$\end{proof}

Consider examples of ``simple'' maps $\omega$, but such that $h$,
which is defined by~(\ref{ex:h:first}), is invertible and consider
the maps $\widetilde{f}_v$, which is defined by commutative
diagram
\begin{equation}\label{eq:13}
\begin{CD}
[0,\, 1] @>f >> & [0,\, 1]\\
@V_{h} VV& @VV_{h}V\\
[0,\, 1] @>\widetilde{f}_v>>& [0,\, 1].
\end{CD}\end{equation}

\begin{lemma}\label{lema:10}If for invertible maps
$h$ of the form~(\ref{ex:h:first}) the diagram~(\ref{eq:13}) is
commutative, then for $x\hm{\in} [0,\, v]$ the equality
$$ \widetilde{f}_v(x) = \frac{x}{v}$$ holds.
\end{lemma}

\begin{proof}It follows from the invertibility of $h$
that $\widetilde{f}_v$ is well-defined, precisely for all $x\in
[0,\, 1]$ the equality $\widetilde{f}_v = h(f(h^{-1}(x)))$ holds.
But by Lemma~\ref{lema:09} for $x\in [0,\, v]$ the inclusion
$h^{-1}(x) \in [0,\, 1/2]$ holds, whence $\widetilde{f}_v(x) =
h(2h^{-1}(x))$.

Since the expression~(\ref{ex:h:first}) is obtained from the
functional equation~(\ref{eq:syst:line}a), then $h$ from the
condition of Lemma satisfies~(\ref{eq:syst:line}a). But the
equation~(\ref{eq:syst:line}a) is equivalent to commutativity of
the diagram~(\ref{eq:13}) for $\widetilde{f}_v(x) = \frac{x}{v}$.
Now Lemma follows from the uniqueness of $\widetilde{f}_v$.
\end{proof}

The simplest case, when $\omega$ is a periodical function with
period 1 is that when it is constant.

If follows from $h(1) = 1$ that if $\omega$ is constant, then
$\omega(x) = 1$ for all $x$.

\begin{example}\label{ex:01}Plot the graph of
$\widetilde{f}_v$, which is defined by the commutative
diagram~(\ref{eq:13}) for the maps $h$ of the
form~(\ref{ex:h:first}), if $\omega$ is constant.
\end{example}

\begin{proof}[Deal of the example] If $\omega(x) =1$
for every $x \in [0,\, 1]$, then $h(x) = x^{-\log_2x}$. Then by
Lemma~\ref{lema:10} (this also can be shown from the direct
calculations) follows that for all $x\in [0,\, v]$ the equality
$\widetilde{f}_v(x) = \frac{x}{v}$ holds.

For such function $\omega$ the equality $h^{-1}(x)= x^{-\log_v2}$
holds, whence for $x\in [v,\, 1]$ obtain
$$ \widetilde{f}_v(x) = \left(2-2x^{-\log_2v}\right)^{-\log_v2}.
$$
The graph of $\widetilde{f}_v$ for $v =3/4$ is given at
Figure~\ref{fig-21}.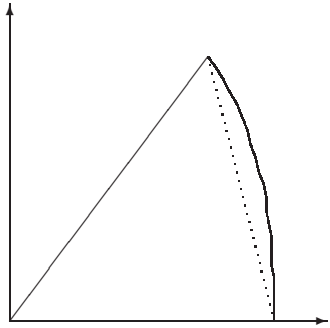
\begin{figure}[htbp]
\begin{minipage}[h]{0.9\linewidth}
\begin{center}
\begin{picture}(100,120)
\put(0,0){\vector(0,1){120}} \put(0,0){\vector(1,0){120}}

\put(0,0){\line(3,4){75}}

\qbezier(75,100)(75,100)(78,96) \qbezier(78,96)(78,96)(81,91)
\qbezier(81,91)(81,91)(83,87) \qbezier(83,87)(83,87)(86,82)
\qbezier(86,82)(86,82)(88,77) \qbezier(88,77)(88,77)(90,72)
\qbezier(90,72)(90,72)(91,67) \qbezier(91,67)(91,67)(93,62)
\qbezier(93,62)(93,62)(94,57) \qbezier(94,57)(94,57)(96,52)
\qbezier(96,52)(96,52)(97,47) \qbezier(97,47)(97,47)(97,42)
\qbezier(97,42)(97,42)(98,37) \qbezier(98,37)(98,37)(99,32)
\qbezier(99,32)(99,32)(99,26) \qbezier(99,26)(99,26)(99,21)
\qbezier(99,21)(99,21)(100,16) \qbezier(100,16)(100,16)(100,11)
\qbezier(100,11)(100,11)(100,6) \qbezier(100,6)(100,6)(100,0)

\qbezier[32](75,100)(87.5,50)(100,0)

\end{picture}
\end{center}
\end{minipage}
\caption{Graph of $\widetilde{f}_v$ for $v =3/4$}\label{fig-21}
\end{figure}
\end{proof}

\begin{example}\label{ex:02}Plot the graph of
$\widetilde{f}_v$, which is defined by the commutative
diagram~(\ref{eq:13}) for the maps $h$ of the
form~(\ref{ex:h:first}), if $\omega$ is a continuous function,
whose graph on $[1/2,\, 1]$ is consisted of two intervals of
linearity.\end{example}

\begin{proof}[Deal of the example] Plug $x=3/4$ into
the commutative diagram~(\ref{eq:13}) and obtain
$$ \begin{CD}
3/4 @>f >> & 1/2\\
@V_{h} VV& @VV_{h}V\\
h(3/4) @>\widetilde{f}_v>>& v =h(1/2).
\end{CD}$$

Define $h(3/4)$ as the bigger pre image of $v$ under $f_v$. This
leads to that the values of $f_v$ and $\widetilde{f}_v$ would
coincide at this bigger pre image.

In other words$$h\left(\frac{3}{4}\right) =
\left(\frac{3}{4}\right)^{-\log_2v}\omega\left( \log_2\left(
\frac{3}{4}\right)\right) = v^2 - v +1.$$

For $v = \frac{3}{4}$ obtain $ \omega(0,584) \approx 0,915. $

Define $\omega$ on the interval $[0,\, 1]$ with the following
rule. $\omega$ is a piecewise linear maps, whose graph consists of
two intervals of linearity and has breaking point at
$$\left( \log_2\frac{3}{4},\, (v^2-v+1)\cdot
\left(\frac{3}{4}\right)^{\log_2v}\right).$$

Graph of $\widetilde{f}_v$ for $v=3/4$, which is constructed by
this $\omega$, is given at Figure~\ref{fig-22}.
\begin{figure}[htbp]
\begin{minipage}[h]{0.9\linewidth}
\begin{center}
\begin{picture}(100,120)
\put(0,0){\vector(0,1){120}} \put(0,0){\vector(1,0){120}}

\put(0,0){\line(3,4){75}} \qbezier[32](75,100)(87.5,50)(100,0)

\qbezier(75,100)(75,100)(76,95) \qbezier(76,95)(76,95)(77,90)
\qbezier(77,90)(77,90)(78,84) \qbezier(78,84)(78,84)(79,80)
\qbezier(79,80)(79,80)(81,75) \qbezier(81,75)(81,75)(83,70)
\qbezier(83,70)(83,70)(85,65) \qbezier(85,65)(85,65)(87,60)
\qbezier(87,60)(87,60)(91,56) \qbezier(91,56)(91,56)(92,52)
\qbezier(92,52)(92,52)(93,46) \qbezier(93,46)(93,46)(96,42)
\qbezier(96,42)(96,42)(96,37) \qbezier(96,37)(96,37)(98,32)
\qbezier(98,32)(98,32)(98,27) \qbezier(98,27)(98,27)(99,21)
\qbezier(99,21)(99,21)(99,17) \qbezier(99,17)(99,17)(100,11)
\qbezier(100,11)(100,11)(100,6) \qbezier(100,6)(100,6)(100,0)

\put(0,0){\line(1,1){100}}
\end{picture}
\end{center}
\end{minipage}
\caption{Graph of $\widetilde{f}_v$ for $v =3/4$}\label{fig-22}
\end{figure}
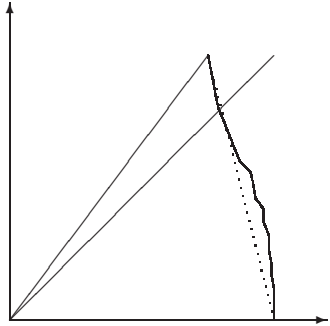
\end{proof}

Examples~\ref{ex:01} and~\ref{ex:02} can be generalized as
follows.

Consider the approximations $\omega_n$ of $\omega$ and use then
for finding approximations $\widetilde{h}_n$ of $h$, which is
given by~(\ref{ex:h:first}), as follows.

For any approximation of $h$ on $x\in \left[\frac{1}{2},\,
1\right]$, we can obtain the approximation of $\omega$ on $[-1,\,
0]$. Then periodicity of $\omega$ would give the approximation of
$h$ on the whole $[0,\, 1]$.

The maps $h_n$, which is constructed in previous sections, moves
points of $A_n$ to $B_n$. Use these $h_n$ to find $\omega$ on
$A_n\cap \left[\frac{1}{2},\, 1\right].$

Define by $\omega_n$ the maps, such that $h_n$ defines its values
at $\{ \log_2x,\, x\in A_n\cap \left[\frac{1}{2},\, 1\right]\}$.
Additionally $\omega_n$ is linear at all points of the set
$A_n\cap \left[\frac{1}{2},\, 1\right]$ and is periodical with
period 1.

Consider the maps
\begin{equation}\label{eq:30} \widetilde{h}_n =
x^{-\log_2v}\omega_n(\log_2x),\end{equation} as the approximation
of $h$.

If the constructed maps $\widetilde{h}_n$ would be invertible,
then there exists a unique maps $\widetilde{f}_n$, such that the
diagram
\begin{equation}\label{eq:27}\begin{CD}
A @>f >> & A\\
@V_{\widetilde{h}_n} VV& @VV_{\widetilde{h}_n}V\\
B @>\widetilde{f}_n>>& B
\end{CD}\end{equation} would be
commutative. This maps $\widetilde{f}_n$ can be given by formula
$$ \widetilde{f}_n = \widetilde{h}_n(f(\widetilde{h}^{-1}_n)).
$$

For instance, the introduced notations give that
$\widetilde{f}_v$, which is constructed in the
example~\ref{ex:02}, is the maps $\widetilde{f}_2$ and the
correspond $h$ is $\widetilde{h}_2$.

Nevertheless, it could happen that the maps $\widetilde{h}_n$
would be non invertible, whence there will not exist
$\widetilde{f}_n$ which would make the diagram~(\ref{eq:27})
commutative.

\begin{example}\label{ex:03}Consider the case, when
the maps $\widetilde{h}_3(x)$, (which is, in fact, dependent on
$v$), is non monotone for some $v$.
\end{example}

\begin{proof}[Deal of the example] The
Figure~\ref{fig-23} contains examples of graphs of maps
$\widetilde{h}_3$ for ${v=0.01}$, ${v=0.025}$, ${v=0.1}$ and
$v=0,2$.

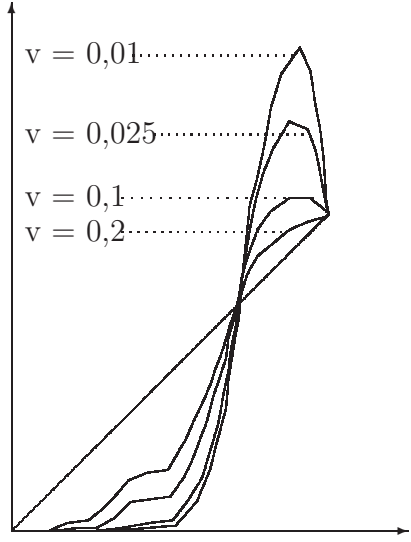
\begin{figure}[htbp]
\begin{minipage}[h]{0.9\linewidth}
\begin{center}
\begin{picture}(150,205)
\put(0,0){\vector(0,1){200}} \put(0,0){\vector(1,0){150}}

\put(5,177){v = 0,01}

\qbezier[20](46,180)(76,180)(106,180)

\put(5,147){v = 0,025}

\qbezier[20](54,150)(82,150)(110,150)

\put(5,123){v = 0,1}

\qbezier[20](42,126)(78,126)(114,126)

\put(5,110){v = 0,2}

\qbezier[20](42,113)(72,113)(102,113)

\qbezier(0,0)(0,0)(7,0) \qbezier(0,0)(0,0)(7,0)
\qbezier(0,0)(0,0)(10,0)    \qbezier(0,0)(0,0)(12,0)
\qbezier(7,0)(7,0)(14,0)    \qbezier(7,0)(7,0)(16,0)
\qbezier(10,0)(10,0)(21,0)  \qbezier(12,0)(12,0)(25,0)
\qbezier(14,0)(14,0)(21,3)  \qbezier(16,0)(16,0)(24,1)
\qbezier(21,0)(21,0)(31,0)  \qbezier(25,0)(25,0)(37,0)
\qbezier(21,3)(21,3)(29,4)  \qbezier(24,1)(24,1)(32,1)
\qbezier(31,0)(31,0)(41,1)  \qbezier(37,0)(37,0)(50,1)
\qbezier(29,4)(29,4)(34,8)  \qbezier(32,1)(32,1)(39,5)
\qbezier(41,1)(41,1)(52,3)  \qbezier(50,1)(50,1)(62,2)
\qbezier(34,8)(34,8)(39,13) \qbezier(39,5)(39,5)(45,11)
\qbezier(52,3)(52,3)(61,4)  \qbezier(62,2)(62,2)(70,11)
\qbezier(39,13)(39,13)(44,19)   \qbezier(45,11)(45,11)(53,12)
\qbezier(61,4)(61,4)(67,12) \qbezier(70,11)(70,11)(75,23)
\qbezier(44,19)(44,19)(51,22)   \qbezier(53,12)(53,12)(60,13)
\qbezier(67,12)(67,12)(72,20)   \qbezier(75,23)(75,23)(78,35)
\qbezier(51,22)(51,22)(59,23)   \qbezier(60,13)(60,13)(65,20)
\qbezier(72,20)(72,20)(75,30)   \qbezier(78,35)(78,35)(81,46)
\qbezier(59,23)(59,23)(62,28)   \qbezier(65,20)(65,20)(68,26)
\qbezier(75,30)(75,30)(78,40)   \qbezier(81,46)(81,46)(82,57)
\qbezier(62,28)(62,28)(66,35)   \qbezier(68,26)(68,26)(71,34)
\qbezier(78,40)(78,40)(80,51)   \qbezier(82,57)(82,57)(84,70)
\qbezier(66,35)(66,35)(69,42)   \qbezier(71,34)(71,34)(74,42)
\qbezier(80,51)(80,51)(82,61)   \qbezier(84,70)(84,70)(86,81)
\qbezier(69,42)(69,42)(72,48)   \qbezier(74,42)(74,42)(76,50)
\qbezier(82,61)(82,61)(83,69)   \qbezier(86,81)(86,81)(87,94)
\qbezier(72,48)(72,48)(75,55)   \qbezier(76,50)(76,50)(78,56)
\qbezier(83,69)(83,69)(85,79)   \qbezier(87,94)(87,94)(88,103)
\qbezier(75,55)(75,55)(77,61)   \qbezier(78,56)(78,56)(81,65)
\qbezier(85,79)(85,79)(86,89)   \qbezier(88,103)(88,103)(90,123)
\qbezier(77,61)(77,61)(80,69)   \qbezier(81,65)(81,65)(82,73)
\qbezier(86,89)(86,89)(88,100)  \qbezier(90,123)(90,123)(93,134)
\qbezier(80,69)(80,69)(82,76)   \qbezier(82,73)(82,73)(84,80)
\qbezier(88,100)(88,100)(89,108) \qbezier(93,134)(93,134)(96,148)
\qbezier(82,76)(82,76)(84,82) \qbezier(84,80)(84,80)(86,89)
\qbezier(89,108)(89,108)(92,124) \qbezier(96,148)(96,148)(98,160)
\qbezier(84,82)(84,82)(87,90) \qbezier(86,89)(86,89)(87,96)
\qbezier(92,124)(92,124)(95,135) \qbezier(98,160)(98,160)(102,173)
\qbezier(87,90)(87,90)(89,97) \qbezier(87,96)(87,96)(89,103)
\qbezier(95,135)(95,135)(99,145)
\qbezier(102,173)(102,173)(109,183) \qbezier(89,97)(89,97)(93,104)
\qbezier(89,103)(89,103)(93,114) \qbezier(99,145)(99,145)(105,155)
\qbezier(109,183)(109,183)(113,174)
\qbezier(93,104)(93,104)(99,109) \qbezier(93,114)(93,114)(98,121)
\qbezier(105,155)(105,155)(112,152)
\qbezier(113,174)(113,174)(115,161)
\qbezier(99,109)(99,109)(105,114)
\qbezier(98,121)(98,121)(105,126)
\qbezier(112,152)(112,152)(115,144)
\qbezier(115,161)(115,161)(117,150)
\qbezier(105,114)(105,114)(112,117)
\qbezier(105,126)(105,126)(113,126)
\qbezier(115,144)(115,144)(117,134)
\qbezier(117,150)(117,150)(118,141)
\qbezier(112,117)(112,117)(119,119)
\qbezier(113,126)(113,126)(119,121)
\qbezier(117,134)(117,134)(119,123)
\qbezier(118,141)(118,141)(119,125)
\qbezier(119,119)(119,119)(120,120)
\qbezier(119,121)(119,121)(120,120)
\qbezier(119,123)(119,123)(120,120)
\qbezier(119,125)(119,125)(120,120)
\qbezier(120,120)(120,120)(0,0) \qbezier(120,120)(120,120)(0,0)
\qbezier(120,120)(120,120)(0,0) \qbezier(120,120)(120,120)(0,0)
\end{picture}
\end{center}
\end{minipage}
\caption{Graphs of $\widetilde{h}_3$ for different values of
$v$}\label{fig-23}
\end{figure}

As a comment to the Figure~\ref{fig-23} notice, that all these
functions satisfy the functional equation~(\ref{eq:syst:line}a) $$
h(2x) = \frac{1}{v}\ h(x),
$$ i.e. for every interval of the form $\left[\frac{1}{2^{k+1}},\,
\frac{1}{2^k}\right]$ the graph of the function repeats its form
on the interval $\left[\frac{1}{2},\, 1\right]$, but it is
compressed $v^{-k}$ times.

The maps $\widetilde{h}_3(x)$, which is calculated for
$v=\frac{1}{2}$, is given by formula $\widetilde{h}_3(x)=x$.

whence, there are some critical value of $v$, such that decreasing
of $v$ after this critical value leads to that
$\widetilde{h}_3(x)$ appears to be non monotone.
\end{proof}

\begin{notation}Denote by $\widehat{h}_n(t)$ the
function such that the equality $$ \widehat{h}_n(\log_2x) =
\widetilde{h}_n(x)
$$ holds.

By numbers $\alpha_k = \alpha_{n,k}$ construct
$\widetilde{\alpha}_k = \log_2\alpha_k$.

Define by $t_k = t_{k,\, n}$ the extremum of the maps
$\widehat{h}_n$ in the interval $(\widetilde{\alpha}_k,\,
\widetilde{\alpha}_{k+1})$. It follows from the construction that
this extremum is unique. We will pay more attention to this fact
later.
\end{notation}

\begin{note}\label{note:new1}Monotonicity of
$\widetilde{h}_n$ is equivalent to that for every $k$ the
inclusion
\begin{equation} \label{eq:29}t_k \in \mathbb{R}\backslash
[\widetilde{\alpha}_k,\, \widetilde{\alpha}_{k+1}] =
\mathbb{R}\backslash [\log_2 \alpha_k,\, \log_2\alpha_{k+1}]
\end{equation} holds.\end{note}

The formula~(\ref{eq:30}) lets to write $\widehat{h}_n(t)$ in the
form
$$ \widehat{h}_n(t) = 2^{-t\log_2v}\omega_n(t).
$$ It follows from the monotonicity
of $y = \log_2x$ that the monotonicity of $\widetilde{h}_n$ is
equivalent to the monotonicity of $\widehat{h}_n$.

It follows from the equality $h_n(\alpha_k) = \beta_k$ that the
following condition for $\omega_n$ holds.
$$\omega_n(\widetilde{\alpha}_k)  =
\beta_k\cdot 2^{\widetilde{\alpha}_k\log_2v}.$$

Denote by $\widetilde{\beta}_k = \beta_k\cdot
2^{\widetilde{\alpha}_k\log_2v} = \beta_kv^{\widetilde{\alpha}_k}.
$

Let the maps $\omega_n$ have the form $$ \omega_n(t) = a_k\cdot t
+ b_k
$$ on the interval $(\widetilde{\alpha}_k,\,
\widetilde{\alpha}_{k+1})$. Then
$$
a_k = \frac{\widetilde{\beta}_{k+1}-\widetilde{\beta}_k}{
\widetilde{\alpha}_{k+1} -\widetilde{\alpha}_k},
$$ $$
b_k = \widetilde{\beta}_k -
\frac{\widetilde{\alpha}_k(\widetilde{\beta}_{k+1}-\widetilde{\beta}_k)}{
\widetilde{\alpha}_{k+1} -\widetilde{\alpha}_k} =
\frac{\widetilde{\beta}_k\widetilde{\alpha}_{k+1} -
\widetilde{\beta}_{k+1}\widetilde{\alpha}_k}{
\widetilde{\alpha}_{k+1} -\widetilde{\alpha}_k}.
$$

Find the extremum of $\widehat{h}_n(t)$ on
$(\widetilde{\alpha}_k,\, \widetilde{\alpha}_{k+1})$. We have that
$$ \widehat{h}_n'(t) = 2^{-t\log_2v}(a_k -\log_2v\ln 2\cdot (a_kt
+b_k)),
$$
whence the necessary extremum of $\widehat{h}_n(t)$ can be found
as

\begin{equation}\label{eq:28}t_k = \frac{a_k -
b_k\log_2v \ln 2}{a_k\log_2v \ln 2} = \frac{1}{\ln v} -
\frac{b_k}{a_k}.
\end{equation}

It follows from the previous calculations that$$ \frac{b_k}{a_k} =
\frac{\widetilde{\beta}_k\widetilde{\alpha}_{k+1} -
\widetilde{\beta}_{k+1}\widetilde{\alpha}_k}
{\widetilde{\beta}_{k+1}-\widetilde{\beta}_k}.
$$

After coming back to former notations obtain
$$ \frac{b_k}{a_k} = \frac{\beta_kv^{\log_2
\alpha_k}\log_2\alpha_{k+1} - \beta_{k+1}v^{\log_2
\alpha_{k+1}}\log_2\alpha_{k}} {\beta_{k+1}v^{\log_2
\alpha_{k+1}}-\beta_kv^{\log_2 \alpha_k}}
$$

Since $\alpha(k,\, n) = \frac{k}{2^n}$ and $A_n = \{ \alpha(k,\,
n-1)\}$ we have that $\alpha_k = \frac{k}{2^{n-1}}$, whence
$$\log_2\alpha_k = \log_2k - n+1.$$ This means that
$$ \frac{b_k}{a_k} = \frac{\beta_kv^{\log_2k -
n+1}\log_2\alpha_{k+1}} {\beta_{k+1}v^{\log_2(k+1) -
n+1}-\beta_kv^{\log_2k - n+1}} - $$ $$ -
\frac{\beta_{k+1}v^{\log_2(k+1) - n+1}\log_2\alpha_{k}}
{\beta_{k+1}v^{\log_2(k+1) - n+1}-\beta_kv^{\log_2k - n+1}} =
$$ $$ =
\frac{\beta_kv^{\log_2k}\log_2\alpha_{k+1}}
{\beta_{k+1}v^{\log_2(k+1)}-\beta_kv^{\log_2k}} - $$
$$ - \frac{\beta_{k+1}v^{\log_2(k+1)}\log_2\alpha_{k}}
{\beta_{k+1}v^{\log_2(k+1)}-\beta_kv^{\log_2k}} =
$$ $$ =
\frac{\beta_kv^{\log_2k}\left( \log_2(k+1) -n+1\right)}
{\beta_{k+1}v^{\log_2(k+1)}-\beta_kv^{\log_2k}}\ -
$$
$$ -\ \frac{\beta_{k+1}v^{\log_2(k+1)}\left(\log_2k -n+1\right)}
{\beta_{k+1}v^{\log_2(k+1)}-\beta_kv^{\log_2k}} =
$$ $$
=n-1 +\ \frac{\beta_kv^{\log_2k}\log_2(k+1)}
{\beta_{k+1}v^{\log_2(k+1)}-\beta_kv^{\log_2k}} - $$ $$-\
\frac{\beta_{k+1}v^{\log_2(k+1)}\log_2k}
{\beta_{k+1}v^{\log_2(k+1)}-\beta_kv^{\log_2k}}.
$$

Example~\ref{ex:03} leads to the assumption about non monotonicity
of $h_n$ for $v\rightarrow 0$ which also stays for big $n$. The
use of formula~(\ref{eq:28}) lets to prove the following
proposition.

\begin{proposition}\label{lema:12} For
every $n\in \mathbb{N}$ there exists $v_0 \in (0,\, 1)$ such that
for every $v\in (0,\, v_0)$ the maps $\widetilde{h}_n(x)$ is non
monotonic on $\left[ \frac{2^{n-1}-1}{2^{n-1}},\, 1\right]$.
\end{proposition}

\begin{proof}Find explicit formulas for $\beta(n,\,
2^{n-1}-1)$. We know that
$$\left\{ \begin{array}{l} f_v^n(\beta(n,\,
2^{n-1}-1))=0;\\ f_v^{n-1}(\beta(n,\, 2^{n-1}-1))=1.
\end{array}\right.$$

The tangent of $f_v^{n-1}$ at the last interval of its
monotonicity is $\frac{1}{v-1}\cdot \frac{1}{v^{n-2}}$, which mens
that $$ \beta(n,\, 2^{n-1}-1) = 1 + v^{n-2}\cdot(v-1).
$$

To use Formula~(\ref{eq:28}) for finding $t_k$, we should consider
$k = 2^{n-1}-1$, $\beta_k = 1+v^{n-2}\cdot(v-1)$, $\beta_{k+1}
=1$, $\alpha_k = 1 - \frac{1}{2^{n-1}}$, $\alpha_{k+1} =1$. Thus
$\log_2\alpha_{k+1}=0$, whence
$$ t_k = \frac{1}{\ln v} - $$ $$-
\frac{\beta_kv^{\log_2 \alpha_k}\log_2\alpha_{k+1} -
\beta_{k+1}v^{\log_2 \alpha_{k+1}}\log_2\alpha_{k}}
{\beta_{k+1}v^{\log_2 \alpha_{k+1}}-\beta_kv^{\log_2 \alpha_k}} =
$$$$ = \frac{1}{\ln v} +
\frac{\log_2\alpha_{k}} {1-(1+v^{n-2}\cdot(v-1))v^{\log_2
\alpha_k}}.
$$

Since $v\in (0,\, 1)$, then $\frac{1}{\ln v}<0$ and $\frac{1}{\ln
v}\rightarrow 0$ if $v\rightarrow 0$.

Since $\log_2\alpha_k<0$, then numerator of the second fraction of
negative. But the same reason mean that
$v^{\log_2\alpha_k}\rightarrow \infty$ for $v\rightarrow 0$,
whence denominator of the second fraction is also negative, whence
second fraction is positive, but tends to 0 if $v\rightarrow 0$,
whence $t_k\rightarrow 0$ for $v\rightarrow 0$.

By Remark~\ref{note:new1}, to prove the Proposition we need to
decide, whether or not the inclusion $t\in
(\widetilde{\alpha}_k,\, \widetilde{\alpha}_{k+1})$ holds. Since
$\widetilde{\alpha}_k = \log_2\left(1-\frac{1}{2^{n-1}}\right)<0$
and $\widetilde{\alpha}_{k+1} = \log_21 = 0$, then we have to
understand which the sign of $t_k$ is for $v\approx 0$ and whether
this sign is fixed, or $t_k$ waves about $0$, changing its signs,
dependently in $v$.

Consider the limit $$ \lim\limits_{v\rightarrow 0}
\frac{\log_2\alpha_{k}} {1-(1+v^{n-2}\cdot(v-1))v^{\log_2
\alpha_k}}:\, \frac{1}{\ln v} =
$$$$
=\lim\limits_{v\rightarrow 0} \frac{-\log_2\alpha_{k}\cdot\,
1/v}{\log_2\alpha_k v^{\log_2\alpha_k-1} +
(n-1+\log_2\alpha_k)v^{n-2+\log_2\alpha_k}-
(n-2+\log_2\alpha_k)v^{n-3+\log_2\alpha_k}}=
$$$$
=\lim\limits_{v\rightarrow 0} \frac{-\log_2\alpha_{k}}
{\log_2\alpha_k v^{\log_2\alpha_k} +
(n-1+\log_2\alpha_k)v^{n-1+\log_2\alpha_k}-
(n-2+\log_2\alpha_k)v^{n-2+\log_2\alpha_k}} =
$$$$
=\lim\limits_{v\rightarrow \infty} \frac{-\log_2\alpha_{k}}
{\log_2\alpha_k v^{-\log_2\alpha_k} +
(n-1+\log_2\alpha_k)v^{-n+1-\log_2\alpha_k}-
(n-2+\log_2\alpha_k)v^{-n+2-\log_2\alpha_k}} \sim
$$$$
\sim \lim\limits_{v\rightarrow \infty} \frac{-\log_2\alpha_{k}}
{\log_2\alpha_k v^{-\log_2\alpha_k}} = \lim\limits_{v\rightarrow
\infty} -v^{\log_2\alpha_k} =0.
$$

This means that $t_k<0$ for $v\approx 0$. Now Proposition follows
from Remark~\ref{note:new1}.
\end{proof}

We can use Formula~(\ref{eq:28}) for studying the
Example~\ref{ex:03} more carefully, precisely study the numbers
$t_2$ and $t_3$ dependently on $v$ for $n=3$.

\begin{example}Consider $t_2(v)$ for ${n=3}$. We
will show an unsuccessful attempt of studying of the limit
$\lim\limits_{v\rightarrow 1} t_2(v)$. Precisely, we will show
that this limit exists and equals 1 in the time, when numerical
methods shows like the limit does not exists.
\end{example}

\begin{proof}For $n=3$ consider the division of the
interval $\left[\frac{1}{2},\, 1\right]$ for two intervals with
end points $\alpha_2 = \frac{1}{2}$, $\alpha_3 =\frac{3}{4}$ and
$\alpha_4 = 1$. Correspondingly, $\beta_2 = v$, $\beta_3 = \max
f_v^{-1}(v) = 1-v(1-v)$ and $\beta_4 =1$.

Then$$ t_2 = \frac{1}{\ln v} -2 -
\frac{\beta_kv^{\log_2k}\log_2(k+1)}
{\beta_{k+1}v^{\log_2(k+1)}-\beta_kv^{\log_2k}} + $$ $$+\
\frac{\beta_{k+1}v^{\log_2(k+1)}\log_2k}
{\beta_{k+1}v^{\log_2(k+1)}-\beta_kv^{\log_2k}} =
$$$$
= \frac{1}{\ln v} -2 + \frac{(v^2-v+1)v^{\log_23} -v^2\log_23}
{(v^2-v+1)v^{\log_23}-v^2} = $$ $$ = \frac{1}{\ln v} -1 +
\frac{v^2 -v^2\log_23}{(v^2-v+1)v^{\log_23}-v^2}$$
$$
t_2 +1 = \frac{1}{\ln v} + \frac{v^2
-v^2\log_23}{(v^2-v+1)v^{\log_23}-v^2}$$

It is necessary and sufficient for being $\widetilde{h}_2$
continuous on $[\alpha_2,\, \alpha_3]$ that the following
inclusion hold.
$$t_2\in \mathbb{R}\backslash [-1,\,
\log_23-2] \approx \mathbb{R}\backslash [-1,\, -0.415].$$

The numerical calculations about $t_2(v)$ let make the conclusion
that it decrease on $v\in [0,\, 0.5]$ and has the asymptote
$v=0.5$, i.e. $\lim\limits_{v\rightarrow 0.5-}t_2 = -\infty$ in
the time, when $\lim\limits_{v\rightarrow 0}t_2 = -1$.

If plot the graph of the maps $t_2(v)$ for $v\in [0,5,\, 1]$ on
the interval $[0,99999,\, 1]$ then see that it is non
non-monotone. Thus, for $v\in [0.6,\, 0.999999]$ the graph of
$t_2(v)$ is given on Figure~\ref{fig-24}.

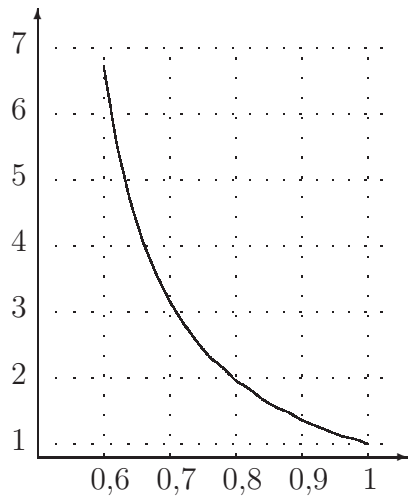
\begin{figure}[htbp]
\begin{minipage}[h]{0.9\linewidth}
\begin{center}
\begin{picture}(150,205)
\put(0,15){\vector(0,1){170}} \put(0,15){\vector(1,0){140}}

\qbezier(25,163)(25,163)(30,133) \qbezier(30,133)(30,133)(35,112)
\qbezier(35,112)(35,112)(40,96) \qbezier(40,96)(40,96)(45,84)
\qbezier(45,84)(45,84)(50,74) \qbezier(50,74)(50,74)(55,66)
\qbezier(55,66)(55,66)(60,59) \qbezier(60,59)(60,59)(65,53)
\qbezier(65,53)(65,53)(70,49) \qbezier(70,49)(70,49)(75,44)
\qbezier(75,44)(75,44)(80,41) \qbezier(80,41)(80,41)(85,37)
\qbezier(85,37)(85,37)(90,34) \qbezier(90,34)(90,34)(95,32)
\qbezier(95,32)(95,32)(100,29) \qbezier(100,29)(100,29)(105,27)
\qbezier(105,27)(105,27)(110,25) \qbezier(110,25)(110,25)(115,23)
\qbezier(115,23)(115,23)(120,22) \qbezier(120,22)(120,22)(125,20)

\qbezier[20](0,20)(70,20)(130,20) \put(-10,18){1}
\qbezier[20](0,45)(70,45)(130,45) \put(-10,43){2}
\qbezier[20](0,70)(70,70)(130,70) \put(-10,68){3}
\qbezier[20](0,95)(70,95)(130,95) \put(-10,93){4}
\qbezier[20](0,120)(70,120)(130,120) \put(-10,118){5}
\qbezier[20](0,145)(70,145)(130,145) \put(-10,143){6}
\qbezier[20](0,170)(70,170)(130,170) \put(-10,168){7}

\qbezier[20](25,16)(25,91)(25,166) \put(20,3){0,6}
\qbezier[20](50,16)(50,91)(50,166) \put(45,3){0,7}
\qbezier[20](75,16)(75,91)(75,166) \put(70,3){0,8}
\qbezier[20](100,16)(100,91)(100,166) \put(95,3){0,9}
\qbezier[20](125,16)(125,91)(125,166) \put(123,3){1}

\end{picture}
\end{center}
\end{minipage}
\caption{Graph of $t_2$ for $v\in [0.6,\, 0.999999]$
}\label{fig-24}
\end{figure}

In the same time the graph of $t_2(v)$ for $v\in M=[0.99999999,\,
0.999999995]$ is given at Figure~\ref{fig-25}. Precisely, this
plot is not a graph of $t_2(v)$ but its values at correspond
points (looking like braking points), which divide $M$ into equal
parts. These points are just connected by line segments to obtain
a graph.

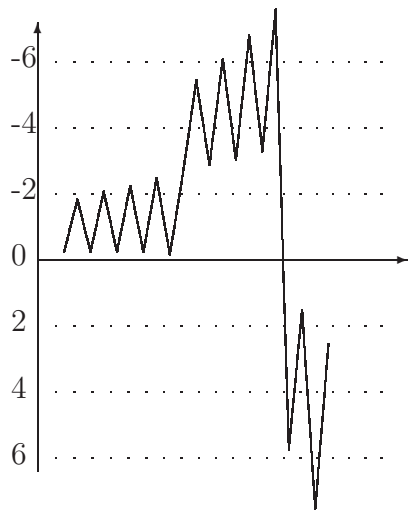
\begin{figure}[htbp]
\begin{minipage}[h]{0.9\linewidth}
\begin{center}
\begin{picture}(150,205)
\put(0,15){\vector(0,1){170}} \put(0,95){\vector(1,0){140}}

\qbezier[20](0,20)(70,20)(130,20) \put(-10,18){6}
\qbezier[20](0,45)(70,45)(130,45) \put(-10,43){4}
\qbezier[20](0,70)(70,70)(130,70) \put(-10,68){2} \put(-10,93){0}
\qbezier[20](0,120)(70,120)(130,120) \put(-10,118){-2}
\qbezier[20](0,145)(70,145)(130,145) \put(-10,143){-4}
\qbezier[20](0,170)(70,170)(130,170) \put(-10,168){-6}

\qbezier(10,98)(10,98)(15,118) \qbezier(15,118)(15,118)(20,98)
\qbezier(20,98)(20,98)(25,121) \qbezier(25,121)(25,121)(30,98)
\qbezier(30,98)(30,98)(35,123) \qbezier(35,123)(35,123)(40,98)
\qbezier(40,98)(40,98)(45,126) \qbezier(45,126)(45,126)(50,97)
\qbezier(50,97)(50,97)(55,128) \qbezier(55,128)(55,128)(60,163)
\qbezier(60,163)(60,163)(65,131) \qbezier(65,131)(65,131)(70,171)
\qbezier(70,171)(70,171)(75,133) \qbezier(75,133)(75,133)(80,180)
\qbezier(80,180)(80,180)(85,136) \qbezier(85,136)(85,136)(90,190)
\qbezier(90,190)(90,190)(95,23) \qbezier(95,23)(95,23)(100,76)
\qbezier(100,76)(100,76)(105,1) \qbezier(105,1)(105,1)(110,63)

\end{picture}
\end{center}
\end{minipage}
\caption{Graph of $t_2$ for $v\in M$}\label{fig-25}
\end{figure}

Notice, that in spite of the evident continuity of $t_2(v)$ we can
see on the graph the points of discontinuity of $\widehat{h}_3$ on
$[\alpha_2,\, \alpha_3]$.

Further experiments show that the maps $t_2(v)$ a point of its
discontinuity on $[\alpha_2,\, \alpha_3]$ on each of the intervals
of the form $\left[ 1-\frac{1}{10^{m+1}},\,
1-\frac{1}{10^{m}}\right]$.

In the same time these experiments does not make any influence to
the mathematical nature of the function under the consideration,
because the limit $$ \lim\limits_{v \rightarrow 1}\left(
\frac{1}{\ln v} + \frac{v^2(1-\log_23)}{(v^2-v+1)v^{\log_23}
-v^2}\right)
$$ is a limit of the difference of expressions,
such that each of them tends to $+\infty$, whence the numerical
investigation of this limit can lead to the mistake and is that,
what has happened in our case.

Since $$ \lim\limits_{v\rightarrow \infty}t_2(v)+1 = \frac{1}{0} +
\frac{0}{0} = \frac{0}{0},
$$ we can write this expression with as a proper fraction and
apply the L'hopitales rule, obtaining as follows
$$ t_2 +1
=
$$ $$= \frac{(v^2-v+1)v^{\log_23}-v^2+ \ln v(v^2 -v^2\log_23)}{\ln
v((v^2-v+1)v^{\log_23}-v^2)}.
$$

Denote by $$ s(v)= (v^2-v+1)v^{\log_23}-v^2+ \ln v(v^2
-v^2\log_23)
$$ and $$
p(v) = \ln v((v^2-v+1)v^{\log_23}-v^2).
$$

Then $ s'(v) = (2v-1)v^{\log_23}
+\log_23(v^2-v\hm{+}1)v^{\log_23-1} - 2v + (v-v\log_23) \hm{+}
2(v-v\log_23)\ln v, $ whence $s'(1) = 0$. Also $p'(v) =
(v-1)v^{\log_23} + v^{\log_23-1}-v \hm{+} ((2v-1)v^{\log_23} +
\log_23(v^2-v+1)v^{\log_23-1} -2v)\ln v$, whence $p'(1) = 0$

The second derivative $s''(v)$ and $p''(v)$ are as follows

$s''(v) = 2v^{\log_23} +(2v-1)v^{\log_23-1}\log_23
+\log_23(\log_23-1)(v^2-v+1)v^{\log_23-2}
+\log_23(2v-1)v^{\log_23-1} -2 +(1-\log_23) +2(1-\log_23)
+2(1-\log_23)\ln v$,whence$s''(1) = 2 +\log_23 +\log_23(\log_23-1)
+\log_23 -2 +(1-\log_23) +2(1-\log_23) = 3 +\log_23(\log_23-1)
-\log_23 = 3 +(\log_23)^2 -2\log_23  \approx 2,3422$.

$p''(v) = v^{\log_23} +\log_23(v-1)v^{\log_23-1} +
(\log_23-1)v^{\log_23-2} -1 + ((2v-1)v^{\log_23-1}
+\log_23(v^2-v+1)v^{\log_23-2} -2) + (\log_23(2v-1)v^{\log_23-1}
+2v^{\log_23} +(2v-1)v^{\log_23-1}
+\log_23(\log_23-1)(v^2-v+1)v^{\log_23-2} -2)\ln v$,whence$p''(1)
= 1 +0 + (\log_23-1) -1 + (1 +\log_23 -2) = 2\log_23 -2 \approx
1,1699$.

That is why, $$ t_2(v) \rightarrow \frac{s''(1)}{p''(1)} -1
\approx 1,0020.$$
\end{proof}

\begin{example}Consider the function $t_3(v)$ for
${n=3}$. We will show, that $\widehat{h}_3(x)$ becomes to be non
monotone for all $v<v_0$, where $v_0 = 0,18867\pm 0,00001$. This
result is concordant with Example~\ref{ex:03}.
\end{example}

\begin{proof}[Deal of the example] Consider the
monotonicity of $\widehat{h}_3(x)$ on the interval $[\alpha_3,\,
\alpha_4]$ for $n=3$. In this case
$$ t_3 = \frac{1}{\ln v} -n+1 -\
\frac{\beta_kv^{\log_2k}\log_2(k+1)}
{\beta_{k+1}v^{\log_2(k+1)}-\beta_kv^{\log_2k}} + $$ $$+\
\frac{\beta_{k+1}v^{\log_2(k+1)}\log_2k}
{\beta_{k+1}v^{\log_2(k+1)}-\beta_kv^{\log_2k}} =
$$$$
= \frac{1}{\ln v} -2 -\ \frac{2(v^2-v+1)v^{\log_23}}
{v^{2}-(v^2-v+1)v^{\log_23}} + $$ $$+\
\frac{v^{2}\log_23}{v^{2}-(v^2-v+1)v^{\log_23}} =
$$$$
= \frac{1}{\ln v} +\ \frac{v^{2}\log_23
-2v^2}{v^{2}-(v^2-v+1)v^{\log_23}}.
$$

The inclusion $$ t_3 \in [\log_23-2,\, 0] \approx [-0.415,\, 0]
$$ is necessary for violating of the monotonicity of
$\widetilde{h}_3$.

The graph of $t_3(v)$ for $v\in [0,\, 0,21]$ is given at
Figure~\ref{fig-26}.

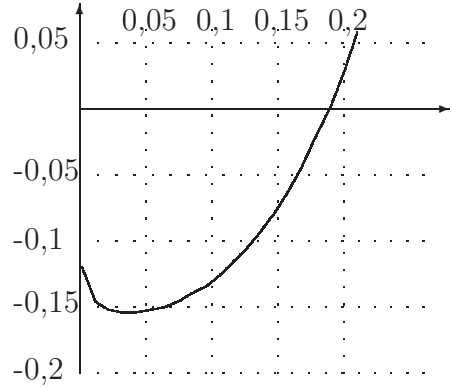
\begin{figure}[htbp]
\begin{minipage}[h]{0.9\linewidth}
\begin{center}
\begin{picture}(150,140)
\put(0,0){\vector(0,1){140}} \put(0,100){\vector(1,0){140}}

\qbezier[20](0,0)(70,0)(130,0) \put(-25,-2){-0,2}
\qbezier[20](0,25)(70,25)(130,25) \put(-25,23){-0,15}
\qbezier[20](0,50)(70,50)(130,50) \put(-25,47){-0,1}
\qbezier[20](0,75)(70,75)(130,75) \put(-25,73){-0,05}
\qbezier[20](0,125)(70,125)(130,125) \put(-25,123){0,05}

\qbezier(1,40)(1,40)(6,27) \qbezier(6,27)(6,27)(11,24)
\qbezier(11,24)(11,24)(16,23) \qbezier(16,23)(16,23)(21,23)
\qbezier(21,23)(21,23)(27,24) \qbezier(27,24)(27,24)(32,25)
\qbezier(32,25)(32,25)(37,27) \qbezier(37,27)(37,27)(42,30)
\qbezier(42,30)(42,30)(48,33) \qbezier(48,33)(48,33)(53,37)
\qbezier(53,37)(53,37)(58,42) \qbezier(58,42)(58,42)(63,47)
\qbezier(63,47)(63,47)(68,53) \qbezier(68,53)(68,53)(74,61)
\qbezier(74,61)(74,61)(79,69) \qbezier(79,69)(79,69)(84,78)
\qbezier(84,78)(84,78)(89,89) \qbezier(89,89)(89,89)(95,101)
\qbezier(95,101)(95,101)(100,114)
\qbezier(100,114)(100,114)(105,129)

\qbezier[20](25,0)(25,65)(25,130) \put(16,130){0,05}
\qbezier[20](50,0)(50,65)(50,130) \put(44,130){0,1}
\qbezier[20](75,0)(75,65)(75,130) \put(66,130){0,15 }
\qbezier[20](100,0)(100,65)(100,130) \put(94,130){0,2}

\end{picture}
\end{center}
\end{minipage}
\caption{Graph of $t_3$ for $v\in [0,\, 0,21]$.}\label{fig-26}
\end{figure}

As it is shown in Example~\ref{ex:03} there is a point near
$v_0\approx 0,2$, where $\widetilde{h}_3$ is non-monotone for all
$\widetilde{h}_3$. Another calculations show that this value
belongs to the interval
$$v_0\in [0.18868,\, 0.18869],$$ i.e. $$ v_0 \approx 0,18867 \pm
0,00001.
$$

As about the graph of $t_3(v)$ on $[0.2,\, 0.5]$, it has asymptote
at the point $0.5$ to positive infinity.
\end{proof}

Do the similar calculation for $n=4$. For this reason for every
$v$ and for each of intervals $[\frac{1}{2},\, \frac{5}{8}]$,
$[\frac{5}{8},\, \frac{3}{4}]$, $[\frac{3}{4},\, \frac{7}{8}]$ and
$[\frac{7}{8},\, 1]$ find the value of $t_k$ by
Formula~(\ref{eq:28}).

For different $v$ and every natural $k\in [4,\, 7]$ find $t_k(v)$
by obtaining $4$ points. These points, just for visibility,
connect by line segments, i.e. for every of these $v$ it
corresponds one line segments. Also add to the Figure (and plot by
bold), two curves, one of which connects numbers $\{
\widetilde{\alpha}_k\}$ and $\{ \widetilde{\alpha}_{k+1}\}$ for
the same $k\in [4,\, 7]$. Taking into account the previous
explanations, the maps $\widehat{h}_{n,\, v}$ would be monotone if
and only if each of the obtained four points would not appear
between the bold corves, i.e. appears either higher then the
higher, of lower then lowest. This construction is given at
Figure~\ref{fig-27}.

\begin{figure}[htbp]
\begin{minipage}[h]{0.9\linewidth}
\begin{center}
\begin{picture}(200,140)
\put(0,0){\vector(0,1){140}}

\linethickness{0.7mm} \qbezier(45,30)(45,30)(95,62)
\qbezier(45,62)(45,62)(95,88) \qbezier(95,62)(95,62)(145,88)
\qbezier(95,88)(95,88)(145,111) \qbezier(145,88)(145,88)(195,111)
\qbezier(145,111)(145,111)(195,130) \linethickness{0.1mm}
\qbezier(45,34)(45,34)(95,28)\put(45,34){\circle*{4}}
\qbezier(45,60)(45,60)(95,46)\put(45,60){\circle*{4}}
\qbezier(45,75)(45,75)(95,59)\put(45,75){\circle*{4}}
\qbezier(45,96)(45,96)(95,80)\put(45,96){\circle*{4}}
\qbezier(45,107)(45,107)(95,97)\put(45,107){\circle*{4}}
\qbezier(45,23)(45,23)(95,21)\put(45,23){\circle*{4}}
\qbezier(45,115)(45,115)(95,113)\put(45,115){\circle*{4}}
\qbezier(45,121)(45,121)(95,121)\put(45,121){\circle*{4}}
\qbezier(95,28)(95,28)(145,51)\put(95,28){\circle*{4}}
\qbezier(95,46)(95,46)(145,67)\put(95,46){\circle*{4}}
\qbezier(95,59)(95,59)(145,78)\put(95,59){\circle*{4}}
\qbezier(95,80)(95,80)(145,94)\put(95,80){\circle*{4}}
\qbezier(95,97)(95,97)(145,105)\put(95,97){\circle*{4}}
\qbezier(95,21)(95,21)(145,45)\put(95,21){\circle*{4}}
\qbezier(95,113)(95,113)(145,115)\put(95,113){\circle*{4}}
\qbezier(95,121)(95,121)(145,121)\put(95,121){\circle*{4}}
\qbezier(145,51)(145,51)(195,62)\put(145,51){\circle*{4}}\put(195,62){\circle*{4}}
\qbezier(145,67)(145,67)(195,77)\put(145,67){\circle*{4}}\put(195,77){\circle*{4}}
\qbezier(145,78)(145,78)(195,85)\put(145,78){\circle*{4}}\put(195,85){\circle*{4}}
\qbezier(145,94)(145,94)(195,98)\put(145,94){\circle*{4}}\put(195,98){\circle*{4}}
\qbezier(145,105)(145,105)(195,107)\put(145,105){\circle*{4}}\put(195,107){\circle*{4}}
\qbezier(145,45)(145,45)(195,57)\put(145,45){\circle*{4}}\put(195,57){\circle*{4}}
\qbezier(145,115)(145,115)(195,115)\put(145,115){\circle*{4}}\put(195,115){\circle*{4}}
\qbezier(145,121)(145,121)(195,121)\put(145,121){\circle*{4}}\put(195,121){\circle*{4}}

\end{picture}
\end{center}
\end{minipage}
\caption{Graphs of $t_k(v)$ for $k\in [4,\, 7]$ and
$n=3$}\label{fig-27}
\end{figure}
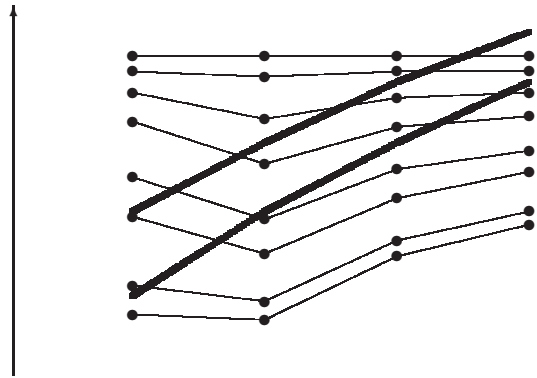

This Figure contains these curves for $v_1=0,15$, $v_2=0,1$,
$v_3=0,07$, $v_4=0,03$, $v_5=0,17$, $v_6=0,01$, $v_7=0,001$ and
$v_8=0,00001$. The curves, which are top at Figure, correspond to
lower values of $v$.

It is seen from the Figure (and it corresponds to
Proposition~\ref{lema:12}), that for $v$ small enough, curves
becomes close to the horizontal line $y=1$ and the last point of
each of them (that, which corresponds to the interval
$[\alpha_7,\, 1]$) tends to 1 for $v\rightarrow 0$ and that is
why, belongs to the interval $(\widetilde{\alpha}_7,\,
\widetilde{\alpha}_8) = (\widetilde{\alpha}_7,\, 0)$. This means
non monotonicity of the maps $\widetilde{h}_n$ on the interval
$(\alpha_7,\, \alpha_8) = (\alpha_7,\, 1)$.

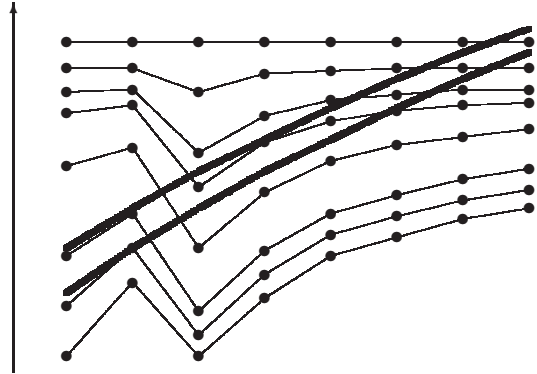
\begin{figure}[htbp]
\begin{minipage}[h]{0.9\linewidth}
\begin{center}
\begin{picture}(200,140)
\put(0,0){\vector(0,1){140}}

\linethickness{0.7mm} \qbezier(20,30)(20,30)(45,47)
\qbezier(20,47)(20,47)(45,62) \qbezier(45,47)(45,47)(70,62)
\qbezier(45,62)(45,62)(70,76) \qbezier(70,62)(70,62)(95,76)
\qbezier(70,76)(70,76)(95,88) \qbezier(95,76)(95,76)(120,88)
\qbezier(95,88)(95,88)(120,100) \qbezier(120,88)(120,88)(145,100)
\qbezier(120,100)(120,100)(145,111)
\qbezier(145,100)(145,100)(170,111)
\qbezier(145,111)(145,111)(170,121)
\qbezier(170,111)(170,111)(195,121)
\qbezier(170,121)(170,121)(195,130)
   \linethickness{0.1mm}

\qbezier(20,25)(20,25)(45,47)\put(20,25){\circle*{4}}
\qbezier(20,44)(20,44)(45,60)\put(20,44){\circle*{4}}
\qbezier(20,78)(20,78)(45,85)\put(20,78){\circle*{4}}
\qbezier(20,98)(20,98)(45,101)\put(20,98){\circle*{4}}
\qbezier(20,106)(20,106)(45,107)\put(20,106){\circle*{4}}
\qbezier(20,115)(20,115)(45,115)\put(20,115){\circle*{4}}
\qbezier(20,125)(20,125)(45,125)\put(20,125){\circle*{4}}
\qbezier(20,6)(20,6)(45,34)\put(20,6){\circle*{4}}
\qbezier(45,47)(45,47)(70,14)\put(45,47){\circle*{4}}
\qbezier(45,60)(45,60)(70,23)\put(45,60){\circle*{4}}
\qbezier(45,85)(45,85)(70,47)\put(45,85){\circle*{4}}
\qbezier(45,101)(45,101)(70,70)\put(45,101){\circle*{4}}
\qbezier(45,107)(45,107)(70,83)\put(45,107){\circle*{4}}
\qbezier(45,115)(45,115)(70,106)\put(45,115){\circle*{4}}
\qbezier(45,125)(45,125)(70,125)\put(45,125){\circle*{4}}
\qbezier(45,34)(45,34)(70,6)\put(45,34){\circle*{4}}
\qbezier(70,14)(70,14)(95,37)\put(70,14){\circle*{4}}
\qbezier(70,23)(70,23)(95,46)\put(70,23){\circle*{4}}
\qbezier(70,47)(70,47)(95,68)\put(70,47){\circle*{4}}
\qbezier(70,70)(70,70)(95,87)\put(70,70){\circle*{4}}
\qbezier(70,83)(70,83)(95,97)\put(70,83){\circle*{4}}
\qbezier(70,106)(70,106)(95,113)\put(70,106){\circle*{4}}
\qbezier(70,125)(70,125)(95,125)\put(70,125){\circle*{4}}
\qbezier(70,6)(70,6)(95,28)\put(70,6){\circle*{4}}
\qbezier(95,37)(95,37)(120,52)\put(95,37){\circle*{4}}
\qbezier(95,46)(95,46)(120,60)\put(95,46){\circle*{4}}
\qbezier(95,68)(95,68)(120,80)\put(95,68){\circle*{4}}
\qbezier(95,87)(95,87)(120,95)\put(95,87){\circle*{4}}
\qbezier(95,97)(95,97)(120,103)\put(95,97){\circle*{4}}
\qbezier(95,113)(95,113)(120,114)\put(95,113){\circle*{4}}
\qbezier(95,125)(95,125)(120,125)\put(95,125){\circle*{4}}
\qbezier(95,28)(95,28)(120,44)\put(95,28){\circle*{4}}
\qbezier(120,52)(120,52)(145,59)\put(120,52){\circle*{4}}
\qbezier(120,60)(120,60)(145,67)\put(120,60){\circle*{4}}
\qbezier(120,80)(120,80)(145,86)\put(120,80){\circle*{4}}
\qbezier(120,95)(120,95)(145,99)\put(120,95){\circle*{4}}
\qbezier(120,103)(120,103)(145,105)\put(120,103){\circle*{4}}
\qbezier(120,114)(120,114)(145,115)\put(120,114){\circle*{4}}
\qbezier(120,125)(120,125)(145,125)\put(120,125){\circle*{4}}
\qbezier(120,44)(120,44)(145,51)\put(120,44){\circle*{4}}
\qbezier(145,59)(145,59)(170,65)\put(145,59){\circle*{4}}
\qbezier(145,67)(145,67)(170,73)\put(145,67){\circle*{4}}
\qbezier(145,86)(145,86)(170,89)\put(145,86){\circle*{4}}
\qbezier(145,99)(145,99)(170,101)\put(145,99){\circle*{4}}
\qbezier(145,105)(145,105)(170,107)\put(145,105){\circle*{4}}
\qbezier(145,115)(145,115)(170,115)\put(145,115){\circle*{4}}
\qbezier(145,125)(145,125)(170,125)\put(145,125){\circle*{4}}
\qbezier(145,51)(145,51)(170,58)\put(145,51){\circle*{4}}
\qbezier(170,65)(170,65)(195,69)\put(170,65){\circle*{4}}\put(195,69){\circle*{4}}
\qbezier(170,73)(170,73)(195,77)\put(170,73){\circle*{4}}\put(195,77){\circle*{4}}
\qbezier(170,89)(170,89)(195,92)\put(170,89){\circle*{4}}\put(195,92){\circle*{4}}
\qbezier(170,101)(170,101)(195,102)\put(170,101){\circle*{4}}\put(195,102){\circle*{4}}
\qbezier(170,107)(170,107)(195,107)\put(170,107){\circle*{4}}\put(195,107){\circle*{4}}
\qbezier(170,115)(170,115)(195,115)\put(170,115){\circle*{4}}\put(195,115){\circle*{4}}
\qbezier(170,125)(170,125)(195,125)\put(170,125){\circle*{4}}\put(195,125){\circle*{4}}
\qbezier(170,58)(170,58)(195,62)\put(170,58){\circle*{4}}\put(195,62){\circle*{4}}
\end{picture}
\end{center}
\end{minipage}
\caption{Graphs of $t_k(v)$ for $k\in [4,\, 7]$ and
$n=5$}\label{fig-28}
\end{figure}

The Figure~\ref{fig-28} contains the analogical calculations for
$n=5$ and values $v_1=0.125$, $v_2=0.1$, $v_3=0.15$, $v_4=0.05$,
$v_5=0.02$, $v_6=0.01$, $v_7=0.001$ and $v_8=0.000000001$.

Now, in the similar manner to previous, consider the solution $h$,
which is defines by~(\ref{ex:h-second}) and is obtained from the
equation~(\ref{eq:syst:line}b).

Consider the function $h$ for $$ x\in \left[\frac{35}{48},\,
\frac{11}{12}\right].
$$ When the value of $x$ runs through this interval,
then the value of the function $\log_2\left(x -
\frac{2}{3}\right)$ runs through the interval $[-4,\, -2]$, which
is of the length 2, which would let to define the function
$\omega^+$ on the whole interval.

The graph of the maps $\omega(x)$ for $x\in [0,\, 2]$ and $v=3/4$
is given on Figure~\ref{fig:18}.

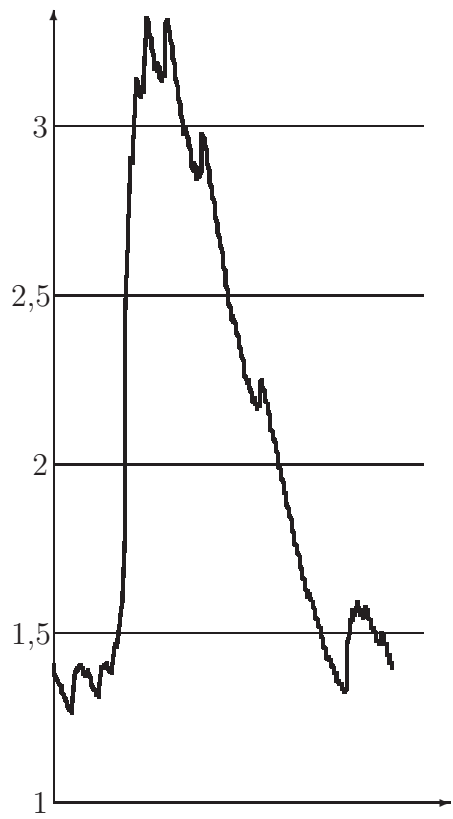
\begin{figure}[htbp]
\begin{minipage}[h]{0.9\linewidth}
\begin{center}
\begin{picture}(150,300)

\put(0,0){\vector(0,1){300}} \put(0,0){\vector(1,0){150}}

\put(0,128){\line(1,0){140}} \put(0,64){\line(1,0){140}}
\put(0,256){\line(1,0){140}} \put(0,192){\line(1,0){140}}

\linethickness{0.4mm}

\qbezier(0,52)(0,52)(0,49) \qbezier(0,49)(0,49)(1,47)
\qbezier(1,47)(1,47)(2,45) \qbezier(2,45)(2,45)(3,44)
\qbezier(3,44)(3,44)(3,42) \qbezier(3,42)(3,42)(4,40)
\qbezier(4,40)(4,40)(5,38) \qbezier(5,38)(5,38)(5,37)
\qbezier(5,37)(5,37)(6,35) \qbezier(6,35)(6,35)(7,34)
\qbezier(7,34)(7,34)(7,37) \qbezier(7,37)(7,37)(8,49)
\qbezier(8,49)(8,49)(9,51) \qbezier(9,51)(9,51)(9,50)
\qbezier(9,50)(9,50)(10,52) \qbezier(10,52)(10,52)(11,51)
\qbezier(11,51)(11,51)(11,49) \qbezier(11,49)(11,49)(12,48)
\qbezier(12,48)(12,48)(13,50) \qbezier(13,50)(13,50)(13,49)
\qbezier(13,49)(13,49)(14,47) \qbezier(14,47)(14,47)(14,45)
\qbezier(14,45)(14,45)(15,43) \qbezier(15,43)(15,43)(16,42)
\qbezier(16,42)(16,42)(16,41) \qbezier(16,41)(16,41)(17,40)
\qbezier(17,40)(17,40)(17,42) \qbezier(17,42)(17,42)(18,51)
\qbezier(18,51)(18,51)(19,52) \qbezier(19,52)(19,52)(19,51)
\qbezier(19,51)(19,51)(20,53) \qbezier(20,53)(20,53)(20,52)
\qbezier(20,52)(20,52)(21,50) \qbezier(21,50)(21,50)(22,49)
\qbezier(22,49)(22,49)(22,50) \qbezier(22,50)(22,50)(23,59)
\qbezier(23,59)(23,59)(23,60) \qbezier(23,60)(23,60)(24,59)
\qbezier(24,59)(24,59)(24,60) \qbezier(24,60)(24,60)(25,68)
\qbezier(25,68)(25,68)(25,69) \qbezier(25,69)(25,69)(26,76)
\qbezier(26,76)(26,76)(27,102) \qbezier(27,102)(27,102)(27,185)
\qbezier(27,185)(27,185)(28,210) \qbezier(28,210)(28,210)(28,215)
\qbezier(28,215)(28,215)(29,239) \qbezier(29,239)(29,239)(29,244)
\qbezier(29,244)(29,244)(30,242) \qbezier(30,242)(30,242)(30,247)
\qbezier(30,247)(30,247)(31,269) \qbezier(31,269)(31,269)(31,274)
\qbezier(31,274)(31,274)(32,272) \qbezier(32,272)(32,272)(32,269)
\qbezier(32,269)(32,269)(33,267) \qbezier(33,267)(33,267)(33,271)
\qbezier(33,271)(33,271)(34,269) \qbezier(34,269)(34,269)(34,273)
\qbezier(34,273)(34,273)(35,293) \qbezier(35,293)(35,293)(35,297)
\qbezier(35,297)(35,297)(36,295) \qbezier(36,295)(36,295)(36,291)
\qbezier(36,291)(36,291)(37,289) \qbezier(37,289)(37,289)(37,286)
\qbezier(37,286)(37,286)(38,282) \qbezier(38,282)(38,282)(38,278)
\qbezier(38,278)(38,278)(39,277) \qbezier(39,277)(39,277)(39,280)
\qbezier(39,280)(39,280)(40,278) \qbezier(40,278)(40,278)(40,275)
\qbezier(40,275)(40,275)(41,273) \qbezier(41,273)(41,273)(41,276)
\qbezier(41,276)(41,276)(42,275) \qbezier(42,275)(42,275)(42,277)
\qbezier(42,277)(42,277)(42,294) \qbezier(42,294)(42,294)(43,296)
\qbezier(43,296)(43,296)(43,295) \qbezier(43,295)(43,295)(44,291)
\qbezier(44,291)(44,291)(44,289) \qbezier(44,289)(44,289)(45,286)
\qbezier(45,286)(45,286)(45,282) \qbezier(45,282)(45,282)(46,279)
\qbezier(46,279)(46,279)(46,277) \qbezier(46,277)(46,277)(46,274)
\qbezier(46,274)(46,274)(47,271) \qbezier(47,271)(47,271)(47,267)
\qbezier(47,267)(47,267)(48,264) \qbezier(48,264)(48,264)(48,261)
\qbezier(48,261)(48,261)(49,257) \qbezier(49,257)(49,257)(49,254)
\qbezier(49,254)(49,254)(49,253) \qbezier(49,253)(49,253)(50,255)
\qbezier(50,255)(50,255)(50,254) \qbezier(50,254)(50,254)(51,251)
\qbezier(51,251)(51,251)(51,249) \qbezier(51,249)(51,249)(52,247)
\qbezier(52,247)(52,247)(52,244) \qbezier(52,244)(52,244)(52,241)
\qbezier(52,241)(52,241)(53,239) \qbezier(53,239)(53,239)(53,242)
\qbezier(53,242)(53,242)(54,240) \qbezier(54,240)(54,240)(54,238)
\qbezier(54,238)(54,238)(54,236) \qbezier(54,236)(54,236)(55,238)
\qbezier(55,238)(55,238)(55,237) \qbezier(55,237)(55,237)(56,239)
\qbezier(56,239)(56,239)(56,251) \qbezier(56,251)(56,251)(56,253)
\qbezier(56,253)(56,253)(57,251) \qbezier(57,251)(57,251)(57,249)
\qbezier(57,249)(57,249)(58,247) \qbezier(58,247)(58,247)(58,244)
\qbezier(58,244)(58,244)(58,242) \qbezier(58,242)(58,242)(59,239)
\qbezier(59,239)(59,239)(59,238) \qbezier(59,238)(59,238)(59,235)
\qbezier(59,235)(59,235)(60,232) \qbezier(60,232)(60,232)(60,229)
\qbezier(60,229)(60,229)(61,227) \qbezier(61,227)(61,227)(61,224)
\qbezier(61,224)(61,224)(61,222) \qbezier(61,222)(61,222)(62,219)
\qbezier(62,219)(62,219)(62,218) \qbezier(62,218)(62,218)(62,216)
\qbezier(62,216)(62,216)(63,213) \qbezier(63,213)(63,213)(63,211)
\qbezier(63,211)(63,211)(64,208) \qbezier(64,208)(64,208)(64,206)
\qbezier(64,206)(64,206)(64,203) \qbezier(64,203)(64,203)(65,201)
\qbezier(65,201)(65,201)(65,198) \qbezier(65,198)(65,198)(65,196)
\qbezier(65,196)(65,196)(66,194) \qbezier(66,194)(66,194)(66,191)
\qbezier(66,191)(66,191)(66,189) \qbezier(66,189)(66,189)(67,187)
\qbezier(67,187)(67,187)(67,185) \qbezier(67,185)(67,185)(67,183)
\qbezier(67,183)(67,183)(68,182) \qbezier(68,182)(68,182)(68,184)
\qbezier(68,184)(68,184)(68,183) \qbezier(68,183)(68,183)(69,181)
\qbezier(69,181)(69,181)(69,180) \qbezier(69,180)(69,180)(69,178)
\qbezier(69,178)(69,178)(70,176) \qbezier(70,176)(70,176)(70,174)
\qbezier(70,174)(70,174)(70,173) \qbezier(70,173)(70,173)(71,171)
\qbezier(71,171)(71,171)(71,169) \qbezier(71,169)(71,169)(72,167)
\qbezier(72,167)(72,167)(72,165) \qbezier(72,165)(72,165)(72,163)
\qbezier(72,163)(72,163)(72,162) \qbezier(72,162)(72,162)(73,160)
\qbezier(73,160)(73,160)(73,159) \qbezier(73,159)(73,159)(73,160)
\qbezier(73,160)(73,160)(74,160) \qbezier(74,160)(74,160)(74,158)
\qbezier(74,158)(74,158)(74,157) \qbezier(74,157)(74,157)(75,155)
\qbezier(75,155)(75,155)(75,154) \qbezier(75,154)(75,154)(75,152)
\qbezier(75,152)(75,152)(76,151) \qbezier(76,151)(76,151)(76,153)
\qbezier(76,153)(76,153)(76,152) \qbezier(76,152)(76,152)(77,150)
\qbezier(77,150)(77,150)(77,149) \qbezier(77,149)(77,149)(77,151)
\qbezier(77,151)(77,151)(78,150) \qbezier(78,150)(78,150)(78,151)
\qbezier(78,151)(78,151)(78,159) \qbezier(78,159)(78,159)(79,160)
\qbezier(79,160)(79,160)(79,159) \qbezier(79,159)(79,159)(79,158)
\qbezier(79,158)(79,158)(79,157) \qbezier(79,157)(79,157)(80,155)
\qbezier(80,155)(80,155)(80,153) \qbezier(80,153)(80,153)(80,152)
\qbezier(80,152)(80,152)(81,151) \qbezier(81,151)(81,151)(81,150)
\qbezier(81,150)(81,150)(81,148) \qbezier(81,148)(81,148)(82,146)
\qbezier(82,146)(82,146)(82,145) \qbezier(82,145)(82,145)(82,143)
\qbezier(82,143)(82,143)(82,141) \qbezier(82,141)(82,141)(83,140)
\qbezier(83,140)(83,140)(83,139) \qbezier(83,139)(83,139)(83,138)
\qbezier(83,138)(83,138)(84,136) \qbezier(84,136)(84,136)(84,135)
\qbezier(84,135)(84,135)(84,133) \qbezier(84,133)(84,133)(85,132)
\qbezier(85,132)(85,132)(85,130) \qbezier(85,130)(85,130)(85,128)
\qbezier(85,128)(85,128)(85,127) \qbezier(85,127)(85,127)(86,126)
\qbezier(86,126)(86,126)(86,124) \qbezier(86,124)(86,124)(86,123)
\qbezier(86,123)(86,123)(87,121) \qbezier(87,121)(87,121)(87,120)
\qbezier(87,120)(87,120)(87,119) \qbezier(87,119)(87,119)(87,117)
\qbezier(87,117)(87,117)(88,117) \qbezier(88,117)(88,117)(88,116)
\qbezier(88,116)(88,116)(88,114) \qbezier(88,114)(88,114)(88,113)
\qbezier(88,113)(88,113)(89,111) \qbezier(89,111)(89,111)(89,110)
\qbezier(89,110)(89,110)(89,109) \qbezier(89,109)(89,109)(90,107)
\qbezier(90,107)(90,107)(90,106) \qbezier(90,106)(90,106)(90,105)
\qbezier(90,105)(90,105)(90,103) \qbezier(90,103)(90,103)(91,102)
\qbezier(91,102)(91,102)(91,101) \qbezier(91,101)(91,101)(91,99)
\qbezier(91,99)(91,99)(91,98) \qbezier(91,98)(91,98)(92,97)
\qbezier(92,97)(92,97)(92,96) \qbezier(92,96)(92,96)(92,95)
\qbezier(92,95)(92,95)(93,93) \qbezier(93,93)(93,93)(93,92)
\qbezier(93,92)(93,92)(93,91) \qbezier(93,91)(93,91)(93,90)
\qbezier(93,90)(93,90)(94,88) \qbezier(94,88)(94,88)(94,87)
\qbezier(94,87)(94,87)(94,86) \qbezier(94,86)(94,86)(94,85)
\qbezier(94,85)(94,85)(95,84) \qbezier(95,84)(95,84)(95,83)
\qbezier(95,83)(95,83)(95,82) \qbezier(95,82)(95,82)(95,81)
\qbezier(95,81)(95,81)(96,80) \qbezier(96,80)(96,80)(96,79)
\qbezier(96,79)(96,79)(96,78) \qbezier(96,78)(96,78)(96,79)
\qbezier(96,79)(96,79)(97,79) \qbezier(97,79)(97,79)(97,78)

\qbezier(97,78)(97,78)(97,77) \qbezier(97,77)(97,77)(98,76)
\qbezier(98,76)(98,76)(98,75) \qbezier(98,75)(98,75)(98,74)

\qbezier(98,74)(98,74)(99,73) \qbezier(99,73)(99,73)(99,72)
\qbezier(99,72)(99,72)(99,71) \qbezier(99,71)(99,71)(99,70)
\qbezier(99,70)(99,70)(100,69) \qbezier(100,69)(100,69)(100,68)

\qbezier(100,68)(100,68)(100,67) \qbezier(100,67)(100,67)(101,66)
\qbezier(101,66)(101,66)(101,65) \qbezier(101,65)(101,65)(101,64)
\qbezier(101,64)(101,64)(101,63) \qbezier(101,63)(101,63)(102,62)
\qbezier(102,62)(102,62)(102,61) \qbezier(102,61)(102,61)(102,60)
\qbezier(102,60)(102,60)(102,59) \qbezier(102,59)(102,59)(103,58)
\qbezier(103,58)(103,58)(103,57)

\qbezier(103,57)(103,57)(103,56) \qbezier(103,56)(103,56)(103,55)
\qbezier(103,55)(103,55)(104,54)

\qbezier(104,54)(104,54)(104,55) \qbezier(104,55)(104,55)(104,54)
\qbezier(104,54)(104,54)(105,54) \qbezier(105,54)(105,54)(105,53)

\qbezier(105,53)(105,53)(105,52) \qbezier(105,52)(105,52)(106,51)

\qbezier(106,51)(106,51)(106,50) \qbezier(106,50)(106,50)(106,49)
\qbezier(106,49)(106,49)(107,49) \qbezier(107,49)(107,49)(107,48)
\qbezier(107,48)(107,48)(107,47) \qbezier(107,47)(107,47)(107,46)

\qbezier(107,46)(107,46)(108,45) \qbezier(108,45)(108,45)(108,46)

\qbezier(108,46)(108,46)(108,45) \qbezier(108,45)(108,45)(109,45)
\qbezier(109,45)(109,45)(109,44) \qbezier(109,44)(109,44)(109,43)

\qbezier(109,43)(109,43)(110,43)

\qbezier(110,43)(110,43)(110,42)

\qbezier(110,42)(110,42)(111,43)

\qbezier(111,43)(111,43)(111,47) \qbezier(111,47)(111,47)(111,60)
\qbezier(111,60)(111,60)(112,64)

\qbezier(112,64)(112,64)(112,68) \qbezier(112,68)(112,68)(112,69)
\qbezier(112,69)(112,69)(112,68) \qbezier(112,68)(112,68)(113,69)
\qbezier(113,69)(113,69)(113,72) \qbezier(113,72)(113,72)(113,73)
\qbezier(113,73)(113,73)(113,72) \qbezier(113,72)(113,72)(114,72)
\qbezier(114,72)(114,72)(114,71) \qbezier(114,71)(114,71)(114,72)
\qbezier(114,72)(114,72)(114,71) \qbezier(114,71)(114,71)(114,72)
\qbezier(114,72)(114,72)(115,75) \qbezier(115,75)(115,75)(115,76)
\qbezier(115,76)(115,76)(115,75)

\qbezier(115,75)(115,75)(115,74) \qbezier(115,74)(115,74)(116,73)

\qbezier(116,73)(116,73)(116,72) \qbezier(116,72)(116,72)(116,71)
\qbezier(116,71)(116,71)(116,72) \qbezier(116,72)(116,72)(117,71)

\qbezier(117,71)(117,71)(117,70) \qbezier(117,70)(117,70)(117,71)
\qbezier(117,71)(117,71)(117,70) \qbezier(117,70)(117,70)(118,71)
\qbezier(118,71)(118,71)(118,74)

\qbezier(118,74)(118,74)(118,73) \qbezier(118,73)(118,73)(119,73)
\qbezier(119,73)(119,73)(119,72) \qbezier(119,72)(119,72)(119,71)

\qbezier(119,71)(119,71)(119,70) \qbezier(119,70)(119,70)(120,69)

\qbezier(120,69)(120,69)(120,68) \qbezier(120,68)(120,68)(120,67)
\qbezier(120,67)(120,67)(120,66) \qbezier(120,66)(120,66)(121,66)
\qbezier(121,66)(121,66)(121,65)

\qbezier(121,65)(121,65)(122,64) \qbezier(122,64)(122,64)(122,63)

\qbezier(122,63)(122,63)(122,62) \qbezier(122,62)(122,62)(122,61)
\qbezier(122,61)(122,61)(123,61)

\qbezier(123,61)(123,61)(123,60)

\qbezier(123,60)(123,60)(124,60)

\qbezier(124,60)(124,60)(124,63) \qbezier(124,63)(124,63)(124,64)
\qbezier(124,64)(124,64)(124,63) \qbezier(124,63)(124,63)(125,63)
\qbezier(125,63)(125,63)(125,62)

\qbezier(125,62)(125,62)(125,61) \qbezier(125,61)(125,61)(125,60)
\qbezier(125,60)(125,60)(126,60) \qbezier(126,60)(126,60)(126,59)
\qbezier(126,59)(126,59)(126,58)

\qbezier(126,58)(126,58)(126,57) \qbezier(126,57)(126,57)(126,56)
\qbezier(126,56)(126,56)(127,56) \qbezier(127,56)(127,56)(127,55)

\qbezier(127,55)(127,55)(127,54) \qbezier(127,54)(127,54)(127,53)
\qbezier(127,53)(127,53)(128,53) \qbezier(128,53)(128,53)(128,52)
\qbezier(128,52)(128,52)(128,51)

\put(-8,252){3}

\put(-8,124){2}

\put(-8,-4){1} \put(-16,60){1,5} \put(-16,188){2,5}

\end{picture}
\end{center}
\end{minipage}
\caption{Graph of $\omega_2(x)$ for $v=3/4$}\label{fig:18}
\end{figure}

\begin{lemma}If the the invertible maps $h$ of the
form~(\ref{ex:h-second}) the diagram~(\ref{eq:13}) is commutative,
then for $x\hm{\in} [v,\, 1]$ the equality
$$ \widetilde{f}_v(x) = \frac{1-x}{1-v}$$ holds.
\end{lemma}

\begin{proof}
The proof of This Lemma if similar to one of Lemma~\ref{lema:10}.
\end{proof}

\begin{example}Plot the graph of the maps
$\widetilde{f}_v$, which is defined by the commutative
diagram~(\ref{eq:13}) for the maps $h$ of the
form~(\ref{ex:h-second}), if $\omega$ is the function of the most
simple form.
\end{example}

\begin{proof}[Deal of the example] The
equations~(\ref{eq:14}) except the case, when the $\omega$ is
constant non zero function. The most simple case for $\omega$, is
that, when $\omega =\omega_0$ for $x<0$ and $\omega =-\omega_0$
for $x>0$.

The it follows from the equation $h(0)=0$ obtain
$$ \omega_0 =
\frac{-1}{2-v}\cdot \left( \frac{2}{3}\right)^{\log_2(1-v)}.
$$

In this case the graph of the maps $\widetilde{f}_v(x) =
h(f(h^{-1}(x)))$, which is defined for $v=\frac{3}{4}$, if given
on the Figure~\ref{fig-29}.
\end{proof}

\begin{figure}[htbp]
\begin{minipage}[h]{0.9\linewidth}
\begin{center}
\begin{picture}(100,120)
\put(0,0){\vector(0,1){120}} \put(0,0){\vector(1,0){120}}

\qbezier(0,0)(0,0)(2,3) \qbezier(2,3)(2,3)(5,9)
\qbezier(5,9)(5,9)(8,14) \qbezier(8,14)(8,14)(10,18)
\qbezier(10,18)(10,18)(12,22) \qbezier(12,22)(12,22)(15,28)
\qbezier(15,28)(15,28)(17,31) \qbezier(17,31)(17,31)(20,36)
\qbezier(20,36)(20,36)(23,41) \qbezier(23,41)(23,41)(25,45)
\qbezier(25,45)(25,45)(28,49) \qbezier(28,49)(28,49)(31,54)
\qbezier(31,54)(31,54)(34,58) \qbezier(34,58)(34,58)(37,62)
\qbezier(37,62)(37,62)(40,66) \qbezier(40,66)(40,66)(42,69)
\qbezier(42,69)(42,69)(45,72) \qbezier(45,72)(45,72)(48,74)
\qbezier(48,74)(48,74)(50,76) \qbezier(50,76)(50,76)(53,78)
\qbezier(53,78)(53,78)(56,79) \qbezier(56,79)(56,79)(59,80)
\qbezier(59,80)(59,80)(62,80) \qbezier(62,80)(62,80)(65,82)
\qbezier(65,82)(65,82)(67,84) \qbezier(67,84)(67,84)(69,86)
\qbezier(69,86)(69,86)(71,89) \qbezier(71,89)(71,89)(72,92)
\qbezier(72,92)(72,92)(74,96) \qbezier(74,96)(74,96)(75,99)
\qbezier(75,99)(75,99)(75,100)

\qbezier[32](0,0)(37.5,50)(75,100)

\qbezier(75,100)(87.5,50)(100,0)

\end{picture}
\end{center}
\end{minipage}
\caption{Graph of $\widetilde{f}_v$ for
$v=\frac{3}{4}$}\label{fig-29}
\end{figure}
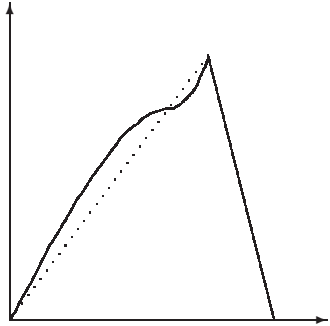

\begin{example}Plot the graph of the maps
$\widetilde{f}_v$, which is defined by the commutative
diagram~(\ref{eq:13}) for the maps $h$ of the
form~(\ref{ex:h-second}), if $\omega$ is the function, for which
$\omega^+$ and $\omega^-$ are of the simplest form, when they are
not constant.
\end{example}

\begin{proof}Assume that functions $\omega^+$ and
$\omega^-$ are continuous.

From the equation $h(0) =0$ obtain
\begin{equation}\label{eq:24} 0 = \frac{1}{2-v} +
\left(\frac{2}{3}\right)^{-\log_2(1-v)}\omega^-\left(
\log_2\frac{2}{3}\right), \end{equation} i.e.
\begin{equation}\label{eq:26} \omega^-(1-\log_23) =
\frac{1}{v-2} \left(\frac{2}{3}\right)^{\log_2(1-v)}
\end{equation}

Since the left pre image of $\frac{1}{2}$ under $f$ is mapped
under $h$ to the left pre image of $v$ under $f_v$, then $$
h\left(\frac{1}{4}\right) = v^2,
$$ i.e.

$$ v^2 = \frac{1}{2-v} +
\left(\frac{5}{12}\right)^{-\log_2(1-v)}\omega^-\left(
\log_2\frac{5}{12}\right),
$$
whence it follows from the periodicity of $\omega^-$ with
period~2, that
\begin{equation}\begin{array}{l}\label{eq:25}
\omega^-\left( \log_2\frac{5}{3}\right) =\\
\\=\displaystyle{ \left(v^2 - \frac{1}{2-v}\right)
\left(\frac{5}{12}\right)^{\log_2(1-v)}}
\end{array}\end{equation}

Plot the function $\omega^-$ as follows. The values of $\omega^-$
at points $\log_2\frac{5}{3} \approx 0,737$ and
$2+\log2\frac{5}{3} \approx 2,737$ are equal and are defined by
the equality~(\ref{eq:25}). Remind that it follows from
equation~(\ref{eq:14}) that the function $\omega^-$ is periodical
with period 2.

The value of $\omega^-$ at point $3-\log_23 \approx 1,415$ is
defined by the equality~(\ref{eq:26}). It follows from periodicity
of $\omega^-$ that the equality~(\ref{eq:26}) can be applied for
finding of the value of $\omega^-$  at this point.

Take the function $\omega^-$ to be linear on each of the intervals
$[\log_2\frac{5}{3},\, 3-\log_23]$ and $[3-\log_23,\,
2+\log2\frac{5}{3}]$ and make it periodical with period 2.

Construct the function $\omega^+$ with the use of $\omega^-$,
using the relation~(\ref{eq:14}).

The graph of $\widetilde{f}_v$, which is constructed as above for
$v=3/4$, if given on the Figure~\ref{fig-30}.

\begin{figure}[htbp]
\begin{minipage}[h]{0.9\linewidth}
\begin{center}
\begin{picture}(100,120)
\put(0,0){\vector(0,1){120}} \put(0,0){\vector(1,0){120}}

\put(56.25,75){\circle*{5}}

\qbezier(75,100)(87.5,50)(100,0)

\qbezier[32](0,0)(37.5,50)(75,100)

\qbezier(0,0)(0,0)(1,3) \qbezier(1,3)(1,3)(3,6)
\qbezier(3,6)(3,6)(5,10) \qbezier(5,10)(5,10)(7,14)
\qbezier(7,14)(7,14)(9,17) \qbezier(9,17)(9,17)(11,21)
\qbezier(11,21)(11,21)(13,24) \qbezier(13,24)(13,24)(15,28)
\qbezier(15,28)(15,28)(17,31) \qbezier(17,31)(17,31)(19,35)
\qbezier(19,35)(19,35)(21,38) \qbezier(21,38)(21,38)(23,42)
\qbezier(23,42)(23,42)(25,45) \qbezier(25,45)(25,45)(27,49)
\qbezier(27,49)(27,49)(30,52) \qbezier(30,52)(30,52)(32,55)
\qbezier(32,55)(32,55)(35,58) \qbezier(35,58)(35,58)(39,61)
\qbezier(39,61)(39,61)(42,64) \qbezier(42,64)(42,64)(45,66)
\qbezier(45,66)(45,66)(49,69) \qbezier(49,69)(49,69)(52,72)
\qbezier(52,72)(52,72)(56,75) \qbezier(56,75)(56,75)(59,78)
\qbezier(59,78)(59,78)(64,80) \qbezier(64,80)(64,80)(68,82)
\qbezier(68,82)(68,82)(71,85) \qbezier(71,85)(71,85)(72,90)
\qbezier(72,90)(72,90)(74,95) \qbezier(74,95)(74,95)(75,100)
\end{picture}
\end{center}
\end{minipage}
\caption{Graph of $\widetilde{f}_v$ for
$v=\frac{3}{4}$}\label{fig-30}
\end{figure}
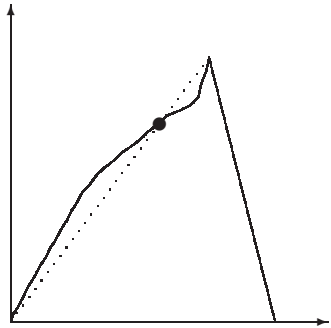

The Figure~\ref{fig-31} contains the result of imposition of the
graph from the previous example for those one, which is
constructed now (the dots are used for the first graph).

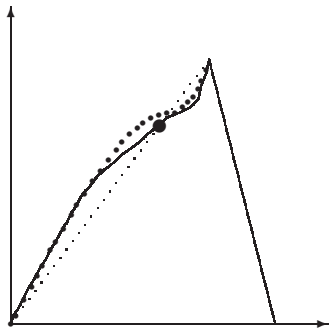
\begin{figure}[htbp]
\begin{minipage}[h]{0.9\linewidth}
\begin{center}
\begin{picture}(100,120)
\put(0,0){\vector(0,1){120}} \put(0,0){\vector(1,0){120}}

\put(56.25,75){\circle*{5}}

\qbezier(75,100)(87.5,50)(100,0)

\qbezier[32](0,0)(37.5,50)(75,100)

\qbezier(0,0)(0,0)(1,3) \qbezier(1,3)(1,3)(3,6)
\qbezier(3,6)(3,6)(5,10) \qbezier(5,10)(5,10)(7,14)
\qbezier(7,14)(7,14)(9,17) \qbezier(9,17)(9,17)(11,21)
\qbezier(11,21)(11,21)(13,24) \qbezier(13,24)(13,24)(15,28)
\qbezier(15,28)(15,28)(17,31) \qbezier(17,31)(17,31)(19,35)
\qbezier(19,35)(19,35)(21,38) \qbezier(21,38)(21,38)(23,42)
\qbezier(23,42)(23,42)(25,45) \qbezier(25,45)(25,45)(27,49)
\qbezier(27,49)(27,49)(30,52) \qbezier(30,52)(30,52)(32,55)
\qbezier(32,55)(32,55)(35,58) \qbezier(35,58)(35,58)(39,61)
\qbezier(39,61)(39,61)(42,64) \qbezier(42,64)(42,64)(45,66)
\qbezier(45,66)(45,66)(49,69) \qbezier(49,69)(49,69)(52,72)
\qbezier(52,72)(52,72)(56,75) \qbezier(56,75)(56,75)(59,78)
\qbezier(59,78)(59,78)(64,80) \qbezier(64,80)(64,80)(68,82)
\qbezier(68,82)(68,82)(71,85) \qbezier(71,85)(71,85)(72,90)
\qbezier(72,90)(72,90)(74,95) \qbezier(74,95)(74,95)(75,100)

\put(0,0){\circle*{2}} \put(2,3){\circle*{2}}
\put(5,9){\circle*{2}} \put(8,14){\circle*{2}}
\put(10,18){\circle*{2}} \put(12,22){\circle*{2}}
\put(15,28){\circle*{2}} \put(17,31){\circle*{2}}
\put(20,36){\circle*{2}} \put(23,41){\circle*{2}}
\put(25,45){\circle*{2}} \put(28,49){\circle*{2}}
\put(31,54){\circle*{2}} \put(34,58){\circle*{2}}
\put(37,62){\circle*{2}} \put(40,66){\circle*{2}}
\put(42,69){\circle*{2}} \put(45,72){\circle*{2}}
\put(48,74){\circle*{2}} \put(50,76){\circle*{2}}
\put(53,78){\circle*{2}} \put(56,79){\circle*{2}}
\put(59,80){\circle*{2}} \put(62,80){\circle*{2}}
\put(65,82){\circle*{2}} \put(67,84){\circle*{2}}
\put(69,86){\circle*{2}} \put(71,89){\circle*{2}}
\put(72,92){\circle*{2}} \put(74,96){\circle*{2}}
\end{picture}
\end{center}
\end{minipage}
\caption{Result of the imposition of one graph onto
another.}\label{fig-31}
\end{figure}

\end{proof}

\newpage
\section{Explicit formulas for topological conjugation}\label{sect-Javni-fornuly}

We will construct in this Section the explicit formulas for the
topological conjugacy of the maps $f,\, f_v:\, [0,\, 1]\rightarrow
[0,\, 1]$, given as follows

\begin{equation}\label{eq:76} f(x) =
\left\{\begin{array}{ll}
2x,& x< 1/2;\\
2-2x,& x\geqslant 1/2
\end{array}\right.
\end{equation}and\begin{equation}
\label{eq:77}
f_v(x) = \left\{\begin{array}{ll} \frac{x}{v},& x\leqslant v;\\
 \frac{1-x}{1-v},&
x>v,
\end{array}\right.
\end{equation}

In other words, we find the homeomorphism $h:\, [0,\,
1]\rightarrow [0,\, 1]$, which is a solution of the functional
equation

\begin{equation}\label{eq:79}h(f) =
f_v(h).
\end{equation}

We have constructed in Section~\ref{sect-Pobudowa} the
homeomorphism $h$ as a limit of piecewise linear homeomorphisms
$h_n$, whose breaking points belonged to the set $A_n$, which is a
solution of the equation $f^n(x)=0$. For any $n \geq 1$ the
equality $h_n(A_n) = B_n$ holds, where $B_n$ is the solution set
of the equation $f^n_v(x)=0$. By Proposition~\ref{lema:An},
$$ A_n = \left\{0,\, \frac{1}{2^{n-1}},\ldots,\,
\frac{2^{n-1}-1}{2^{n-1}},\, 1\right\},
$$ By Theorem~\ref{theor:10}, the equality $h(A_n) =B_n$ holds.

By Theorem~\ref{theor:10} the equality $h(x)=h_n(x)$ holds for
every $x\in A_n$ and the conjugacy $h$ increase.

We have denoted in Section~\ref{sect-dyffer} the elements of $A_n$
by $\alpha_{n,k},\, 0\leq k\leq 2^{n-1}$ such that
$\alpha_{n,k_1}<\alpha_{n,k_2}$ for $k_1<k_2$. By
Proposition~\ref{lema:An}, the equality
$$\alpha_{n,k}=\frac{k}{2^{n-1}}$$ holds.
Also we have in Section~\ref{sect-dyffer} the elements of $B_n$ by
$\beta_{n,k},\, 0\leq k\leq 2^{n-1}$ such that
$\beta_{n,k_1}<\beta_{n,k_2}$ for $k_1<k_2$. In these notations,
if follows from Theorem~\ref{theor:10} that
$h(\alpha_{n,k})=\beta_{n,k}$ for all $n\in \mathbb{N}$ and all
$k,\, 0\leq l\leq 2^{n-1}$.

Notice that the following evident property of $\alpha_{n,k}$ and
$\beta_{n,k}$ holds.
\begin{equation}\label{eq:3}
\left\{\begin{array}{l} \alpha_{n+t,2^tk} = \alpha_{n,k},\\
\beta_{n+t,2^tk} = \beta_{n,k}.
\end{array} \right.
\end{equation}

Notice, that the number, which is formed of the first $n$ digits
of the binary decomposition of $x$ is $\alpha_{n,k}$ (for some
$k$) and $k =\left[2^nx\right]$, where brackets denote the integer
part of a number, i.e. $[2^nx]$ is the biggest integer, which is
not grater than $2^nx$. Thus, the following equality holds for
$h$.
\begin{equation}\label{h}h(x) = \lim\limits_{n
\rightarrow \infty}\beta_{n,\, \left[2^nx\right]}.
\end{equation}

\subsection{The first way of finding of explicit formulas}\label{subs-h-worse}

Remind, that the graph of $f_v^n$ is piecewise linear and consists
of $2^n$ intervals of monotonicity, each of them maps some
subinterval of $[0,\, 1]$ to the whole $[0,\, 1]$.

Consider an integer $x= 1,\ldots, 2^n$ and calculate the tangent
of the branch of monotonicity number $k,\, 0\leq k\leq 2^{n-1}-1$
of the maps $f_v^n$. Let the binary decomposition of $k$ be
$$ k = x_1x_2\ldots x_n.
$$
Let $\widetilde{x}\in [0,\, 1]$ be a point of the interval of
monotonicity under consideration, i.e. $\widetilde{x}\in
[\alpha_{n,k},\, \alpha_{n,k+1})$.

The graph of $f_v^n$ increase on the intervals with even numbers
(the numeration starts with zero) and decrease on the intervals
with odd number. In other words, $(f_v^n)'(x)>0$ for $x_n=0$ and
$(f_v^n)'(x)<0$ for $x_n=1$.

The maps $f_v$ acts on each of the branch of the linearity of it's
correspond iteration as follows. Independently on increasing or
decreasing of the branch, it transforms to two branches such that
the left one increase and the right one decrease.

1. If the branch increased, then the tangent of the new increasing
branch if the former tangent, multiplied by $\frac{1}{v}$. The
tangent of the new decreasing is obtained from the former tangent
by multiplying it by $\frac{1}{v-1}$.

2. If the branch increased, then the tangent of the new increasing
branch if the former tangent, multiplied by $\frac{1}{v-1}$. The
tangent of the new decreasing is obtained from the former tangent
by multiplying it by $\frac{1}{v}$.

Notice that the $n$-th iteration of $f_v$ at $\widetilde{x}$ is a
composition of linear maps, whence the tangent of $f_v^n$ at this
point would be the product of tangents os correspond branches of
linearity of $f_v$ at points
$\alpha_k\hm{=}f_v^k(\widetilde{x}),\, k=1,\ldots n$. In other
words,
$$ (f_v^n)'(\widetilde{x}) = \prod\limits_{k=1}^n\alpha_k,
$$ where $\alpha_k$ can be found as follows.\\

\systema%
{\alpha_1=\frac{1}{v} & \text{for }x_{1}=0}%
{\alpha_1=\frac{1}{v-1} & \text{for }x_{1}=0} For bigger $k$ we
have
\systema%
{\alpha_k=\frac{1}{v} & \text{for }x_{k-1}+x_k \equiv 0\ \mod 2}%
{\alpha_k=\frac{1}{v-1} & \text{for }x_{k-1}+x_k \equiv 1\ \mod 2}

Take $x_0=0$ for making this formula correct for $k=1$.

We will need two additional notations for the following result.
Denote $\psi_1(x)$ as follows
$$ \psi_1(x)=
\left\{
\begin{array}{ll}
0 & \text{for } x \equiv 0\ \mod 2\\
1 & \text{for } x \equiv 1\ \mod 2,
\end{array}
\right.
$$ and denote $$
\psi_2(x) = 1-\psi_1(x).
$$

\begin{lemma}Functions $\psi_1(x)$ and $\psi_2(x)$
can be given as
$$ \psi_1(x) = 2\left\{ \frac{x}{2}\right\};\
\ \ \psi_2(x) = 2\left\{ \frac{x+1}{2}\right\},
$$ where figure brackets denote the fractional part of a number.
\end{lemma}

\begin{proof}Lemma is evident.
\end{proof}

With the use of $\psi_1$ and $\psi_2$ we can rewrite the formula
for $\alpha_k$ as follows
$$ \alpha_k =
\frac{1}{v-1}\cdot\psi_1(x_k+x_{k-1})
+\frac{1}{v}\cdot\psi_2(x_k+x_{k-1}) =
$$
$$
= \frac{2}{v-1}\cdot\left\{\frac{x_k+x_{k-1}}{2}\right\}
+\frac{2}{v}\cdot\left\{\frac{x_k+x_{k-1}+1}{2}\right\}.
$$

Thus, the tangent of the $k$-th branch of monotonicity of $f_v^n$
(for $k=1,\ldots, 2^n$), can be calculated as
\begin{equation}\label{zeta}\zeta_{n,k} =
\prod\limits_{k=1}^n\left(
\frac{2}{1-v}\cdot\left\{\frac{x_k+x_{k-1}}{2}\right\}
+\frac{2}{v}\cdot\left\{\frac{x_k+x_{k-1}+1}{2}\right\}\right).
\end{equation}

Since the graph of $f_v^n$ consists of $2^n$ branches, and each of
them maps some subinterval of $[0,\, 1]$ into the whole $[0,\,
1]$, then the following expressions for $\beta_{n,k}$ hold.
\begin{equation}\label{beta}\beta_{n,k}
=\sum\limits_{t=1}^k \frac{1}{\zeta(t,n)} = \exp\left(
\prod\limits_{t=1}^k \frac{1}{\zeta(t,n)}\right).
\end{equation}

Remind that $x_1,\ldots, x_n$ in the formula for $\zeta_{n,k}$ if
the binary decomposition of $k$. We will obtain the formulas,
which would express each of these digits in terms of $k$.

The number, which is consisted of the last $t$ binary digits of
$k$, can be found by the formula $\left[\frac{k}{2^t}\right].$

If one change the last $t$ digits of $x$ to zeros, then would
obtain the number
$$k-2^t\cdot\left[\frac{k}{2^t}\right].$$

If delete the last $t$ digits of the obtained number, then obtain
$$\frac{k-2^t\cdot\left[\frac{k}{2^t}\right]}{2^t} =\frac{k}{2^t}
-\left[\frac{k}{2^t}\right] = \left\{ \frac{k}{2^t}\right\}.$$ The
last $t+1$-st digit of $x$ can be calculated as the last digit of
the the number, which is obtained from $x$ by deleting the last
$t$ digits. In other words, it ca be found as
$$ \psi_1\left( \left\{ \frac{k}{2^t}\right\}\right) = 2
\left\{\frac{\left\{ \frac{k}{2^t}\right\}}{2}\right\}.
$$

If the number $k$ is consisted of $n$ digits, then each of them
can be found as$$ x_p =2 \left\{\frac{\displaystyle{\left\{
\frac{k}{2^{n-p}}\right\}}}{2}\right\}.
$$

These computations let us to rewrite the formula~(\ref{zeta}) for
$\zeta_{n,k}$ as follows

\begin{equation}\label{eq:80}
\begin{array}{c} \zeta_{n,k} =
\prod\limits_{t=1}^n\left( \frac{2}{1-v}\cdot\left\{
\left\{\displaystyle{\left\{ \frac{k}{2^{n-t}}\right\}}/2\right\}+
\left\{\displaystyle{\left\{
\frac{k}{2^{n-t+1}}\right\}}/2\right\}\right\} +\right.\\
+\left.\frac{2}{v}\cdot\left\{\left\{\displaystyle{\left\{
\frac{k}{2^{n-t}}\right\}}/2\right\}+ \left\{\displaystyle{\left\{
\frac{k}{2^{n-t+1}}\right\}}/2\right\}
+\frac{1}{2}\right\}\right).\end{array}\end{equation}

Thus, using~(\ref{beta}), we have the following theorem.

\begin{theorem}\label{theor:18}The
homeomorphic solution $h:\, [0,\, 1]\rightarrow [0,\, 1]$ of the
equation~(\ref{eq:79}) can be expressed by the formula
$$
h(x) = \lim\limits_{n \rightarrow
\infty}\beta\left(\displaystyle{\left[2^nx\right],\, n}\right) =
\lim\limits_{n \rightarrow \infty}
\sum\limits_{t=1}^{[2^nx]}\frac{1}{\zeta_{n,t}},
$$ where $\zeta(k,n)$ is given by~(\ref{eq:80}).
\end{theorem}

\begin{note}\label{note:2}In spite that
the formula for $h(x)$ is quite complicated and contains a limit,
this limit exists and the function is defined correctly.
\end{note}

\begin{proof}The existence of the limit follows from the
same reasonings, which where done in the proof of
Theorem~\ref{theor:homeom-jed}.
\end{proof}

\begin{note}\label{note:3}Notice that the
formula from Theorem~\ref{theor:18} has the following properties.

1. Even in the case when $x\in A$ the calculation of $h$ by this
formula would need the $n$ summands for obtaining the
approximation $[2^nx]\in A_n$. The condition $x\in A$ yield that
the approximate values of $h(x)$ would not change with increasing
$n$ after some $n$, huge enough.

2. The stabilizing of $h(x)$, which is under consideration, means
that for some $n\in \mathbb{N}$ the equality
$$ \sum\limits_{t=1}^{[2^nx]}\frac{1}{\zeta_{n,t}} =
\sum\limits_{t=1}^{[2^{n+1}x]}\frac{1}{\zeta_{n+1,t}}
$$ holds. From
another hand, this equality means that the $x\in A_n$ and the
exact value of $h(x)$ is found.
\end{note}

The first of the properties from the Remark~\ref{note:3} can be
considered as its deficiency. We will find in the next section
another formula for $\beta_{n,k}$, which would not have this
deficiency.

\subsection{The second way of finding of explicit formulas}\label{subs-h-better}

Clearly, the values of $\beta_{1,k},\, 0\leq k\leq 2$ as as
follows: $\beta_{1,0} = 0$, $\beta_{1,1} = v$ and $\beta_{1,2} =
1$.

Similarly to as it was done in Section~\ref{subs-h-worse}, write
$k$ as follows
$$ k =
\sum\limits_{i=1}^nx_i2^{n-i}.
$$

Consider the sequence $\xi_t=\xi_{t,n,k}$ of the left ends of the
interval $[\beta_{t,s},\, \beta_{t,s+1})\ (t<n)$, if is given that
that is contains the point $\beta(n,\, k)$.

Let for some $t$ the number $\xi_t$ is found and let $p_t$ be the
tangent from the right of the maps $f^t$. Precisely, $\xi_0 = 0$,
$p_0 = 1$.

\begin{lemma}If $x_1 =1$, then $\xi_1 =v$, and the
maps $f^1$ decrease at the fight neighborhood of $\xi_1$ and $p_1
= \frac{1}{v-1}$.

If $x_1=\ldots =x_{t-1} =0$ and $x_t = 1$, then $\xi_t = v^{t}$
and the maps $f^t$ decrease in the right neighborhood of $\xi_t$
and $p_t = \left(\frac{1}{v}\right)^{t-1}\frac{1}{v-1}$.
\end{lemma}

Let $x = \frac{k}{2^n}$, then
$$ x =
\sum\limits_{i=1}^nx_i2^{-i}.
$$

Notice, that the number of the $t$-th first digit $1$ of the
expression for $x$ can be found as follows
$$t = [-\log_2x].$$

If $\{ \log_2 x\}=0$ then $x$ is a power of $2$, i.e.
$$
((-1)^{[-\{ \log_2 x\}]}-1)/2 = \left\{ \begin{array}{ll} 0 & x =
2^{t}\\
-1 & x\neq 2^t
\end{array}\right.
$$

\begin{lemma}Let for some $t$ the maps $f^t$
decrease in the right neighborhood of $\xi_t$ and $x_t =1$. Then
the following implications hold.

If $x_{t+1} =1$, then $\xi_{t+1} = \xi_t - \frac{1}{p_{t+1}}$ and
$p_{t+1} = p_t\frac{1}{v}$.

If $x_{t+1} = x_{t+2} =\ldots = x_{t+s} =0$ ($s\geq 1$) and
$x_{t+s+1} =1$, then $\xi_{t+s+1} = \xi_t - \frac{1}{p_{t+s}}$ and
$p_{t+s+1} =
p_t\left(\frac{1}{v-1}\right)^2\left(\frac{1}{v}\right)^{s-1}$.
\end{lemma}

\begin{lemma}Let for some $t$ the maps $f^t$
decrease in the right neighborhood of $\xi_t$ and $x_t =1$. Let
also $x_{t+1} = x_{t+2} =\ldots = x_{t+s} =0$ ($s\geq 1$) and
$x_{t+s+1} =1$. Then $$\xi_{t+s+1} = \xi_t -
\frac{1}{\zeta_{[2^{t+s+1} x],\, t}},$$ where $\zeta(k,n)$ is
expressed by~(\ref{eq:80}).
\end{lemma}

The number of the second digit $1$ of the expression of $x$ is$$
t_2 = [-\log_2(x-2^{-t})].$$

Continuing this way, we can obtain the formula for the value
$h(x)$ for every $x\in [0,\, 1]$.

Thus, we have proved the following theorem.

\begin{theorem}\label{theor:19}
The homeomorphic solution $h:\, [0,\, 1]\rightarrow [0,\, 1]$
of~(\ref{eq:79}) can be expressed by formula
$$ h(x) = \sum\limits_{i=1}^{\infty}\frac{(2^{i-1}
x)((-1)^{[-\{ \log_2 [2^ix]\}]}-1)}{\zeta_{i,\,
[2^{i+[-\log_2(2^{i+1}x)]+1} x]}},
$$ where $\zeta_{n,k}$ expressed by~(\ref{eq:80}).
\end{theorem}

\begin{note}\label{note:4}
Notice,that if $x\in \mathcal{A}$, then the formula from
Theorem~\ref{theor:19} contains only finite number of summands.
\end{note}

\newpage

\section{Conjugateness of piecewise linear unimodal maps}\label{sect:KuskowoLin}

We will consider in this Section the map\begin{equation}
\label{eq:89} f(x) = \left\{\begin{array}{ll}
2x,& \text{if } 0\leq x< 1/2,\\
2-2x,& \text{if } 1/2 \leqslant x\leqslant 1.
\end{array}\right.
\end{equation} and a continuous map $g: \, [0,\,
1]\rightarrow [0,\, 1]$ of the form \begin{equation} \label{eq:g}
g(x) = \left\{\begin{array}{ll}
g_l,& \text{if } 0\leq x< v,\\
g_r,& \text{if } v \leqslant x\leqslant 1,
\end{array}\right.
\end{equation} where
$g_l(0)=g_r(1)=0$, $g_r(v)=1$ and functions $g_l$ and $g_r$ are
monotone piecewise linear.

We will consider a homeomorphism $h:\, [0,\, 1]\rightarrow [0,\,
1]$, such that the following diagram
$$\begin{CD}
[0,\, 1] @>f >> & [0,\, 1]\\
@V_{h} VV& @VV_{h}V\\
[0,\, 1] @>g>>& [0,\, 1],
\end{CD}$$ is commutative, i.e. the equality
\begin{equation}\label{eq:eq-top-eq-ffv} h(f(x)) =
g(h(x))
\end{equation} holds for every $x \in [0,\, 1]$.

\subsection{Continuous differentiability of the conjugation}\label{sect:KuskowoLin-2}

The main result of this section is the following theorem.

\begin{theorem}\label{theor-dyfer-lin}
Let the map $f$, which is given by~(\ref{eq:89}), be topologically
conjugated with piecewise linear map $g,\, [0,\, 1]\rightarrow
[0,\, 1]$, and let $h$ be the conjugacy such
that~(\ref{eq:eq-top-eq-ffv}) holds. If $h$ is continuously
differentiable on $(\alpha,\, \beta)$ for some $0\leq
\alpha<\beta\leq 1$, then $h$ is piecewise linear on $[0,\, 1]$.
\end{theorem}

We will formulate some lemmas at the very beginning. These lemmas
are, in fact, the steps of the proof of
Theorem~\ref{theor-dyfer-lin}.

\begin{lemma}\label{lema:37}
If the homeomorphism $h$ is piecewise differential on some
interval $(\alpha,\, \beta)\subset (0,\, 1)$ and is a conjugacy of
$f$ and the piecewise linear $g$, then there exists open disjoint
intervals $A_1,\ldots,\, A_s$ such that
$\bigcup\limits_{i=1}^{s}\overline{A}_i =[0,\, 1]$ and $h$ is
continuously differentiable on each $A_i$.
\end{lemma}

\begin{lemma}\label{lema:38}
If the homeomorphism $h$ is continuously differentiable on some
interval $(\alpha,\, \beta)\subset (0,\, 1)$ and $A_1,\ldots,\,
A_s$ are open intervals from Lemma~\ref{lema:37}, then for every
$i,\, 1\leq i\leq s$ there exists $k_i\in \mathbb{R}$ such that
for every $x\in A_i$ the maps $h$ is differentiable at $f(x)$ and
the equality $$ h'(f(x)) = k_ih'(x)$$ holds.
\end{lemma}

\begin{lemma}\label{lema:39}
Let $x^*$ be a periodical point of $f$ of the period $n$, whose
orbit belongs to $\bigcup\limits_{i=1}^{s}A_i$ and $h'(x^*)\neq
0$. For every $i,\, 1\leq i< n$ denote by $x_i = f^i(x^*)$ the
trajectory of $x^*$. Let $A_{c_0},\, A_{c_1},\ldots,\,
A_{c_{n-1}}$ be sets from Lemma~\ref{lema:37}, which contain the
trajectory (i.e.  $x_i\in A_{c_i}$). Then the equality
$k_{c_0}\cdot k_{c_1}\cdot \ldots \cdot k_{c_{n-1}} = 1$ holds,
where $k_{c_0},\, k_{c_1},\, \ldots ,\, k_{c_{n-1}}$ are
constructed for intervals $A_{c_0},\, A_{c_1},\ldots,\,
A_{c_{n-1}}$ as in Lemma~\ref{lema:38}.
\end{lemma}

\begin{lemma}\label{lema:40}
There exists an interval $[a,\, b]\subset [0,\, 1]$, where the
derivative $h'$ is constant.
\end{lemma}

\begin{lemma}\label{lema:41}
Let the maps $f$ be of the form~(\ref{eq:89}) and let $h$ be the
conjugacy of $f$ with the piecewise linear unimodal $g$. Assume
that $h$ is piecewise continuously differentiable on an interval
$(\alpha,\, \beta)\subset (0,\, 1)$. Then there exists numbers
$\alpha =\alpha_1< \alpha_2 <\hm{\ldots} < \alpha_t =\beta $ such
that for every $p\in 1,\ldots,\, t-1$ there exists $k_p$ such that
$h'(w) =k_ph'(x)$ for all $x\in (\alpha_p,\, \alpha_{p+1})$, where
$w=f(x)$.
\end{lemma}

\begin{proof}
Without loss of generality assume that $1/2\not\in (\alpha,\,
\beta)$. For avoiding two cases whether $\beta<1/2$, of
$\alpha>1/2$ denote $f(x)=ax+b$ for $x\in (\alpha,\, \beta)$, as
$f$ is linear on $(\alpha,\, \beta)$. Let
$\alpha=\alpha_1<\alpha_2<\ldots \alpha_t=\beta$ be such numbers,
that all braking points of $g$ belong to $h(\alpha_1),\ldots,\,
h(\alpha_t)$. Let $g(x)=a_px+b_p$ be the formula for $g$ for $x\in
(h(\alpha_p),\, h(\alpha_{p+1}))$.

Let $x\in (\alpha,\, \beta)\backslash \{\alpha_2,\ldots,\,
\alpha_{t-1}\}$ be fixed. Consider a sequence $\{ x_n\}$ such that
$\lim\limits_{n\rightarrow \infty}x_n = x$ and $x_n\neq x$ for all
$n$. Consider the following equalities
\begin{equation}\label{eq:82} \left\{
\begin{array}{l}
h(f(x)) = g(h(x)),\\
h(f(x_n)) = g(h(x_n)).
\end{array}\right.
\end{equation}

Since $x\in (\alpha_p,\, \alpha_{p+1})$ for some $p,\, 1\leq p\leq
t-1$, then equalities~(\ref{eq:82}) can be rewritten as
\begin{equation}\label{eq:83}
\begin{CD}
x @>f >> & ax+b\\
@V_{h} VV& @VV_{h}V\\
h(x) @>g>> & a_ph(x)+b_p
\end{CD}\hskip 1cm\text{ and }\hskip 1cm \begin{CD}
x_n @>f >> & ax_n+b\\
@V_{h} VV& @VV_{h}V\\
h(x_n) @>g>> & a_ph(x_n)+b_p
\end{CD}\end{equation}

Without loos of generality assume that $x_n\in (\alpha_p,\,
\alpha_{p+1})$ for all $n$. Consider a number $w = ax+b$ and the
sequence $w_n = ax_n +b_n$. Since $\lim\limits_{n\rightarrow
\infty}x_n = x$, then $\lim\limits_{n\rightarrow \infty}w_n = w$.
Since $x_n\neq x$ for all $n$, then $w_n\neq w$ for all $n$. Prove
the existence of the limit $\lim\limits_{n\rightarrow
\infty}\frac{h(w)-h(w_n)}{w-w_n}.$ It follows from the
commutativity of diagrams~(\ref{eq:83}) that $h(w) = a_ph(x)+b_p$
and $h(w_n) = a_ph(x_n)+b_p$. Then
$$\frac{h(w)-h(w_n)}{w-w_n} = \frac{a_p(h(x)-h(x_n))}{a(x-x_n)},$$
which prove Lemma if take $k_p = \frac{a_p}{a}.$
\end{proof}

\begin{proof}
[Prove of Lemma~\ref{lema:38}] consider numbers
$\alpha_1,\ldots,\, \alpha_t$ from Lemma~\ref{lema:41}. Since for
every $p$ the maps $h$ is differentiable on $f(\alpha_p,\,
\alpha_{p+1})$,we may repeat the proof of Lemma~\ref{lema:41} for
the intervals, whose union is $f(\alpha,\, beta)$. Evidently,
there is a finite $k$ such that $f^k(\alpha,\, \beta) = (0,\, 1)$
and we will divide the each interval into a finitely many sub
intervals, whence the necessary $A_1,\ldots,\, A_s$ would appear.
\end{proof}

We will need the following two remarks for the further reasonings.

\begin{note}\label{note:15}
The set of periodical points of $f$ is dense in $[0,\, 1]$.
\end{note}

\begin{note}\label{note:16}
Non of numbers $k_1,\ldots,\, k_t$ from Lemma~\ref{lema:38} does
not equal to 0.
\end{note}

\begin{proof}
[Proof of Remark~\ref{note:15}] The graph of the $n$-th iteration
of $f$ consists of $2^n$ line segments, each of them has tangent
either $2^n$, or $-2^n$ and maps some subinterval of $[0,\, 1]$
onto the whole $[0,\, 1]$. The lengthes of ``domains'' of these
line segments are $\frac{1}{2^n}$ and tend to $0$ if $n\rightarrow
\infty$. Since each of these line segments intersects the line
$y=x$, then correspond domain contains a fixes point of $f^n$,
which is periodical for $f$.
\end{proof}

\begin{proof}
[Proof of Remark~\ref{note:16}] The equality $k_i=0$ yields that
$h$ is constant on $f(A_i)$, which contradicts to that $h$ is a
homeomorphism.
\end{proof}

Denote $\mathbb{A} =\bigcup\limits_{i=1}^{s}A_i$. Since the set
$[0,\, 1]\backslash \mathbb{A}$ is finite, then the following
corollary follows from Remark~\ref{note:15}.

\begin{corollary}
The set of periodical points of $f$, whose trajectories belong to
$\mathbb{A}$, is dense in $[0,\, 1]$.
\end{corollary}

\begin{lemma}
Let $x^*$ be a periodical point of $f$, such that its trajectory
belongs to $\mathbb{A}$ and $h'(x^*)\neq 0$. Then there is a
neighborhood of $x^*$, where the derivative $h'$ is constant.
\end{lemma}

\begin{proof}
Let $n$ be a period of $x^*$. Denote by $x_i$ the trajectory of
$x^*$, i.e. $x_i = f^i(x^*)$, and $x_0 = x^* = x_n$. Since the
trajectory of $x^*$ belongs to $\mathbb{A}$, then for every $i$
the derivative $h'(x_i)$ exists. It follows from
Lemma~\ref{lema:38} and Remark~\ref{note:16} that $h'(x_i)\neq 0$
for every $i$. Let $A_{c_0},\, A_{c_1},\ldots,\, A_{c_{n-1}}$ be
the sets from Lemma~\ref{lema:37}, which contains the points of
trajectory of $x^*$, i.e. $x_i\in A_{c_i}$. Then it follows from
Lemma~\ref{lema:38} that $ h'(x^*) \hm{=} k_{c_0}\cdot
k_{c_1}\cdot \ldots \cdot k_{c_{n-1}} h'(x^*). $ Since
$h'(x^*)\neq 0$, then
\begin{equation}\label{eq:84} k_{c_0}\cdot k_{c_1}\cdot \ldots
\cdot k_{c_{n-1}} = 1.
\end{equation}

Since the set $A_{c_0}$ is open and $x^*\in A_{c_0}$, then there
exists $\varepsilon>0$, such that $(x^*-\varepsilon,\,
x^*+\varepsilon)\subset A_{c_0}$. Since the trajectory of $x^*$
belongs to $\mathbb{A}$ then without loss of generality we may
assume that the first $n$ points of the trajectory of each point
from $(x^*-\varepsilon,\, x^*+\varepsilon)$ also belong to
$\mathbb{A}$. Precisely, for each point $\widetilde{x}\in
(x^*-\varepsilon,\, x^*+\varepsilon)$ and for every $i,\, 1\leq
i\leq n$ the inclusion $f^i(\widetilde{x})\in A_{c_i}$ holds.
Without loos of generality assume that $f^n$ is linear on
$(x^*-\varepsilon,\, x^*+\varepsilon)$, because otherwise decrease
$\varepsilon$. Notice that $f^n(x^*) = x^*$. There is a
neighborhood of $x^*$, where $f^n$ is given either by $f^n(x) =
2^n(x-x^*)$, or $f^n(x) = -2^n(x-x^*)$.

Assume that $f^n(x) =  2^n(x-x^*)$ in some neighborhood of $x^*$.
Consider an arbitrary $\widetilde{x}\in (x^*-\varepsilon,\, x^*)$
and the sequence of its pre images $\widetilde{x}_i$, which is
given by the equalities $\widetilde{x}_{0} = x^*$ and
$\widetilde{x}_{i+1} = (f^n)^{-1}(\widetilde{x}_i)$, where
$(f^n)^{-1}$ means the maps, which is inverse to $f^n$ on
$(x^*-\varepsilon,\, x^*+\varepsilon)$. Since the derivative of
the maps $f^n$ on the correspondent interval is $>1$, then the
sequence ${\widetilde{x}_i}$ increase and tends to $x^*$. It
follows from the equality~(\ref{eq:84}) and Lemma~\ref{lema:38}
that $h'(\widetilde{x}_i) = h'(\widetilde{x}_{i-1})$ for every
$i\geq 0$. Sice $h$ is continuously differentiable, then it
follows from that $h'(\widetilde{x}_i)=h'(x^*)$ for every $i\geq
0$ and $\lim\limits_{i \rightarrow \infty}\widetilde{x}_i = x^*$
that $h'(\widetilde{x})=h'(x^*)$. The arbitrariness of
$\widetilde{x}\in (x^*-\varepsilon,\, x^*)$ yields that derivative
of $h$ is constant on $(x^*-\varepsilon,\, x^*)$.

The case when $f^n(x) = -2^n(x-x^*)$ is similar if consider the
images instead of pre images of $\widetilde{x}\in
(x^*-\varepsilon,\, x^*)$.
\end{proof}

Now Theorem~\ref{theor-dyfer-lin} follows from the proved lemmas.

\begin{proof}
[Proof of Theorem~\ref{theor-dyfer-lin}] By Lemma~\ref{lema:40}
let $[a,\, b]\subset [0,\, 1]$ be an interval,  where the
derivative $h'$ is constant, denote $p_1$. Without loos of
generality assume that $[a,\, b]\subset A_i$ for some $i,\, 1\leq
i\leq s$ where $\{ A_i\}$ are defined in Lemma~\ref{lema:37}. By
Lemma~\ref{lema:38}, there exists $k_i\in \mathbb{R}$ (which is
dependent only on $i$), such that $h'(f(x)) = k_if(x)$ for every
$x\in A_i$. Now we can write $ f([a,\, b]) \hm{=}
\left(\bigcup\limits_{i=1}^s \left[ A_i\cap f([a,\, b])
\right]\right)\cup C, $ where $C$ is a finite set whence $h$ is
piecewise linear on $f([a,\, b])$. Since tangents of $f$ are $2$
and  $-2$ only, then $\{ f(x)|\, x\in [a,\, b]\} $ is of
Lebesgue's measure $\min\{ 2(b-a),\, 1\}$ and Theorem follows from
that $[0,\, 1]$ is an image of $[a,\, b]$ under finitely many
iterations.
\end{proof}

Lemma below follows from Theorem~\ref{theor:10}.

\begin{lemma}\label{lema:45}
1. Homeomorphism $h$ increase and $h(1/2)=v$;

2. The equation~(\ref{eq:eq-top-eq-ffv}) can be rewritten as
$$\left\{
\begin{array}{ll} h(f_l(x)) =
g_l(h(x))& \text{ for }x\leqslant 1/2,\\
h(f_r(x)) = g_l(h(x))& \text{ for }x\geqslant 1/2.
\end{array}\right.$$
\end{lemma}

The following propositions contain some properties of $g$, which
follow from that conjugacy is piecewise linear.

\begin{proposition}\label{lema:36}
The maps $g$ has exactly two fixed points: one of these points is
zero and another belongs to $(v,\, 1)$.
\end{proposition}

\begin{proof}The existence of the mentioned points
follows from the construction of $g$ and the theorem about the
mean value of a continuous function.

Prove that there is no other fixed points of $g$.

Non-existence of other fixed points of $g$ on $(v,\, 1)$ follows
from that $g$ decrease on $(v,\, 1)$ and the fixed point if the
point of intersection of the graph of $g$ with $y=x$.

Assume that $g$ has a fixed point $x_0$ on $(0,\, v)$.
Applying~(\ref{eq:eq-top-eq-ffv}) for $x=h^{-1}(x_0)$, obtain that
$h^{-1}(x_0)$ is a fixed point of $f_l$ on $(0,\, 1/2)$, which
contradicts to that $f$ has no fixed points on this interval/.
\end{proof}

\begin{corollary}\label{corol:4}
Trajectory of an arbitrary $x\in (0,\, 1)$ tends to $0$ under
$g_l^{-1}$.
\end{corollary}

\subsection{Piecewise linearity of the conjugation}\label{sect:KuskowoLin-3}

Remind that it is obtained in Theorem~\ref{theor:ksklin-tpleqv}
that if the map $g:\, [0,\, 1]\rightarrow [0,\, 1]$ of the
form~(\ref{eq:g}) is conjugated to $f:\, [0,\, 1]\rightarrow [0,\,
1]$ of the form~(\ref{eq:89}) via piecewise linear conjugacy then
$g'(0) =2$. Theorem~\ref{theor:ksklin-tpleqv} was the inspiration
of this Section.

\begin{theorem}\label{theor:21}
For an arbitrary $v\in (0,\, 1)$ and increasing piecewise linear
maps $g:\, [0,\, v]\to [0,\, 1]$ such that $g(0)=0$, $g(v)=1$ and
$g'(0)=2$ there exists it's unique continuation $\widetilde{g}:\,
[0,\, 1]\to [0,\, 1]$, which is conjugated to $f$ of the
form~(\ref{eq:89}) via piecewise linear conjugation.
\end{theorem}

\begin{theorem}\label{theor:22}
For an arbitrary $v\in (0,\, 1)$ and arbitrary decreasing
piecewise linear maps $g:\, [v,\, 1]\mapsto [0,\, 1]$ such that
$g(v)=1$, ${g(1)=0}$ and $(g^2)'(x_0) = 4$, where $x_0$ is a
positive fixed points of $g$, there exists it's unique
continuation $\widetilde{g}:\, [0,\, 1]\rightarrow [0,\, 1]$,
which is topologically conjugated to $f$ of the form~(\ref{eq:89})
via piecewise linear conjugation.
\end{theorem}

We will prove Theorems~\ref{theor:21} and~\ref{theor:22}
constructively. In the proof of Theorem~\ref{theor:21} we will
show that the increasing part of $g$ defines uniquely the
piecewise linear conjugacy $h$, which
satisfies~(\ref{eq:eq-top-eq-ffv}). After this,
equation~(\ref{eq:eq-top-eq-ffv}) defines uniquely the map $g$ via
$f$ and the found conjugacy. Theorems~\ref{theor:21}
and~\ref{theor:22} also can be understood in the manner that $g_l$
itself and $g_r$ itself defines the conjugacy $h$. Nevertheless,
Theorem~\ref{theor:ksklin-tpleqv} states, that $g$ can not be
arbitrary, for instance, $g_l'(0)=2$.

\begin{notation}\label{notation:3}
For the piecewise linear conjugacy of $f$ of the
form~(\ref{eq:89}) and unimodal piecewise linear map $g$, denote
by $k$ the derivative of $h$ at $0$. Also denote by $\varepsilon$
the $x$-coordinate of the first break of $g$, i.e. a). $h(x)=kx$
at some neighborhood of $x=0$; b).~$g(x)=2x$ for all $x\in (0,\,
\varepsilon)$ and c). $x = \varepsilon$ is a break point of $g$.
\end{notation}

Let $A_1,\ldots,\, A_s$ be open intervals of linearity of $g$ on
$[0,\, v]$ and $\bigcup\limits_{i=1}^s \overline{A}_i = [0,\, v]$.

\begin{lemma}\label{lema:42}
For every $x\in [0,\, \frac{\varepsilon}{k}]$ the equality $h(x) =
kx$ holds. Moreover, $\frac{\varepsilon}{k}$ is a breakpoint of
$h$.
\end{lemma}

\begin{proof}
Since $h$ is piecewise linear, then there exists $\delta$ such
that the equality~(\ref{eq:eq-top-eq-ffv}) is of the form
$h(2x)=2\cdot kx$ for $x\in [0,\, \delta]$ and $\delta$ is
determined from that $g(x)=2x$ for all $x\in (0,\, k\delta)$. From
another hand, it follows from the formula for $h$ on $x\in [0,\,
\delta]$ that $h(x)=kx$ for all $x\in (0,\, 2\delta)$ and this
lets to consider the functional equation $h(2x)= g(kx)$ for $x\in
(0,\, 2\delta)$. For these $x$ the range of $g(kx)$ is $[0,\,
y_0]$, where $y_0 = h(4\delta) = g(2k\delta)$. The maximum length
of the interval for $x$, such that functional
equation~(\ref{eq:eq-top-eq-ffv}) is of the form $h(2x)=2\cdot kx$
is $x\in [0,\, \frac{\varepsilon}{2k}]$. Further increasing of the
interval for $x$ leads to that functional
equation~(\ref{eq:eq-top-eq-ffv}) transforms to $h(2x) = a_2\cdot
kx + b_2$, where $g(x)=a_2x+b_2$ is the formula for $g$ on the
next interval of linearity of $g$ after $(0,\, \varepsilon)$.
\end{proof}

In fact, Lemma~\ref{lema:42} states that $g_l$ defines the length
of the first interval of linearity of the conjugacy $h$
dependently on its tangent on its the first interval of linearity.

\begin{lemma}\label{lema:43}
Let for some $v\in (0,\, 1)$ and piecewise linear maps $g,\,
\widetilde{g}:\, [0,\, 1]\rightarrow [0,\, 1]$, which are
conjugated to $f$, equalities $g(v) = \widetilde{g}(v) = 1$ and
$g(0) = g(1) = \widetilde{g}(0) = \widetilde{g}(1) = 0$ holds and,
furthermore, $g(x) = \widetilde{g}(x)$ for all $x\in [0,\, v]$.
Then tangents of correspond conjugacies at $0$ coincide.
\end{lemma}

This lemma is a partial case of Theorem~\ref{theor:21}. We will
prove Lemma~\ref{lema:43} by calculating the tangent at $0$ for
the conjugacy $h$ of $f$ and $g$. We will use in our calculations
only maps $g_l$, but not $g_r$. We will use one more notation.

\begin{notation}
Let $G_l(\alpha,\, \beta) = (f_l^{-1}(\alpha),\, g_l^{-1}(\beta))$
and $G_r(\alpha,\, \beta) = (f_r^{-1}(\alpha),\, g_r^{-1}(\beta))$
be maps of the square $[0,\, 1]^2$ to itself.
\end{notation}

The following lemma follows from the increasing of $h$.

\begin{lemma}\label{lema:44}
Let numbers $\alpha,\, \beta\in [0,\, 1]$ be such that $h(\alpha)
= \beta$. Then the graph of $h$ is invariant under the action of
$G_l$ and $G_r$, i.e. $h(f_l^{-1}(\alpha)) = g_l^{-1}(\beta)$ and
$h(f_r^{-1}(\alpha)) = g_r^{-1}(\beta)$.
\end{lemma}

\begin{proof}
[Proof of Lemma~\ref{lema:43}] Notice that $h(1)=1$ and consider
the trajectory of $(1,\, 1)$ under $G_l$, obtaining $G_l^{n}(1,\,
1) \hm{=} ( 1/2^n,\, g_l^{-n}(1) ). $ By Lemma~\ref{lema:44}
obtain that $h(1/2^n) \hm{=} g_l^{-n}(1)$ for all $n$. By
Corollary~\ref{corol:4}, the sequence $\{ g_l^{-n}(1)\}$ decrease
to 0. Find the tangent $k$ of the conjugacy $h$ at the
neighborhood of $0$. By Lemma~\ref{lema:42}, $h(x)=kx$ for $x\in
(0,\, \varepsilon/k)$, where $\varepsilon$ is taken from
Notation~\ref{notation:3}. For $n$ huge enough the inclusion
$1/2^n\in (0,\, \varepsilon/k)$, which means that the sequence
$k_n = 2^ng_l^{-n}(1)$ stabilizes on $k$.
\end{proof}

\begin{note}
Notice that in the same way as it was done in the proof of
Lemma~\ref{lema:43}, we can show that for every $\alpha,\,
\beta\in (0,\, 1)$ such that $h(\alpha)=\beta$ the sequence $k_n =
\frac{2^ng_l^{-n}(\beta)}{\alpha}$ stabilizes independently on
whether $h$ is piecewise linear, or not. Piecewise linearity of
$h$ means that for every pair $\alpha,\, \beta\in (0,\, 1)$ the
obtained sequence stabilizes on those $k$, which is obtained in
the proof of Lemma~\ref{lema:43} and is independent on $\alpha$
and $\beta$.
\end{note}

\begin{proof}
[Proof of Theorem~\ref{theor:21}] Let $h$ be the conjugacy of $f$
and $g$. By reasonings from the proof of Lemma~\ref{lema:43}, find
by $g_l$ the tangent $k$ of $h$ at $0$. Then by
Lemma~\ref{lema:42} the conjugacy is expressed by $h(x)=kx$ for
$x\in [0,\, \frac{\varepsilon}{k}]$, where $\varepsilon$ is the
first break point of $g$. Consider $n$, such that $\frac{1}{2^n}<
\frac{\varepsilon}{k}$. For every $t\geq 0$ from the
equation~(\ref{eq:eq-top-eq-ffv}) for $x\in \left[
\frac{1}{2^{n-t+1}},\, \frac{1}{2^{n-t}}\right]$ and from values
of $h$, known earlier on $\left[ \frac{1}{2^{n-t+1}},\,
\frac{1}{2^{n-t}}\right]$, find the values of $h$ on $\left[
\frac{1}{2^{n-t+1}},\, \frac{1}{2^{n-t}}\right]$. For $t=n-1$ the
equation~(\ref{eq:eq-top-eq-ffv}) will be defined for $x\in \left[
\frac{1}{4},\, \frac{1}{2}\right]$, whence we find $h$ on the
whole $[0,\, 1]$. After this the map $g_r$ can be found from
$h(f_r(x)) = g_r(h(x))$ for $x\in [1/2,\, 1]$ by the equation
$g_r(x) = h(f_r(h^{-1}(x)))$. This finishes the proof.
\end{proof}

Proposition~\ref{theor:ksklin-tpleqv} contains the restriction on
$g$ in the neighborhood of $0$ for existing of piecewise linear
conjugacy $h$ of $f$ and $g$. The analog of this proposition holds
for the positive fixed point of~$g$.

\begin{proposition}\label{theor:23}
Let $x_0$ be a fixed point of unimodal piecewise liner map $g$,
which is conjugated to $f$ of the form~(\ref{eq:89}) via piecewise
linear conjugacy. Then there exists $\varepsilon>0$, such that for
every $x\in (x_0-\varepsilon,\, x_0+\varepsilon)\backslash
\{x_0\}$ the equality $(g^2)'(x_0)=4$ holds, where $g^2$ is the
second iteration of $g$.
\end{proposition}

We know that $h(2/3) = x_0$, since $2/3$ is a fixed point of $f$.
Let at the left neighborhood of $x_0$ the map $g$ is given by
$g_1(x)=a_1x + b_1$, and in the right neighborhood of $x_0$ it is
given by $g_2(x)=a_2x + b_2$.

\begin{note}\label{note:17}
Using this notations claim that Proposition~\ref{theor:23} is
equivalent to $a_1a_2=4$.
\end{note}

\begin{proof}
[Explanation of Remark~\ref{note:17}.] Since $g$ decrease on
$[v,\, 1]$, then it is enough to prove that there exists
$\varepsilon>0$ such that $g(x_0,\, x_0+\varepsilon)\subset (v,\,
x_0)$ and $g(x_0-\varepsilon,\, x_0)\subset (x_0,\, 1)$.

Then for the left neighborhood $(x_0-\varepsilon,\, x_0)$ the maps
$g^2$ is defined by $g^2(x)=g_2(g_1(x))$, and in the right
neighborhood $(x_0,\, x_0+\varepsilon)$ it is defined by
$g^2(x)=g_1(g_2(x))$. Nevertheless, in both first and second case
the maps is linear in $x_0$ and its tangent is~$a_1a_2$.
\end{proof}

\begin{proof}
[Proof of Proposition~\ref{theor:23}] Left in the left
neighborhood of $2/3$ the maps $h$ is given by $h_1(x)=a_3x +
b_3$, and in the right neighborhood of $2/3$ it is given by
$h_2(x)=a_4x + b_4$.

For an $\varepsilon$ small enough consider an arbitrary number
$\beta\in (x_0,\, x_0+\varepsilon)$ and consider the sequence
$G_r^n(\alpha,\, \beta)$, where $\beta = h(\alpha)$. This sequence
converges to $(2/3,\, x_0 )$. The first three elements of the
sequence are $(\alpha,\, \beta) \rightarrow ( f_r^{-1}(\alpha),\,
g_r^{-1}(\beta) ) \hm{\rightarrow} ( f_r^{-2}(\alpha),\,
g_r^{-2}(\beta) )$, i.e. $( \alpha,\, \beta) \rightarrow
(f_r^{-1}(\alpha),\, g_1^{-1}(\beta)) \rightarrow (
f_r^{-2}(\alpha),\, g_2^{-2}(\beta)), $ because $\beta>x_0$ and
the map $g$ decrease in the neighborhood of~$x_0$.

Express the maps $g_r^{-1}(\beta)$ and $g_r^{-2}(\beta)$ as
$\displaystyle{g_r^{-1}(\beta) = \frac{\beta-b_1}{a_1}}$ and $
g_r^{-2}(\beta) = \frac{\frac{\beta-b_1}{a_1}-b_2}{a_2}
\hm{=}\frac{\beta-b_1 -a_1b_2}{a_1a_2}.$ Moreover,
$\displaystyle{f_r^{-1}(\alpha) =1 -\frac{\alpha}{2}}$ and
$\displaystyle{f_r^{-2}(\alpha) =\frac{1}{2}+\frac{\alpha}{4}}$.
Also the following equalities hold.
$$\left\{
\begin{array}{l}
\beta = a_4\alpha +b_4,\\
g_r^{-1}(\beta) =a_3f_r^{-1}(\alpha) + b_3,\\
g_r^{-2}(\beta) =a_4f_r^{-2}(\alpha) + b_4.
\end{array}\right.
$$

By plugging of the first two equalities to the third one gives
\begin{equation}\label{eq:85} \frac{(a_4\alpha +b_4)-b_1
-a_1b_2}{a_1a_2} =a_4\left( \frac{1}{2}+\frac{\alpha}{4}\right) +
b_4. \end{equation}

If follows from the arbitrariness of $\beta$ in correspond
neighborhood of $x_0$ and from the continuity of $h$ that the set
of $\alpha$, such that $h(\alpha)=\beta$ is some neighborhood of
$2/3$. That is why, the equality~(\ref{eq:85}) is the identity in
correspond neighborhood of $\alpha$. Saying that coefficients near
$\alpha$ in the left and right hand side of~(\ref{eq:85}) are
equal, obtain that $a_1a_2=4$.
\end{proof}

\begin{proof}
[Proof of Theorem~\ref{theor:22}] Let $x_0$ be positive fixed
point of $g$ and $\widetilde{g}$. Let $h$ be a conjugation of $f$
and $g$. Since $h(2/3)=x_0$ by Lemma~\ref{lema:45}, then there
exists $\varepsilon>0$ such that $h(x)=k(x-2/3)+x_0$ for all $x\in
(2/3,\, 2/3+\varepsilon)$. Consider the trajectory of $(1,\, 1)$
under $G_l$. Denote by $(x_n,\, y_n )$ the $n$-th iteration of
this point, i.e. $G_l^{n}(1,\, 1) = (x_n,\, y_n ).$ It follows
from Lemma~\ref{lema:44} that for each point of the trajectory one
have that
\begin{equation}\label{eq:86}h(x_n)=y_n.
\end{equation}

The sequence $x_n$ can be given by $x_n = f_r^{-1}(x_{n-1})$ and
tends to $2/3$. Furthermore, for every $k\in \mathbb{N}$ the
inclusions $(f_r^{-1})^{2k}(1)\in (1/2,\, 2/3)$ and
$(f_r^{-1})^{2k+1}(1)\in (2/3,\, 1)$ holds and
$\lim\limits_{i\rightarrow \infty}x_n = 2/3$. Since for huge odd
$n$ the inclusion $x_n\in (2/3,\, 2/3+\varepsilon)$ holds, then
the equality~(\ref{eq:86}) can be rewritten as $y_n =
k(x_n-2/3)+x_0$, whence
\begin{equation}\label{eq:88}k =
\frac{y_n-x_0}{x_n-2/3}.\end{equation} Moreover, by
Proposition~\ref{theor:23} without loss of generality we may
assume that for these $n$ and all $y_n$ then maps $g^{2}$ is given
by equality
\begin{equation}\label{eq:87}g^{2}(x)
=4(x-x_0)+x_0.\end{equation}

Similarly as in the proof of Proposition~\ref{theor:23} we can
show that if $x_n\in (2/3,\, 2/3+\varepsilon)$ then
 $
\frac{y_n-x_0}{x_n-2/3} = \frac{y_{n+2}-x_0}{x_{n+2}-2/3}.$

It follows from equation~(\ref{eq:87}) that $y_{n+2} =
\frac{y_n+3x_0}{4}.$ Also $x_{n+2}$ can be found from the formula
$x_{n+2} = (f_r^{-1})^2(x_n)$, i.e. $x_{n+2}
=\frac{1}{2}+\frac{x_n}{4}.$ Plugging of these expressions into
$\frac{y_{n+2}-x_0}{x_{n+2}-2/3}$ gives, after simplifications
that $ \frac{y_{n+2}-x_0}{x_{n+2}-2/3} =
\frac{y_n-x_0}{x_n-\frac{2}{3}}.$ Thus, $h(2/3)=x_0$ and $h$ has
the tangent $k$, defined by~(\ref{eq:88}) for $x\in (2/3,\,
2/3+\varepsilon)$.

Denote $\mathbb{A}_1 = (2/3,\, 2/3+\varepsilon)$ and consider the
sequence of sets $\mathbb{A}_{i+1} = f(\mathbb{A}_i)$. Clearly,
$f^t(\mathbb{A}_1) = [0,\, 1]$ for some finite $t\in \mathbb{N}$.
For every $x\in \mathbb{A}_i$ the
equation~(\ref{eq:eq-top-eq-ffv}) lets to determine the values of
$h$ on $\mathbb{A}_{i+1}$ in the assumption that its values on
$\mathbb{A}_i$ is known. Sequent applying of these reasonings let
to find $h$ on the whole $[0,\, 1]$. In fact, we have proved that
the maps $g$ defines $h$ by it values on $[v,\, 1]$ and, then $h$
defines the increasing part of $g$. The last finishes the proof.
\end{proof}

\subsection{Example of non convex maps, which
is conjugated to standard hat-map}\label{arch-sect-Non-convex}

It is assumed in Theorem~\ref{Theor:9}, that the function $g$,
which is conjugated to $f$, is convex. Nevertheless, convexity is
not used in the proof of this Theorem.

In this section we will use in details the techniques from the
proof of Theorems~\ref{theor:21} and~\ref{theor:22} for obtaining
the example of non-convex $g$, which is conjugated to $f$.

Let $g_1:\, [0,\, 5/8]\rightarrow [0,\, 1]$ be monotone,
increasing, piecewise linear such that $(1/4,\, 1/2)$, $(1/2,\,
5/8)$, $(5/8,\, 1)$ are all breaking points of $g_1$. Evidently,
this $g_1$ is not convex. Since the proof of
Theorem~\ref{theor:21} is constructive, we may apply it to find
the maps $g_2:\, [5/8,\, 1]\rightarrow [0,\, 1]$ such that $g$ of
the form~(\ref{eq:g}), which is constructed with these $g_1$ and
$g_2$ would be conjugated to $f$. Assume, that $g$ is already
constructed and $h:\, [0,\, 1] \rightarrow [0,\, 1]$ is a
homeomorphic solution of the functional
equation~(\ref{eq:eq-top-eq-ffv}).

Notice, that it follows from Lemma~\ref{lema:45}, that $h$
increase.

\begin{lemma}\label{arch-lema-3}
$h(x)=2x$ for all $x\in [0,\, 1/4]$ and $(1/4,\, 1/2)$ is a
breaking points of $h$.
\end{lemma}

\begin{proof}
In fact, this Lemma means that $k$ from Lemma~\ref{lema:42} equals
to 2.

With the use of Lemma~\ref{lema:44}, consider the trajectory of
$1$ under $g_l^{-1}(1)$ and notice, that since $\frac{1}{2^n} =
f_l^{-n}(1)$, whence $h\left(\frac{1}{2^n}\right) = g_l^{-n}(1)$.

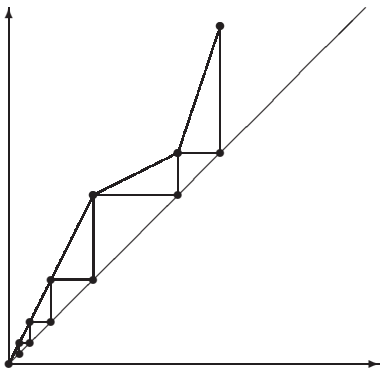
\begin{figure}[htbp]
\begin{center}
\begin{center}
\begin{picture}(140,135)
\put(0,0){\vector(1,0){140}} \put(0,0){\vector(0,1){135}}
\put(0,0){\line(1,1){135}}

\put(80,128){\circle*{3}}
\Vidr{80}{128}{80}{80}\put(80,80){\circle*{3}}
\VidrTo{64}{80}\put(64,80){\circle*{3}}
\VidrTo{64}{64}\put(64,64){\circle*{3}}
\VidrTo{32}{64}\put(32,64){\circle*{3}}
\VidrTo{32}{32}\put(32,32){\circle*{3}}
\VidrTo{16}{32}\put(16,32){\circle*{3}}
\VidrTo{16}{16}\put(16,16){\circle*{3}}
\VidrTo{8}{16}\put(8,16){\circle*{3}}
\VidrTo{8}{8}\put(8,8){\circle*{3}}
\VidrTo{4}{8}\put(4,8){\circle*{3}}
\VidrTo{4}{4}\put(4,4){\circle*{3}}

\put(0,0){\circle*{3}} \Vidr{0}{0}{32}{64}\put(32,64){\circle*{3}}
\VidrTo{64}{80}\put(64,80){\circle*{3}}
\VidrTo{80}{128}\put(80,128){\circle*{3}}

\end{picture}
\end{center}
\end{center} \caption{Graph of $g_l$}\label{fig:26}
\end{figure}

The trajectory of $1$ under $g_l^{-1}$ is given at
Figure~\ref{fig:26}. It is clear, that $g^{-3}(1)$ is the
$x$-coordinate if the first breaking points of $g_l$, whence
$h(1/8)=1/4$ and $k=2$.
\end{proof}

Let $A,\, B\subseteq [0,\, 1]$ are closed intervals. Denote by $
\xymatrix{ A \ar@{=>}^{\varphi}[r] & B }$ the fact the maps
$\varphi$ is linear and increasing on $A$ such that $\varphi(A) =
B$. Also denote by $ \xymatrix{ A \ar@{-->}^{\varphi}[r] & B }$
that $\varphi(A) = B$, $\varphi$ is linear on $A$, but decreasing.

In this notation Lemma~\ref{arch-lema-3} means that following
diagram
$$
\xymatrix{ [0,\, 1/8] \ar@{=>}^{f}[r] \ar@{=>}_{h}[d]
 & [0,\, 1/4] \ar@{=>}^{h}[d] \\
[0,\, 1/4] \ar@{=>}^{g}[r]
& [0,\, 1/2] \\
 }$$
is commutative. The following technical lemma is obvious.

\begin{lemma}\label{arch-lema-4}
Assume that for intervals $A$, $B$, $C$ and $D$ and maps
$\varphi_1$, $\varphi_2$, $\varphi_3$ and $\varphi_4$ the
following diagram
$$
\xymatrix{ A \ar@{=>}^{\varphi_1}[r] \ar@{=>}_{\varphi_2}[d]
 & B \ar^{\varphi_4}[d] \\
C \ar@{=>}^{\varphi_3}[r]
& D \\
 }$$
is commutative. Then $ \xymatrix{ B \ar@{=>}^{\varphi_4}[r] & D
}$, i.e. $\varphi_4$ is linear on $B$.
\end{lemma}

\begin{lemma}\label{arch-lema-5}
All breaking points of $h$ are $(1/4,\, 1/2)$ and $(1/2,\, 5/8)$.
\end{lemma}

\begin{proof}
By Lemma~\ref{arch-lema-3} the following diagram
\begin{equation}\label{arch-eq-03}
\xymatrix{ [1/8,\, 1/4] \ar@{=>}^{f}[r] \ar@{=>}_{h}[d]
 & [1/4,\, 1/2] \ar^{h}[d] \\
[1/4,\, 1/2] \ar@{=>}^{g_1}[r]
& [1/2,\, 5/8] \\
 }\end{equation} is commutative. By Lemma~\ref{arch-lema-4} the
 diagram~(\ref{arch-eq-03}) can be continued as
 $$\xymatrix{ [1/8,\, 1/4] \ar@{=>}^{f}[r] \ar@{=>}_{h}[d]
 & [1/4,\, 1/2] \ar@{=>}^{h}[d] \ar@{=>}^{f}[r] & [1/2,\, 1]
 \ar^{h}[d]\\
[1/4,\, 1/2] \ar@{=>}^{g_1}[r]
& [1/2,\, 5/8] \ar@{=>}^{g_1}[r] & [5/8,\, 1].\\
 }$$ Repeating application of Lemma~\ref{arch-lema-4} finishes the
 proof.
\end{proof}

\begin{figure}[htbp]
\begin{center}
\begin{center}
\begin{picture}(140,135)
\put(0,0){\vector(1,0){140}} \put(0,0){\vector(0,1){135}}
\put(0,0){\line(1,1){135}}

\put(0,0){\circle*{3}} \Vidr{0}{0}{32}{64}
\put(32,64){\circle*{3}} \VidrTo{64}{80} \put(64,80){\circle*{3}}
\VidrTo{80}{128} \put(80,128){\circle*{3}} \VidrTo{104}{80}
\put(104,80){\circle*{3}} \VidrTo{116}{64}
\put(116,64){\circle*{3}} \VidrTo{128}{0} \put(128,0){\circle*{3}}

\end{picture}
\end{center}
\end{center} \caption{Graph of $g$}\label{arch-fig:2}
\end{figure}
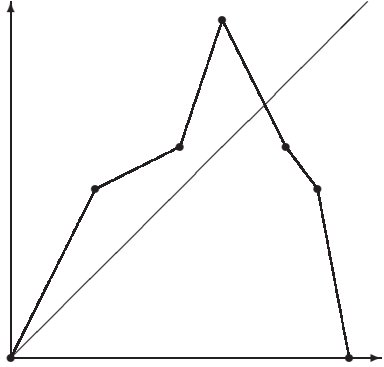

\begin{lemma}
All the breaking points of $g_2$ are $(13/16,\, 5/8)$ and
$(29/32,\, 1/2)$.
\end{lemma}

\begin{proof}
The maps $g_2$ can be found from the commutative diagram
\begin{equation}\label{arch-eq-04} \xymatrix{ [1/2,\, 1] \ar@{-->}^{f}[r] \ar@{=>}_{h}[d]
 & [0,\, 1] \ar^{h}[d] \\
[5/8,\, 1] \ar@{-->}^{g_2}[r]
& [0,\, 1] \\
 }
\end{equation} By Lemma~\ref{arch-lema-5}, diagram~(\ref{arch-eq-04}) may be
rewritten as the following three diagrams.
$$
\xymatrix{ [1/2,\, 3/4] \ar@{-->}^{f}[r] \ar@{=>}_{h}[d]
 & [ 1/2,\, 1] \ar@{=>}^{h}[d] & [3/4,\, 7/8] \ar@{-->}^{f}[r] \ar@{=>}_{h}[d]
 & [1/4,\, 1/2] \ar@{=>}^{h}[d] & [7/8,\, 1]  \ar@{-->}^{f}[r] \ar@{=>}_{h}[d]
 & [0,\, 1/4] \ar@{=>}^{h}[d] \\
[ 5/8,\, 13/16] \ar^{g_2}[r] & [5/8,\, 1] & [13/16,\, 29/32]
\ar^{g_2}[r] & [1/2,\, 5/8] & [29/32,\, 1]  \ar^{g_2}[r] & [0,\,
1/2] \\
 }$$
This finishes the proof.
\end{proof}

The graph of the maps $g$, constructed above, is given at
Figure~\ref{arch-fig:2}.

\subsection{Types of piecewise linearity}\label{sect:type-linear}

The example in Section~\ref{arch-sect-Non-convex} contained the
maps $g:\, [0,\, 1]\rightarrow [0,\, 1]$ of the form~(\ref{eq:g}),
whose increasing and decreasing part contained 3 parts of
linearity each.

\begin{definition}\label{def:03}
Let for the piecewise linear $g$ of the form~(\ref{eq:g}) be
topologically conjugated to $f$ of the form~(\ref{eq:89}). Let the
number of pieces of linearity of $g_1$ and $g_2$ be $p$ and $q$
correspondingly. Call the pair $(p,\, q)$ the \textbf{type of
piecewise linearity} of $g$.
\end{definition}

\begin{definition}\label{def:04}
If for a pair $(p,\, q)$ there exists a maps $g$ of the
form~(\ref{eq:g}) with type of piecewise linearity $(p,\, q)$,
then call this type \textbf{admissible}. If follows from
Proposition~\ref{theor:ksklin-tpleqv} and~\ref{theor:23} that pair
$(1,\, q)$ is admissible only if $q=1$ and the pair $(p,\, 1)$ is
admissible only if $p=1$. In both cases the equality $g=f$ holds.
\end{definition}

\begin{lemma}\label{arch-lema-6}
For any $q\geq 2$ the type $(2,\, q)$ is admissible.
\end{lemma}

\begin{proof}
For an arbitrary $n\in \mathbb{N},\, n>2$  consider the maps
$g_1$, whose graph pass through points $(0,\, 0)$,
$(\frac{1}{n},\, \frac{2}{n})$ and $(\frac{n-1}{n},\, 1)$. Let
$h:\, [0,\, 1]\rightarrow [0,\, 1]$ be a piecewise linear
homeomorphism, which defines he conjugation of $f$ and $g$ with
$g_1$. The existence and uniqueness of this $h$ follows from
Theorem~\ref{theor:21}. We will find the correspond $g_2$ and
whence calculate the type of linearity of $g$.

Notice, that $$ g_1^{n-1}\left(\frac{1}{n}\right) = 1.
$$

Since $h$ increase and satisfy~(\ref{eq:eq-top-eq-ffv}), we have
\begin{equation}\label{arch-eq-05}
h\left(\frac{1}{2^t}\right) = 1 - \frac{t}{n},\ 1\leq t\leq n-1.
\end{equation}

For instance, for $t=n-1$ we have $$
h\left(\frac{1}{2^{n-1}}\right) = \frac{1}{n}.
$$

The reasonings, which are similar to those in the prove of
Lemmas~\ref{lema:45} and~\ref{arch-lema-3} give that $$
h(x)=\frac{2^{n-1}}{n}\, x,\ x\in \left[0,\,
\frac{1}{2^{n-2}}\right].
$$

It is evident by induction by $k$, that the following diagram
$$ \xymatrix{\left[\frac{2^k}{2^{n-1}},\,
 \frac{2^{k+1}}{2^{n-1}}\right] \ar@{=>}^{f}[rr] \ar@{=>}_{
h}[d]
 && \left[\frac{2^{k+1}}{2^{n-1}},\, \frac{2^{k+2}}{2^{n-1}}\right] \ar@{=>}^{h}[d] \\
\left[\frac{k+1}{n},\,
 \frac{k+2}{n}\right] \ar@{=>}^{g_1}[rr]
&& \left[\frac{k+2}{n},\,
 \frac{k+3}{n}\right] \\
 }
$$ is commutative for all $k,\, 0\leq k\leq n-3$. This prove
that~(\ref{arch-eq-05}) is the complete set of breaking points of
$h$, precisely, $h$ is linear on $[1/2,\, 1]$.

Define $g_2$ from the commutative diagram
\begin{equation}\label{arch-eq-06}
\xymatrix{\left[\frac{1}{2},\,
 1\right] \ar@{-->}^{f}[rr] \ar@{=>}_{
h}[d]
 && \left[0,\, 1\right] \ar^{h}[d] \\
\left[\frac{n-1}{n},\,
 1\right] \ar@{-->}^{g_2}[rr]  && \left[0,\,
 1\right] \\
 }
\end{equation}

Since $h$ has $n-2$ break points, then it has $n-1$ pieces of
linearity, call $\mathcal{P}_i,\, 1\leq i\leq n-1$. Whence,
diagram~(\ref{arch-eq-06}) break into $n$ diagrams of the form
\begin{equation}\label{arch-eq-07} \xymatrix{\mathcal{Q}_i \ar@{-->}^{f}[rr]
\ar@{=>}_{ h}[d] && \mathcal{P}_i \ar@{=>}^{h}[d] \\
h(\mathcal{Q}_i) \ar@{-->}^{g_2}[rr]  && h(\mathcal{P}_i), \\
 }
\end{equation} where $\mathcal{Q}_i$ are pre-images of $\mathcal{P}_i$ under
$x\mapsto 2-2x$ and evidently, $\mathcal{Q}_i\subset [1/2,\, 1]$,
and $\bigcup\limits_{i=1}^{n-1}\mathcal{Q}_i = [1/2,\, 1]$.

These reasonings show, that the type of piecewise linearity of $g$
is $(2,\, n)$.

The notice, that $n>1$ is arbitrary and taking $q=n-1$ finishes
the proof. \end{proof}

\begin{lemma}\label{arch-lema-8}
For any $p,\, q$ such that $2\leq p\leq q$ the type $(p,\, q)$ is
admissible.
\end{lemma}

\begin{proof}
Analogically to it was done at the proof of
Lemma~\ref{arch-lema-6} denote $n=q+1$. Consider the following
intervals. $\mathcal{P}'_0 = \left[0,\, \frac{1}{2^{n-2}}\right]$,
$\mathcal{P}_t = \left[\frac{2^t}{2^{n-1}},\,
\frac{2^{t+1}}{2^{n-1}}\right]$ for $0\leq t\leq n-2$. Observe,
that $|\mathcal{P}_t| = \frac{2^t}{2^{n-1}}$ for $0\leq t\leq n-2$
and $f(\mathcal{P}_t) = \mathcal{P}_{t+1}$ for $0\leq t\leq n-3$.

Let $k_t,\, 1\leq t\leq n-2$ be positive numbers such that
$$ 2k_0+\sum\limits_{t=1}^{n-2}k_t2^t
=2^{n-1}$$ and $k_t\neq k_{t+1}$ for all $t,\, 0\leq t\leq n-2$.
Notice, that these conditions let us to consider an increasing
piecewise linear homeomorphism $h:\ [0,\, 1]\rightarrow [0,\, 1]$
which is linear on $\mathcal{P}'_0$ and such that $k_t$ is the
tangent of $h$ on $\mathcal{P}_t$ for all $t$. Denote the maps $g
= h^{-1}(f(h))$. Evidently $g:\, [0,\, 1]\rightarrow [0,\, 1]$ and
$g$ is of the form~(\ref{eq:g}). clearly, $g$ is dependent on
$\mathcal{K} =\{ k_t ,\, 0\leq t\leq n-2 \}$.

Since in this case the diagram~(\ref{arch-eq-07}) will be
commutative, then the type of piecewise linearity of $g$ will be
$(\widetilde{p},\, q)$ for some $\widetilde{p}$.

Denote by $\widetilde{k}_t$ the tangent of $g_1$ on
$h(\mathcal{P}_t)$. It follows from Lemma~\ref{eq:eq-top-eq-ffv}
that $\widetilde{k}_0 =2$.

For any $t,\, 0\leq t\leq n-3$ it follows from commutative diagram
$$
\xymatrix{\mathcal{P}_t \ar@{=>}^{f}[rr]
\ar@{=>}_{ h}[d] && \mathcal{P}_{t+1} \ar@{=>}^{h}[d] \\
h(\mathcal{P}_t) \ar@{=>}^{g_1}[rr]  && h(\mathcal{P}_{t+1}) \\
 }
 $$ that $$
\widetilde{k}_t = \frac{2k_{t+1}}{k_t}.
$$
Since, $k_2\neq k_1$, then $\widetilde{k}_2\neq 2$. Evidently, the
condition $\widetilde{k}_t = \widetilde{k}_{t+1}$ is equivalent to
$$ k_{t+2} = \frac{k_{t+1}^2}{k_t}. $$

Assume that $p=q$. Take an arbitrary positive reals $r_0,\ldots,\,
r_{p-1}$ such that: $$r_t\neq r_{t+1}\, \text{ for }\ 0\leq t\leq
p-2 \ \text{ and }\ r_{t+2}\neq \frac{r_{t+1}^2}{r_t}\, \text{ for
}\ 1\leq t\leq p-3$$.

Denote \begin{equation}\label{arch-eq-11}r = \frac{r_0}{2^{n-2}} +
\sum\limits_{t=1}^{n-2}\frac{r_t2^t}{2^{n-1}} \end{equation} and
consider
\begin{equation}\label{arch-eq-08} k_t = \frac{r_t}{r},\ 0\leq t\leq
q-1
\end{equation} as tangents of $h:\, [0,\, 1]\rightarrow [0,\, 1]$
on $\mathcal{P}_t$ such that $h$ is linear on $\mathcal{P}_0'$.
Then the type of piecewise linearity of $g = h^{-1} (f(h))$ would
be $(p,\, q)$.

Assume now that $1<p<q$. Then construct the maps $g$ in the same
way as in the previous case, but with $\widetilde{k}_0=2$,
$\widetilde{k}_1\neq \widetilde{k}_0$,
$\widetilde{k}_1=\widetilde{k}_2=\ldots =\widetilde{k}_{q-p+1}$,
$\widetilde{k}_t\neq \widetilde{k}_{t+1}$ for $q-p+1\leq t\leq
q-1$ Remark, that $q-1 = n-2$. For this reason take an arbitrary
$r_0,\ \ldots,\, r_{n-2}$ such that
\begin{equation}\label{arch-eq-10}r_t\neq r_{t+1}\ \text{ for }\ 0\leq t\leq
n-3,\end{equation}
\begin{equation}\label{arch-eq-09} r_{t+2} = \frac{r_{t+1}^2}{r_t},\
\text{ for }\ 1\leq t\leq q-p,
\end{equation} and $
r_{t+2} \neq \frac{r_{t+1}^2}{r_t},\ \text{ for }\ q-p+1\leq t\leq
n-4. $ Notice, that there is no contradiction
between~(\ref{arch-eq-10}) and~(\ref{arch-eq-09}), because
$r_{t+1}\neq r_t$ implies $\frac{r_{t+1}^2}{r_t}\neq r_{t+1}$,
whence $r_{t+2}\neq r_{t+1}$ for $r_{t+2}$ found
by~(\ref{arch-eq-10}).

Now define $r$ and $k_t,\, 0\leq t\leq n-2$ by~(\ref{arch-eq-11})
and~(\ref{arch-eq-08}) and consider $k_t$ as tangents of $h:\,
[0,\, 1]\rightarrow [0,\, 1]$ on $\mathcal{P}_t$ such that $h$ is
linear on $\mathcal{P}_0'$. Then the type of piecewise linearity
of $g = h^{-1} (f(h))$ would be $(p,\, q)$.
\end{proof}

\begin{lemma}\label{arch-lema-7}
For any $p,\, q$ such that $2\leq q<p$ the type $(p,\, q)$ is
admissible.
\end{lemma}

Before the proof of Lemma~\ref{arch-lema-7}, make some observation
about the proof of Lemma~\ref{arch-lema-8}.

Denote by $f_l(x) =  2x$ for $x\in [0,\, 1/2]$, which is the
decreasing part of $f$ and denote by $f_r(x)=2-2x$ for $x\in
[1/2,\, 1]$, which is the decreasing part of $f$. For every $t$
denote by $f_l^{-t}$ and $f_r^{-t}$ the $t$-th iteration of the
maps $f_l^{-1}:\, [0,\, 1]\rightarrow [0,\, 1/2]$ and $f_r^{-1}:\,
[0,\, 1]\rightarrow [1/2,\, 1]$ correspondingly. Naturally,
$f_l^{-1}$ and $f_r^{-1}$ are inverse maps to $f_l$ and $f_r$
correspondingly. Then it is evident, that in the notations of the
proof of Lemma~\ref{arch-lema-8} we have that $\mathcal{P}_{n-2} =
[1/2,\, 1]$ and $\mathcal{P}_{n-2-k} =
f_l^{-k}(\mathcal{P}_{n-2})$ for all $k,\, 0\leq k\leq n-2$.
Denoting $t=n-2-k$ obtain that $\mathcal{P}_{t} =
f_l^{t+2-n}(\mathcal{P}_{n-2}) =f_l^{t+2-n}([1/2,\, 1])$.

\begin{proof}[Proof of Lemma~\ref{arch-lema-7}]
Denote $n=p+1$ and $\mathcal{Q}_{t} = f_r^{t+2-n}[0,\, 1/2]$ for
$0\leq t\leq n-2$ and $\mathcal{Q}_0' = [1/2,\, 1]\backslash
\bigcup\limits_{t=0}^{n-2}\mathcal{Q}_{t}$. Notice, that the fixed
point of $f_r$ belong to $\mathcal{Q}_0'$ and the length of
$\mathcal{Q}_t$ is $\frac{2^t}{2^{n-1}}$ for all $t,\, 0\leq t\leq
n-2$.

The continuation of the proof is analogical to the proof of
Lemma~\ref{arch-lema-8}.
\end{proof}

The following theorem follows from
Proposition~\ref{theor:ksklin-tpleqv} and~\ref{theor:23} and
Lemmas~\ref{arch-lema-8} and~\ref{arch-lema-7}.

\begin{theorem}\label{theor:24}
1. For any $p\geq 2$ and $q\geq 2$ the type of linearity $(p,\,
q)$ is admissible.

2. A type of linearity $(p,\, 1)$ and $(1,\, q)$ is admissible
only if it is $(1,\, 1)$. In this case the maps $g$ coincides with
$f$.
\end{theorem}

\newpage

\section{Semi conjugation of unimodal maps}\label{sect:SemiConj}

We will consider in this section the semi conjuation of maps $f:\,
[0,\, 1]\rightarrow [0,\, 1]$, which is given by
\begin{equation}\label{eq:92} f(x) = \left\{\begin{array}{ll}
2x,& x< 1/2;\\
2-2x,& x\geqslant 1/2,
\end{array}\right.
\end{equation} and the map
$f_v:\, [0,\, 1]\rightarrow [0,\, 1]$, which is given by
\begin{equation}\label{eq:12}
f_v(x) = \left\{\begin{array}{ll}
\frac{x}{v},& 0\leqslant x\leqslant v;\\
 \frac{1-x}{1-v},&
v<x\leqslant 1,
\end{array}\right.
\end{equation} for an
arbitrary $v\in [0,\, 1]\backslash \{ 1/2\}$.

Consider the functional equation
\begin{equation}\label{eq:11} \psi(f(x)) =
f_v(\psi(x))
\end{equation} for the
unknown continuous function $\psi:\, [0,\, 1]\rightarrow [0,\,
1]$. In this section we will study the properties of the solutions
of~(\ref{eq:11}), and will not assume the invertibility of them.

Remind that continuous surjective solution of~(\ref{eq:11}) is
called the semi conjugation of $f$ and $f_v$.

\subsection{Monotone solutions of functional equation}\label{sect:SemiConj-1}

Let $\eta:\, [0,\, 1] \rightarrow [0,\, 1]$ be continuous (but not
necessary surjective) solution of functional
equation~(\ref{eq:11}). We will prove, that if $\eta$ monotone,
then it is a conjugation between $f$ and $f_v$, defined
by~(\ref{eq:92}) and~(\ref{eq:12}).

\begin{lemma}\label{lema:h-0}$\eta(0) = 0$, or
$\eta(0) = \frac{1}{2-v}$.\end{lemma}

\begin{proof}
Plug $x=0$ into~(\ref{eq:11}), and obtain
$$ \eta(f(0)) = f_v(\eta(0)).
$$
Whence, $\eta(0)$ is a fixed point of $f_v$. Since $f_v$ has two
branches of monotonicity, which are defined by formulas $f_v(x)
=\frac{x}{v}$ and $f_v = \frac{1-x}{1-v}$, then foxed points of
$f_v$ are $0$ and $\frac{1}{2-v}.$
\end{proof}

\begin{lemma}\label{lema-2}
If $\eta(0)=0$, then either $\eta(1) = 0$, or $\eta(1) = 1$.
\end{lemma}

\begin{proof}
Plug $x=1$ into~(\ref{eq:11}) and obtain $\eta(f(1)) =
f_v(\eta(1)).$ Now lemma follows from $f(1)=0$ and the fact that
$f_v$ equals $0$ only at 0 and 1.
\end{proof}

\begin{corollary}
If $\eta(1) =1$, then $\eta$ is a homeomorphism.
\end{corollary}

\begin{lemma}
If $\eta(0)=\frac{1}{2-v}$, then either $\eta(1) = \frac{1}{2-v}$,
or $\eta(1) =  \frac{v}{2-v}$.
\end{lemma}

\begin{proof}
Lemma follows from the equality $f_v(\eta(1))=\frac{1}{2-v},$
which is obtained from~(\ref{eq:11}) by plugging $x=1$.
\end{proof}

\begin{corollary}
If $\eta$ is monotone and $\eta(0)=\frac{1}{2-v}$, then $\eta(1)
=\frac{v}{2-v}$.
\end{corollary}

\begin{lemma}
If $\eta(0) = \frac{1}{2-v}$ and $\eta(1) =  \frac{v}{2-v}$, then
$\eta$ is non monotone.\end{lemma}

\begin{proof}
It follows from the functional equation~(\ref{eq:11}) for $x=1/2$
that conditions of Lemma imply $f_v(\eta(1/2)) = \frac{v}{2-v}$.

Remind that graph of $f_v$ consists of two branches of
monotonicity and $x=\frac{v}{2-v}$ is its fixed point. In other
words, $\eta(1/2)$ is a fixed point of $f_v$. But $\eta(1/2)
=\frac{v}{2-v}$ contradicts to monotonicity of $\eta$, whence
$\eta(1/2) = \frac{v^2}{2-v}.$ But this also contradicts to
monotonicity of $\eta$, because implies that $\eta(1/2) < \eta(0)$
and $\eta(1/2) < \eta(1)$. The last proved Lemma.
\end{proof}

\begin{lemma}
If $\eta(0) = \frac{1}{2-v}$ and $\eta(1) =  0$, then $\eta$ is
non monotone.
\end{lemma}

\begin{proof}
It follows from~(\ref{eq:11}) for $x=1/2$ and from the condition
of Lemma that
$$ f_v(\eta(1/2))=0,
$$ whence $\eta(1/2) \in \{ 0,\, 1\}.$ This
contradicts to monotonicity of $\eta$.
\end{proof}

The next proposition follows from the proved lemmas.

\begin{proposition}\label{prop:2}
Every continuous monotone solution $\eta:\, [0,\, 1]\rightarrow
[0,\, 1]$ of the functional equation~(\ref{eq:11}) is the
conjugacy of $f$ and $f_v$ forms~(\ref{eq:92}) and~(\ref{eq:12})
correspondingly.
\end{proposition}

\subsection{Problem of the self-semiconjugation}\label{sect:SemiConj-2}

Let $h:\, [0,\, 1]\rightarrow [0,\, 1]$ be the conjugacy of $f$
and $f_v$, i.e. the solution of the functional
equation~(\ref{eq:11}). By Theorem~\ref{theor:homeom-jed} this
solution exists and is the unique. Also $h$ increase.

For an arbitrary (not necessary homeomorphic) solution $\eta:\,
[0,\, 1] \rightarrow [0,\, 1]$ of the functional
equation~(\ref{eq:11}) consider a commutative diagram

\begin{equation}\label{eq:09}\xymatrix{ [0,\, 1]
\ar@/_3pc/@{-->}_{\xi}[dd] \ar^{f}[r]
\ar_{\eta}[d] & [0,\, 1] \ar^{\eta}[d] \ar@/^3pc/@{-->}^{\xi}[dd]\\
[0,\, 1] \ar^{f_v}[r] \ar^{h^{-1}}[d] & [0,\, 1] \ar_{h^{-1}}[d]\\
[0,\, 1] \ar^{f}[r] & [0,\, 1]}\end{equation} and denote $\xi =
h^{-1}(\eta)$.

Let $\xi$ be an arbitrary (not necessary homeomorphic) solution of
a functional equation
\begin{equation}\label{eq:08}\psi(f) = f(\psi).
\end{equation} for
an unknown function $\psi:\, [0,\, 1] \rightarrow [0,\, 1]$.

Notice, that~(\ref{eq:08}) means that diagram $$ \xymatrix{
[0,\, 1] \ar^{f}[r] \ar_{\psi}[d] & [0,\, 1] \ar^{\psi}[d]\\
[0,\, 1] \ar^{f}[r] & [0,\, 1] }
$$ is commutative.

From commutative diagram~(\ref{eq:09}) obtain, that for any
solution $\xi$ of functional equation~(\ref{eq:08}), the function
$\eta = h(\psi)$ is a solution of functional
equation~(\ref{eq:11}). Whence, the uniqueness of homeomorphic
solutions of~(\ref{eq:11}) yields the one to one correspondence
between non-homeomorphic solutions of~(\ref{eq:11}) and
non-homeomorphic solutions of more simple equation~(\ref{eq:08}).

Let $\xi:\, [0,\, 1]\rightarrow [0,\, 1]$ be a continuous solution
of~(\ref{eq:08}).

\begin{note}\label{note:14}
Consider some examples of solutions $\xi$ of the functional
equation~(\ref{eq:08}):

1. $\xi(x)=x$ for all $x\in [0,\, 1]$;

2. $\xi$ is constant, which is one of fixed points of $f$;

3. $\xi$ is some iteration of $f$.
\end{note}

\begin{proposition}\label{theorNeStale}
If a a continuous solution $\xi$ of the functional
equation~(\ref{eq:08}) is constant on some $M = [\alpha,\,
\beta]$, then it is constant on the whole $[0,\, 1]$.
\end{proposition}

\begin{proof}
Prove that if $\xi$ is constant on $M$, then it is piecewise
constant on some interval $[\widetilde{\alpha},\,
\widetilde{\beta}]$, which is either of the length $2(\beta -
\alpha)$, or $\widetilde{\alpha} =0$.

is $[0,\, \gamma]$.

Consider the commutative diagram
$$\begin{CD}
[\alpha,\, \beta] @>f >> & f([\alpha,\, \beta])\\
@V_{\xi} VV& @VV_{\xi}V\\
\xi([\alpha,\, \beta]) @>f>>& \xi(f([\alpha,\, \beta]))  =
f(\xi([\alpha,\, \beta]))
\end{CD}$$ and consider
two cases, whether $[\alpha,\, \beta]$ contains $0.5$, or not. If
$0.5 \in [\alpha,\, \beta]$, then change $M$ into either $M_1 =
[\alpha,\, 0,5]$, or $M_2 =[0,5,\, \beta]$. For both $M_1$ and
$M_2$ we
$$ \xi(f(M_i)) = f(\xi(M_i)),
$$ whence $\xi$ is constant on
$f(M_i)$.

If $M$ contains $\frac{1}{2}$, then $\xi$ is constant on $f(M)$,
which is of the form $[\delta,\, 1]$ and, applying the same
commutative diagram, obtain that $\xi$ is constant on $[0,\,
\widetilde{\beta}]$. Applying finitely many times the same
reasonings for $[0,\, \widetilde{\beta}]$ obtain that $\xi$ is
constant of the whole $[0,\, 1]$.

If $1/2\not\in M$, then $f(M)$ is an interval, whose length is two
times more than one of $M$ and $\xi$ is constant on $f(M)$.
Repeating these reasonings finitely many times, we obtain either
an interval, which contains $1/2$.
\end{proof}

\begin{proposition}\label{theorNePriama}
If the graph of $\xi$ is a segment of a line on some $M =
[\alpha,\, \beta]$ then $\xi$ is piecewise linear on $[0,\, 1]$.
\end{proposition}

\begin{proof}
Proof of this proposition is analogical to the proof of
Proposition~\ref{theorNeStale}. We just have to change the word
``constant'' to ``is piecewise linear'' through the whole proof.
\end{proof}

\subsection{Piecewise linearity of continuous
self-semiconjugation}\label{sect:SemiConj-3}

Since interval $[0,\, 1]$ is compact, then continuity of $\xi$
implies the uniformly continuity. For every $n \in \mathbb{N}$ it
follows from the uniformly continuity of $\xi$ that there exists
$m$ such that if the first $m$ binary digits of $a,\, b\in [0,\,
1]$ coincide, then \begin{equation}\label{eq:05} |\xi(a)-\xi(b)|<
2^{-n}.
\end{equation} Denote this $m$ by $m_\xi(n)$.

\begin{lemma}\label{lema:MainR}
For every $n\in \mathbb{N}$ if numbers $a,\, b\in [0,\, 1]$ have
the same $m_\xi(n)+1$ the first binary digits, then
$|\xi(a)-\xi(b)|< 2^{-(n+1)}$.
\end{lemma}

\begin{proof}
Let $m=m_\xi(n)$ and the binary decompositions of $a$ and $b$ are
$$ a = 0,\, c_1\, c_2\, \ldots c_{m+1}\,  x_1\, x_2\, \ldots
$$ and $$
b = 0,\, c_1\, c_2\, \ldots c_{m+1}\,  y_1\, y_2\, \ldots\, .
$$

Since the first $m$ digits of $a$ and $b$ coincide, then the
inequality $|\xi(a)-\xi(b)|< 2^{-n}$ holds.

Without loss of generality assume that $\xi(a)\geq \xi(b)$. Assume
that $\xi(a)-\xi(b)\geq 2^{-(n+1)}$. This assumption leads to the
following two cases of the form of $\xi(a)$ and $\xi(b)$.

Case 1: \begin{equation}\label{eq:06}
\begin{array}{l}
\xi(a) = 0, \underbrace{M\, 1}_n\, 0\, A,\\
\xi(b) = 0, \underbrace{M\, 0}_n\, 1\, B,
\end{array}
\end{equation} where
$M,\, A$ and $B$ are blocks of digits, the length of $M$ is $n-1$
and $0,A \geq 0,B$. Notice, that in these notations $\xi(a) -
\xi(b) = 2^{-(n+1)} + 2^{-(n+2)}\cdot(0.A - 0.B)$.

Case 2: \begin{equation}\label{eq:07}
\begin{array}{l}
\xi(a) = 0, \underbrace{M}_n\, 1\, A,\\
\xi(b) = 0, \underbrace{M}_n\, 0\, B,
\end{array}
\end{equation}
where $M,\, A$ and $B$ are blocks of digits, the quantity of
numbers in the block $M$ is $n$ and $0,A \geq 0,B$.

If the first $m+1$ binary digits of $a$ and $b$ coincide, then the
first $m$ digits of $f(a)$ and $f(b)$. Fom the equality $$
\xi(f(x)) = f(\xi(x)),
$$ obtain that~(\ref{eq:05}) implies
\begin{equation}\label{eq:05-a} |f(\xi(a))
-f(\xi(b))|<2^{-n}.\end{equation}

Denote by $M'$ the block $M$ without its the first digit. By the
name of block with a line above (for instance $\overline{A}$,
$\overline{M}'$ etc.) denote the block, which is obtained from the
former after the inversion of all its digits.

Consider the first case, i.e. $\xi(a)$ and $\xi(b)$ are of the
form~(\ref{eq:06}). If the first digit of $M$ is zero, then
$$ f(\xi(a)) = 0, \underbrace{M'\, 1}_{n-1}\, 0\,
A,
$$
$$ f(\xi(b)) = 0, \underbrace{M'\,
0}_{n-1}\, 1\, B.
$$
This contradicts to~(\ref{eq:05-a}), because it appears to be that
\begin{equation}\label{eq:04}
f(\xi(a)) - f(\xi(b)) = 2^{-n} + 2^{-n-1}\cdot (0, A - 0,B).
\end{equation}

If the first digit of $M$ equals 1, then
$$ f(\xi(a)) = 0, \underbrace{\overline{M}'\, 0}_{n-1}\, 1\, \overline{A},
$$
$$ f(\xi(b)) = 0, \underbrace{\overline{M}'\,
0}_{n-1}\, 1\, \overline{B}.
$$ This means that \begin{equation}\label{eq:04-a}
f(\xi(b)) - f(\xi(a)) = 2^{-n} + 2^{-n-1}\cdot (0,\overline{B} -
0,\overline{A}) \end{equation} and also
contradicts~(\ref{eq:05-a}).

Consider the second case, i.e. when $\xi(a)$ and $\xi(b)$ are of
the form~(\ref{eq:07}). Similarly to the first case consider
whether the first digits of $M$ is $0$ or $1$.

If the first digit of $M$ equals 0, then
$$
\begin{array}{l}
f(\xi(a)) = 0, \underbrace{M'}_{n-1}\, 1\, A,\\
f(\xi(b)) = 0, \underbrace{M'}_{n-1}\, 0\, B,
\end{array}
$$ whence the difference $f(\xi(a)) - f(\xi(b))$
is of the form~(\ref{eq:04}) and we have a contradiction
with~(\ref{eq:05-a}).

If the first digit of $M$ is 1, then
$$
\begin{array}{l}
f(\xi(a)) = 0, \underbrace{\overline{M}'}_n\, 0\, \overline{A},\\
f(\xi(b)) = 0, \underbrace{\overline{M}'}_n\, 1\, \overline{B},
\end{array}
$$ and the difference $f(\xi(b)) - f(\xi(a))
$ is of the form~(\ref{eq:04-a}) and it is also contradicts
to~(\ref{eq:05-a}).
\end{proof}

\begin{corollary}\label{corLemaMain}
Let $t,\, n\in \mathbb{N}$ and $m=m_\xi(n)$. If the first $m+t$
binary digits of $a,\, b\in [0,\, 1]$ coincide, then
$|\xi(a)-\xi(b)|< 2^{-(n+t)}$.
\end{corollary}

\begin{proof}
This corollary should be proved by induction of $t$ with the same
reasonings an in Lemma~\ref{lema:MainR}.
\end{proof}

Corollary~\ref{corLemaMain} admits the following geometrical
interpretation.

Let the number $n$ be fixed. Notice, that the coincidence of the
first $m$ digits of $a$ and $b$ means that these numbers are
between two neighbor points of $A_{m+1}$. Fix arbitrary neighbor
points of $A_m$, say $\alpha(t,\, m)$ and $\alpha(t+1,\, m)$.

It follows from the construction of $m$ by $n$ that for
\begin{equation}\label{eq:02} x\in [\alpha(t,\, m),\,
\alpha(t+1,\, m)]
\end{equation} the
graph of $h$ belongs to the rectangle of the height $2^{-n}$,
whose sides are parallel to coordinate axes.

Lemma~\ref{lema:MainR} means that if we divide this rectangle into
four ones by the lines, which are parallel to sides and pass
through the middle points of sides, then the graph of $\xi$ would
belong to exactly two of the obtained rectangles.

Corollary~\ref{corLemaMain} means that if each of the smaller
rectangles, which contains the graph of $\xi$, will be divided in
the same manner, then the the graph would be contained exactly in
two new rectangles and the process can be continued to infinity.

\begin{definition}\label{def:02}
Let $n\in \mathbb{N}$, $m = m_\xi(n)$ and $t,\, 0\leqslant
t\leqslant 2^{m-1}$ be a number. We shall call the following line
segments the \textbf{lines of the net}.

1. Lines, which are parallel to $x$-axis for
$x\in\nobreak\discretionary{}{\hbox{\ensuremath{\in}}}{}
[\alpha(t,\, m),\, \alpha(t+1,\, m)]$ and bound the graph of
$\xi$, if the distance between them is $2^{-n}$;

2. Vertical line segments, which connect the ends of lines, which
are described in the item 1 above.

3. Each of line segments, which is constructed in the geometrical
interpretation of Lemma~\ref{lema:MainR} and
Corollary~\ref{corLemaMain}.

The points of intersection of the lines of the net will be called
the \textbf{knots of the net}.

We shall call the rectangles, which are obtained in the item 3
above by new lines of the net and old lines of the net, the
\textbf{rectangles of the net}. Notice, that new lines of the net
are constructed only in the case when the old rectangle contains
the graph of $\xi$.
\end{definition}

\begin{note}Notice, that lines of the net are
defines in item 1. of Definition~\ref{def:02} not unambiguously.
The only thing, which is unambiguous, is the distance between
them.

All other is defined unambiguously.
\end{note}

The following lemma follows from Corollary~\ref{corLemaMain} and
the notions above.

\begin{lemma}\label{noteSitka}
If after the recurrent division of the rectangle of the net,
mentioned at item 3 of Definition~\ref{def:02}, the graph does not
contains in two new neighbor vertical rectangles, then it passes
through the knot of the net, which is the intersection of new
lines of the net.
\end{lemma}

Figure~\ref{Fig-proof} contains the interval $[\alpha(t,\, m),\,
\alpha(t+1,\, m)]$ and the graph of $\xi$.

We consider the upper and lower bound of the rectangle as the
bounds, which some from the uniformly continuity of $\xi$.

\begin{figure}[htbp]
\begin{center}
\begin{center}
\begin{picture}(100,120)
\put(100,10){\line(0,1){100}} \put(0,110){\line(1,0){100}}
\put(0,10){\line(0,1){100}} \put(0,10){\line(1,0){100}}

\qbezier[25](0,60)(50,60)(100,60)
\qbezier[25](50,10)(50,60)(50,110)

\qbezier[25](0,35)(50,35)(100,35)
\qbezier[12](25,10)(25,35)(25,60)

\qbezier[12](50,45.5)(75,45.5)(100,45.5)
\qbezier[7](75,35)(75,47.5)(75,60)

\qbezier(0,20)(50,55)(100,40)

\put(-14,0){$\alpha(t,\, m)$}  \put(80,0){$\alpha(t+1,\, m)$}

\put(21,30){\LARGE{A}} \put(71,30){\LARGE{D}}
\put(21,80){\LARGE{B}} \put(71,80){\LARGE{C}}

\put(11,19){1} \put(11,45){2} \put(37,45){3} \put(37,19){4}
\end{picture}
\end{center}
\end{center} \caption{Picture}\label{Fig-proof}
\end{figure}

The rectangle is divided into 4 rectangles, which are called $A,\,
B,\, C,\, D$. The graph of $\xi$ contains in the rectangles $A$
and $D$.

Rectangle $A$ is divides into $4$ parts, which are named ba
numbers from 1 till 4. It is clear that the graph of $\xi$ is
contained in rectangles 1 and 3. By Lemma~\ref{def:02} and
Corollary~\ref{corLemaMain} this means that this graph passes
through the point of intersection of new lines of the net and the
graph of $\xi$ has no points in the rectangles $2$ and $4$.

\begin{lemma}\label{lemaVyzolSitky}
If the map $\xi$ is not constant, then for any interval $J$ of the
form~(\ref{eq:02}) the graph of $\xi$ pathes through the new knot
of the net.
\end{lemma}

\begin{proof}
By Proposition~\ref{theorNeStale}, the maps $\xi$ is not constant
on $J$. Then denote $a = \min\limits_{x\in J}\xi(x)$ and $b =
\max\limits_{x\in J}\xi(x)$ and notice that $a\neq b$.

Since the height of each rectangle is divided by 2 on each step,
then there will be a step, when this height would be less then
$b-a$. By Lemma~\ref{noteSitka} is means that not later then after
this number of steps the graph of $\xi$ would pass through a knot
of the net.
\end{proof}

\begin{corollary}
If the maps $\xi$ is not constant, then on any $J$ of the
form~(\ref{eq:02}) there is at least two different knots of the
net, which belong to the graph of $\xi$.
\end{corollary}

\begin{proof}
Proof the this corollary is the same as one of
Lemma~\ref{lemaVyzolSitky}.
\end{proof}

\begin{lemma}\label{lema:03}
If the graph of $\xi$ is not constant on sone interval $J$ of the
form~(\ref{eq:02}) and passes through different knots of the net,
then $\xi$ is piecewise linear on the whole $[0,\, 1]$ with finite
number of intervals of linearity.
\end{lemma}

\begin{proof}
Let $a,\, b,\ (a<b)$ be points of $J$, where the graph of $\xi$
passes through knots of the net.

By Corollary~\ref{corLemaMain} graph of the map $\xi$ passes
through the point $(c,\, \xi(c))$, which is the middle of the line
segment with ends at $(a,\, \xi(a))$ and $(b,\, \xi(b))$.

We can apply Corollary~\ref{corLemaMain} to line segments $(a,\,
c)$ and $(c,\, b)$ and continue the reasonings infinitely, whence
obtain that the graph of $\xi$ has a dense set of point on the
line segment with ends at $(a,\, \xi(a))$ and $(b,\, \xi(b))$.

It follows from continuity of $\xi$ that its graph is a line
segment for $x\in [a,\, b]$.

Now lemma follows from Proposition~\ref{theorNePriama}.
\end{proof}

Lemma~\ref{lema:03} the following Theorem.

\begin{theorem}\label{theor:01}
Every continuous solution $\xi:\, [0,\, 1]\rightarrow [0,\, 1]$ of
equation~(\ref{eq:09}) is piecewise linear.
\end{theorem}

\subsection{Piecewise linearity of monotone
self-semiconjugation}\label{pr-Sect-monotone}

Assume that $\xi:\, [0,\, 1]\rightarrow [0,\, 1]$ is a piecewise
linear solution of~(\ref{eq:08}).

In the same way as we have proved Proposition~\ref{theorNeStale}
we can prove the following proposition.

\begin{proposition}\label{pr-Theor-1}
If $\xi$ is monotone on some interval $M=[\alpha,\, \beta]\subset
[0,\, 1]$, then $\xi$ is piecewise monotone on $[0,\, 1]$.
\end{proposition}

\begin{proof}
Assume, that $\xi$ is piecewise monotone on $[0,\, b]$ for some
$b,\, b\leq 1/2$. Then equation~(\ref{eq:08}) can be rewritten as
$$ \xi(2x)=f(\xi(x)).
$$ As $f$ is piecewise monotone, then $\xi$ follows to be
piecewise monotone on $[0,\, 2b]$. Repeating these reasonings for
several times if necessary, obtain the proof of the theorem.

Assume, that $\xi$ is piecewise monotone on $[a,\, 1]$, where
$a\geq 1/2$.  Then equation~(\ref{eq:08}) can be rewritten as
$$ \xi(2-2x)=f(\xi(x)).
$$ Nevertheless, this equation determines $\xi$ to be piecewise
monotone on $f([a,\, 1]) = [0,\, 2-2a]$ and the case is reduced to
previous one.

Case of $\xi$ being piecewise monotone on $[a,\, b]\subseteq [0,\,
1/2]$ and on $[a,\, b]\subseteq [1/2,\, 1]$ should be considered
similarly.
\end{proof}

\begin{lemma}\label{pr-lema-4}
The inclusion $\xi(0)\in \{ 0,\, 2/3\}$ holds.
\end{lemma}

Notice, that $2/3$ is positive fixed point of $f$.

\begin{proof}

It follows from the functional equation~(\ref{eq:08}), that the
following diagram is commutative
$$ \xymatrix{
0 \ar^{f}[r] \ar_{\xi}[d] & 0 \ar^{\xi}[d]\\
\xi(0) \ar^{f}[r] & \xi(0). }
$$
It means, that $\xi(0)$ is a fixed point of $f$, and proves the
lemma.
\end{proof}

\begin{lemma}\label{pr-lema-6}
Either $\xi(x)=2/3$ for all $x\in [0,\, 1]$, or $\xi(0)=0$.
\end{lemma}

\begin{proof}
Prove this lemma by contradiction. By Lemma~\ref{pr-lema-4}
assume, that $\xi(0)=2/3$.

Consider the interval call $[0,\, a]$ such that $\xi$ is monotone
on it. If $\xi$ increase on $[0,\, a]$, then for any $x\in (0,\,
\min\{ a/2,\, 1/2\})$ consider the commutative diagram $$
\xymatrix{
x \ar^{f}[r] \ar_{\xi}[d] & 2x \ar^{\xi}[d]\\
\xi(x) \ar^{f}[r] & \xi(2x). }
$$ As $\xi(x)>2/3$, then $\xi(2x)=2-2\xi(x)$. Next, as
$\xi(x)>2/3$, then $2-2\xi(x)<2/3$, whence $\xi(2x)<2/3$. It
contradicts to that $\xi(x)>2/3$ for all $x\in (0,\, a)$.

Consider now the case, when $\xi$ decrease on $[0,\, a]$.

Assume, that there exists $x_0 \in (0,\, a]$ such that
$\xi(x_0)\geq 1/2$. Then for any $x\in (0,\, \min\{ x_0/2,\,
1/2\})$ consider the commutative diagram $$ \xymatrix{
x \ar^{f}[r] \ar_{\xi}[d] & 2x \ar^{\xi}[d]\\
\xi(x) \ar^{f}[r] & \xi(2x). }
$$ As $\xi(x)<2/3$, then $\xi(2x)=2-2\xi(x)$. Next, as
$\xi(x)<2/3$, then $2-2\xi(x)>2/3$, whence $\xi(2x)>2/3$. It
contradicts to that $\xi(x)<2/3$ for all $x\in (0,\, a)$.

Assume at last, that $\xi(x)<1/2$ for all $x\in (0,\, a]$,
$xi(0)=2/3$ and $\xi$ decrease on $[0,\, a]$. For any $x_0\in
(0,\, \min\{ a/2,\, 1/2\})$ consider the commutative diagram $$
\xymatrix{
x_0/2 \ar^{f}[r] \ar_{\xi}[d] & x_0 \ar^{\xi}[d]\\
\xi(x_0/2) \ar^{f}[r] & \xi(x_0). }
$$ As $\xi(x_0/2)<1/2$, then $\xi(x_0)=2\xi(x_0/2)$, whence
$\xi(x_0/2)>\xi(x_0)$, which contradicts to that $\xi$ decrease on
$[0,\, a]$.
\end{proof}

\begin{lemma}\label{pr-lema-5}
For any $n\geq 2$ the inclusion $\xi(A_n)\subseteq A_n$ holds.
\end{lemma}

\begin{proof} Consider the continuation to the right of the commutative diagram,
which is equivalent to functional equation~(\ref{eq:08}), and
obtain
$$\xymatrix{
[0,\, 1] \ar^{f}[r] \ar_{\xi}[d] \ar@/^2pc/@{-->}^{f^n}[rrr] &
[0,\, 1] \ar^{\xi}[d] \ar^{f}[r] & \ldots \ar^{f}[r] & [0,\, 1]
\ar_{\xi}[d]\\
[0,\, 1] \ar^{f}[r] \ar@/_2pc/@{-->}^{f^n}[rrr] & [0,\, 1]
\ar^{f}[r] & \ldots \ar^{f^n}[r] & [0,\, 1]. }$$

The corollary of this diagram is $$ \xymatrix{
[0,\, 1] \ar^{f^n}[r] \ar_{\xi}[d] & [0,\, 1] \ar^{\xi}[d]\\
[0,\, 1] \ar^{f^n}[r] & [0,\, 1]. }
$$
Using Lemma~\ref{pr-lema-6}, it follows from the last diagram,
that for any $x\in A_n$ the following diagram commutes $$
\xymatrix{ x \ar^{f^n}[r] \ar_{\xi}[d] & 0
\ar^{\xi}[d]\\
\xi(x) \ar^{f^n}[r] & \xi(0), }
$$ and this proves the lemma.
\end{proof}

\begin{lemma}\label{pr-lema-7}
Let $\xi$ be monotone on some $A_{n,k}$ and for any $m\geq n$ the
equality \begin{equation}\label{pr-eq:12}\xi(\alpha_{m,p}) +
\xi(\alpha_{m,p+1}) = 2\xi(\alpha_{m+1,2p+1})\end{equation}
follows from the inclusion $A_{m,p}\subset A_{n,k}$. Then $\xi$ is
linear on $A_{n,k}$.
\end{lemma}

\begin{proof}
From the induction on $m$ and Proposition~\ref{lema:An} it
follows, that $\xi$ is linear on $A\cap A_{n,k}$. More precisely,
for any $t\in \mathbb{N}$ denote by $\xi_t$ the linear
approximation of $\xi$ such that $\xi_t(x) = \xi(x)$ for every
$x\in A_t$ and for every $s,\, 0\leq s\leq 2^{t-1}-1$ the maps
$\xi_t$ is linear on $A_{t,s}$. Then equation~(\ref{pr-eq:12})
mean that tangents of $\xi_{m+1}$ and $\xi_{m}$ coincide on
$A_{m,p}$, i.e. $\xi_{m+1}$ is linear in $A_{m,p}$ for all $p$
such that $A_{m,p}\subset A_{n,k}$. Application of induction on
$m$ leads to that $\xi_{m+1}$ is linear on~$A_{n,k}$.

Now the lemma follows from density of $A\cap A_{n,k}$ in
$A_{n,k}$.
\end{proof}

\begin{lemma}\label{pr-lema-8}
There is $A_{n,k}$ such that $\xi$ is linear on $A_{n,k}$.
\end{lemma}

\begin{proof}
If $\xi$ is monotone on some interval, then this interval contains
a set $A_{n,k}$ for some $n$ and $k$. By Lemma~\ref{pr-lema-5},
there exists $s,\, t$ such that $\xi(\alpha_{n,k})=\alpha_{n,t}$
and $\xi(\alpha_{n,k+1})=\alpha_{n,s}$. Denote $d = |\alpha_{n,s}
- \alpha_{n,t}|\cdot 2^{n-1}$ and prove the lemma by induction on
$d$. Notice, that $d$ equals to the quantity of intervals of the
form $A_{n,p}$ such that $\xi(A_{n,k})$ is union of these
intervals, i.e. $$ A_{n,k} = \bigcup\limits_{i=1}^dA_{n,k_i},
$$ where $k_{i_1},\ldots ,\, k_{i_d}$ are sequent integers. In
other words, $d+1$ is the number of points in $A_{n,k}\cap A_n$.

If $d=1$ then the lemma follows from Proposition~\ref{lema:An} and
Lemmas~\ref{pr-lema-5},~\ref{pr-lema-7}.

For every $t\geq 1$ consider the sets $A_{n+t,2^tk},\ldots
,A_{n+t,2^tk+2^{t-1}}$. By Proposition~\ref{lema:An}, the equality
$$ A_{n,k} = \bigcup\limits_{s=0}^{2^t-1}A_{n+t,2^tk+s}
$$ holds. For any $t$ and $s$ calculate $$
d_{t,s} = |\xi(\alpha_{n+t,2^tk+s})
-\xi(\alpha_{n+t,2^tk+s+1})|\cdot 2^{n+t-1}.
$$
Evidently, it follows from the monotonicity of $\xi$ on $A_{n,k}$,
that $$ 2^t\cdot d =  \sum\limits_{s=0}^{2^t-1}d_{t,s}\ .
$$

If for all $t,\, s$ the equality $d = d_{t,s}$ holds, then the
Lemma follows from Lemma~\ref{pr-lema-7}. Otherwise, there is $t,
s$, such that $d_{t,s}<d$ and we can apply the induction to
$A_{n+t,2^tk+s}$.
\end{proof}

The following theorem follows from Proposition~\ref{pr-Theor-1}
and Lemma~\ref{pr-lema-8}.

\begin{theorem}\label{theor:02}
If the function $\xi:\, [0,\, 1]\rightarrow [0,\, 1]$, which is a
solution of a functional equation~(\ref{eq:08}), is monotone on an
interval $M\subset [0,\, 1]$, then $\xi$ is piecewise linear on
$[0,\, 1]$.
\end{theorem}

\subsection{Piecewise linear
self-semiconjugation}\label{sect:SemiConj-4}

Assume that $\xi:\, [0,\, 1]\rightarrow [0,\, 1]$ is a piecewise
linear solution of equation~(\ref{eq:08}). The main result of this
subsection of the following theorem. Because of
Lemma~\ref{pr-lema-6} assume, that $\xi(0)=0$. Denote by $k$ the
tangent of $\xi$ at $0$.

\begin{theorem}\label{pr-theor-2}
Let $\xi$ be piecewise linear solution of the functional
equation~(\ref{eq:08}). Then $\xi$ is one of the following
functions

\noindent 1. $\xi(x)=x_0$ for all $x\in [0,\, 1]$, where $x_0=0$,
or $x_0=2/3$;

\noindent 2. for some $k\in \mathbb{N}$
$$\xi(x)=\displaystyle{\frac{1 - (-1)^{[kx]}}{2}
+(-1)^{[kx]}\{kx\}},$$ where $\{\cdot \}$ denotes fractional part
and $[\cdot ]$ denotes integer part. More then this for any $k\in
\mathbb{N}$ the function $\xi(x)$ of the form above is a solution
of~(\ref{eq:08}).
\end{theorem}

The fact, that all the maps, mentioned in Theorem~\ref{pr-theor-2}
satisfy the functional equation~(\ref{eq:08}) is evident.

\begin{lemma}\label{pr-lema-9}
Let $\xi(1)\in \{ 0,\, 1\}$.
\end{lemma}

\begin{proof}
It follows from Lemma~\ref{pr-lema-6} and functional
equation~(\ref{eq:08}), that the following diagram is commutative
$$ \xymatrix{
1 \ar^{f}[r] \ar_{\xi}[d] & 0 \ar^{\xi}[d]\\
\xi(1) \ar^{f}[r] & 0. }
$$
This proves lemma.
\end{proof}

\begin{lemma}\label{pr-lema-10}
$k\geq 1$ and $\xi$ is linear on the interval $[0,\,
\frac{1}{k}]$.
\end{lemma}

\begin{proof}
Evidently, there exists $n\in \mathbb{N}$ such that $\xi$ is
linear on $M= [0,\, \frac{1}{k\cdot 2^n}]$. Whence, for every
$x\in M$ functional equation~(\ref{eq:08}) is equivalent to
commutativity of the following diagram
$$ \xymatrix{
x \ar^{f^n}[r] \ar_{\xi}[d] & 2^n\cdot x \ar^{\xi}[d]\\
kx \ar^{f^n}[r] & 2^n\cdot kx. }
$$ Linearity of $\xi$ on $[0,\,
\frac{1}{k}]$ follows from the diagram. From this and
Lemma~\ref{pr-lema-9} obtain the proof.
\end{proof}

\begin{lemma}\label{pr-lema-11}
Let for $a<b\leq 1/2$ the following holds: $\xi(a)=0$, $\xi(b)=1$
and $\xi$ is linear with tangent $k$ for $x\in [a,\, b]$. Then
$\xi'(x)=k$ for all $x\in (2a,\, a+b)$ and $\xi'(x)=-k$ for all
$x\in (a+b,\, 2b)$.
\end{lemma}

\begin{proof}
For any $x\in (a,\, \frac{a+b}{2})$ functional
equation~(\ref{eq:08}) is equivalent to commutativity of the
following diagram
$$ \xymatrix{
x \ar^{f}[r] \ar_{\xi}[d] & 2x \ar^{\xi}[d]\\
k(x-a) \ar^{f}[r] & 2k(x-a). }
$$ Whence, $\xi'(x)=k$ for all $x\in (2a,\, a+b)$.

For any $x\in (\frac{a+b}{2},\, b)$ functional
equation~(\ref{eq:08}) is equivalent to commutativity of the
following diagram
$$ \xymatrix{
x \ar^{f}[r] \ar_{\xi}[d] & 2x \ar^{\xi}[d]\\
k(x-a) \ar^{f}[r] & 2-2k(x-a). }
$$ Whence, $\xi'(x)=-k$ for all $x\in (2a,\, a+b)$.
\end{proof}

\begin{lemma}\label{pr-lema-12}
$k\in \mathbb{N}$.
\end{lemma}

\begin{proof}
If follows from Lemmas~\ref{pr-lema-10} and~\ref{pr-lema-11}, that
the graph of $\xi$ consists of consequent line segments, which
correspond to intervals of the form $M=[\frac{s}{k},\,
\frac{s+1}{k}]$ for integer $s$, such that $\xi(M) = [0,\, 1]$ and
$|\xi'(x)|=k$ for all $x\in M$. Because of these reasonings, the
lemma follows from Lemma~\ref{pr-lema-9}.
\end{proof}

Now Theorem~\ref{pr-theor-2} follows from Lemmas~\ref{pr-lema-9}
-- \ref{pr-lema-12}.

\newpage
\section{The length of the graph of the conjugacy}\label{pr-sect-length}

\subsection{Finding of the length of the graph of
the conjugation by dynamical reasonings}\label{pr-sect-length-1}

We find in this section the length of the graph of the conjugation
$h:\, [0,\, 1]\rightarrow [0,\, 1]$ which is a solutions of the
functional equation \begin{equation}\label{eq:33} h(f) = f_v(h),
\end{equation} where $f:\, [0,\, 1]\rightarrow [0,\, 1]$ is
given by \begin{equation}\label{eq:34} f(x) =
\left\{\begin{array}{ll}
2x,& x< 1/2;\\
2-2x,& x\geqslant 1/2,
\end{array}\right.
\end{equation} and $f_v:\, [0,\, 1]\rightarrow [0,\,
1]$ is given by \begin{equation}\label{eq:35}
f_v(x) = \left\{\begin{array}{ll} \frac{x}{v},& 0\leqslant x\leqslant v;\\
 \frac{1-x}{1-v},&
v<x\leqslant 1,
\end{array}\right.
\end{equation} for arbitrary $v\in (0,\, 1)$.

The existence of the conjugation $h$ was proved in
Theorem~\ref{theor:homeom-jed}. Remind, that we have introduced in
Section~\ref{sect-Pobudowa-1} the set $A_n,\, n\geq 1$, which is a
solution of the equation $f^n(x)=0$. By Proposition~\ref{lema:An}
the equality $$ A_n = \left\{ \frac{k}{2^{n-1}},\, 0\leq k\leq
2^{n-1}\right\}
$$ holds. Denote
$$A_{n,k}=\left[\frac{k}{2^{n-1}},\, \frac{k+1}{2^{n-1}}\right].$$

Remind, that we have introduced in Section~\ref{sect-Pobudowa-1}
the sets $B_n$ which consists of solution of the equation
$f_v^n(x)=0$ and by Theorem~\ref{theor:10} the equality $$
h(A_n)=B_n
$$ holds for all $n\geq 1$. As in Section~\ref{sect-Pobudowa-1}
denote by $h_n:\, [0,\, 1]\rightarrow [0,\, 1]$ the piecewise
linear maps which coincide with $h$ on $A_n$ and all whose
breaking points belong to $A_n$. In other words, $h_n$ in linear
on $A_{n,k}$ for all admissible $k$.

Let $n\in \mathbb{N}$ and $k,\, 0\leq k\leq 2^{n-1}-1$ be given.
Let $\alpha$  be a tangent of $h_n$ in $A_{n,k}$.

As in the proof of Theorem~\ref{lema:32} for any $x\in [0,\, 1]$
denote its binary decomposition by
\begin{equation}\label{pr-eq:6} x = 0,x_1x_2\ldots x_k\ldots\,
\end{equation} and the only case when
there exists $n_0$ such that $x_n = x_{n+1} =1$ for all $n>n_0$ is
if $x = 1$.

For a number $x$ of the form~(\ref{pr-eq:6}) denote $x_0 = 0$ and
for every $i\geq 2$ denote
\begin{equation}\label{pr-eq:11}\alpha_i(x) =\left\{
\begin{array}{ll} 2v& \overline{x_{i-1}x_{i-2}}\in \{ \overline{00},\, \overline{11}\}\\
2(1-v)& \overline{x_{i-1}x_{i-2}}\in \{ \overline{01},\,
\overline{10}\}.
\end{array}\right.
\end{equation}

It follows from Theorem~\ref{lema:32} that tangents $\tau_1$ and
$\tau_2$ on $h_{n+1}$ on intervals $A_{n+1,2k}$ and $A_{n+1,2k+1}$
are determined only by $\alpha$ and evenness of $k$ and are
independent on $n$ and exact value of $k$. Precisely, one of
$\tau_1$ and $\tau_2$ equals $2v\cdot \tan \alpha$ and another
equals $2(1-v)\cdot \tan\alpha$. The only thing which depends on
evenness of $k$ is that which of mentioned values is equal to
$2v\cdot \tan \alpha$. Whence the length of approximation
$h_{n+1}$ on $A_{n,k}$ is determined by $v,\, n$ and the length of
$h_n$ on this interval.

Consider points $F_1(\frac{k}{2^{n-1}},\, \beta_{n,k})$ and
$F_2(\frac{k+1}{2^{n-1}},\, \beta_{n,k+1})$ of the graph of $h_n$.
Denote \begin{equation}\label{pr-eq:10}t_{n,k} =
\frac{F_1F_2}{2^{n-1}+(\beta_{n,k+1} -
\beta_{n,k})}.\end{equation}

Evidently, $t<1$. Since $2^{n-1}=l\cos\alpha$ and $\beta_{n,k+1} -
\beta_{n,k}=l\sin\alpha$, then
$$ t_{n,k} =\frac{1}{\sin\alpha +\cos\alpha} =
\frac{1}{\sqrt{2}\sin(\alpha +\pi/4)},
$$ whence $t\in [\frac{\sqrt{2}}{2},\, 1]$. Denote by
$\xi_{n,k}(t_{n,k})$ the ratio of the length of the graph of
$h_{n+1}$ on $A_{n+1,2k}$ over the sum $2^{n-1}+(\beta_{n,k+1} -
\beta_{n,k})$.

The increasing function $\xi(t)$, which can be constructed in the
following Lemma, can be considered as lower bound of $\xi_{n,k}$.

\begin{lemma}
There exists a continuous function $\xi_v:
\left[\frac{\sqrt{2}}{2},\, 1\right] \rightarrow
\left[\frac{\sqrt{2}}{2},\, 1\right]$, dependent on $v$, but
independent on $n,\, k$, which satisfy the following properties:

1. $\xi_v(t)>t$ for every $t\in \left[\frac{\sqrt{2}}{2},\,
1\right)$;

2. $\xi_v(1)=1$;

3. For every $n,\, k$ the inequality
$\xi_v(t_{n,k})>\xi_{n,k}(t_{n,k})$ holds.
\end{lemma}

\begin{proof}
Let for given $n,\, k$ the points $F_1$, $F_2$ and the number $t =
t_{n,k}$ is described in the beginning of the chapter. For making
the further reasonings more short, denote $a = 2^{n-1}$, $b =
\beta_{n,k+1} - \beta_{n,k}$ and $l = F_1F_2$. Then
$a=l\cos\alpha$ and $b=l\sin\alpha$.

Denote $H(\frac{2k+1}{2^n},\, \beta_{n+1,\, 2k+1})$ and
$C(\frac{2k+1}{2^n},\, \beta_{n,\, k})$. By
Lemmas~\ref{lema:vlastfv1} and~\ref{lema:vlastfv2} the point $H$
coincides either with $H_1$ or $H_2$ such that $H_1C = (1-v)b$ and
$H_2C = vb$ (see Fig.~\ref{pr-fig:25}).

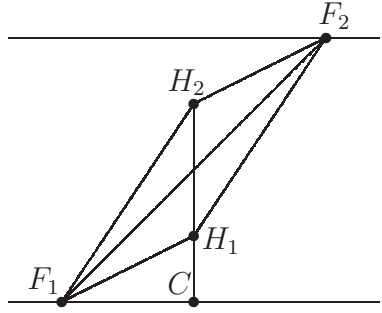
\begin{figure}[htbp]
\begin{center}
\begin{picture}(140,120)
\put(0,100){\line(1,0){140}} \put(0,0){\line(1,0){140}}

\put(20,0){\circle*{4}} \put(7,5){$F_1$}
\put(120,100){\circle*{4}} \put(117,105){$F_2$}

\Vidr{20}{0}{120}{100} \VidrTo{70}{75} \VidrTo{20}{0}
\VidrTo{70}{25} \VidrTo{120}{100} \Vidr{70}{75}{70}{0}

\put(70,25){\circle*{4}} \put(73,20){$H_1$}

\put(70,75){\circle*{4}} \put(60,80){$H_2$}

\put(70,0){\circle*{4}} \put(60,3){$C$}
\end{picture}
\end{center}
\caption{Finding of the length of $h_{n+1}$ by given $h_n$}
\label{pr-fig:25}
\end{figure}

Since $F_1H_2F_2H_1$ is a parallelogram, then $F_1H_2 + H_2F_2 =
F_1H_1 + H_1F_2$. In other words the length of the graph of
$h_{n+1}$ on $\left(\frac{k}{2^{n-1}},\, \frac{k+1}{2^{n-1}}
\right)$ is independent on those which of equalities $H=H_1$ or
$H=H_2$ holds.

The length of the graph $h_{n+1}$ on $\left(\frac{k}{2^{n-1}},\,
\frac{k+1}{2^{n-1}} \right)$ can be calculated as
$$ F_1H_2 + H_2F_2 =\sqrt{\frac{a^2}{4} + b^2v^2} +
\sqrt{\frac{a^2}{4} + b^2(1-v)^2} =
$$$$
=l\left(\sqrt{\frac{\cos^2\alpha}{4} + v^2\sin^2\alpha} +
\sqrt{\frac{\cos^2\alpha}{4} + (1-v)^2\sin^2\alpha}\right).
$$

Then $$ \xi_v(t) =\frac{AC + CB}{a+b} = \frac{AC + CB}{l/ t} = $$
$$=
t\cdot \left(\sqrt{\frac{\cos^2\alpha}{4} + v^2\sin^2\alpha} +
\sqrt{\frac{\cos^2\alpha}{4} + (1-v)^2\sin^2\alpha}\right)
$$

$$
t^2 = \frac{1}{1+\sin2\alpha}
$$$$
\sin2\alpha = \frac{1}{t^2}-1
$$
$$
\cos2\alpha = \pm\sqrt{1 -\left(\frac{1}{t^2}-1\right)^2}
=\pm\sqrt{\frac{2}{t^2} -\frac{1}{t^4}} =
\frac{\pm\sqrt{2t^2-1}}{t^2}.
$$

Consider two cases, when $$ \cos2\alpha_1 =
\frac{\sqrt{2t^2-1}}{t^2}
$$ and $$ \cos2\alpha_2
= \frac{-\sqrt{2t^2-1}}{t^2}
$$

Since  $$ \cos2\alpha = 2\cos^2\alpha -1,
$$ then $$
\cos^2\alpha_1 = \frac{1+\frac{\sqrt{2t^2-1}}{t^2}}{2} = \frac{t^2
+\sqrt{2t^2-1}}{2t^2};
$$
$$
\cos^2\alpha_2 = \frac{1-\frac{\sqrt{2t^2-1}}{t^2}}{2} =
\frac{t^2-\sqrt{2t^2-1}}{2t^2}.
$$

Since $$ \cos2\alpha = 1-2\sin^2\alpha,
$$ then $$
\sin^2\alpha_1 = \frac{1-\frac{\sqrt{2t^2-1}}{t^2}}{2} = \frac{t^2
-\sqrt{2t^2-1}}{2t^2}
$$ and $$
\sin^2\alpha_2 =\frac{t^2 +\sqrt{2t^2-1}}{2t^2}.
$$

Now for the case $\alpha_1$ we have $$ \xi_v^{(1)}(t) =
t\cdot\left(\sqrt{\frac{t^2 +\sqrt{2t^2-1}}{8t^2} +\frac{v^2(t^2
-\sqrt{2t^2-1})}{2t^2}}\right. + $$
$$ +\left.\sqrt{\frac{t^2 +\sqrt{2t^2-1}}{8t^2} +\frac{(1-v)^2(t^2
-\sqrt{2t^2-1})}{2t^2}}\, \right) =
$$$$ =
\left(\sqrt{\frac{(1+4v^2)t^2}{8}
+\frac{(1-4v^2)\sqrt{2t^2-1}}{8}}\right. + $$
$$ +\left.\sqrt{\frac{(1+4(1-v)^2)t^2}{8}
+\frac{(1-4(1-v)^2)\sqrt{2t^2-1}}{8}}\, \right).
$$

For the case $\alpha_2$ we have

$$ \xi_v^{(2)}(t) =
\left(\sqrt{\frac{(1+4v^2)t^2}{8}
-\frac{(1-4v^2)\sqrt{2t^2-1}}{8}}\right. + $$
$$ +\left.\sqrt{\frac{(1+4(1-v)^2)t^2}{8}
-\frac{(1-4(1-v)^2)\sqrt{2t^2-1}}{8}}\, \right).
$$

Consider the expression $$ \varphi_v(t) = (1+4v^2)t^2
+(1-4v^2)\sqrt{2t^2-1}.
$$
$$
\varphi_v'(t) = 2(1+4v^2)t +\frac{2(1-4v^2)t}{\sqrt{2t^2-1}}
=\frac{2t\cdot ((1+4v^2)\sqrt{2t^2-1} +(1-4v^2))}{\sqrt{2t^2-1}} =
$$$$
=\frac{2t\cdot (1+4v^2)\cdot \left(\sqrt{2t^2-1}
-\frac{4v^2-1}{4v^2+1}\right)}{\sqrt{2t^2-1}}
$$
As $\sqrt{2t^2-1}>1$ and $\frac{4v^2-1}{4v^2+1}<1$, then
$\varphi_v'(t)>0$.

Since $\xi_v^{(1)}(t) = \sqrt{\varphi_v(t)} +
\sqrt{\varphi_{1-v}(t)}$, then $\xi_v^{(1)}(t)$ increase.

Consider the expression $$ \psi_v(t) = (1+4v^2)t^2
-(1-4v^2)\sqrt{2t^2-1}.
$$

$$ \psi'_v(t) = 2t(1+4v^2)  -\frac{2(1-4v^2)t}{\sqrt{2t^2-1}} =
 \frac{2t\cdot ((1+4v^2)\sqrt{2t^2-1} -(1-4v^2))}{\sqrt{2t^2-1}} =
$$$$
=\frac{2t\cdot (1+4v^2)\cdot \left(\sqrt{2t^2-1}
-\frac{1-4v^2}{1+4v^2}\right)}{\sqrt{2t^2-1}}>0,
$$ because $\sqrt{2t^2-1}>1$ and $\frac{1-4v^2}{1+4v^2}<1$.

Since $\xi_v^{(2)}(t) = \sqrt{\psi_v(t)} + \sqrt{\psi_{1-v}(t)}$,
then $\xi_v^{(2)}(t)$ increase.

Now let $\xi_v(t) = \min\{ \xi_v^{(1)}(t),\, \xi_v^{(2)}(t)\}$ and
this finishes the proof.
\end{proof}

Evidently, that $\xi_v(t) = \xi_{1-v}(t)$  for all $v$ and $t$.

Graphs of $y = \xi_v(t)-t$ for $v = 0.1$, $v=0.2$, $v=0.3$ and
$v=0.4$ are given on Figure~\ref{pr-fig:26}. Notice, that for
$\{v_1,\, v_2\}<0.5$ and all $t\in t\in
\left[\frac{\sqrt{2}}{2},\, 1\right)$ the inequality $\xi_{v_1}(t)
> \xi_{v_2}(t)$ holds for $v_1<v_2$.
More then this $\xi_{0.5}(t)-t =0$ for all$t$.

\begin{figure}[htbp]
\begin{center}\begin{picture}(140,240)

\Vidr{57}{208}{62}{169} \VidrTo{66}{146} \VidrTo{70}{127}
\VidrTo{75}{111} \VidrTo{79}{97} \VidrTo{83}{84} \VidrTo{88}{73}
\VidrTo{92}{64} \VidrTo{97}{55} \VidrTo{101}{48} \VidrTo{105}{42}
\VidrTo{110}{37} \VidrTo{114}{33} \VidrTo{119}{29}
\VidrTo{123}{26} \VidrTo{127}{24} \VidrTo{132}{22}
\VidrTo{136}{21} \VidrTo{141}{20} \VidrTo{145}{20}

\Vidr{57}{122}{62}{103} \VidrTo{66}{91} \VidrTo{70}{80}
\VidrTo{75}{71} \VidrTo{79}{63} \VidrTo{83}{56} \VidrTo{88}{50}
\VidrTo{92}{45} \VidrTo{97}{40} \VidrTo{101}{36} \VidrTo{105}{32}
\VidrTo{110}{30} \VidrTo{114}{27} \VidrTo{119}{25}
\VidrTo{123}{23} \VidrTo{127}{22} \VidrTo{132}{21}
\VidrTo{136}{20} \VidrTo{141}{20} \VidrTo{145}{20}

\Vidr{57}{64}{62}{57} \VidrTo{66}{52} \VidrTo{70}{47}
\VidrTo{75}{43} \VidrTo{79}{39} \VidrTo{83}{36} \VidrTo{88}{33}
\VidrTo{92}{31} \VidrTo{97}{29} \VidrTo{101}{27} \VidrTo{105}{26}
\VidrTo{110}{24} \VidrTo{114}{23} \VidrTo{119}{22}
\VidrTo{123}{21} \VidrTo{127}{21} \VidrTo{132}{21}
\VidrTo{136}{20} \VidrTo{141}{20} \VidrTo{145}{20}

\Vidr{57}{31}{62}{29} \VidrTo{66}{28} \VidrTo{70}{27}
\VidrTo{75}{26} \VidrTo{79}{25} \VidrTo{83}{24} \VidrTo{88}{23}
\VidrTo{92}{23} \VidrTo{97}{22} \VidrTo{101}{22} \VidrTo{105}{21}
\VidrTo{110}{21} \VidrTo{114}{21} \VidrTo{119}{21}
\VidrTo{123}{20} \VidrTo{127}{20} \VidrTo{132}{20}
\VidrTo{136}{20} \VidrTo{141}{20} \VidrTo{145}{20}

\put(25,20){\vector(0,1){220}} \put(25,20){\vector(1,0){150}}

\Vidr{55}{20}{55}{220} \Vidr{85}{20}{85}{220}
\Vidr{115}{20}{115}{220} \Vidr{145}{20}{145}{220}

\Vidr{25}{80}{145}{80} \Vidr{25}{140}{145}{140}
\Vidr{25}{200}{145}{200}

\put(48,8){0.7} \put(78,8){0.8} \put(108,8){0.9} \put(143,8){1}

\put(2,77){0.02} \put(2,137){0.04} \put(2,197){0.06}

\end{picture}
\end{center}
\caption{Graphs of $y= \xi_v(t)-t$ for $v = 0.1$, $v=0.2$, $v=0.3$
and $v=0.4$} \label{pr-fig:26}
\end{figure}
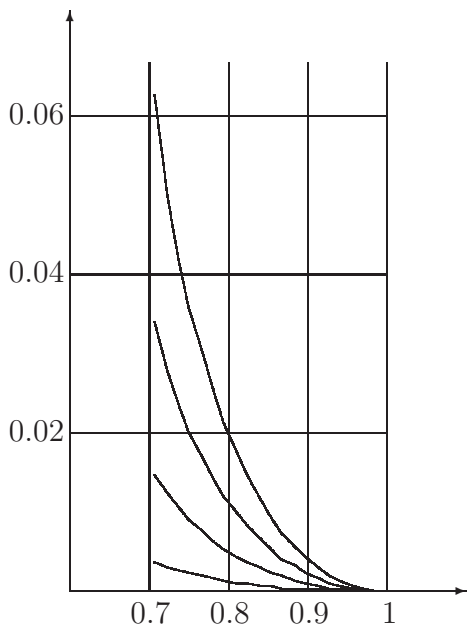

\begin{theorem}\label{pr-Theor-3}
The length of the graph of $h_n$ tends to 2 for $n \rightarrow
\infty$.
\end{theorem}

\begin{proof} As the sequence $l_n$ of lengthes of the graphs of $h_n$
increase and is bounded by $2$, then limit $L
=\lim\limits_{n\rightarrow \infty}l_n$ exists and $L\leq 2$.

Assume, that $L<2$. For any $n\geq 2$ consider the graph of $h_n$,
which is linear on each of intervals $A_{n,k}$ for $0\leq k\leq
2^{n-1}-1$.

Consider all the numbers $t_{n,k}$ for $0\leq k\leq 2^{n-1}-1$,
defined by~(\ref{pr-eq:10}). If for some $k$ the inequality
$t_{n,k}>\frac{L}{2}$ holds, then for every $n'>n$ the length of
$h_{n'}$ on $A_{n,k}$ is more than $\frac{L}{2}$ times more, than
$2^{n-1}+\beta_{n,k+1}-\beta_{n,k}$.

Denote by $C_n$ the union of all $A_{n,k}$ such that
$t_{n,k}>\frac{L}{2}$ and denote by $D_n =
\bigcup\limits_{n=2}^nC_n$ and $D = \bigcup\limits_{n=2}^\infty
C_n$.

We will show, that $\lambda(D)$, where $\lambda$ is the Lebesgue
measure.

For any $i\geq 2$ consider the function $\gamma_i(x)$ such that
$\gamma_i(x)=1$ if $\alpha_i(1)=2v$ and $\gamma_i(x)=0$ if
$\alpha_i(1)=2(1-v)$, where $\alpha_i$ is defined
by~(\ref{pr-eq:11}). Consider the function
$$\psi_n(x)=\sum\limits_{i=2}^n \frac{\gamma_i(x)}{2^i} + 2^n\cdot
\left[\frac{x}{2^n}\right].$$

Consider the function $\psi(x) =\lim\limits_{n\rightarrow
\infty}\psi_n(x)$. As for any $n$ the function $\psi_n$ permutes
intervals of $[0,\, 1]$, then for any Lebesgue measurable set
$M\subseteq [0,\ ,1]$ the equality $\lambda(\psi_n(M)) =
\lambda(M)$, whence $\lambda(\psi(M)) = \lambda(M)$. Notice, that
this fact, together with that $\psi$ and all $\psi_n$ are
bijections means that $\psi$ and $\psi_n$ preserve the Lebesgue
measure.

There is a correspondence between a point $x\in [0,\, 1]$ (or
$\psi(x)\in [0,\, 1]$) and the following pathes in the quadrant
$\mathcal{K}_1 =\{ (x,\, y)\in \mathbb{Z}^+\times \mathbb{Z}^+:\,
x\geq 0,\, y\geq 0\}$. We start from $(0,\, 0)$ and do 1 right, if
the first binary digit of $\psi(x)$ is 1 and go 1 up, if this
digit is $0$. Then consider the next digit of $\psi(x)$ and repeat
infinitely the process.

This leads us to considering $x$ as a Markov process
(see~\cite{Markov-1} and~\cite{Markov-2}),  with 2 states $p_0,\,
p_1$ and transition matrix
$$ \left(
\begin{array}{ll}
0.5 & 0.5\\
0.5 & 0.5
\end{array}\right).
$$ The state $p_1$ means moving right and $p_2$ means moving up.

For any $n\geq 2$ the chain of the first $n-1$ states corresponds
to a set $A_{n,k}$ for some $k$, dependent on $n$. Let $s$ be sum
of first $n-1$ digits of $\psi(x)$. Then $h'(x)=(2v)^s\cdot
(2-2v)^{n-s-1}$ for $x\in A_{n,k}$. This point corresponds to the
point $P_n(x)\in \mathcal{K}_1$ with coordinates $(s,\, n-s-1)$,
which belong to the line $y+x = n-1$ .

Describe the conditions on $s$ such that $A_{n,k}\not\in C_n$.
Denote the angle between the graph of $h_n$ on $A_{n,k}$ and the
$x$-axis by $\alpha$. Then the condition $$ t_{n,k} =
\frac{1}{\sin\alpha +\cos\alpha}<\frac{L}{2}
$$ is equivalent to $A_{n,k}\not\in C_n$.
Transformations of the last inequality give $$ \sin\left(\alpha +
\frac{\pi}{4}\right) >\frac{\sqrt{2}}{L};
$$$$
\frac{\pi}{2}> \frac{3\pi}{4} - \arcsin\frac{\sqrt{2}}{L} >\alpha
> \arcsin\frac{\sqrt{2}}{L} - \frac{\pi}{4} >0
$$

$$\sin\left( \frac{3\pi}{4} - \arcsin\frac{\sqrt{2}}{L} \right) =
\frac{\sqrt{2}}{2}\cos\left(\arcsin\frac{\sqrt{2}}{L}\right) +
\frac{\sqrt{2}}{2}\sin\left(\arcsin\frac{\sqrt{2}}{L}\right) =
$$ $$
= \frac{\sqrt{2}}{2}\left( \sqrt{1-\frac{2}{L^2}} +
\frac{\sqrt{2}}{L}\right) = \frac{\sqrt{2}}{2}\cdot
\frac{\sqrt{L^2 -2} +\sqrt{2}}{L}
$$

$$\cos\left( \frac{3\pi}{4} - \arcsin\frac{\sqrt{2}}{L} \right) =
\frac{-\sqrt{2}}{2}\cos\left(\arcsin\frac{\sqrt{2}}{L}\right) +
\frac{\sqrt{2}}{2}\sin\left(\arcsin\frac{\sqrt{2}}{L}\right) =
$$
$$
= \frac{\sqrt{2}}{2}\cdot \frac{\sqrt{2} -\sqrt{L^2 -2}}{L}
$$
$$
\frac{\sqrt{L^2 -2} +\sqrt{2}}{\sqrt{2} -\sqrt{L^2 -2}} > \tan
\alpha > \frac{\sqrt{2} -\sqrt{L^2 -2}}{\sqrt{L^2 -2} +\sqrt{2}}
$$

Denote by $$ T^+ = \frac{\sqrt{L^2 -2} +\sqrt{2}}{\sqrt{2}
-\sqrt{L^2 -2}}
$$ and $$
T^- = \frac{\sqrt{2} -\sqrt{L^2 -2}}{\sqrt{L^2 -2} +\sqrt{2}}.
$$ Whence, $$
T^+ > 2^n\cdot v^s\cdot (1-v)^{n-s-1} > T^-
$$ $$
\ln(T^+) > n\ln 2 + s\ln v + (n-s-1)\ln (1-v)
> \ln(T^-)
$$ $$
\ln(T^+) > s(\ln v -\ln (1-v)) + n\ln 2 + (n-1)\ln (1-v)
> \ln(T^-)
$$

If $v>1/2$, then $\ln v - \ln(1-v) > 0$, whence $$ \frac{\ln(T^+)
- n\ln 2 - (n-1)\ln (1-v)}{\ln v -\ln (1-v)} > s > \frac{\ln(T^-)
- n\ln 2 - (n-1)\ln (1-v)}{\ln v -\ln (1-v)}.
$$

If $v<1/2$, then $\ln v - \ln(1-v) < 0$, whence $$ \frac{\ln(T^+)
- n\ln 2 - (n-1)\ln (1-v)}{\ln v -\ln (1-v)} < s < \frac{\ln(T^-)
- n\ln 2 - (n-1)\ln (1-v)}{\ln v -\ln (1-v)}.
$$

Nevertheless, independently on the sign of the expression $v-1/2$,
the condition $A_{n,k}\not\in C_n$ is equivalent to that $s$
belongs to some interval of the length not more then $$
T=\left[\frac{\ln (T^+) - \ln (T^-)}{\left\| \ln v -\ln
(1-v)\right\|}\right],
$$ where $||x ||$ denotes the absolute value of $x$ and $[x]$
denotes integer part of $x$.

Resume, that we know, that $[0,\, 1]\backslash C_n$ consists of
not more then $T$ intervals. Consider all the passes, which end at
$y=n-1-x$, i.e. which have the length $n-1$. There are $2^{n-1}$
such pathes. For any $s,\, 0\leq s\leq n-1$ there are $C_{n-1}^s$
pathes, which end at $(s,\, n-s-1)$, where $C_{n-1}^s$ is a
binomial coefficient. If $s_1,\ldots,\, s_T$ are numbers of these
sets, then $$ \lambda(C_n) =
\frac{\sum\limits_{i=1}^TC_{n-1}^{s_i}}{2^{n-1}}.
$$ As $C_{n-1}^{[(n-1)/2]}$ is the biggest binomial coefficient of
the form $C_{n-1}^s$, then we have the restriction $$ \lambda(C_n)
< \frac{T\cdot C_{n-1}^{[n-1/2]}}{2^{n-1}}.
$$

Assume, that $n=2m+1$ is odd. Then Stirling approximation gives $$
\frac{T\cdot C_{n-1}^{[n-1/2]}}{2^{n-1}} = \frac{T\cdot
C_{2m}^{m}}{2^{2m}} \sim \frac{T\cdot \sqrt{2\pi \cdot 2m}\cdot
\left(\frac{2m}{e}\right)^{2m}}{2\pi m\cdot
\left(\frac{m}{e}\right)^{2m} 2^{2m}} = $$ $$=\frac{T\cdot
C_{2m}^{m}}{2^{2m}} \sim \frac{T\cdot \sqrt{2\pi \cdot 2m}\cdot
2^{2m}}{2\pi m\cdot 2^{2m}} \sim \frac{1}{\sqrt{m}}.
$$

The result for the case, when $n$ is even is the same, whence $$
\lim\limits_{n\rightarrow \infty}\lambda (C_n) = 1,
$$ whence $$
\lambda(D)=1,
$$ which contradicts to the fact, that lengthes of the graphs
$h_n$ are bounded by $L<2$.
\end{proof}

\subsection{Finding of the length of the graph of
the conjugation by combinatorial
reasonings}\label{pr-sect-length-2}

Now we will find the evident formula for the length of the graph
of $h_n$, which approximates the conjugation $h_n$ of $f$ and
$f_v$. The graph of $h_{n+1}$ divides $[0,\, 1]$ into $2^n$ equal
parts, where it is linear. At each of these parts the derivative
of $h_{n+1}$ equals $\prod\limits_{i=2}^{n+1} \alpha_i$. Each of
these $n$ multipliers can be either $2v$, of $2(1-v)$.

Let for some interval of the length $2^{-n}$ the derivative is
$t$. It means, that $\tan \alpha = t$, where $\alpha$ is an angle
between the graph and $x$-axis. Then $\cos\alpha =
\sqrt{\frac{1}{1+t^2}}$ and the length of the graph on this
interval is $\frac{1}{2^n\cos\alpha} = \frac{1}{2^n}\sqrt{
1+t^2}$.

As all the both values of $\alpha_i$ are equally expected (for the
random uniformly distributed number), then the probability, that
$\prod\limits_{i=2}^{n+1} \alpha_i = (2v)^k(2-2v)^{n-k}$ equals
$\frac{C_n^k}{2^n}$, where $C_n^k$ is the binomial coefficient.

These reasonings prove the following proposition.

\begin{proposition}\label{prop:3}
The following equality holds
\begin{equation}\label{pr-eq:20} l_{n+1}(v) =
\frac{1}{2^{n}}\cdot\sum\limits_{k=0}^nC_n^k \cdot
\sqrt{1+2^{2n}v^{2k}(1-v)^{2(n-k)}}\end{equation} for the length
$l_{n+1}(v)$ of $h_{n+1}$.
\end{proposition}

The following combinatorial fact follows from
Theorem~\ref{pr-Theor-3}.

\begin{theorem}\label{theor:03}
For every $v\in (0,\, 1)\backslash \{ 0.5\}$ the limit $
\lim\limits_{n\rightarrow \infty}l_n(v)=2$ holds, where $l_n(v)$
are defined by~(\ref{pr-eq:20}).
\end{theorem}

Notice, that expression~(\ref{pr-eq:20}) has cense also for $v =
0,\, 0.5$ and $1$. Obviously, $l_n(0)=l_n(1)=1$ and
$l_n(0.5)=\sqrt{2}$ for $l_n(v)$ be defined by~(\ref{pr-eq:20}).
The case $l_n(0.5)$ corresponds to the trivial conjugation $y=x$
of the maps $f$ with itself.

Prof. Georgiy Shevchenko from Taras Shevchenko National University
of Kyiv noticed us that Theorem~\ref{theor:03} can be simply
proven with the use of probability theory reasonings.

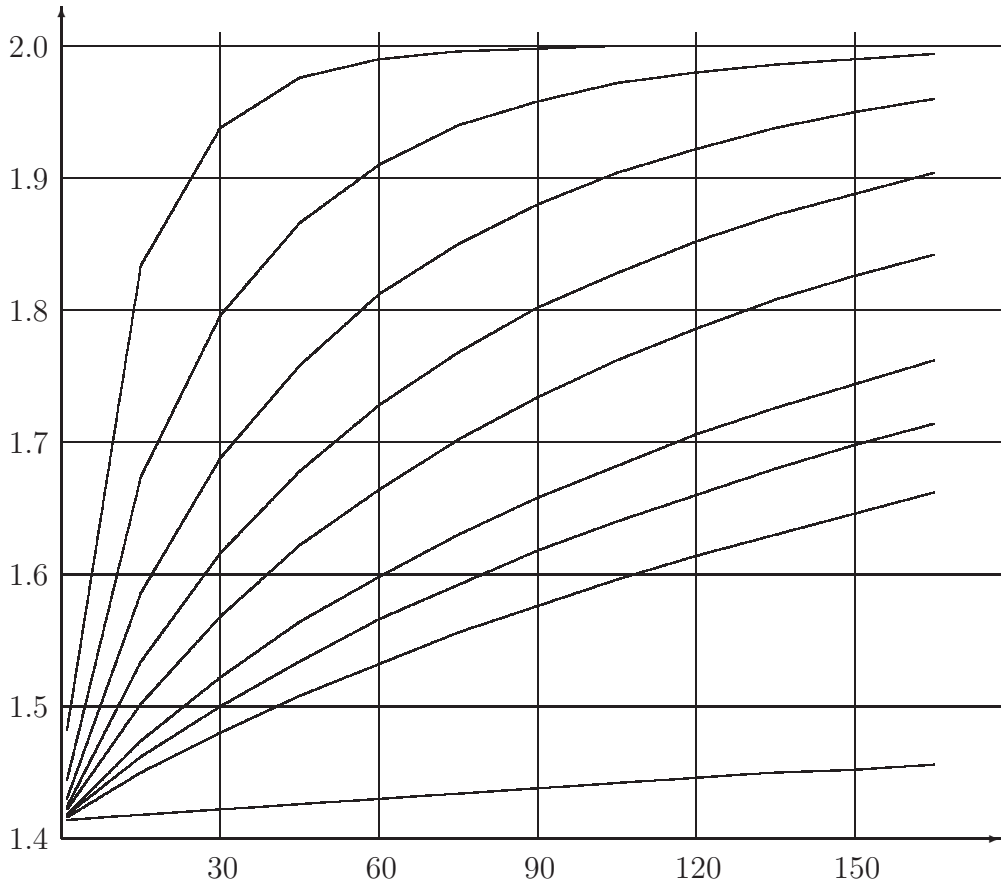
\begin{figure}[ht]
\begin{center}
\begin{picture}(380,335)

\Vidr{27}{61}{55}{237}  \VidrTo{85}{289}    \VidrTo{115}{308}
\VidrTo{145}{315}   \VidrTo{175}{318}   \VidrTo{205}{319}
\VidrTo{235}{320}   \VidrTo{265}{320}   \VidrTo{295}{320}
\VidrTo{325}{320}   \VidrTo{355}{320}

\Vidr{27}{42}{55}{157}  \VidrTo{85}{218}    \VidrTo{115}{253}
\VidrTo{145}{275}   \VidrTo{175}{290}   \VidrTo{205}{299}
\VidrTo{235}{306}   \VidrTo{265}{310}   \VidrTo{295}{313}
\VidrTo{325}{315}   \VidrTo{355}{317}

\Vidr{27}{35}{55}{113}  \VidrTo{85}{164}    \VidrTo{115}{199}
\VidrTo{145}{226}   \VidrTo{175}{245}   \VidrTo{205}{260}
\VidrTo{235}{272}   \VidrTo{265}{281}   \VidrTo{295}{289}
\VidrTo{325}{295}   \VidrTo{355}{300}

\Vidr{27}{32}{55}{87}   \VidrTo{85}{128}    \VidrTo{115}{159}
\VidrTo{145}{184}   \VidrTo{175}{204}   \VidrTo{205}{221}
\VidrTo{235}{234}   \VidrTo{265}{246}   \VidrTo{295}{256}
\VidrTo{325}{264}   \VidrTo{355}{272}

\Vidr{27}{31}{55}{71}   \VidrTo{85}{104}    \VidrTo{115}{131}
\VidrTo{145}{152}   \VidrTo{175}{171}   \VidrTo{205}{187}
\VidrTo{235}{201}   \VidrTo{265}{213}   \VidrTo{295}{224}
\VidrTo{325}{233}   \VidrTo{355}{241}

\Vidr{27}{29}{55}{57}   \VidrTo{85}{81} \VidrTo{115}{102}
\VidrTo{145}{119}   \VidrTo{175}{135}   \VidrTo{205}{149}
\VidrTo{235}{161}   \VidrTo{265}{173}   \VidrTo{295}{183}
\VidrTo{325}{192}   \VidrTo{355}{201}

\Vidr{27}{29}{55}{51}   \VidrTo{85}{70} \VidrTo{115}{87}
\VidrTo{145}{103}   \VidrTo{175}{116}   \VidrTo{205}{129}
\VidrTo{235}{140}   \VidrTo{265}{150}   \VidrTo{295}{160}
\VidrTo{325}{169}   \VidrTo{355}{177}

\Vidr{27}{28}{55}{45}   \VidrTo{85}{60} \VidrTo{115}{74}
\VidrTo{145}{86}    \VidrTo{175}{98}    \VidrTo{205}{108}
\VidrTo{235}{118}   \VidrTo{265}{127}   \VidrTo{295}{135}
\VidrTo{325}{143}   \VidrTo{355}{151}

\Vidr{27}{27}{55}{29}   \VidrTo{85}{31} \VidrTo{115}{33}
\VidrTo{145}{35}    \VidrTo{175}{37}    \VidrTo{205}{39}
\VidrTo{235}{41}    \VidrTo{265}{43}    \VidrTo{295}{45}
\VidrTo{325}{46}    \VidrTo{355}{48}

\put(25,20){\vector(0,1){315}} \put(25,20){\vector(1,0){355}}
\put(25,70){\line(1,0){355}} \put(25,120){\line(1,0){355}}
\put(25,170){\line(1,0){355}} \put(25,220){\line(1,0){355}}
\put(25,270){\line(1,0){355}} \put(25,320){\line(1,0){355}}

\put(85,20){\line(0,1){305}} \put(145,20){\line(0,1){305}}
\put(205,20){\line(0,1){305}} \put(265,20){\line(0,1){305}}
\put(325,20){\line(0,1){305}}

\put(5,16){1.4} \put(5,66){1.5} \put(5,116){1.6} \put(5,166){1.7}
\put(5,216){1.8} \put(5,266){1.9} \put(5,316){2.0}

\put(80,5){30} \put(140,5){60} \put(200,5){90} \put(257,5){120}
\put(317,5){150}

\end{picture}
\end{center}
\caption{Graphs of $l_{n+1}(v)$} \label{fig_6}
\end{figure}

Let $l_{n+1}(v)$ be given by~(\ref{pr-eq:20}). The graphs of $y(n)
=l_{n+1}(v)$ for $$v\in \{ 0.52;\, 0.54;\, 0.56;\, 0.57;\, 0.58;\,
0.6;\, 0.62;\, 0.65;\, 0.7;\, 0.8\}$$ are given at
Figure~\ref{fig_6}. Here $l_{n+1}(v)$ is the length of the graph
of $h_{n+1}$, dependent on $v$.

Clearly, it is not evident from the graph at Figure~\ref{fig_6}
that, for instance, $\lim\limits_{n\rightarrow \infty}l_n(0.52)
=2$. Nevertheless, graphs of $y = l_{n+1}(0.51),\, y
=l_{n+1}(0.512),\, y =l_{n+1}(0.515)$ and $y = l_{n+1}(0.52)$ are
given at Figure~\ref{fig_13}, where $n$ is considered from 20,000
(for $y = l_{n+1}(0.51)$ with $l_{20,001}(0.52) = 1,99611$) to
50,000 (for $y = l_{n+1}(0.51)$ with $l_{50,001}(0.51) =
1,97903$).

\setlength{\unitlength}{0.85pt}

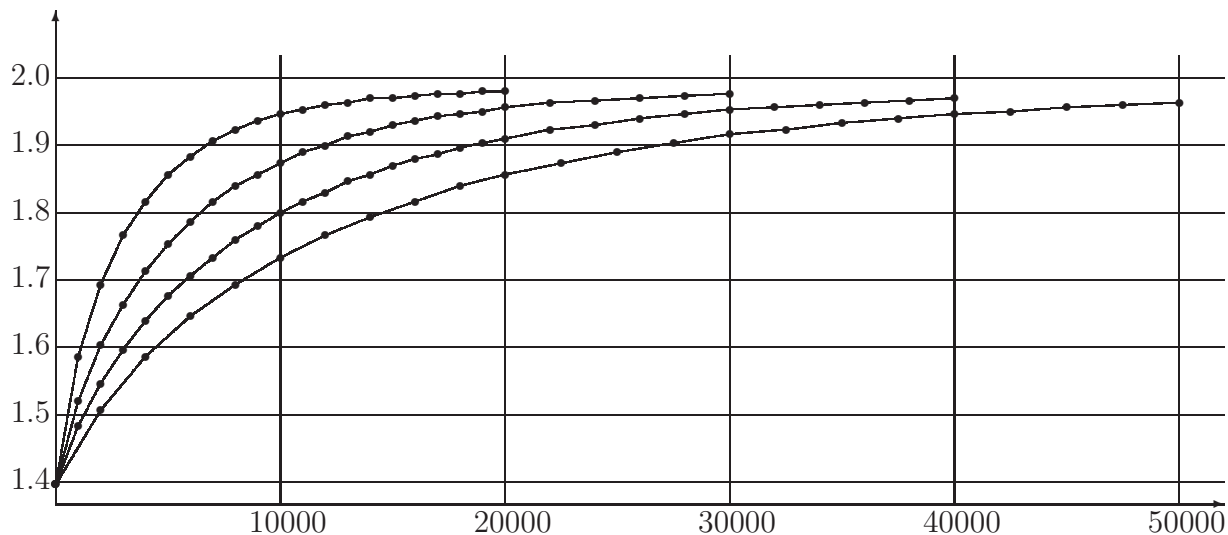
\begin{figure}[htbp]
\begin{center}
\begin{picture}(545,245)

\put(25,20){\vector(0,1){220}} \put(25,20){\vector(1,0){520}}

\put(125,20){\line(0,1){200}} \put(225,20){\line(0,1){200}}
\put(325,20){\line(0,1){200}} \put(425,20){\line(0,1){200}}
\put(525,20){\line(0,1){200}}

\put(25,30){\line(1,0){520}} \put(25,60){\line(1,0){520}}
\put(25,90){\line(1,0){520}} \put(25,120){\line(1,0){520}}
\put(25,150){\line(1,0){520}} \put(25,180){\line(1,0){520}}
\put(25,210){\line(1,0){520}}

\put(111,8){10000} \put(211,8){20000} \put(311,8){30000}
\put(411,8){40000} \put(511,8){50000}

\put(5,27){1.4} \put(5,57){1.5}  \put(5,87){1.6} \put(5,117){1.7}
\put(5,147){1.8} \put(5,177){1.9} \put(5,207){2.0}

\put(25,29){\circle*{3}} \put(45,62){\circle*{3}}
\put(65,86){\circle*{3}} \put(85,104){\circle*{3}}
\put(105,118){\circle*{3}} \put(125,130){\circle*{3}}
\put(145,140){\circle*{3}} \put(165,148){\circle*{3}}
\put(185,155){\circle*{3}} \put(205,162){\circle*{3}}
\put(225,167){\circle*{3}} \put(250,172){\circle*{3}}
\put(275,177){\circle*{3}} \put(300,181){\circle*{3}}
\put(325,185){\circle*{3}} \put(350,187){\circle*{3}}
\put(375,190){\circle*{3}} \put(400,192){\circle*{3}}
\put(425,194){\circle*{3}} \put(450,195){\circle*{3}}
\put(475,197){\circle*{3}} \put(500,198){\circle*{3}}
\put(525,199){\circle*{3}}

\Vidr{25}{29}{45}{62} \VidrTo{65}{86} \VidrTo{85}{104}
\VidrTo{105}{118} \VidrTo{125}{130} \VidrTo{145}{140}
\VidrTo{165}{148} \VidrTo{185}{155} \VidrTo{205}{162}
\VidrTo{225}{167} \VidrTo{250}{172} \VidrTo{275}{177}
\VidrTo{300}{181} \VidrTo{325}{185} \VidrTo{350}{187}
\VidrTo{375}{190} \VidrTo{400}{192} \VidrTo{425}{194}
\VidrTo{450}{195} \VidrTo{475}{197} \VidrTo{500}{198}
\VidrTo{525}{199}

\put(25,29){\circle*{3}} \put(35,86){\circle*{3}}
\Vidr{25}{29}{35}{86} \put(45,118){\circle*{3}}   \VidrTo{45}{118}
\put(55,140){\circle*{3}}   \VidrTo{55}{140}
\put(65,155){\circle*{3}}   \VidrTo{65}{155}
\put(75,167){\circle*{3}}   \VidrTo{75}{167}
\put(85,175){\circle*{3}}   \VidrTo{85}{175}
\put(95,182){\circle*{3}}   \VidrTo{95}{182}
\put(105,187){\circle*{3}}  \VidrTo{105}{187}
\put(115,191){\circle*{3}}  \VidrTo{115}{191}
\put(125,194){\circle*{3}}  \VidrTo{125}{194}
\put(135,196){\circle*{3}}  \VidrTo{135}{196}
\put(145,198){\circle*{3}}  \VidrTo{145}{198}
\put(155,199){\circle*{3}}  \VidrTo{155}{199}
\put(165,201){\circle*{3}}  \VidrTo{165}{201}
\put(175,201){\circle*{3}}  \VidrTo{175}{201}
\put(185,202){\circle*{3}}  \VidrTo{185}{202}
\put(195,203){\circle*{3}}  \VidrTo{195}{203}
\put(205,203){\circle*{3}}  \VidrTo{205}{203}
\put(215,204){\circle*{3}}  \VidrTo{215}{204}
\put(225,204){\circle*{3}}  \VidrTo{225}{204}

\put(25,29){\circle*{3}} \put(35,66){\circle*{3}}
\Vidr{25}{29}{35}{66} \put(45,91){\circle*{3}}    \VidrTo{45}{91}
\put(55,109){\circle*{3}}   \VidrTo{55}{109}
\put(65,124){\circle*{3}}   \VidrTo{65}{124}
\put(75,136){\circle*{3}}   \VidrTo{75}{136}
\put(85,146){\circle*{3}}   \VidrTo{85}{146}
\put(95,155){\circle*{3}}   \VidrTo{95}{155}
\put(105,162){\circle*{3}}  \VidrTo{105}{162}
\put(115,167){\circle*{3}}  \VidrTo{115}{167}
\put(125,172){\circle*{3}}  \VidrTo{125}{172}
\put(135,177){\circle*{3}}  \VidrTo{135}{177}
\put(145,180){\circle*{3}}  \VidrTo{145}{180}
\put(155,184){\circle*{3}}  \VidrTo{155}{184}
\put(165,186){\circle*{3}}  \VidrTo{165}{186}
\put(175,189){\circle*{3}}  \VidrTo{175}{189}
\put(185,191){\circle*{3}}  \VidrTo{185}{191}
\put(195,193){\circle*{3}}  \VidrTo{195}{193}
\put(205,194){\circle*{3}}  \VidrTo{205}{194}
\put(215,195){\circle*{3}}  \VidrTo{215}{195}
\put(225,197){\circle*{3}}  \VidrTo{225}{197}
\put(245,199){\circle*{3}}  \VidrTo{245}{199}
\put(265,200){\circle*{3}}  \VidrTo{265}{200}
\put(285,201){\circle*{3}}  \VidrTo{285}{201}
\put(305,202){\circle*{3}}  \VidrTo{305}{202}
\put(325,203){\circle*{3}}  \VidrTo{325}{203}

\put(25,29){\circle*{3}} \put(35,55){\circle*{3}}
\Vidr{25}{29}{35}{55} \put(45,74){\circle*{3}}    \VidrTo{45}{74}
\put(55,89){\circle*{3}}    \VidrTo{55}{89}
\put(65,102){\circle*{3}}   \VidrTo{65}{102}
\put(75,113){\circle*{3}}   \VidrTo{75}{113}
\put(85,122){\circle*{3}}   \VidrTo{85}{122}
\put(95,130){\circle*{3}}   \VidrTo{95}{130}
\put(105,138){\circle*{3}}  \VidrTo{105}{138}
\put(115,144){\circle*{3}}  \VidrTo{115}{144}
\put(125,150){\circle*{3}}  \VidrTo{125}{150}
\put(135,155){\circle*{3}}  \VidrTo{135}{155}
\put(145,159){\circle*{3}}  \VidrTo{145}{159}
\put(155,164){\circle*{3}}  \VidrTo{155}{164}
\put(165,167){\circle*{3}}  \VidrTo{165}{167}
\put(175,171){\circle*{3}}  \VidrTo{175}{171}
\put(185,174){\circle*{3}}  \VidrTo{185}{174}
\put(195,176){\circle*{3}}  \VidrTo{195}{176}
\put(205,179){\circle*{3}}  \VidrTo{205}{179}
\put(215,181){\circle*{3}}  \VidrTo{215}{181}
\put(225,183){\circle*{3}}  \VidrTo{225}{183}
\put(245,187){\circle*{3}}  \VidrTo{245}{187}
\put(265,189){\circle*{3}}  \VidrTo{265}{189}
\put(285,192){\circle*{3}}  \VidrTo{285}{192}
\put(305,194){\circle*{3}}  \VidrTo{305}{194}
\put(325,196){\circle*{3}}  \VidrTo{325}{196}
\put(345,197){\circle*{3}}  \VidrTo{345}{197}
\put(365,198){\circle*{3}}  \VidrTo{365}{198}
\put(385,199){\circle*{3}}  \VidrTo{385}{199}
\put(405,200){\circle*{3}}  \VidrTo{405}{200}
\put(425,201){\circle*{3}}  \VidrTo{425}{201}

\end{picture}
\end{center}
\caption{Graphs of $y = l_{n+1}(0.51),\, y =l_{n+1}(0.512),\, y
=l_{n+1}(0.515)$ and $y = l_{n+1}(0.52)$} \label{fig_13}
\end{figure}

\setlength{\unitlength}{1pt}

\newpage

\section{Piecewise linear approximations of self semi
conjugation}\label{sect-KuskLin}

In this section we consider the topological self-semiconjugation
of the maps $f:\, [0,\, 1]\rightarrow [0,\, 1]$, given by
\begin{equation}\label{eq:37} f(x) = \left\{\begin{array}{ll}
2x,& 0\leq x< 1/2;\\
2-2x,& 1/2 \leqslant x\leqslant 1,
\end{array}\right.
\end{equation}
i.e. we consider the properties of continuous surjective $h:\,
[0,\, 1]\rightarrow [0,\, 1]$ of the functional equation
\begin{equation}\label{eq:38} \xi(f) = f(\xi).
\end{equation}

We have considered in the Section~\ref{sect-Pobudowa-1} the sets
$$
A_n = \left\{ \frac{t}{2^{n-1}},\, 0\leq t\leq 2^{n-1}\right\}
$$ and according to Proposition~\ref{lema:An}, these $A_n$ are
solutions of the equation $f^n(x)=0$.

For any $n\geq 1$ consider the map $\xi_n:\, A_n \rightarrow A_n$
consider the possibility of its continuation to a
self-semiconjugation of $f$.

\begin{notation}
The maps $\xi_n:\, A_n \rightarrow [0,\, 1]$ is called
\textbf{admissible}, if the equality
\begin{equation}\label{eq:44}\xi(f(x)) = f(\xi(x)) \end{equation} for all
$x\in A_n$.

Notice, that since $f(A_n)\subseteq A_n$ for all $n$,
then~(\ref{eq:44}) is defined for all $x\in A_n$.
\end{notation}

\begin{notation}
The maps $\xi_n:\, A_n \rightarrow A_n$ is called
\textbf{continuable}, if the there is a self-semiconjugation $h:\,
[0,\, 1]\rightarrow [0,\, 1]$ of $f$, which coincides with $\xi_n$
on $A_n$.

In this case the maps $h$ can be considered as continuous
surjective continuation of $\xi_n$.
\end{notation}

\subsection{Admissible self-semiconjugations}
\label{sect-KuskLin-1}

Let $\xi_n:\, A_n \rightarrow [0,\, 1]$ be an admissible maps.
Denote
$$x_0 = \xi(0).
$$ From the commutative diagram $$\begin{CD}
0 @>f >> & 0\\
@V_{h} VV& @VV_{h}V\\
x_0 @>f>>& x_0
\end{CD}$$ obtain that $x_0$ is a fixed point of $f$, whence
\begin{equation}\label{eq:94}x_0\in \{0,\, 2/3\}.\end{equation}

Denote $\varphi_0(x)=2x$ for $x\in [0,\, 1/2]$ and denote
$\varphi_2(x)=2-2x$ for $x\in [1/2,\, 1]$. This notation let us to
rewrite~(\ref{eq:37}) as \begin{equation}\label{eq:52} f(x) =
\left\{\begin{array}{ll}
\varphi_0,& 0\leq x< 1/2;\\
\varphi_1,& 1/2 \leqslant x\leqslant 1.
\end{array}\right.
\end{equation}
Notice, that maps $\varphi_i,\, i=0,\, 1$ are invertible.

Consider the commutative diagram
$$\begin{CD}
1 @>f >> & 0\\
@V_{\psi} VV& @VV_{\psi}V\\
h(1) @>f>>& x_0
\end{CD}$$ and conclude that
\begin{equation}\label{eq:93}h(1) = \varphi_{i_1}^{-1}(x_0)\ \text{ for some }\ i_1\in
\{0,\, 1\}.\end{equation}

From equalities~(\ref{eq:94}) and~(\ref{eq:93}) we obtain the
following lemma.

\begin{lemma}\label{lema:04}
For any $x\in A_1$ there exist $i_1\in \{0,\, 1\}$ such that
$\xi_n(x)=\varphi_{i_1}^{-1}(x_0)$.
\end{lemma}

Denote by $i_0\in \{0,\, 1\}$ such number that
\begin{equation}\label{eq:97}x_0 = \varphi_{i_0}(x_0).\end{equation}

This lemma can be inductively generalized as follows.

\begin{lemma}\label{lema:06}
For any $m\leq n$ and any $x\in A_m$ there exist $i_1,\ldots,\,
i_m \in \{0,\, 1\}$ such that $$\xi_n(x)=\varphi_{i_m}^{-1}(\ldots
(\varphi_{i_1}^{-1}(x_0))\ldots ).$$
\end{lemma}

\begin{proof}
The base of induction (the case $m=1$) follows from
Lemma~\ref{lema:04}.

Assume that for $m=k$ the lemma is proved. For any $x\in A_{k+1}$
consider the commutative diagram $$\begin{CD}
x @>f >> & f(x)\\
@V_{\xi} VV& @VV_{\xi}V\\
\xi(x) @>f>> & \xi(f(x))
\end{CD}$$

Since $f(x)\in A_k$, then by the assumption of induction obtain
$\xi(f(x)) =\varphi_{i_k}^{-1}(\ldots
(\varphi_{i_1}^{-1}(x_0))\ldots )$, whence $$ f(\xi(x)) =
\varphi_{i_k}^{-1}(\ldots (\varphi_{i_1}^{-1}(x_0))\ldots ).
$$ The last equality means that there exists $i_{k+1}$ such that
$$
\xi(x) = \varphi_{i_{k+1}}^{-1}(\ldots
(\varphi_{i_1}^{-1}(x_0))\ldots ).
$$
\end{proof}

\begin{note}\label{note:18}
From the definition of $A_n$ obtain that $x\in A_n$ if and only if
there exist $j_1,\ldots,\, j_n$ such that
\begin{equation}\label{eq:95} x = \varphi_{j_n}^{-1}(\ldots
(\varphi_{j_1}^{-1}(0))\ldots ).
\end{equation}
\end{note}

\begin{notation}
For any $m\geq 1$ denote by $\mathcal{B}_m$ the set of sequences
the the length $m$, consisted of $0$-s and $1$-s. Also denote
$\mathcal{B} = \bigcup\limits_{i=1}^\infty\mathcal{B}_i$.
\end{notation}

If follows from Lemma~\ref{lema:06} that for any $m\geq 1$ and
$j_1,\ldots,\, j_m$ there exist $i_1,\ldots,\, i_m$ such that
\begin{equation}\label{eq:98}
\xi(\varphi_{j_n}^{-1}(\ldots (\varphi_{j_1}^{-1}(0))\ldots )) =
\varphi_{i_m}^{-1}(\ldots (\varphi_{i_1}^{-1}(x_0))\ldots ).
\end{equation}

This equality means that $\xi$ generates the maps
$\widetilde{\xi}:\, \mathcal{B} \rightarrow \mathcal{B}$ such that
for any $m\geq 1$ the inclusion
$\widetilde{\xi}(\mathcal{B}_m)\subseteq \mathcal{B}_m$ holds and
\begin{equation}\label{eq:96} \widetilde{\xi}(j_1,\ldots,\, j_m) = (i_1,\ldots,\,
i_m)
\end{equation} for $1\leq m\leq n$, and $i_1,\ldots i_m,\, j_1,\ldots
,\,j_m$ such that~(\ref{eq:98}) holds.

\begin{lemma}\label{lema:05}
Let $\widetilde{\xi}$ be defined by~(\ref{eq:96}) for an
admissible $\xi$. For $m\leq n$ and $t<m$ the equality $$
\widetilde{\xi}(j_1,\ldots,\, j_m) = (i_1,\ldots,\, i_m)
$$ yields $$
\widetilde{\xi}(j_1,\ldots,\, j_{m-t}) = (i_1,\ldots,\, i_{m-t}).
$$
\end{lemma}

\begin{proof}
Prove the lemma for $t=1$. Consider the commutative diagram

$$\begin{CD}
\varphi_{j_m}^{-1}(\ldots (\varphi_{j_1}^{-1}(0))\ldots ) @>f >> &
\varphi_{j_{m-1}}^{-1}(\ldots
(\varphi_{j_1}^{-1}(0))\ldots )\\
@V_{\xi} VV& @VV_{\xi}V\\
\varphi_{i_m}^{-1}(\ldots (\varphi_{i_1}^{-1}(x_0))\ldots ) @>f>>
& \varphi_{i_{m-1}}^{-1}(\ldots (\varphi_{i_1}^{-1}(x_0))\ldots ).
\end{CD}$$

For $t>1$ the reasonings are the same.
\end{proof}

\begin{note}
Notice, that~(\ref{eq:95}) is not one to one correspondence
between $A_n$ and finite sequences $j_n,\ldots,\, j_n$,
whence~(\ref{eq:96}) does not define the one to one correspondence
between admissible $\xi$ and the the functions, whose domain and
gange is the finite sequence of $0$-s and $1$-s. Indeed, if
$j_1,\ldots,\, j_k=0$ and $j_{k+1}\neq 0$, then
$$ \varphi_{j_n}^{-1}(\ldots (\varphi_{j_1}^{-1}(0))\ldots ) =
\varphi_{j_n}^{-1}(\ldots (\varphi_{j_2}^{-1}(0))\ldots ) =\ldots
\, = \varphi_{j_n}^{-1}(\ldots (\varphi_{j_k}^{-1}(0))\ldots ),
$$ and
$$
\varphi_{j_n}^{-1}(\ldots (\varphi_{j_k}^{-1}(0))\ldots ) \neq
\varphi_{j_n}^{-1}(\ldots (\varphi_{j_{k+1}}^{-1}(0))\ldots ).
$$
\end{note}

These reasonings prove the following lemma.

\begin{lemma}\label{lema:07}
Let $\widetilde{\xi}$ be defined by an admissible $\xi$. Let for
some $m\geq 1$ and $j_1,\ldots,\, j_m,\, i_1,\ldots,\, i_m$ the
equality~(\ref{eq:96}) holds. Let $i_0$ be defined
by~(\ref{eq:97}). Then for any $k,\, 1\leq k\leq m$ the equality
$j_k=0$  yields $i_k=i_0$.
\end{lemma}

\begin{proof}
Prove the lemma for $k=1$. Consider $x\in [0,\, 1]$ defined
by~(\ref{eq:95}). The case $k=1$ means that $j_1=0$, whence
$\varphi_{j_1}^{-1}(0)=0$, which means that $$
\varphi_{j_m}^{-1}(\ldots (\varphi_{j_1}^{-1}(0))\ldots ) =
\varphi_{j_m}^{-1}(\ldots (\varphi_{j_2}^{-1}(0))\ldots ).
$$

The equality $$ \xi(\varphi_{j_m}^{-1}(\ldots
(\varphi_{j_1}^{-1}(0))\ldots )) = \xi(\varphi_{j_m}^{-1}(\ldots
(\varphi_{j_2}^{-1}(0))\ldots ))
$$ means that $$
\varphi_{i_m}^{-1}(\ldots (\varphi_{i_1}^{-1}(x_0))\ldots ) =
\varphi_{i_m'}^{-1}(\ldots (\varphi_{i_2'}^{-1}(x_0))\ldots ).
$$

Applying $f^{m-1}$ to both sides of the obtained equality obtain
$$ f(\varphi_{i_1}^{-1}(x_0)) = x_0,
$$ which means that $i_1=i_0$.

The case $k>1$ follows from the case $k=1$ and
Lemma~\ref{lema:05}.
\end{proof}

\begin{theorem}\label{theor:04}
There is one to one correspondence between admissible self-semi
conjugations $\xi:\, A_n  \rightarrow [0,\, 1]$ and maps
$\widetilde{\xi}:\, \bigcup\limits_{i=1}^n\mathcal{B}_i
\rightarrow \bigcup\limits_{i=1}^n\mathcal{B}_i$ with the
following properties:

(1) For any $m,\, 1\leq m\leq n$, the inclusion
$\widetilde{\xi}(\mathcal{B}_m) \subseteq \mathcal{B}_m$ holds.

(2) For any $m,\, 2\leq m\leq n$ the equality $$
\widetilde{\xi}(j_1,\ldots,\, j_m) = (i_1,\ldots,\, i_m)
$$ yields $$
\widetilde{\xi}(j_1,\ldots,\, j_{m-1}) = (i_1,\ldots,\, i_{m-1}).
$$

(3) If the equality $$ \widetilde{\xi}(j_1,\ldots,\, j_n) =
(i_1,\ldots,\, i_n)
$$ holds for some $i_1,\ldots,\, i_n,\, j_1,\ldots,\, j_n$, then
for any $k,\, 1\leq k\leq n$ the equality $i_k=0$ yields
$j_k=i_0$, where $i_0\in \{0,\, 1\}$ is a fixed number.
\end{theorem}

\begin{proof}
If follows from Lemmas~\ref{lema:06}, \ref{lema:05}
and~\ref{lema:07} that maps $\widetilde{\xi}$, which is defined by
$\xi$ via the equalities~(\ref{eq:98}) and~(\ref{eq:96}) satisfies
the conditions (1), (2) and (3) of Theorem for $i_0$ such that
$\varphi_{i_0}(\xi(0))  = \xi(0)$.

Let $\widetilde{\xi}$ satisfy conditions (1), (2) and (3) of
Theorem. Define an admissible $\xi:\, A_n \rightarrow [0,\, 1]$ as
follows.

If $i_0=0$ then let $x_0=0$, otherwise let $x_0=1$.

For any $x\in A_n$ define $\xi(x)$ as follows. By the
Remark~\ref{note:18} there exist $j_1,\ldots,\, j_n$ such that $$
x = \varphi_{j_n}^{-1}(\ldots (\varphi_{j_1}^{-1}(0))\ldots ).$$
Let $i_1,\ldots,\, i_n$ be such that \begin{equation}\label{eq:99}
\widetilde{\xi}(j_1,\ldots,\, j_n) = (i_1,\ldots,\, i_n).
\end{equation} Now define $\xi(x)$ as $$ \xi(x)=
\varphi_{i_n}^{-1}(\ldots (\varphi_{i_1}^{-1}(x_0))\ldots ).$$

Prove that the obtained $\xi$ would be admissible.

Evidently $f(\xi(x)) = \varphi_{i_{n-1}}^{-1}(\ldots
(\varphi_{i_1}^{-1}(x_0))\ldots )$. From the other hand,
\begin{equation}\label{eq:23}f(x) = \varphi_{j_{n-1}}^{-1}(\ldots
(\varphi_{j_1}^{-1}(0))\ldots ).
\end{equation}

Consider $$x^* = \varphi_{j_{n-1}}^{-1}(\ldots
(\varphi_{j_1}^{-1}(0))\ldots ) \in A_{n-1}$$.

From~(\ref{eq:99}) and from condition (2) of Theorem obtain that
$$\widetilde{\xi}(j_1,\ldots,\, j_{n-1}) = (i_1,\ldots,\,
i_{n-1}),
$$ whence
\begin{equation}\label{eq:100}
\xi(x^*) = \varphi_{j_{n-1}}^{-1}(\ldots
(\varphi_{j_1}^{-1}(0))\ldots ).\end{equation}

Now Theorem follows from the comparing of equations~(\ref{eq:23})
and~(\ref{eq:100}).
\end{proof}

\begin{corollary}\label{corol:1}
For any $n\geq 1$ the number of admissible self-semi conjugations
$\xi:\ A_n \rightarrow [0,\, 1]$ is
$$\sum\limits_{k=0}^n2^{k+1}\cdot C_n^k$$.
\end{corollary}

\begin{proof}
By Theorem~\ref{theor:04} we can calculate the maps
$\widetilde{\xi}$ from the Theorem instead of admissible maps.

By condition (2) of Theorem~\ref{theor:04} it is enough to define
$\widetilde{\xi}$ only on $\mathcal{B}_n$.

Let $\mathcal{J} =(j_1,\ldots,\, j_n)\in \mathcal{B}_n$ be a
typical element of $\mathcal{B}_n$. Let $k$ be the quantity of
$1$-s in $\mathcal{J}$. There are $C_n^k$ possibilities to choose
positions $j_{s_1},\ldots,\, j_{s_k}$ of these $1$-s and there are
$2^k$ possibilities to choose correspond $i_{s_1},\ldots,\,
i_{s_k}$. All another elements of $\mathcal{J}$ will be $i_0$, and
we have $2$ possibilities for $i_0$.

Whence, there are $$ \sum\limits_{k=0}^n2^{k+1}\cdot C_n^k.
$$ possibilities to choose maps $\widetilde{\xi}$, which
satisfy the conditions of Theorem~\ref{theor:04}.
\end{proof}

\subsection{Continuable self-semiconjugations}
\label{sect-KuskLin-2}

We will pay our attention to continuable self-semiconjugations
$\xi:\, A_n \rightarrow [0,\, 1]$ in this section, where $n$ is
considered to be arbitrary fixed.

\begin{lemma}\label{lema:14}
Let $\xi_1,\, \xi_2:\, A_n \rightarrow [0,\, 1]$ be admissible
self-semiconjugations. If $\xi_1(x) = \xi_2(x)$ for all $x\in
A_n\setminus A_{n-1}$, then $\xi_1(x) = \xi_2(x)$ for all $x\in
A_n$.
\end{lemma}

\begin{proof}
According to Proposition~\ref{lema:An}, $$ A_n\setminus A_{n-1}
=\left\{ \frac{2t+1}{2^{n-1}},\ t\in [0,\, 2^{n-2}-1]\right\}.
$$

For $t\in [0,\, 2^{n-2}-1]$ and $x = \frac{2t+1}{2^{n}}$ consider
the following commutative diagram for $\xi_i,\, i=1,\, 2$:
$$\begin{CD}
x @>f >> & f(x)\\
@V_{\xi_i} VV& @VV_{\xi_i}V\\
\xi(x) @>f>>& f(\xi_i(x)).
\end{CD}$$

Since $\xi_1(x) = \xi_2(x)$ for all $x\in A_n \setminus A_{n-1}$,
then it follows from the commutative diagram, that $\xi_1(x) =
\xi_2(x)$ for all $x\in f(A_n \setminus A_{n-1})$. It follows from
the definition of $A_n$, that $f(A_k)=A_{k-1}$ for all $k\geq 1$,
whence $f(A_n \setminus A_{n-1}) = A_{n-1}\setminus A_{n-2}$.

Applying $n-1$ times the reasonings above obtain that $\xi_1(x) =
\xi_2(x)$ for all $x\in A_n$.
\end{proof}

If follows from Theorems~\ref{theor:01} and~\ref{pr-theor-2} that
if $\xi:\, A_n\rightarrow [0,\, 1]$ is continuable, then either
$\xi(x)=2/3$ for all $x\in A_n$, or $\xi(0)=0$. More then this,
the maps $h:\, [0,\, 1]\rightarrow [0,\, 1]$ from
Theorem~\ref{pr-theor-2} are all possible continuations of
admissible $\xi:\, A_n \rightarrow [0,\, 1]$.

\begin{lemma}
If $\xi:\, A_n \rightarrow [0,\, 1]$ is a continuous
self-semiconjugation of $f$, then $\xi(A_n)\subseteq A_n$.
\end{lemma}

\begin{proof}
Since $\xi(0)=0$ by Theorems~\ref{theor:01} and~\ref{pr-theor-2}
then for any $x\in A_n$ it follows from the definition of
admissibility of $\xi$ that the following diagram $$\begin{CD}
x @>f^n >> & 0\\
@V_{\xi} VV& @VV_{\xi}V\\
\xi(x) @>f^n>>& 0
\end{CD}
$$ is commutative. This proves the Lemma.
\end{proof}

\begin{lemma}\label{lema:13}
Let $\xi:\, A_n \rightarrow A_n$ be continuable
self-semiconjugation of $f$ and $\xi(\alpha_{n,2s+1}) =
\alpha_{n,p}$ for some $\alpha_{n,2s+1}$ and $\alpha_{n,p}$.

1. If $h$ is a continuation of $\xi$ and $k$ is its tangent at
$0$, then either
\begin{equation}\label{eq:102}k(2s+1) - p \equiv 0\mod 2^{n},
\end{equation} or \begin{equation}\label{eq:104}k(2s+1) + p \equiv
0\mod 2^{n}.
\end{equation}

2. If $k$ satisfies either~(\ref{eq:102}) or~(\ref{eq:104}), then
there is a continuation $h:\, [0,\, 1]\rightarrow [0,\, 1]$ of
$\xi$, whose tangent at $0$ is $k$.
\end{lemma}

\begin{proof}
There exist $q\in \mathbb{N}\cup \{0\},\, q\leq 2^{n-1}-1$ such
that $\alpha_{n,2s+1}\in [\alpha_{n,q},\, \alpha_{n,q+1})$.

Consider the possibilities when $q$ is even and when it is odd.

If $q=2t$ for some $t$, then $h$ increase on $[\alpha_{n,q},\,
\alpha_{n,q+1}]$, whence \begin{equation}\label{eq:101}
\frac{\alpha_{n,2s+1}-\frac{2t}{k}}{\frac{1}{k}} = \alpha_{n,p}.
\end{equation} We can simplify this equality as $$ \frac{k(2s+1)}{2^{n-1}}
-2t = \frac{p}{2^{n-1}},
$$ $$
k(2s+1) - p = 2t\cdot 2^{n-1}.
$$

Notice, that the last equation means that~(\ref{eq:101}) is
equivalent to~(\ref{eq:102}).

Consider another case, i.e. $q=2t+1$ for some $t$. Then $h$
decrease on $[\alpha_{n,q},\, \alpha_{n,q+1}]$, whence
\begin{equation}\label{eq:103} \frac{\alpha_{n,2s+1}-\frac{2t+1}{k}}{\frac{1}{k}} =
1-\alpha_{n,p}. \end{equation} We can simplify this equality as $$
\frac{k(2s+1)}{2^{n-1}} - (2t+1) = 1 - \frac{p}{2^{n-1}},
$$ $$
k(2s+1) + p = 2(t+1)\cdot 2^{n-1}.
$$
The last equation means that~(\ref{eq:103}) is equivalent
to~(\ref{eq:104}).
\end{proof}

\begin{lemma}\label{corol:5}
1. The set of natural $k$, which satisfy the
congruence~(\ref{eq:102}) is $k = k_1 + 2^nt,\, t\in \mathbb{N}$
for some $k_1,\, 1\leq k_1\leq 2^{n}$.

2. The set of natural $k$, which satisfy the
congruence~(\ref{eq:104}) is $k = k_2 + 2^nt,\, t\in \mathbb{N}$
$k_2,\, 1\leq k_2\leq 2^{n}$.

3. $k_1 +k_2 = 2^n$, or $k_1+k_2 = 2^{n+2}$ for $k_1$ and $k_2$
from the previous items of this corollary.
\end{lemma}

\begin{proof}
1. Since $2s+1$ and $2^n$ are reciprocal simple then there exist
integers $\gamma_1$ and $\gamma_2$ such that $$
\gamma_1\cdot(2s+1) + \gamma_2 \cdot 2^n =1,
$$ whence $$
p\cdot \gamma_1\cdot (2s+1) \equiv p\mod 2^n.
$$ Since for every $p,\, 1\leq p\leq 2^n$ there exists $k_1,\, 1\leq
k_1\leq 2^n$, which satisfy~(\ref{eq:102}), then for each $p$ the
correspond $k$ is unique up to adding a number, which is divisible
by $2^n$.

2. The proof of second item is the came as of the first one.

3. If $k_1$ and $k_2,\, 1\leq k_1,\, k_2\leq 2^n$ are solutions
of~(\ref{eq:102}) and~(\ref{eq:104}) correspondingly then there
exist $t_1$ and $t_2$ such that $$ \left\{ \begin{array}{l}
k_1(2s+1) - p = 2^n\cdot t_1\\
k_2(2s+1) + p = 2^n\cdot t_2.
\end{array}\right.
$$ Adding these two equalities obtain that $k_1+k_2$ is divisible
by $2^n$.
\end{proof}

The main result of this section is the following theorem.

\begin{theorem}\label{theor-main-dopust}
1. For every $x\in A_n\setminus A_{n-1}$ and for every $y\in A_n$
there exists a continuable $\xi: A_n\rightarrow A_n$ such that
$\xi(x)=y$.

2. Let $\xi_1,\, \xi_2:\, A_n \rightarrow A_n$ be continuable
self-semiconjugations of $f$ of the form~(\ref{eq:37}) and
$\xi_1(x) = \xi_2(x)$ for some $x\in A_n\setminus A_{n-1}$. Then
$\xi_1(x)=\xi_2(x)$ for all $x\in A_n$.
\end{theorem}

\begin{proof}
The first part of Theorem follows from Lemma~\ref{lema:13} and
reasonings, which are similar to those from the proof of item 1 of
Lemma~\ref{corol:5}.

The second part of Theorem follows from the following
observations. For any $\alpha_{n,2s+1}\in A_n$ consider all
possible $\alpha_{n,p}$ such that $\xi(\alpha_{n,2s+1})
=\alpha_{n,p}$. According to item 3. of Lemma~\ref{corol:5}, each
$p$ (of all $2^{n-1}$ possible) divides the set of all possible
$k$ (up to adding a number, which is divisible by $2^n$) to pairs
$\{ k_1,\, k_2,\, 1\leq k_1,\, k_2\leq 2^n\}$ such that $k_1+k_2$
is divisible by $2^n$. This gives $2^{n-1}$ pairs. From another
hand, by Lemma~\ref{lema:13} each pair defines $k$ as a tangent of
the continuation $h$ of $\xi$ uniquely up to adding a number,
which is divisible by $2^n$.

Now if $\xi_1(x)=\xi_2(x)$ for continuable $\xi_1,\, \xi_2$, then
there exists a pair $\{ k_1,\, k_2\}$, mentioned above and this
pair defines $\xi_1$ and $\xi_2$ in the unique way at all points
of $A_n\setminus A_{n-1}$. Applying Lemma~\ref{lema:14} finishes
the proof.
\end{proof}

\begin{corollary}\label{corol:6}
For every $n\geq 1$ there are $2^{n-1}$ continuable self-semi
conjugations of $f$ of the form~(\ref{eq:37}).
\end{corollary}

\begin{proof}
By item 1 of Theorem~\ref{theor-main-dopust} for every $x\in
A_n\setminus A_{n-1}$ and for every $y\in A_n$ there exist a
continuable $\xi:\, A_n \rightarrow A_n$ such that $\xi(x)=y$.

If consider $x\in A_n\setminus A_{n-1}$ to be fixed, then by item
2 of Theorem~\ref{theor-main-dopust} every $y\in A_n$ defines a
continuable $\xi$ in the unique way.
\end{proof}

\newpage

\tableofcontents
\end{document}